\documentclass[reqno,11pt]{amsart}

\usepackage[top=2.0cm,bottom=2.0cm,left=3cm,right=3cm]{geometry}
\usepackage{amsthm,amsmath,amssymb,dsfont}
\usepackage{mathrsfs,amsfonts,functan,extarrows,mathtools}
\usepackage[colorlinks]{hyperref}

\usepackage{marginnote}
\usepackage{xcolor}
\usepackage{stmaryrd}
\usepackage{esint}
\usepackage{graphicx}
\usepackage{bbm}
\usepackage{wasysym}
\usepackage{textcomp}
\usepackage{float}
\usepackage{wasysym}

\usepackage{tikz}
\usetikzlibrary{calc,decorations.pathreplacing}

\newtheorem{theorem}{Theorem}[section]

\newtheorem{corollary}{Corollary}[section]
\newtheorem{remark}{Remark}[section]
\newtheorem{lemma}{Lemma}[section]

\numberwithin{equation}{section}
\allowdisplaybreaks

\def\d{\mathrm{d}}
\def\no{\nonumber}
\def\R{\mathbb{R}}
\def\eps{\varepsilon}

\def\u{\mathbf{u}}

\def\exp{\mathrm{exp}}
\def\M{\mathfrak{M}}
\def\A{\mathcal{A}}
\def\B{\mathcal{B}}

\def\u{\mathfrak{u}}
\def\L{\mathcal{L}}

\def\P{\mathcal{P}}
\def\I{\mathcal{I}}

\def\LL{[\![}
\def\RR{]\!]}
\def\n{{\mathfrak{n}_\gamma}}

\newcounter{wronumber}

\begin{document}
\title[Knudsen layer equation]
			{Knudsen boundary layer equations for full ranges of cutoff collision kernels: Maxwell reflection boundary with all accommodation coefficients in [0,1]}

\author[Ning Jiang]{Ning Jiang}
\address[Ning Jiang]{\newline School of Mathematics and Statistics, Wuhan University, Wuhan, 430072, P. R. China}
\email{njiang@whu.edu.cn}

\author[Yi-Long Luo]{Yi-Long Luo${}^*$}
\address[Yi-Long Luo]
{\newline School of Mathematics, South China University of Technology, Guangzhou, 510641, P. R. China}
\email{luoylmath@scut.edu.cn}

\author[Yulong Wu]{Yulong Wu}
\address[Yulong Wu]{\newline School of Mathematics and Statistics, Wuhan University, Wuhan, 430072, P. R. China}
\email{yulong\_wu@whu.edu.cn}

\thanks{${}^*$ Corresponding author \quad \today}

\maketitle

\begin{abstract}
   In this paper, we prove the existence and uniqueness of the Knudsen layer equation imposed on Maxwell reflection boundary condition with full ranges of cutoff collision kernels and accommodation coefficients (i.e., $- 3 < \gamma \leq 1$ and $0 \leq \alpha_* \leq 1$, respectively) in the $L^\infty_{x,v}$ framework. Moreover, the solution enjoys the exponential decay $\exp \{- c x^\frac{2}{3 - \gamma} - c |v|^2 \}$ for some $c > 0$. In order to study the general angular cutoff collision kernel $-3 < \gamma \leq 1$, we should introduce a $(x,v)$-mixed weight $\sigma$. The biggest difficulty in this paper is the nondissipative boundary condition, hence, the boundary temperature and velocity $(T_w, u_w)$ on $\{ x = 0 \}$ and $(T, \u)$ on $\{ x = + \infty \}$ do not guarantee the nonnegativity of the $L^2$ boundary energy. We also do not assume that $(T_w, u_w)$ and $(T, \u)$ are very closed to each other. We first derive the Nondissipative boundary lemma to pull the boundary energy to the interior weighted $L^2$ norms with higher power of $x$-polynomial weights. Then a so-called spatial-velocity indices iteration approach is developed to shift the higher power $x$-polynomial weights to $|v|$-polynomial weights. Finally, we construct an interleaved iteration process such that the boundary energy is successfully dominated. \\


\end{abstract}

%
%


\section{Introduction and main results}

\subsection{The description of the problem}

When studying the hydrodynamic limits of the scaled Boltzmann equation in the domain with boundary, there is an essential kinetic boundary layer equation, called the Knudsen layer equation, will be generated, see \cite{Aoki-2017-JSP,Guo-Huang-Wang-2021-ARMA,JLT-2021-arXiv,JLT-2021-arXiv-2,Sone-Book-2022,Sone-Book-2007} for instance. In this paper, we consider the Knudsen layer equation over $(x,v) \in \R_+ \times \R^3$ with the Maxwell boundary condition at $x = 0$:
\begin{equation}\tag{KL}\label{KL}
	\left\{
	  \begin{aligned}
	  	& v_3 \partial_x f + \L f = S \,, \\
	  	& f (0,v) |_{v_3 > 0} = (1 - \alpha_*) f (0,R_0 v) + \alpha_* \mathcal{D}_w f (0, v) + f_b ( v) \,, \ \lim_{x \to + \infty} f (x,v) = 0 \,,
	  \end{aligned}
	\right.
\end{equation}
where $S = S(x,v) \in \mathrm{Null}^\perp (\L)$, $f_b (v) = 0$ if $v_3 < 0$. At $x = 0$, we define an external normal vector $n (0) = (0,0,-1)$. $R_0 v = v - 2 (n(0) \cdot v) n(0) = (v_1, v_2, - v_3)$ represents the specular reflection at $x = 0$. $f (0,R_0 v)$ then characterizes the specular reflection effect at the boundary $x = 0$. The operator $\mathcal{D}_w f (0, v)$ is the diffusive reflection operator, which is defined by
\begin{equation}\label{Dw}{\small
	\begin{aligned}
		\mathcal{D}_w f (0, v) = \tfrac{M_w (v)}{\sqrt{\M (v)}} \int_{v_3' < 0} (- v_3') f (0, v') \sqrt{\M (v')} \d v' \,.
	\end{aligned}}
\end{equation}
Here $M_w (v)$ is the Maxwellian of the boundary with the expression
\begin{equation*}
	\begin{aligned}
		M_w (v) = \sqrt{\tfrac{2 \pi}{ T_w }} \tfrac{1}{ ( 2 \pi T_w )^\frac{3}{2} } \exp \big[ - \tfrac{|v - u_w|^2}{ 2 T_w } \big] \,,
	\end{aligned}
\end{equation*}
where $T_w > 0$ is the temperature of the boundary, and $u_w = (u_{w, 1}, u_{w, 2}, u_{w, 3}) \in \R^3$ with $u_{w, 3} = 0$ is the velocity of the boundary. Remark that $M_w (v)$ satisfies
\begin{equation*}
	\begin{aligned}
		\int_{v_3 > 0} v_3 M_w (v) \d v = 1 \,,
	\end{aligned}
\end{equation*}
which means that the particles absorbed by the boundary will be completely released by the way of Gaussian distribution. The constant $\alpha_* \in [0, 1]$ is the accommodation coefficient weighting the specular reflection effect and the diffusive reflection effect.

The symbol $\L$ is the linearized Boltzmann collision operator $\B (F, F)$ around the Maxwellian
\begin{equation*}
	\begin{aligned}
		\M (v) = \tfrac{\rho}{(2 \pi T)^{3/2}} \exp \big[ - \tfrac{|v - \u|^2}{2 T} \big] \,,
	\end{aligned}
\end{equation*}
where $\rho, T > 0$ and $\u = (\u_1, \u_2, \u_3) \in \R^3$ with $\u_3 = 0$ are all constant. We remark that $T > 0$ and $\u \in \R^3$ exactly are the temperature and velocity of the far-field boundary of the Knudsen layer equation. It can be intuitively observed in the nonlinear problem \eqref{KL-NL} later.

Furthermore, the Boltzmann collision operator $\mathcal{B} (F, F)$ is defined by
\begin{equation}
	\begin{aligned}
		\mathcal{B} (F, F) = \iint_{\mathbb{R}^3 \times \mathbb{S}^2} (F' F_*' - F F_*) b (\omega, v_* - v) \d \omega \d v_* \,.
	\end{aligned}
\end{equation}
Here $\omega \in \mathbb{S}^2$ is a unit vector, $\d \omega$ is the rotationally invariant surface integral on $\mathbb{S}^2$, while $F_*'$, $F'$, $F_*$ and $F$ are the number density $F(\cdot)$ evaluated at the velocities $v_*'$, $v'$, $v_*$ and $v$ respectively, i.e.,
\begin{equation*}
	\begin{aligned}
		F_*' = F(v_*') \,, \ F' = F(v') \,, \ F_* = F(v_*) \,, \ F = F(v) \,.
	\end{aligned}
\end{equation*}
Here $(v_*', v')$ are the velocities after an elastic binary collision between two molecules with velocities $(v_*, v)$ before the collision, or vice versa. Since both momentum and energy are conserved during the elastic collision, $v_*'$ and $v'$ can be expressed in terms of $v_*$ and $v$ as
\begin{equation}
	\begin{aligned}
		v_*' = v_* - [ (v_* - v) \cdot \omega ] \omega \,, \ v' = v + [ (v_* - v) \cdot \omega ] \omega \,,
	\end{aligned}
\end{equation}
where the unit vector $\omega \in \mathbb{S}^2$ is parallel to the deflections $v_*' - v_*$ and $v'-v$, and is therefore perpendicular to the plane of reflection. In the collision term $\mathcal{B} (F, F)$, $F_*'F'$ is the gain term, while $F_* F$ is called the loss term.

By Grad's work \cite{Grad-1963}, the collision kernel $b (\omega, v_* - v)$ has the factored form
\begin{equation}\label{b}
	\begin{aligned}
		b (\omega, v_* - v) = \tilde{b} (\theta) |v_* - v|^\gamma \,, \ \cos \theta = \tfrac{(v_* - v) \cdot \omega}{|v_* - v|} \,, \ -3 < \gamma \leq 1 \,,
	\end{aligned}
\end{equation}
where $\tilde{b} (\theta)$ satisfies the small deflection cutoff condition
\begin{equation}\label{Cutoff}
	\begin{aligned}
		\int_{\mathbb{S}^2} \tilde{b} (\theta) \d \omega = 1 \,, \ 0 \leq \tilde{b} (\theta) \leq \tilde{b}_0 |\cos \theta|
	\end{aligned}
\end{equation}
for some constant $\tilde{b}_0 > 0$. The cases $-3 < \gamma < 0$ and $0 \leq \gamma \leq 1$ are respectively referred to as the ``soft" and ``hard" potential cases. In particular, $\gamma = 0$ is the Maxwell potential case, and $\gamma = 1$ is the hard sphere case, in which $\tilde{b} (\theta) = \tilde{b}_0 |\cos \theta|$.

Considering the perturbation $F = \M + \sqrt{\M} f$ around $\M$, the linearized Boltzmann collision operator $\L$ is defined by
\begin{equation}\label{Lf}
	\begin{aligned}
		\L f = - \M^{- \frac{1}{2}} (v) \big[ \mathcal{B} (\M, \sqrt{\M} f) + \mathcal{B} ( \sqrt{\M} f , \M ) \big] = \nu (v) f (v) - K f (v) \,,
	\end{aligned}
\end{equation}
where the collision frequency $\nu (v)$ is defined by
\begin{equation}\label{nu}
	\begin{aligned}
		\nu (v) = \iint_{\mathbb{R}^3 \times \mathbb{S}^2} \M (v_*) b (\omega, v_* - v) \d \omega \d v_* \overset{\eqref{Cutoff}}{=} \int_{\R^3} |v_* - v|^\gamma \M (v_*) \d v_* \,,
	\end{aligned}
\end{equation}
and the operator $K f (v)$ can be decomposed into two parts:
\begin{equation}
	\begin{aligned}\label{K-K1-K2}
		K f (v) = - K_1 f (v) + K_2 f (v) \,.
	\end{aligned}
\end{equation}
Here the loss term $K_1 f (v)$ is
\begin{equation}\label{K1}
	\begin{aligned}
		K_1 f(v) = \M^\frac{1}{2} (v) \iint_{\mathbb{R}^3 \times \mathbb{S}^2} f (v_*) \M^\frac{1}{2} (v_*) b (\omega, v_* - v) \d \omega \d v_* \,,
	\end{aligned}
\end{equation}
and the gain term $K_2 f (v)$ is
\begin{equation}\label{K2}
	\begin{aligned}
		K_2 f(v) = & \M^\frac{1}{2} (v) \iint_{\mathbb{R}^3 \times \mathbb{S}^2} \big[ \M^{- \frac{1}{2}} (v') f (v') + \M^{- \frac{1}{2}} (v_*') f(v_*') \big] \M (v_*) b (\omega, v_* - v) \d \omega \d v_* \\
		= & \iint_{\mathbb{R}^3 \times \mathbb{S}^2} \big[ \M^{\frac{1}{2}} (v'_*) f (v') + \M^{\frac{1}{2}} (v') f(v_*') \big] \M^\frac{1}{2} (v_*) b (\omega, v_* - v) \d \omega \d v_* \,,
	\end{aligned}
\end{equation}
where the last equality is derived from the collisional invariant $\M (v) \M (v_*) = \M (v') \M (v_*')$. Remark that the collision frequency is invariant under the specular reflection operator, i.e.,
\begin{equation}\label{nu-rf}
	\begin{aligned}
		\nu (R_0 v) = \nu (v) \,.
	\end{aligned}
\end{equation}

The null space $\mathrm{Null} (\L)$ of the operator $\L$ is spanned by the basis
\begin{equation}\label{psi-i-bases}
	\begin{aligned}
		\psi_0 = \tfrac{1}{\sqrt{\rho}} \sqrt{\M} \,, \ \psi_i =  \tfrac{v_i - \u_i}{\sqrt{\rho T}} \sqrt{\M} \, (i = 1,2,3) \,, \ \psi_4 = \tfrac{1}{\sqrt{6 \rho}} \left( \tfrac{|v - \u|^2}{T} - 3 \right) \sqrt{\M} \,.
	\end{aligned}
\end{equation}
The basis is orthonormal in the Hilbert space $L^2 : = L^2 (\R^3)$. Namely, $\int_{\R^3} \psi_i \psi_j \d v = \delta_{ij}$ for $i,j= 0, 1,2,3,4$. Let $\mathrm{Null}^\perp (\L)$ be the orthogonal space of $\mathrm{Null} (\L)$ in $L^2$, namely,
\begin{equation*}
	\begin{aligned}
		\mathrm{Null} (\L) \oplus \mathrm{Null}^\perp (\L) = L^2 \,.
	\end{aligned}
\end{equation*}
We define the projection $\P : L^2 \to \mathrm{Null} (\L)$ by
\begin{equation}\label{Projt-P}
	\begin{aligned}
		\P f = \sum_{i = 0}^4 a_i (f) \psi_i \,, \ a_i (f) = \int_{\R^3} f \psi_i \d v \,, i = 0, 1, 2, 3, 4 \,.
	\end{aligned}
\end{equation}

Based on the projection $\P$, over the space $\mathrm{Null} (\L)$, we introduce the operator
\begin{equation}\label{A-operator}
	\begin{aligned}
		\A = \P v_3 \P \,,
	\end{aligned}
\end{equation}
which is the 5-dimensional linear bounded self adjoint operator. It is easy to find that $\A$ possesses the eigenvalues
\begin{equation}
	\lambda_i = 0 \ (i = 0,1,2) \,, \ \lambda_3 = \sqrt{\tfrac{5}{3} T} \,, \lambda_4 = - \sqrt{\tfrac{5}{3} T} \,.
\end{equation}
The corresponding unit eigenvectors are
\begin{equation}\label{psi*-basis}
	\begin{aligned}
		& \psi_0^* (v) = \tfrac{1}{\sqrt{\rho}} ( \tfrac{|v - \u|^2}{T} - 5 ) \sqrt{\M} \,, \ \psi_i^* (v) = \tfrac{v_i - \u_i}{\sqrt{\rho T}} \sqrt{\M} (i = 1,2) \,, \\
		& \psi_3^* (v) = \tfrac{1}{\sqrt{\rho}} ( \tfrac{|v - \u|^2}{T} + \sqrt{15} \tfrac{v_3}{\sqrt{T}} ) \sqrt{\M} \,, \ \psi_4^* (v) = \tfrac{1}{\sqrt{\rho}} ( \tfrac{|v - \u|^2}{T} - \sqrt{15} \tfrac{v_3}{\sqrt{T}} ) \sqrt{\M} \,,
	\end{aligned}
\end{equation}
which forms an orthonormal basis of the null space $\mathrm{Null} (\L)$. Then one can define the projection $\P^+$ from $\mathrm{Null} (\L)$ to the subspace associated with positive eigenvalues of $\A$ by
\begin{equation}\label{P+}
	\begin{aligned}
		\P^+ f = \int_{\R^3} \psi_3^* f \d v \psi_3^* \,.
	\end{aligned}
\end{equation}

We now introduce the Burnett functions $\mathbb{A} \in \R^{3 \times 3}$ and $\mathbb{B} \in \R^3$ with the entries
\begin{equation}\label{AB-1}
	\begin{aligned}
		& \mathbb{A}_{ij} = \big\{ \tfrac{ (v_i - \u_i) ( v_j - \u_j ) }{ T } - \delta_{ij} \tfrac{|v - \u|^2}{3 T} \big\} \sqrt{\M} \ (1 \leq i,j \leq 3) \,, \\
		& \mathbb{B}_i = \tfrac{v_i - \u_i}{2 \sqrt{T}} ( \tfrac{|v - \u|^2}{T} - 5 ) \sqrt{\M} \ (1 \leq i \leq 3) \,.
	\end{aligned}
\end{equation}
The entries above belonging to $\mathrm{Null}^\perp (\L)$ are all orthogonal each other. We also define
\begin{equation}\label{AB-2}
	\begin{aligned}
		\widehat{\mathbb{A}}_{ij} = \L^{-1} \mathbb{A}_{ij} \ (1 \leq i,j \leq 3) \,, \ \widehat{\mathbb{B}}_i = \L^{-1} \mathbb{B}_i \ (1 \leq i \leq 3) \,,
	\end{aligned}
\end{equation}
where the notation $\L^{-1}$ is the inverse of $\L$ restricted on the orthogonal space $\mathrm{Null}^\perp (\L)$, whose decay properties can be found in \cite{JLT-2022-arXiv}. As shown in Golse's work \cite{Golse-2008-BIMA}, we introduce the projection $\P^0$ by
\begin{equation}\label{P0}
	\begin{aligned}
		\P^0 f = \frac{ \int_{\R^3} v_3 \widehat{\mathbb{B}}_3 f \d v }{ \int_{\R^3} \widehat{\mathbb{B}}_3 \mathbb{B}_3 \d v } \psi_0^* + \frac{ \int_{\R^3} v_3 \widehat{\mathbb{A}}_{13} f \d v }{ \int_{\R^3} \widehat{\mathbb{A}}_{13} \mathbb{A}_{13} \d v } \psi_1^* + \frac{ \int_{\R^3} v_3 \widehat{\mathbb{A}}_{23} f \d v }{ \int_{\R^3} \widehat{\mathbb{A}}_{23} \mathbb{A}_{23} \d v } \psi_2^* \,,
	\end{aligned}
\end{equation}
which is actually a map from $L^2$ to the subspace of $\mathrm{Null} (\L)$ associated with the zero eigenvalues. Moreover, Golse \cite{Golse-2008-BIMA} also defined the operator $\mathbb{P}$ from $\mathrm{Null}^\perp (\L)$ to $\mathrm{Span} \{ \mathbb{A}_{12}, \mathbb{A}_{13}, \mathbb{B}_3 \} $ by
\begin{equation}\label{P-AB}
	\begin{aligned}
		\mathbb{P} f = \frac{ \int_{\R^3} \widehat{\mathbb{B}}_3 f \d v }{ \int_{\R^3} \widehat{\mathbb{B}}_3 \mathbb{B}_3 \d v } \mathbb{B}_3 + \frac{ \int_{\R^3} \widehat{\mathbb{A}}_{13} f \d v }{ \int_{\R^3} \widehat{\mathbb{A}}_{13} \mathbb{A}_{13} \d v } \mathbb{A}_{13} + \frac{ \int_{\R^3} \widehat{\mathbb{A}}_{23} f \d v }{ \int_{\R^3} \widehat{\mathbb{A}}_{23} \mathbb{A}_{23} \d v } \mathbb{A}_{23} \,.
	\end{aligned}
\end{equation}
By Lemma 3 of \cite{Golse-2008-BIMA}, the operators $\P^0$ and $\mathbb{P}$ enjoy the following properties:
\begin{itemize}
	\item $ \mathbb{P}^2 = \mathbb{P} $, $(\P^0)^2 = \P^0$;
	\item $\mathrm{Im} \mathbb{P} = \mathrm{Span} \{ \mathbb{A}_{13}, \mathbb{A}_{23}, \mathbb{B}_3 \} \subset \mathrm{Null}^\perp (\L)$, $\mathrm{Im} \P^0 = \mathrm{Span} \{ \psi_0^*, \psi_1^*, \psi_2^* \} \subset \mathrm{Null} (\L)$;
	\item $\mathbb{P} (v_3 f) = v_3 \P^0 f$;
	\item $\mathbb{P} (\L f) = 0$ if $v_3 f \perp \mathrm{Null} (\L)$;
	\item $\int_{\R^3} f \P^0 f \d v \geq 0$, and the map $f \mapsto ( \int_{\R^3} f \P^0 f \d v )^\frac{1}{2} $ defines a norm on $\mathrm{Span} \{ \psi_0^*, \psi_1^*, \psi_2^* \}$.
\end{itemize}

\subsection{Overdetermined far-field condition $\lim_{x \to + \infty} f (x, v) = 0$}

As inspired in Theorem 3.3 of \cite{Bardos-Caflisch-Nicolaenko-1986-CPAM}, the function $f (x,v)$ solving the problem
\begin{equation}\label{KL-d}
	\left\{
	  \begin{aligned}
	  	& v_3 \partial_x f + \L f = S \,, \ x > 0 \,, \\
	  	& f (0,v) |_{v_3 > 0} = (1 - \alpha_*) f (0,R_0 v) + \alpha_* \mathcal{D}_w f (0, v) + f_b ( v)
	  \end{aligned}
	\right.
\end{equation}
will enjoy the asymptotic behavior
\begin{equation}\label{Asymptotic-f}
	\begin{aligned}
		\lim_{x \to + \infty} f (x,v) = a_\infty \psi_0 (v) + b_{1 \infty} \psi_1 (v) + b_{2 \infty} \psi_2 (v) + b_{3 \infty} \psi_3 (v) + c_\infty \psi_4 (v)
	\end{aligned}
\end{equation}
for some constants $a_\infty$, $b_{1 \infty}$, $b_{2 \infty}$, $b_{3 \infty}$ and $c_\infty$, where the functions $\psi_i (v)$ are given in \eqref{psi-i-bases}. In other words, the far-field condition $\lim_{x \to + \infty} f (x, v) = 0$ in \eqref{KL} is overdetermined for general source terms $S \in \mathrm{Null}^\perp (\L)$ and $f_b$.

However, the far-field condition $\lim_{x \to + \infty} f (x, v) = 0$ is necessary in proving the hydrodynamic limits (including the compressible/incompressible limits) of the Boltzmann equation with Maxwell reflection boundary condition. In order to overcome the overdetermination of zero far-field condition, some further assumptions on the source terms $S \in \mathrm{Null}^\perp (\L)$ and $f_b$ are required. Now we introduce a so-called {\em vanishing sources set} $\mathbb{VSS}_{\alpha_*}$ defined by
\begin{equation}\label{ASS}
	\begin{aligned}
		\mathbb{VSS}_{\alpha_*} = \big\{ (S, f_b) ; S (x,v) \in \mathrm{Null}^\perp (\L) \textrm{ and } f_b (v) \textrm{ in \eqref{KL-d} such that } \lim_{x \to + \infty} f (x, v) = 0 \big\} \,.
	\end{aligned}
\end{equation}
As shown in \cite{Golse-Perthame-Sulem-1988-ARMA}, one knows that for $\alpha_* = 0$,
\begin{equation*}
	\begin{aligned}
		\mathbb{VSS}_0 = \Big\{ (S, f_b) ; S (x,v) \in \mathrm{Null}^\perp (\L) \,, \ \int_{v_3 > 0} v_3 \psi_i (v) f_b (v) \d v = 0 \,, \ i = 0, 1, 2, 4 \Big\} \neq \varnothing \,.
	\end{aligned}
\end{equation*}
For the case $0 < \alpha_* \leq 1$, by employing the isotropic properties of $\L$ (see Appendix B of \cite{Sone-Book-2022} and Appendix A.2 of \cite{Sone-Book-2007}), one can find some special expressions of $ S (x,v) $ and $f_b (v)$ belonging to the so-called vanishing sources set $\mathbb{VSS}_{\alpha_*}$. For example, special representations of $S$ and $f_b$ can be found in Problems (106)-(109), Section 5.2 of \cite{Aoki-2017-JSP}, and Problems (4.16)-(4.17), Section 4 of \cite{HJW-2024-preprint}. Therefore, by summarizing above arguments, we know that for $0 \leq \alpha_* \leq 1$,
\begin{equation}
	\begin{aligned}
		\mathbb{VSS}_{\alpha_*} \neq \varnothing \,.
	\end{aligned}
\end{equation}
Lately, the structure of the vanishing sources set $ \mathbb{VSS}_{\alpha_*} $ with codimension 4 will be clearly characterized. From now on, we will assume $(S, f_b) \in \mathbb{VSS}_{\alpha_*}$ such that the far-field condition $\lim_{x \to + \infty} f (x, v) = 0$ holds.

\subsection{Toolbox}

In this subsection, we will collect all notations, functions spaces and energy functionals that will be utilized in whole paper.

\subsubsection{Notations}

We employ the symbol $A \lesssim B$ meaning that $A \leq C B$ for some harmless constant $C > 0$. Moreover, $A \thicksim B$ denotes by $C_1 B \leq A \leq C_2 B$ for some harmless constants $C_1, C_2 > 0$. As inspired in \cite{Chen-Liu-Yang-2004-AA}, we introduce the following $(x,v)$-mixed weight
\begin{equation}\label{sigma}{\small
	\begin{aligned}
		\sigma (x,v) = 5 (\delta x + l)^\frac{2}{3 -\gamma} \Big[ 1 - \Upsilon \Big( \tfrac{\delta x + l}{( 1 + |v - \u| )^{3 - \gamma}} \Big) \Big] + \Big( \tfrac{\delta x + l}{(1 + |v - \u|)^{1 - \gamma}} + 3 |v - \u|^2 \Big) \Upsilon \Big( \tfrac{\delta x + l}{( 1 + |v - \u| )^{3 - \gamma}} \Big)
	\end{aligned}}
\end{equation}
for small $\delta > 0$ and large $l > 0$ to be determined, where the cutoff function $\Upsilon (\cdot)$ is defined in \eqref{Ups-Cut}. This weight can also be employed in the works \cite{Wang-Yang-Yang-2006-JMP,Wang-Yang-Yang-2007-JMP,Yang-2008-JMAA}. For any small $\hbar > 0$ to be determined, denote by
\begin{equation}
	\begin{aligned}
		f_\sigma : = e^{\hbar \sigma} f \ (\forall f = f (x,v)) \,.
	\end{aligned}
\end{equation}
We also introduce a weighted function
\begin{equation}\label{wv}
	\begin{aligned}
		w_{\beta, \vartheta} (v) = (1 + |v|)^\beta e^{\vartheta |v - \u|^2}
	\end{aligned}
\end{equation}
for the constants $\beta \in \R$ and $ \vartheta \geq 0 $. Then, for $\alpha \in \R$, we denote a weight function
\begin{equation}\label{z-alpha}
	z_\alpha (v) =
	\left\{
	\begin{array}{cl}
		|v_3|^\alpha \,, & |v_3| < 1 \,, \\
		1 \,, & |v_3| \geq 1 \,. \\
	\end{array}
	\right.
\end{equation}
We remark that the weights $\sigma (x,v)$ and $w_{\beta, \vartheta} (v)$ satisfy
\begin{equation}
	\begin{aligned}\label{sigma-w-inv}
		\sigma (x,R_U v) = \sigma (x,v) \,, \ w_{\beta, \vartheta} (R_U v) = w_{\beta, \vartheta} (v) \,, \ z_\alpha (R_U v) = z_\alpha (v) \textrm{ for } U = 0, A \,,
	\end{aligned}
\end{equation}
due to $\u_3 = 0$. Here $R_U v = v - 2 [ n (U) \cdot v ] n (U)$, $n (A) = (0,0,1)$. We now define the constant $\beta_\gamma$ is given by
\begin{equation}\label{beta-gamma}
	\begin{aligned}
		\beta_\gamma =
		\left\{
		\begin{aligned}
			0 \,, \quad & \textrm{if } 0 \leq \gamma \leq 1 \,, \\
			- \tfrac{\gamma}{2} \,, \quad & \textrm{if } - 3 < \gamma < 0 \,,
		\end{aligned}
		\right.
	\end{aligned}
\end{equation}
which will be utilized later while designing the various weighted norms.

\subsubsection{Functions spaces}

Based on the above weights, we introduce the spaces $X^\infty_{\beta, \vartheta} (A)$ and $Y^\infty_{m, \beta, \vartheta} (A)$ over $(x,v) \in {\Omega_A} \times \R^3$ for $m, \beta, \vartheta \in \R$ and $0 < A \leq \infty$ with the weighted norms
\begin{equation}\label{XY-space}{\small
	\begin{aligned}
		& \| f \|_{A; \beta, \vartheta} : = \| w_{\beta, \vartheta} f \|_{L^\infty_{x,v}} = \sup_{(x, v) \in \Omega_A \times \R^3} | w_{\beta, \vartheta} (v) f (x,v) | \,, \ \LL f \RR_{A; m, \beta, \vartheta} : = \| \sigma_x^\frac{m}{2} w_{\beta, \vartheta} f \|_{L^\infty_{x,v}} \,,
	\end{aligned}}
\end{equation}
respectively. On the boundary $\Sigma : = \partial {\Omega_A} \times \R^3 $, we introduce the spaces $L^\infty_\Sigma$, $X^\infty_{ \beta, \vartheta, \Sigma}$ and $Y^\infty_{m, \beta, \vartheta, \Sigma}$ endowed with the norms
\begin{equation}\label{XY-Sigma-space}
	\begin{aligned}
		& \| f \|_{L^\infty_\Sigma} : = \sup_{(x, v) \in \Sigma} |f (x,v)| \,, \ \| f \|_{ \beta, \vartheta, \Sigma} : = \| w_{\beta, \vartheta} f \|_{L^\infty_\Sigma} \,, \ \LL f \RR_{m, \beta, \vartheta, \Sigma} : = \| \sigma_x^\frac{m}{2} w_{\beta, \vartheta} f \|_{L^\infty_\Sigma} \,,
	\end{aligned}
\end{equation}
respectively. Let $L^p_x$ and $L^p_v$ with $1 \leq p \leq \infty$ be the standard $L^p$ space over $x \in \Omega_A$ and $v \in \R^3$, respectively. Moreover, the norm $\| \cdot \|_{L^r_v (L^p_x)}$ is defined by
\begin{equation*}
	\begin{aligned}
		\| \cdot \|_{L^r_v (L^p_x)} = \| \big( \| f (\cdot, v) \|_{L^x_p} \big) \|_{L^r_v} \,.
	\end{aligned}
\end{equation*}
We also introduce the space $L^\infty_x L^2_v$ endowed with the norm
\begin{equation*}
	\begin{aligned}
		\| f \|_{L^\infty_x L^2_v} = \sup_{x \in \Omega_A} \| f (x, \cdot) \|_{L^2_v} \,.
	\end{aligned}
\end{equation*}
Furthermore, the space $L^2_{x,v}$ over $\Omega_A \times \R^3$ endows with the norm
\begin{equation}
	\begin{aligned}
		\| f \|_A = \big( \iint_{\Omega_A \times \R^3} |f (x,v) |^2 \d v \d x \big)^\frac{1}{2} \,.
	\end{aligned}
\end{equation}

One introduces the out normal vector $n (0) = (0,0,-1)$ and $n (A) = (0,0,1)$ of the boundary $\partial \Omega_A$. Define
\begin{equation*}
	\begin{aligned}
		\Sigma_\pm = \{ (x, v) \in \Sigma ; \pm n (x) \cdot v > 0 \} \,.
	\end{aligned}
\end{equation*}
We then define the space ${L^2_{\Sigma_\pm}}$ over $(x,v) \in \Sigma_\pm$ as follows:
\begin{equation}\label{L2-Sigma+}{\footnotesize
	\begin{aligned}
		\| f \|^2_{{L^2_{\Sigma_\pm}}} = \| f \|^2_{{L^2_{\Sigma_\pm^0}}} + \| f \|^2_{{L^2_{\Sigma_\pm^A}}} \,, \ \| f \|^2_{{L^2_{\Sigma_\pm^0}}} = \int_{ \pm v_3 < 0} |f (0, v)|^2 \d v \,, \ \| f \|^2_{{L^2_{\Sigma_\pm^A}}} = \int_{\pm v_3 > 0} |f (A, v)|^2 \d v \,.
	\end{aligned}}
\end{equation}

\subsubsection{Energy functionals}\label{Subsubsec:EF}

We now introduce total energy functional as
\begin{equation}\label{Eg-lambda}
	\begin{aligned}
		\mathscr{E}^A (g) = \mathscr{E}_\infty^A (g) + \mathscr{E}_{\mathtt{cro}}^A (g) + \big[ \mathscr{E}_{\mathtt{NBE}}^A (g) \big]^\frac{1}{2} + \big[ \mathscr{E}_2^A (g) \big]^\frac{1}{2}
	\end{aligned}
\end{equation}
for $0 < A \leq \infty$. The weighted $L^\infty_{x,v}$ energy functional $\mathscr{E}_\infty^A (g)$ is defined as
\begin{equation}\label{E-infty}
	\begin{aligned}
		\mathscr{E}_\infty^A (g) = \LL g \RR_{A; m, \beta_*, \vartheta} \,.
	\end{aligned}
\end{equation}
The cross energy functional $ \mathscr{E}_{\mathtt{cro}}^A (g) $ connecting the $L^\infty_{x,v}$ and $L^2_{x,v}$ estimates is defined as
\begin{equation}\label{E-cro}
	\begin{aligned}
		\mathscr{E}_{\mathtt{cro}}^A (g) = \| \nu^\frac{1}{2} z_{- \alpha} \sigma_x^\frac{m}{2} w_{- \gamma + \beta_\gamma, \vartheta} g \|_A + \| \nu^{ - \frac{1}{2} } z_{- \alpha} \sigma_x^\frac{m}{2} z_1 w_{- \gamma + \beta_\gamma, \vartheta} \partial_x g \|_A \,.
	\end{aligned}
\end{equation}
The nondissipative boundary energy functional $\mathscr{E}_{\mathtt{NBE}}^A (g)$ is expressed by
\begin{equation}\label{E-NBE}{\small
	\begin{aligned}
		\mathscr{E}_{\mathtt{NBE}}^A (g) = & \Xi^A [ \tfrac{59 - 24 \gamma}{24 (3 - \gamma)}; \tfrac{15}{8 (3 - \gamma)}, 0 ] (g) + \Xi^A [ \tfrac{9 + \gamma}{3 (3 - \gamma)}; \tfrac{7}{8 (3 - \gamma)}, \beta_\gamma + \tfrac{1}{2} ] (g) + \Xi^A [ \tfrac{ 107 - 24 \gamma }{12 (3  - \gamma)}; \tfrac{7}{4 (3 - \gamma)}, 0 ] (g) \\
& + \Xi^A [ \tfrac{ 297 - 40 \gamma }{36 (3 - \gamma)}; \tfrac{3}{4 (3 - \gamma)}, \beta_\gamma + \tfrac{1}{2} ] (g) + \Xi^A [ \tfrac{189 - 52 \gamma}{36 ( 3 - \gamma )}; - \tfrac{1}{8 (3 - \gamma)}, 2 \beta_\gamma + 1 ] (g) \,,
	\end{aligned}}
\end{equation}
where the constant $\beta_\gamma$ is given in \eqref{beta-gamma}, and
\begin{equation}\label{Xi-A}{\footnotesize
	\begin{aligned}
		\Xi^A [ \mathbf{a}; \mathbf{b}, \mathbf{c} ] (g) : = & l^{ \mathbf{a} } \Big( \delta \hbar \| ( l^{- 6} \delta x + l )^\mathbf{b} ( \delta x + l )^{ - \frac{ 1 - \gamma }{ 2 ( 3 - \gamma ) } }  \P w_{\mathbf{c} , 0 } g \|_A^2 + \| ( l^{- 6} \delta x + l )^\mathbf{b} \nu^\frac{1}{2} \P^\perp w_{\mathbf{c} , 0 } g \|_A^2 \Big)
	\end{aligned}}
\end{equation}
with the weight $w_{\beta, \vartheta}$ introduced in \eqref{wv}. The weighted $L^2_{x,v}$ energy functional $\mathscr{E}_2^A (g)$ is represented by
\begin{equation}\label{E2-lambda}
	\begin{aligned}
		\mathscr{E}_2^A (g) = & {\delta \hbar} \| ( \delta x + l )^{ - \frac{1 - \gamma}{2 ( 3 - \gamma ) } } \P w_{\beta, \vartheta} g \|_A^2 + \| \nu^\frac{1}{2} \P^\perp w_{\beta, \vartheta} g \|_A^2 \,.
	\end{aligned}
\end{equation}

We further introduce the source energy functional $\mathscr{A}^A (h)$ and the boundary source energy functional $\mathscr{B} (\varphi)$ as follows:
\begin{equation}\label{Ah-lambda}
	\begin{aligned}
		& \mathscr{A}^A (h) = \mathscr{A}_\infty^A (h) + \mathscr{A}_{\mathtt{cro}}^A (h) + \big[ \mathscr{A}_{\mathtt{NBE}}^A (h) \big]^\frac{1}{2} + \big[ \mathscr{A}_2^A (h) \big]^\frac{1}{2} \,, \\
		& \mathscr{B} (\varphi) = \mathscr{B}_\infty (\varphi) + \mathscr{B}_{\mathtt{cro}} (\varphi) + \big[ \mathscr{B}_{\mathtt{NBE}} (\varphi) \big]^\frac{1}{2} + \big[ \mathscr{B}_2 (\varphi) \big]^\frac{1}{2} \,,
	\end{aligned}
\end{equation}
where
\begin{equation}\label{As-def}
	\begin{aligned}
		\mathscr{A}_\infty^A (h) = & \LL \nu^{-1} h \RR_{A; m, \beta_*, \vartheta} \,, \ \mathscr{A}_{\mathtt{cro}}^A (h) = \| \nu^{ - \frac{1}{2} } z_{- \alpha} \sigma_x^\frac{m}{2} w_{- \gamma + \beta_\gamma, \vartheta} h \|_A \,, \\
		\mathscr{A}_{\mathtt{NBE}}^A (h) = & \Lambda^A [ \tfrac{15}{8 (3 - \gamma)}, 0 ] (h_\sigma) + \Lambda^A [ \tfrac{7}{4 (3 - \gamma)}, 0 ] (h_\sigma) + \Lambda^A [ \tfrac{7}{8 (3 - \gamma)}, \beta_\gamma + \tfrac{1}{2} ] (h_\sigma) \\
		& + \Lambda^A [ \tfrac{3}{4 (3 - \gamma)}, \beta_\gamma + \tfrac{1}{2} ] (h_\sigma) + \Lambda^A [ - \tfrac{1}{8 (3 - \gamma)}, 2 \beta_\gamma + 1 ] (h_\sigma) \,, \\
		\mathscr{A}_2^A (h) = & (\delta \hbar )^{- 1 } \| ( \delta x + l )^{ \frac{1 - \gamma}{2 ( 3 - \gamma ) } } \P w_{\beta, \vartheta} h \|_A^2 + \| \nu^{ - \frac{1}{2} } \P^\perp w_{\beta, \vartheta} h \|_A^2 \,,
	\end{aligned}
\end{equation}
and
	\begin{align}\label{B-def}
		\no \mathscr{B}_\infty (\varphi) = & \| \sigma_x^\frac{m}{2} (A, \cdot) w_{\beta, \vartheta} \varphi \|_{L^\infty_v } \,, \ \mathscr{B}_2 (\varphi) = \| |v_3|^\frac{1}{2} w_{\beta, \vartheta} \varphi \|_{L^2_{ \Sigma_-^A }}^2 \,, \\ 
         \no \mathscr{B}_{\mathtt{cro}} (\varphi) = & \| |v_3|^\frac{1}{2} \sigma_x^\frac{m}{2} z_{- \alpha} w_{\beta + \beta_\gamma, \vartheta} \varphi \|_{L^2_{ \Sigma_-^A }} \,, \\
		\no \mathscr{B}_{\mathtt{NBE}} (\varphi) = & \Lambda_\star [ \tfrac{15}{8 (3 - \gamma)}, 0 ] (\varphi_{A, \sigma}) + \Lambda_\star [ \tfrac{7}{4 (3 - \gamma)}, 0 ] (\varphi_{A, \sigma}) + \Lambda_\star [ \tfrac{7}{8 (3 - \gamma)}, \beta_\gamma + \tfrac{1}{2} ] (\varphi_{A, \sigma}) \\
		& + \Lambda_\star [ \tfrac{3}{4 (3 - \gamma)}, \beta_\gamma + \tfrac{1}{2} ] (\varphi_{A, \sigma}) + \Lambda_\star [ - \tfrac{1}{8 (3 - \gamma)}, 2 \beta_\gamma + 1 ] (\varphi_{A, \sigma}) \,.
	\end{align}
Here
\begin{equation*}{\small
	\begin{aligned}
		\Lambda^A [ \mathbf{a}, \mathbf{b} ] (h_\sigma) : = (\delta \hbar)^{-1} \| ( l^{-6} \delta x + l )^{\mathbf{a}} ( \delta x + l )^\frac{ 1 - \gamma }{ 2 ( 3 - \gamma ) } \P w_{\mathbf{b} , 0 } h_\sigma \|_A^2 + \| ( l^{-6} \delta x + l )^{\mathbf{a}} \nu^{- \frac{1}{2}} \P^\perp w_{\mathbf{b} , 0 } h_\sigma \|_A^2 \,,
	\end{aligned}}
\end{equation*}
and
\begin{equation*}
	\begin{aligned}
		\Lambda_\star [ \mathbf{a}, \mathbf{b} ] (\varphi_{A, \sigma}) : = \| |v_3|^\frac{1}{2} ( l^{-6} \delta x + l )^\mathbf{a} w_{\mathbf{b}, 0} \varphi_{A, \sigma} \|^2_{L^2_{\Sigma_-^A}} \,.
	\end{aligned}
\end{equation*}
At the end, we define the following norm associated with the boundary source term $\widetilde{f}_b (v)$ as
\begin{equation}\label{fb-N}
	\begin{aligned}
		\| \widetilde{f}_b \|_{\mathfrak{N}} : = & \| w_{1 + \beta_* , \vartheta + 3 \hbar } \widetilde{f}_b \|_{L^\infty_v}  + \| \nu^\frac{1}{2} z_{- \alpha} w_{1 - \gamma + \beta_\gamma, \vartheta + 3 \hbar} \widetilde{f}_b \|_{L^2_v} + \| \P w_{1 + \beta, \vartheta + 3 \hbar} \widetilde{f}_b \|_{L^2_v} \\
		& + \| \nu^{- \frac{1}{2} } z_{1 - \alpha} w_{1 - \gamma + \beta_\gamma, \vartheta + 3 \hbar} \widetilde{f}_b \|_{L^2_v} + \| \nu^\frac{1}{2} \P^\perp w_{1 + \beta, \vartheta + 3 \hbar} \widetilde{f}_b \|_{L^2_v} \,.
	\end{aligned}
\end{equation}

\subsection{Main results}

\subsubsection{Linear problem \eqref{KL}}

We now state the existence result on the Knudsen layer equation \eqref{KL}. Before doing this, we first display the parameters assumptions occurred in $\mathscr{E}^A (g)$ (see \eqref{Eg-lambda}) and $ \mathscr{A}^A (h) $, $\mathscr{B} (\varphi)$ (see \eqref{Ah-lambda}) above.

{\bf Parameters Hypotheses (PH):} For $ \{ \gamma, \alpha_*, \beta, \beta_*, \alpha, m, \delta, \hbar, \vartheta, l, \rho, T, T_w, \u, u_w \} $,
we assume that
\begin{itemize}
	\item $- 3 < \gamma \leq 1$, $0 \leq \alpha_* \leq 1$, $m \geq 1$, $0 < \delta < 1$, integer $\beta_* \geq 0$;
	\item $\beta \geq 3 (\beta_\gamma + \frac{1}{2})$, where $\beta_\gamma$ is given in \eqref{beta-gamma};
	\item $0 < \alpha < \mu_\gamma$, where $\mu_\gamma > 0$ is given in Lemma \ref{Lmm-Kh-L2} below;
	\item $0 \leq \vartheta \leq \vartheta_0$, where small $\vartheta_0 > 0$ is given in Lemma \ref{Lmm-APE-A3} below;
	\item $0 < \hbar \leq o (1) \delta$ and $l \geq O (1) (\delta \hbar)^{ - \frac{512 (3 - \gamma)}{3 (3 + \gamma)} }$, where sufficiently small $0 < o (1) \ll 1$ and large enough $O (1) \gg 1$ are both independent of $\delta, \hbar, l$ ;
	\item $\rho > 0$, $\u, w_w \in \R^3$ with $\u_3 = u_{w 3} = 0$, and
	\begin{equation}\label{Assmp-T}
		\begin{aligned}
			0 < T_w < 2 T \,.
		\end{aligned}
	\end{equation}
\end{itemize}

For the source term, we introduce a functional space $\mathfrak{X}_\gamma^\infty$ by
\begin{equation}\label{X-gamma-infty}
  \begin{aligned}
    \mathfrak{X}_\gamma^\infty = \big\{ (S, f_b); S \in \mathrm{Null}^\perp (\L) , \mathscr{E}^\infty ( (\delta x + l)^\frac{1 - \gamma}{3 - \gamma} e^{\hbar \sigma} S) + \| {f}_b \|_{\mathfrak{N}} < \infty \big\} \,,
  \end{aligned}
\end{equation}
where the quantity $\mathscr{E}^\infty ( \cdot )$ is defined in \eqref{Eg-lambda} with $A = \infty$ and the norm $\| {f}_b \|_{\mathfrak{N}}$ is introduced in \eqref{fb-N}.

We now state the first result of this paper.

\begin{theorem}\label{Thm-Linear}
Assume that
\begin{itemize}
  \item the mixed weight $\sigma (x,v)$ be given in \eqref{sigma};
  \item the parameters $\{ \gamma, \alpha_*, \beta, \beta_*, \alpha, m, \delta, \hbar, \vartheta, l, \rho, T, T_w, \u, u_w \}$ satisfy the hypotheses {\bf (PH)};
  \item the source terms $(S, f_b) \in \mathbb{VSS}_{\alpha_*} \cap \mathfrak{X}_\gamma^\infty $, where $\mathbb{VSS}_{\alpha_*}$ is defined in \eqref{ASS}.
\end{itemize}	
Then the linear Knudsen layer equation \eqref{KL} admits a unique mild solution $f (x, v)$ satisfying
	\begin{equation}\label{Bnd-f}
		\begin{aligned}
			\mathscr{E}^\infty ( e^{ \hbar \sigma } f ) \leq C ( \mathscr{E}^\infty ( (\delta x + l)^\frac{1 - \gamma}{3 - \gamma} e^{\hbar \sigma} S) + \| {f}_b \|_{\mathfrak{N}} )
		\end{aligned}
	\end{equation}
	for some constant $C > 0$. Moreover, the set $\mathbb{VSS}_{\alpha_*} \cap \mathfrak{X}_\gamma^\infty$ is the subspace of $ \mathfrak{X}_\gamma^\infty $ with codimension $4$.
\end{theorem}

\begin{remark}[Exponential decay]
	Since $ \sigma (x,v) \geq c ( \delta x + l )^\frac{2}{3 - \gamma} $ for $- 3 < \gamma \leq 1$, one has
	\begin{equation*}
		\begin{aligned}
			\mathscr{E}^\infty ( e^{ \hbar \sigma } f ) \geq \| \sigma_x^\frac{m}{2} w_{\beta, \vartheta} e^{ \hbar \sigma } f \|_{L^\infty_{x,v}} \geq C' e^{ \frac{1}{2} c ( \delta x + l )^\frac{2}{3 - \gamma} } e^{ \frac{1}{2} \vartheta |v|^2 } | f (x,v) |
		\end{aligned}
	\end{equation*}
	uniformly in $(x,v) \in \R_+ \times \R^3$. Together with the bound \eqref{Bnd-f} in Theorem \ref{Thm-Linear}, the solution $f (x,v)$ to the problem \eqref{KL} enjoys the pointwise decay behavior
	\begin{equation}\label{Decay}
		\begin{aligned}
			| f (x,v) | \lesssim e^{ - \frac{1}{2} c ( \delta x + l )^\frac{2}{3 - \gamma} } e^{ - \frac{1}{2} \vartheta |v|^2 } \,.
		\end{aligned}
	\end{equation}
\end{remark}

\begin{remark}[Background temperature $T$ vs boundary temperature $T_w$]
	The assumption \eqref{Assmp-T}, i.e., $0 < T_w < 2 T$, is such that the factor $\frac{M_w (v)}{\sqrt{\M (v)}}$ in the diffusive operator $\mathcal{D}_w$ exponentially decays at $|v| \to + \infty$. Indeed,
	\begin{equation*}{\small
		\begin{aligned}
			\tfrac{M_w (v)}{\sqrt{\M (v)}} = & \tfrac{\sqrt{2 \pi} T^\frac{3}{2}}{ \rho
				T_w^2 } \exp \big\{ - ( \tfrac{1}{2 T_w} - \tfrac{1}{4 T} ) |v - u_w|^2 - \tfrac{(u_w - \u) \cdot (v - u_w)}{2 T} - \tfrac{|u_w - \u|^2}{4 T} \big\} \\
			= & \mathcal{O} \big( \exp \{ - ( \tfrac{1}{2 T_w} - \tfrac{1}{4 T} )^- |v - u_w|^2 \} \big) \to 0
		\end{aligned}}
	\end{equation*}
	as $|v| \to + \infty$ under the assumption \eqref{Assmp-T}. This reasonable assumption is consistent with the desired exponential decay \eqref{Decay} about the Maxwell reflection boundary condition $ f (0, v) |_{v_3 > 0} = (1 - \alpha_* ) f (0, R_0 v) + \alpha_* \tfrac{M_w (v)}{\sqrt{\M (v)}} \int_{v_3' < 0} ( - v_3' ) f ( 0, v' ) \sqrt{M} (v') \d v' + f_b (v) $. In other words, $ \tfrac{M_w (v)}{\sqrt{\M (v)}} $ should enjoy the same decay behavior of the quantity $ f (0, v) |_{v_3 > 0} - (1 - \alpha_* ) f (0, R_0 v) - f_b (v) $.
\end{remark}

\begin{remark}[Re-characterization of the vanishing sources set $\mathbb{VSS}_{\alpha_*}$]\label{Rmk-L-ASS}
  The structure of the vanishing sources set $\mathbb{VSS}_{\alpha_*} $ given in \eqref{ASS} is unclear. Actually, it can be characterized by a clearer way. Let $\mathbb{I}_\gamma (S, f_b) = f$ be the solution operator of the following damped problem:
  \begin{equation}\label{dd}{\small
    \begin{aligned}
      \left\{
      \begin{aligned}
        & v_3 \partial_x f_1 + \L f_1 + \mathbf{D} f_1 = (\mathbb{I} - \mathbb{P}) S \,, \ v_3 \partial_x f_2 = \mathbb{P} S \,, \\
        & f = f_1 + f_2 \,, f_1 = (\I - \P^0) f \,, f_2 = \P^0 f \,, \\
        & f (0, v) |_{v_3 > 0} = (1 - \alpha_*) f (0, R_0 v) + \alpha_* \mathcal{D}_w f (0, v) + f_b (v) \,, \\
        & \lim_{x \to + \infty} f_1 (x, v) = \lim_{x \to + \infty} f_2 (x, v) = 0 \,,
      \end{aligned}
      \right.
    \end{aligned}}
  \end{equation}
  where the damping operator $\mathbf{D}$ is given in \eqref{D-operator} later. The solution to \eqref{dd} is exactly that to \eqref{KL} if and only if $\mathbf{D} f_1 (x,v) = 0$ for all $x \geq 0$ and $v \in \R^3$, which is equivalent to \eqref{Solvb-1} (replacing $f_*$ by $f_1$) later. It is easy to see that the condition \eqref{Solvb-1} can be equivalently expressed by
  \begin{equation}\label{SBC}{\small
		\begin{aligned}
			\int_{\R^3}
			\left(
			\begin{array}{c}
				\psi_3^* \\
				\widehat{\mathbb{B}}_3 \\
				\widehat{\mathbb{A}}_{13}\\
				\widehat{\mathbb{A}}_{23}
			\end{array}
			\right) \Big( v_3 f (0,v) + \int_0^\infty S (z, v) \d z \Big) \d v = 0_4 \,.
		\end{aligned}}
	\end{equation}
  As a result, the set $\mathbb{VSS}_{\alpha_*} $ can be re-characterized by
  \begin{equation}\label{ASS-equv}
    \begin{aligned}
      \mathbb{VSS}_{\alpha_*} = \big\{ (S, f_b); \mathbb{I}_\gamma (S, f_b) = f \textrm{ such that \eqref{SBC} holds } \big\} \,.
    \end{aligned}
  \end{equation}
\end{remark}

\subsubsection{Application to nonlinear problem}

In this part, we will employ the linear theory constructed in Theorem \ref{Thm-Linear} to investigate the following nonlinear problem
\begin{equation}\tag{KL-NL}\label{KL-NL}{\small
	\begin{aligned}
		\left\{
		    \begin{aligned}
		    	& v_3 \partial_x F = \mathcal{B} (F, F) + H \,, \ x > 0, v \in \R^3 \,, \\
		    	& F (0, v) |_{v_3 > 0 } = ( 1 - \alpha_* ) F ( 0, R_0 v ) + \alpha_* M_w (v) \int_{v_3' < 0} ( - v_3' ) F (0, v') \d v' + F_b (v) \,, \\
		    	& \lim_{x \to + \infty} F (x,v) = \M (v) \,.		
	    	\end{aligned}
		\right.
	\end{aligned}}
\end{equation}
Obviously, $T_w > 0$ and $u_w \in \R^3$ are the temperature and velocity of the boundary $\{ x = 0 \}$. However, $T > 0$ and $\u \in \R^3$ are actually the temperature and velocity of the far-field boundary $\{ x = + \infty \}$. Let $f = \tfrac{F - \M}{\sqrt{\M}}$, $\widehat{h} = \tfrac{H}{\sqrt{\M}}$ and $\widehat{f}_b = \tfrac{F_b}{\sqrt{\M}}$. Then the problem \eqref{KL-NL} can be equivalently rewritten as
\begin{equation}\label{KL-NL-f}{\small
	\left\{
	    \begin{aligned}
	    	& v_3 \partial_x f + \L f = \Gamma ( f, f ) + \widehat{h} \,, \ x > 0 \,, v \in \R^3 \,, \\
	    	& f (0, v) |_{v_3 > 0} = (1 - \alpha_*) f (0, R_0 v) + \alpha_* \mathcal{D}_w f (0, v) + \widehat{f}_b (v) \,, \ \lim_{x \to + \infty} f (x,v) = 0 \,,
	    \end{aligned}
	\right.}
\end{equation}
where the nonlinear operator $ \Gamma (f, f) $ is defined as
\begin{equation}\label{Gamma-ff}
	\begin{aligned}
		\Gamma (f, f) = \tfrac{1}{\sqrt{\M}} \mathcal{B} ( f \sqrt{\M}, f \sqrt{\M} ) \,.
	\end{aligned}
\end{equation}
As the same as the linear problem \eqref{KL}, the far-field condition $ \lim_{x \to + \infty} f (x,v) = 0 $ in the nonlinear problem \eqref{KL-NL-f} is also overdetermined. We also introduce the vanishing sources set $\widetilde{\mathbb{VSS}}_{\alpha_*}$ by
\begin{equation}\label{ASS-NL}
  \begin{aligned}
    \widetilde{\mathbb{VSS}}_{\alpha_*} = \big\{ (\widehat{h}, \widehat{f}_b) ; \widehat{h} \in \mathrm{Null}^\perp (\L) \textrm{ and } \widehat{f}_b \textrm{ in \eqref{KL-NL-f} such that } \lim_{x \to + \infty} f (x,v) = 0 \big\} \,.
  \end{aligned}
\end{equation}

For the source term, we introduce the functional space $\mathfrak{X}_\gamma^{\varsigma_0}$ by
\begin{equation}\label{X-gamma-0}
  \begin{aligned}
    \mathfrak{X}_\gamma^{\varsigma_0} = \big\{ (\widehat{h}, \widehat{f}_b); \widehat{h} \in \mathrm{Null}^\perp (\L) , \varsigma := \mathscr{E}^\infty ( (\delta x + l)^\frac{1 - \gamma}{3 - \gamma} e^{\hbar \sigma} \widehat{h}) + \| \widehat{f}_b \|_{\mathfrak{N}} \leq \varsigma_0 \big\} 
  \end{aligned}
\end{equation}
for $\varsigma_0 > 0$, where the quantity $\mathscr{E}^\infty ( \cdot )$ is defined in \eqref{Eg-lambda} with $A = \infty$ and the norm $\| \cdot \|_{\mathfrak{N}}$ is introduced in \eqref{fb-N}.

More precisely, we can establish the following result.

\begin{theorem}\label{Thm-Nonlinear}
Assume that
\begin{itemize}
  \item the mixed weight $\sigma (x,v)$ be given in \eqref{sigma};
  \item the parameters $\{ \gamma, \alpha_*, \beta, \beta_*, \alpha, m, \delta, \hbar, \vartheta, l, \rho, T, T_w, \u, u_w \}$ satisfy the hypotheses {\bf (PH)};
  \item there is a small $\varsigma_0 > 0$ such that $(\tfrac{H}{\sqrt{\M}}, \tfrac{F_b}{\sqrt{\M}}) \in \widetilde{\mathbb{VSS}}_{\alpha_*} \cap \mathfrak{X}_\gamma^{\varsigma_0}$.
\end{itemize}
	Then the nonlinear problem \eqref{KL-NL} admits a unique solution $F (x, v)$ enjoying the bound
	\begin{equation}
		\begin{aligned}
			\mathscr{E}^\infty ( e^{ \hbar \sigma } \tfrac{F - \M}{\sqrt{\M}} ) \leq C \big[ \mathscr{E}^\infty ( (\delta x + l)^\frac{1 - \gamma}{3 - \gamma} e^{\hbar \sigma} \tfrac{H}{\sqrt{\M}}) + \| \tfrac{F_b}{\sqrt{\M}} \|_{\mathfrak{N}} \big]
		\end{aligned}
	\end{equation}
	for some constant $C > 0$, where the functionals $ \mathscr{E}^\infty ( \cdot ) $ is defined in \eqref{Eg-lambda} with $A = \infty$ and the norm $ \| \cdot \|_{ \mathfrak{N} } $ is given in \eqref{fb-N}. Moreover, the set $\widetilde{\mathbb{VSS}}_{\alpha_*} \cap \mathfrak{X}_\gamma^{\varsigma_0}$ is the subspace of $ \mathfrak{X}_\gamma^{\varsigma_0}$ with codimension $4$.
\end{theorem}

\begin{remark}[Re-characterization of  the vanishing sources set $\widetilde{\mathbb{VSS}}_{\alpha_*}$]\label{Rmk-NL-ASS}
  As the similar as the linear problem, let $\mathcal{I}_\gamma (\widehat{h}, \widehat{f}_b) = f$ be the solution operator of the following nonlinear damped problem:
  \begin{equation}\label{KL-NL-f-damp}{\small
  \left\{
  \begin{aligned}
    & v_3 \partial_x f_1 + \L f_1 + \mathbf{D} f_1 = (\mathbb{I} - \mathbb{P}) \Gamma (f, f) + (\mathbb{I} - \mathbb{P}) \widehat{h} \,, \ v_3 \partial_x f_2 = \mathbb{P} \Gamma (f, f) + \mathbb{P} \widehat{h} \,, f = f_1 + f_2 \,, \\
    & f (0, v) |_{v_3 > 0} = (1 - \alpha_*) f (0, R_0 v) + \alpha_* \mathcal{D}_w f (0, v) + \widehat{f}_b \,, \\ 
    & \lim_{x \to + \infty} f_1 (x,v) = \lim_{x \to + \infty} f_2 (x,v) = 0 \,,
  \end{aligned}
  \right.}
\end{equation}
where the artificial damping operator $\mathbf{D}$ is defined in \eqref{D-operator} later. The solution to \eqref{KL-NL-f-damp} is exactly that to \eqref{KL-NL-f} (equivalently \eqref{KL-NL}) if and only if $\mathbf{D} f_1 (x,v) = 0$ for all $x \geq 0$ and $v \in \R^3$, which is further equivalent to the relations \eqref{Solvb-NL} later. As a result,
\begin{equation*}{\small
  \begin{aligned}
    \widetilde{\mathbb{VSS}}_{\alpha_*} = \Big\{ & (\widehat{h}, \widehat{f}_b) ; \widehat{h} \in \mathrm{Null}^\perp (\L), \mathcal{I}_\gamma (\widehat{h}, \widehat{f}_b) = f \,, \\
    & \int_{\R^3}
    \left(
    \begin{array}{c}
      \psi_3^* \\
      \widehat{\mathbb{B}}_3 \\
      \widehat{\mathbb{A}}_{13} \\
      \widehat{\mathbb{A}}_{23}
    \end{array}
    \right) \Big( v_3 f (0, v) + \int_0^\infty \big[ \Gamma (f, f) + \widehat{h} \big] (z, v) \d z \Big) \d v = 0_4 \Big\} \,.
  \end{aligned}}
\end{equation*}
\end{remark}

\subsection{Outline of existence of the solutions to the system $\eqref{KL}$}\label{Subsec:OESKL}

We now sketch the rough process of solving the Knudsen layer equation \eqref{KL}.

First, by the definitions of the operators $\P^0$ in \eqref{P0} and $\mathbb{P}$ in \eqref{P-AB}, together with the relation $\mathbb{P} (v_3 f) = v_3 \P^0 f$ (see Lemma 3 of \cite{Golse-2008-BIMA}),the equation \eqref{KL} can be decomposed as
\begin{equation}\label{P0f-eq}{\small
	\begin{aligned}
		v_3 \partial_x \P^0 f = \mathbb{P} S \,, \ \lim_{x \to + \infty} \P^0 f = 0 \,,
	\end{aligned}}
\end{equation}
and, by employing the notation $f_* =  ( \I - \P^0 ) f$ for simplicity,
\begin{equation}\label{KL-P0}{\small
	\left\{
	  \begin{aligned}
	  	& v_3 \partial_x f_* + \L f_* = ( \mathbb{I} - \mathbb{P} ) S \,, \\
	  	& f_* (0, v) |_{v_3 > 0} = ( 1 - \alpha_* ) f_* (0, R_0 v) + \alpha_* \mathcal{D}_w f_* (0, v) + \widetilde{f}_b (v) \,, \ \lim_{x \to + \infty} f_* (x, v) = 0 \,,
	  \end{aligned}
	\right.}
\end{equation}
where
\begin{equation}\label{fb-tilde}
	\begin{aligned}
		\widetilde{f}_b (v) = f_b (v) - \P^0 f (0,v) \mathbf{1}_{v_3 > 0} + ( 1 - \alpha_* ) \P^0 f (0, R_0 v) \mathbf{1}_{v_3 > 0} + \alpha_* \mathcal{D}_w \P^0 f (0, v) \mathbf{1}_{v_3 > 0} \,.
	\end{aligned}
\end{equation}
The equation \eqref{P0f-eq} is actually an ordinary differential equation, which can be explicitly solved by
\begin{equation}\label{P0f-sol}{\small
	\begin{aligned}
		\P^0 f (x,v) = - \int_x^{+ \infty} \tfrac{1}{v_3} \mathbb{P} S (x', v) \d x' \,.
	\end{aligned}}
\end{equation}
Then the function $\widetilde{f}_b (v)$ is determined by the boundary source term $f_b (v)$ and part of source term $\mathbb{P} S$. Then we will focus on the problem \eqref{KL-P0} later.

Second, because the operator $\L$ is coercive merely on the orthogonal space $\mathrm{Null}^\perp (\L)$, we should add an artificial damping on the fluid part $\mathrm{Null} (\L)$ to the problem \eqref{KL-P0}. Namely, we consider the following damped problem
\begin{equation}\tag{KLd}\label{KL-Damped-f}
	\left\{	
	\begin{aligned}
		& v_3 \partial_x f_* + \L f_* + \mathbf{D} f_* = ( \mathbb{I} - \mathbb{P} ) S \,, \\
		& f_* (0,v) |_{v_3 > 0} = (1 - \alpha_*) f_* (0, R_0 v) + \alpha_* \mathcal{D}_w f_* (0, v) + \widetilde{f}_b (v) \,, \ \lim_{x \to + \infty} f_* (x,v) = 0 \,,
	\end{aligned}
	\right.
\end{equation}
where the artificial damping operator $\mathbf{D} f_*$ is defined as
\begin{equation}\label{D-operator}
	\begin{aligned}
		\mathbf{D} f_* = \ss_+ (\delta x + l)^{ - \frac{1 - \gamma}{3 - \gamma} } \P^+ ( v_3 f_* ) + \ss_0 ( \delta x + l )^{- \frac{1 - \gamma }{3 - \gamma }} \P^0 f_*
	\end{aligned}
\end{equation}
for all $f_* \in L^2_{x,v}$. Here the constants $\ss_+, \ss_0 > 0$ will be determined later. We remark that the artificial damping penalizes the decay of subspace of $\mathrm{Null} (\L)$ associated with the nonnegative eigenvalues of the operator $\A$ given in \eqref{A-operator}. Once one justifies the existence of the equation \eqref{KL-Damped-f}, the artificial damping will be removed by employing the structure of the vanishing sources set $\mathbb{VSS}_{\alpha_*}$ so that the existence of the problem \eqref{KL-P0} is established.

Third, we homogenize the Maxwell reflection boundary condition in \eqref{KL-Damped-f}. Let $\Upsilon : \R_+ \to [0,1]$ be a smooth monotone function satisfying
\begin{equation}\label{Ups-Cut}
	\begin{aligned}
		\Upsilon (x) =
		\left\{
		\begin{aligned}
			1 \quad & \textrm{for } \ 0 \leq x \leq 1 \\
			0 \quad & \textrm{for } \ x \geq 2 \,.
		\end{aligned}
		\right.
	\end{aligned}
\end{equation}
Denote by
\begin{equation}\label{gh-def}
	\begin{aligned}
		g (x,v) = & f_* (x,v) - \Upsilon (x) \widetilde{f}_b (v) \,, \\
		h (x,v) = & ( \mathbb{I} - \mathbb{P} ) S(x,v) - v_3 \partial_x \Upsilon (x) \widetilde{f}_b (v) - \Upsilon (x) ( \L + \mathbf{D} ) \widetilde{f}_b (v) \,.
	\end{aligned}
\end{equation}
It is easy to see that
\begin{equation*}
	\begin{aligned}
		v_3 \partial_x g + \L g + \mathbf{D} g = h \,.
	\end{aligned}
\end{equation*}
Recall that $f_b (v) = 0$ for $v_3 < 0$, which means that $\widetilde{f}_b (v) |_{v_3 < 0} = 0$ and $ \widetilde{f}_b (R_0 v) |_{v_3 > 0} = 0 $. Then
\begin{equation*}
	\begin{aligned}
		g (0, R_0 v) |_{v_3 > 0} = f_* (0, R_0 v) |_{v_3 > 0} - \Upsilon (0 ) \widetilde{f}_b ( R_0 v) |_{v_3 > 0} = f_* (0, R_0 v) |_{v_3 > 0} \,.
	\end{aligned}
\end{equation*}
Moreover, $g (0, v) |_{v_3 > 0} = f_* (0, v) |_{v_3 > 0} - \Upsilon (0) \widetilde{f}_b (v) |_{v_3 > 0} = f_* (0, v) - \widetilde{f}_b (v) ) |_{v_3 > 0}$, where the fact $\Upsilon (0) = 1$ has been used. It is further derived that
\begin{equation*}
	\begin{aligned}
		g (0, v') |_{v_3' < 0} = f_* (0, v') |_{v_3' < 0} - \Upsilon (0) \widetilde{f}_b (v') |_{v_3' < 0} = f_* (0, v') |_{v_3' < 0} \,,
	\end{aligned}
\end{equation*}
which infers that $\mathcal{D}_w g (0, v) = \mathcal{D}_w f_* (0, v)$. As a result, $g (x,v)$ satisfies the boundary condition
\begin{equation*}
	\begin{aligned}
		g (0, v) |_{v_3 > 0} = (1 - \alpha_*) g (0, R_0 v) + \alpha_* \mathcal{D}_w g (0, v) \,.
	\end{aligned}
\end{equation*}
Observing that $\Upsilon (x) = 0$ for $x \geq 2$, one has $ \lim_{x \to + \infty} g (x,v) = \lim_{x \to + \infty} f_* (x,v) = 0 $. Therefore, $g(x,v)$ satisfies
\begin{equation}\label{KL-Damped}
	\left\{	
	\begin{aligned}
		& v_3 \partial_x g + \L g + \mathbf{D} g = h \,, \\
		& g (0,v) |_{v_3 > 0} = (1 - \alpha_*) g (0, R_0 v) + \alpha_* \mathcal{D}_w g (0, v) \,, \ \lim_{x \to + \infty} g (x,v) = 0 \,,
	\end{aligned}
	\right.
\end{equation}
where $h (x,v)$ is defined in \eqref{gh-def}. So, the problem \eqref{KL-Damped} is the equivalent form of the damped problem \eqref{KL-Damped-f}.

Forth, in order to prove the existence of \eqref{KL-Damped}, we consider the problem in a finite slab with Maxwell reflection condition at $x = 0$ and incoming boundary condition at $x = A$. We call it the so-called connection auxiliary equation. Namely, for $(x, v) \in \Omega_A \times \R^3$ with $\Omega_A = \{ x; 0 < x < A \}$,
\begin{equation}\tag{CA-eq}\label{A1}
	\left\{	
	\begin{aligned}
		& v_3 \partial_x g + \L g + \mathbf{D} g = h \,, \\
		& g (0,v) |_{v_3 > 0} = (1 - \alpha_*) g (0, R_0 v) + \alpha_* \mathcal{D}_w g (0, v) \,, \ g (A, v) |_{v_3 < 0} = \varphi_A (v) \,.
	\end{aligned}
	\right.
\end{equation}
The incoming data $g (A, v) |_{v_3 < 0} = 0$ is the most closed approximation of the far-field condition $\lim_{x \to + \infty} g (x,v) = 0$ in \eqref{KL-Damped}. Here we write down the general $\varphi_A (v)$ will be used to prove the corresponding uniqueness results. We will prove the solution of the approximate problem \eqref{A1} converges to that of \eqref{KL-Damped} with $\varphi_A = 0$ as $A \to + \infty$. The existence of the problem \eqref{A1} will be shown by employing the Hanh-Banach Theorem and the Lax-Milgram Theorem.

\subsection{Methodology and novelties}\label{Subsec:MI}

  In this subsection, we methodologically sketch the proof of the main results and illustrate the novelties of this paper. As shown in the previous subsection, the core of current work is to derive the uniform estimates of the connection auxiliary equation \eqref{A1} associated with the parameter $A > 1$. The main ideas are displayed as follows.

  {\bf (I) Choice of the mixed weight $\sigma (x,v)$.} The major part of the equation \eqref{KL} reads
  	\begin{equation}\label{MajorPart}
  		\begin{aligned}
  			v_3 \partial_x f + \nu (v) f = s.o.t. \textrm{ (some other terms)} \,.
  		\end{aligned}
  	\end{equation}
  	For the hard sphere model $ \gamma = 1 $, $ v_3 $ and $ \nu (v) $ have the same order as $ |v| \gg 1 $, i.e. $ |v_3| \lesssim  \nu (v) $. Roughly speaking, the $ x $-decay can be expected from
  	\begin{equation*}
  		\begin{aligned}
  			\partial_x f + c f = s.o.t. \,,
  		\end{aligned}
  	\end{equation*}
  	which means that the expected $ x $-decay is $ e^{- c x} $. For example, for the hard sphere model $\gamma = 1$, there have been some related works as follows: Golse-Perthame-Sulem \cite{Golse-Perthame-Sulem-1988-ARMA} proved the exponential decay $e^{- c x}$ in the space $ L^2 ( e^{c x} \d x; L^2 ( (1 + |v|)^\frac{1}{2} \d v ) ) \cap L^\infty ( e^{cx} \d x ; L^2 (\d v) ) $ for specular reflection condition ($\alpha_* = 0$); Coron-Golse-Sulem \cite{Coron-Golse-Sulem-1988-CPAM} verified the exponential decay $e^{- c x}$ in $L^\infty ( e^{cx} \d x ; L^2 ( |v_3| \d v ) )$ for general Maxwell reflection boundary condition ($0 \leq \alpha_* \leq 1$); Huang-Wang \cite{HW-2022-SIMA} proved the exponential decay $ e^{- c x} e^{ - a |v|^2 } $ in $L^\infty_{x,v}$ space for diffusive reflection condition ($ \alpha_* = 1 $); Huang-Jiang-Wang \cite{HJW-2023-AMASES} also shown the exponential decay $ e^{- c x} e^{ - a |v|^2 } $ in $L^\infty_{x,v}$ space for specular reflection condition ($ \alpha_* = 0 $); He-Jiang-Wu \cite{HJW-2024-preprint} recently proved the exponential decay $ e^{- c x} e^{ - a |v|^2 } $ in $L^\infty_{x,v}$ space for Maxwell reflection boundary condition ($ 0 < \alpha_* \leq 1 $). All of previous works have essentially used the fact $|v_3| \lesssim \nu (v)$ which only holds for $\gamma = 1$.
  	
  	For the cases $ - 3 < \gamma < 1 $, $ |v_3| \lesssim \nu (v) |v_3|^{1 - \gamma} $. Note that \eqref{MajorPart} implies
  	\begin{equation*}
  		\begin{aligned}
  			\partial_x ( e^{\frac{\nu (v)}{v_3} x} f ) = s.o.t. \,.
  		\end{aligned}
  	\end{equation*}
  	This inspires us to introduce an $ (x,v) $-mixed weight $ \sigma (x,v) $ to deal with the power $ \frac{\nu (v)}{v_3} x $. More precisely, \eqref{MajorPart} reduces to
  	\begin{equation*}
  		\begin{aligned}
  			v_3 \partial_x ( e^{ \hbar \sigma } f ) + ( \nu (v) - \hbar v_3 \sigma_x ) e^{ \hbar \sigma } f = s.o.t. \,,
  		\end{aligned}
  	\end{equation*}
  	which inspires us to find a weight $\sigma (x,v)$ such that $v_3 \sigma_x (x,v) \thicksim \nu (v)$. As in Chen-Liu-Yang's work \cite{Chen-Liu-Yang-2004-AA}, the weight $ \sigma (x,v) $ in \eqref{sigma} is introduced, which satisfies
  	$$ |v_3| \sigma_x (x,v) \lesssim \nu (v) \,, \sigma (x,v) \geq c (\delta x + l)^\frac{2}{3 - \gamma} \,, \sigma_x (x,v) \lesssim (\delta x + l)^{- \frac{1-\gamma}{3-\gamma}} \,. $$
  	The derivative $\sigma_x$ of $ \sigma $ actually balances the disparity of $ |v_3| $ and $ \nu (v) $. We remark that the work \cite{Chen-Liu-Yang-2004-AA} investigated the Knudsen layer equation with the nondegenerate moving boundary condition and hard potential collision kernel $0 < \gamma \leq 1$. Here we, similarly as in \cite{Chen-Liu-Yang-2004-AA}, design the mixed weight $\sigma (x,v) $ in \eqref{sigma} to verify the exponential decays both $x$ and $v$ variables for all cases $- 3 < \gamma \leq 1$ and $0 \leq \alpha_* \leq 1$.
  	
  	The intuition idea of choosing the mixed weight $\sigma (x,v)$ can be illustrated as follows:\footnote{This intuition idea comes from the valuable discussion with Prof. Tong Yang.} We set
  	\begin{equation*}
  		\begin{aligned}
  			\sigma (x,v) = \underbrace{ 5 ( \delta x + l )^{\gamma_1} \big( 1 - \Upsilon ( \tfrac{\delta x + l}{( 1 + |v - \u| )^{\gamma_2} } ) \big) }_{\sigma_I} + \underbrace{ \big( \tfrac{\delta x + l}{ ( 1 + |v - \u| )^{\gamma_3} } + 3 |v - \u|^2 \big) \Upsilon ( \tfrac{\delta x + l}{( 1 + |v - \u| )^{\gamma_2} } ) }_{\sigma_{I\!I}} \,,
  		\end{aligned}
  	\end{equation*}
    where the indices $\gamma_1$, $\gamma_2$ and $\gamma_3$ will be determined later. The first part $\sigma_I$ depends only $x$ when $x \gg 1$ because we want to show decay like $e^{- \sigma} = e^{ - c ( \delta x + l )^{\gamma_1} }$. The factor $3 |v - \u|^2$ in the second part $\sigma_{I\!I}$ bears the best (Maxwellian type) decay in $v$ as $|v| \gg 1$. The factor $\tfrac{\delta x + l}{ ( 1 + |v - \u| )^{\gamma_3} }$ in the second part $\sigma_{I\!I}$ recovers the relation between $\sigma_x$ and $v$-polynomials, i.e., $\sigma_x \thicksim (1 + |v - \u|)^{- \gamma_3}$, which will be applied to balance the disparity of $|v_3|$ and $\nu (v)$. Let $g = e^{\sigma} f$. The equation \eqref{MajorPart} can be rewritten as
    \begin{equation}\label{MajorPart-g}
    	\begin{aligned}
    		v_3 \partial_x g + ( - v_3 \sigma_x + \nu (v) ) g = s.o.t. \,.
    	\end{aligned}
    \end{equation}
    Because higher order dissipation in $v$ than $\nu (v) \thicksim ( 1 + |v| )^\gamma$ can not be expected, we set $ v_3 \sigma_x \thicksim \nu (v) $, which means that $ \sigma_x \thicksim ( 1 + |v| )^{\gamma - 1} $. Together with $\sigma_x \thicksim (1 + |v - \u|)^{- \gamma_3}$, one has
    \begin{equation}\label{G3}
    	\begin{aligned}
    		- \gamma_3 = \gamma - 1 \,.
    	\end{aligned}
    \end{equation}
    In the transition regime $ \delta x + l \thicksim ( 1 + |v - \u| )^{\gamma_2}  $, one should have $ \tfrac{\delta x + l}{ ( 1 + |v - \u| )^{\gamma_3} } \thicksim |v - \u|^2 $. This indicates that
    \begin{equation}\label{G2}
    	\begin{aligned}
    		\gamma_2 - \gamma_3 = 2 \,.
    	\end{aligned}
    \end{equation}
    Moreover, the ODE \eqref{MajorPart-g} in the transition regime is like
    \begin{equation*}
    	\begin{aligned}
    		\partial_x g + ( 1 + |v - \u| )^{ \gamma - 1 } g \thicksim \partial_x g + ( \delta x + l )^\frac{\gamma - 1}{\gamma_2} g = s.o.t. \,,
    	\end{aligned}
    \end{equation*}
    which reveals that the decay rate in $x$ should be $g \thicksim e^{ - ( \delta x + l )^{ \frac{\gamma - 1}{\gamma_2} + 1 } }$. Consequently, one obtains
    \begin{equation}\label{G1}
    	\begin{aligned}
    		\gamma_1 = \tfrac{\gamma - 1}{\gamma_2} + 1 \,.
    	\end{aligned}
    \end{equation}
    Then the relations \eqref{G3}-\eqref{G2}-\eqref{G1} imply that $\gamma_1 = \frac{2}{3 - \gamma}$, $\gamma_2 = 3 - \gamma$ and $\gamma_3 = 1 - \gamma$. The mixed weight $\sigma (x,v)$ in \eqref{sigma} is therefore constructed.

  	{\bf (II) Artificial damping.} It is well-known that the linearized Boltzmann operator $\L$ does not supply the coercivity structure in the null space $\mathrm{Null} (\L)$, i.e., macroscopic damping effect. In the previous literature, various artificial damping quantities are introduced to obtain the macroscopic coercivity, and then remove additional damping quantities by proper ways. For example, in the works \cite{Coron-Golse-Sulem-1988-CPAM,Golse-Perthame-Sulem-1988-ARMA,HJW-2024-preprint,HJW-2023-AMASES,HW-2022-SIMA}, the damping $\eps f$ was introduced to deal with the linear Knudsen layer equation for hard sphere collision case with various boundary conditions. Then it is removed by taking $\eps \to 0$ via compact arguments. Moreover, in the works \cite{Chen-Liu-Yang-2004-AA,Wang-Yang-Yang-2006-JMP,Wang-Yang-Yang-2007-JMP,Yang-2008-JMAA} associated with the incoming data boundary condition, the artificial damping $- \gamma P_0^+ \xi_1 f$ ($\gamma > 0$) coming from the eigenspace corresponding to positive eigenvalues of the linear operator $P_0 \xi_1 P_0$ are applied to deal with the incoming boundary conditions. In order to return the original equation, a further assumption on the incoming data $a (\xi)$ (the data $a (\xi)$ vanishes in the eigenspace corresponding to positive eigenvalues of the linear operator $P_0 \xi_1 P_0 $, i.e., $ P_0^+ \xi_1 P_0 a = 0 $) will be imposed such that the artificial damping $- \gamma P_0^+ \xi_1 f$ vanishes by the Gr\"onwall inequality argument. Remark that the series of works \cite{Chen-Liu-Yang-2004-AA,Wang-Yang-Yang-2006-JMP,Wang-Yang-Yang-2007-JMP,Yang-2008-JMAA} considered the nondegenerate case, i.e., the far-field velocity such that the operator $P_0 \xi_1 P_0 $ admits no zero eigenvalue. For the corresponding degenerate case in the hard sphere model ($\gamma = 1$), Golse \cite{Golse-2008-BIMA} studied the $L^\infty$ decay theory of the Knudsen layer problem by adding a damping $\alpha \Pi_+ (v_1 \cdot) + \beta \mathbf{p} - \gamma v_1 I$ for some constants $\alpha, \beta, \gamma > 0$, where $\Pi_+$ is the orthogonal projection on the positive eigensubspace, $\mathbf{p}$ is that on the zero eigensubspace, and $- v_1 I$ supplies the damping effect on the negative eigensubspace.
  
Inspired by Golse's work, we add the artificial damping $\mathbf{D} f_*$ given in \eqref{D-operator} in the problem \eqref{KL-Damped-f}, i.e.,
\begin{equation*}
    \begin{aligned}
	    \mathbf{D} f_* = \ss_+ (\delta x + l)^{ - \frac{1 - \gamma}{3 - \gamma} } \P^+ ( v_3 f_* ) + \ss_0 ( \delta x + l )^{- \frac{1 - \gamma }{3 - \gamma }} \P^0 f_* \,.
	\end{aligned}
\end{equation*}	
Here $\P+$ and $\P^0$ supply the damping effect in the positive and zero eigensubspaces, respectively. What different with Golse's damping is that we do not add a damping effect on the negative eigensubspace. The reason is that there is an intrinsic damping mechanism in the negative eigensubspace, see \eqref{Neg-Damp} later, hence,
\begin{equation*}
  \begin{aligned}
    - \tfrac{10 \delta \hbar}{3 - \gamma} \int_{\R^3} v_3 ( a_4^* \psi_4^* )^2 \d v = \sqrt{\tfrac{3}{5 T}} \tfrac{10 \delta \hbar}{3 - \gamma} (a_4^*)^2 \int_{\R^3} v_3 \psi_4^* \P v_3 \P \psi_4^* \d v \geq \mu_2 \delta \hbar  (a_4^*)^2 \,,
  \end{aligned}
\end{equation*}	
where $\P v_3 \P \psi_4^* = - \sqrt{\frac{5}{3} T} \psi_4^* $, and $\psi_4^*$ is the negative unit eigenvector of the eigenvalue $\lambda_4 = - \sqrt{\frac{5}{3} T} < 0 $ of the operator $\P v_3 \P$ (see \eqref{psi*-basis}). For the precise application, we first rescale the equation \eqref{A1} by using the factor $e^{ \hbar \sigma }$, \eqref{A1} can be equivalently expressed by \eqref{A3-lambda} below, i.e.,
  	\begin{equation*}{\small
  		\begin{aligned}
  			v_3 \partial_x g + \mathbf{D}_\hbar g - \hbar \sigma_x v_3 g + \nu (v) g - \lambda K_\hbar g = h \,.
  		\end{aligned}}
  	\end{equation*}
  	The term $\nu (v) g - \lambda K_\hbar g $ supplies the coercivity structure $ \| \nu^\frac{1}{2} \P^\perp w_{\beta, \vartheta} g \|^2_{L^2_v} $ in $ \mathrm{Null}^\perp ( \L ) $. By using the property of $\sigma (x,v)$ and the damping mechanism of $\mathbf{D}$, the term $ \mathbf{D}_\hbar g - \hbar \sigma_x v_3 g $ implies the macroscopic damping $ \hbar ( \delta x + l )^{ - \frac{1 - \gamma}{3 - \gamma} } \| \P w_{\beta, \vartheta} g \|^2_{L^2_v} $ by losing a small microscopic quantity $ \hbar^\frac{1}{2} \| \nu^\frac{1}{2} \P^\perp w_{\beta, \vartheta} g \|^2_{L^2_v}$ (the smallness comes from the small parameter $\hbar > 0$). To be more precise, the macroscopic damping is constructed in Lemma \ref{Lmm-Dh} below, i.e.,
  	\begin{equation*}{\footnotesize
  		\begin{aligned}
  			\int_{\R^3} w^2_{\beta, \vartheta} ( \mathbf{D}_\hbar g - \hbar \sigma_x v_3 g ) \d v \geq \mu_0 \delta \hbar ( \delta x + l )^{ - \frac{1 - \gamma}{3 - \gamma} } \| \P w_{\beta, \vartheta} g \|^2_{L^2_v} - C \hbar^\frac{1}{2} \| \nu^\frac{1}{2} \P^\perp w_{\beta, \vartheta} g \|^2_{L^2_v} \,.
  		\end{aligned}}
  	\end{equation*}

    {\bf (III) Nondissipative boundary condition.} While multiplying the equation \eqref{KL-Damped} by $g (x,v)$ and integrating the resultant equation over $(x,v) \in \R_+ \times \R^3$, one can obtain the boundary $L^2$ energy
    \begin{equation*}{\small
    	\begin{aligned}
    		\mathcal{E}_{\!B\!C} := - \int_{\R^3} v_3 |g (0, v)|^2 \d v \,.
    	\end{aligned}}
    \end{equation*}
    If the boundary condition $g (0, v) |_{v_3 > 0} = \mathcal{K} g (0,v)$ is such that $ \mathcal{E}_{\!B\!C} \geq 0 $, we call that the boundary condition $g (0, v) |_{v_3 > 0} = \mathcal{K} g (0,v)$ is {\em dissipative}. Otherwise, we call that the corresponding boundary is {\em nondissipative}. In this paper, we consider
    \begin{equation*}{\small
    	\begin{aligned}
    		\mathcal{K} g (0,v) = ( 1 - \alpha_* ) g (0, R_0 v) + \alpha_* \tfrac{M_w (v) }{ \sqrt{\M (v)} } \int_{ v_3' < 0 } ( - v_3' ) g (0, v') \sqrt{\M (v')} \d v' \,.
    	\end{aligned}}
    \end{equation*}

    To the best knowledge of the authors, all the known results of the boundary value problem of Boltzmann equation studied the dissipative boundary or nearly dissipative boundary. For instance, Guo \cite{Guo-2010-ARMA} studied the Boltzmann equation in three dimensional bounded domain with in-flow, bounce-back, specular reflection ($\alpha_* = 0$) and diffusive reflection ($\alpha_* = 1$) boundary conditions associated with the global Maxwellian $M_w (v) = \M (v) = e^{ - \frac{|v|^2}{2} }$. We remark that the in-flow boundary condition responds to the case $ \mathcal{K} = 0 $ by homogenizing the boundary condition as the similar operations as in Subsection \ref{Subsec:OESKL} above. Li-Lu-Sun \cite{LLS-2017-MC} investigated one dimensional half-space linear steady Boltzmann equation with general dissipative boundary condition. Moreover, the all hard sphere model $\gamma = 1$ mentioned in Part (I) and (II) of this subsection studied the dissipative Maxwell type boundary conditions. Moreover, Esposito-Guo-Kim-Marra \cite{EGKM-2013-CMP} studied the steady Boltzmann equation of hard potential model ($0 \leq \gamma \leq 1$) with non-isothermal diffusive boundary condition ($\alpha_* = 1$) in the bounded spatial domain $\Omega \subset \R^3$, in which the boundary temperature $\theta = \theta (x)$ for $x \in \partial \Omega$. For general $\theta = \theta (x)$, the diffusive boundary condition is of course nondissipative. However, for a fixed constant $\theta_0 > 0$, there was an assumption $\sup_{x \in \partial \Omega} |\theta (x) - \theta_0| < \delta_0$ in \cite{EGKM-2013-CMP} for a sufficiently small $\delta_0 > 0$. It is actually a nearly dissipative boundary. The ideas in \cite{EGKM-2013-CMP} were to avoid the boundary $L^2$ energy $\mathcal{E}_{\!B\!C}$ by employing the trajectory approach. They first controlled the weighted $L^\infty$ norm. Then the required $L^2$ norm can be obtained by the $L^\infty$ estimates, see Lemma 3.1 of \cite{EGKM-2013-CMP}. However, this method essentially relied on the boundedness of the domain $\Omega \subset \R^3$. In our work, the spatial domain $\R_+$ is unbounded, so that the approach of \cite{EGKM-2013-CMP} do not work for our issue.

    We now focus on the Maxwell reflection boundary condition of \eqref{KL-Damped} considered in our paper. Actually, under the boundary condition in \eqref{KL-Damped},
    \begin{equation*}{\footnotesize
    	\begin{aligned}
    		\mathcal{E}_{\!B\!C} = & 2 \alpha_* ( 1 - \alpha_* ) \Big[ \int_{ v_3 < 0 } |v_3| g^2 (0, v) \d v - \int_{ v_3 < 0 } |v_3| \tfrac{M_w}{ \sqrt{\M} } g (0, v) \d v \int_{ v_3 < 0 } |v_3| g (0, v) \sqrt{\M (v)} \d v \Big] \\
    		& + \alpha_*^2 \Big[ \int_{ v_3 < 0 } |v_3| g^2 (0, v) \d v - \int_{ v_3 < 0 } |v_3| \tfrac{M_w^2}{ \M } \d v \big( \int_{ v_3 < 0 } |v_3| g (0, v) \sqrt{\M (v)} \d v \big)^2 \Big] \,.
    	\end{aligned}}
    \end{equation*}
    If $\alpha_* = 0$, $\mathcal{E}_{\!B\!C} = 0$, which means that the specular reflection boundary condition is dissipative. If $0 < \alpha_* \leq 1$, the fact $ \mathcal{E}_{\!B\!C} \geq 0 $ for any $ g (0, v) $ is equivalent to dissipative condition
    \begin{equation}\label{DBC}{\footnotesize
    	\begin{aligned}
    		\int_{ v_3 < 0 } |v_3| \tfrac{M_w^2}{\M} \d v \int_{ v_3 < 0 } |v_3| \M \d v \leq 1 \,.
    	\end{aligned}}
    \end{equation}
    However, the general parameters $(T, T_w, \u, u_w)$ given in the hypotheses {\bf(PH)} above do not meet the dissipative condition \eqref{DBC}. It is easy to check that \eqref{DBC} will hold by setting $T = T_w$ and $\u = u_w$. In the works \cite{Coron-Golse-Sulem-1988-CPAM,Golse-Perthame-Sulem-1988-ARMA,Guo-2010-ARMA,HJW-2024-preprint,HW-2022-SIMA}, the special case $( \u, T ) = (u_w, T_w) = (\vec{0}, 1) $, which subjected to the dissipative condition \eqref{DBC}, was considered.

    In our work, we focus on the general case WITHOUT the dissipative condition \eqref{DBC}. We only assume that $\{ T, \u, T_w, u_w \}$ satisfies the hypotheses given in {\bf (PH)}, hence,
    \begin{equation*}
      \begin{aligned}
        0 < T_w < 2 T \,, u_w, \u \in \R^3 \,, u_{w3} = \u_3 = 0 \,.
      \end{aligned}
    \end{equation*}
    So we do not require that the temperature and velocity of the boundaries $\{ x =0 \}$ and $\{ x = + \infty \}$ are closed each other. Our way will be divided into the following three steps.
    
    \underline{\em (III.1) Nondissipative boundary lemma.} The nondissipation of $\mathcal{E}_{\!B\!C}$ comes from the integral form $ \int_{ v_3 < 0 } |v_3| g (0, v) \mathfrak{w} (v) \d v $ for some exponential decay function $\mathfrak{w} (v)$. In order to dominate the previous integral form, in the so-called nondissipative boundary lemma (see Lemma \ref{Lmm-NDBL} below), we subtly construct a useful nondissipative boundary inequality
    \begin{equation*}{\footnotesize
    	\begin{aligned}
    		| \int_{v_3 < 0} |v_3| \mathfrak{w} (v) g_\sigma (0, v) \d v | \lesssim & \delta^{- \frac{1}{2}} l^{ - ( a - \frac{1 - \gamma}{2 (3 - \gamma)} ) } \| ( l^{- \mathbf{Z}} \delta x + l)^a (\delta x + l)^{- \frac{1 - \gamma}{2 (3 - \gamma)}} \nu^\frac{1}{2} g_\sigma \|_A \\
    & + \delta^{- \frac{1}{2}} l^{- 50} \| (\delta x + l)^{ - \frac{1 - \gamma}{2 (3 - \gamma)}} \nu^\frac{1}{2} g_\sigma \|_A + \cdots
    	\end{aligned}}
    \end{equation*}
    for any $a > 0$ and $\mathbf{Z} \geq 0$. Once $l \geq 1$ is large enough such that $ \delta^{- \frac{1}{2}} l^{- 50} $ is sufficiently small, the quantity $\delta^{- \frac{1}{2}} l^{- 50} \| (\delta x + l)^{ - \frac{1 - \gamma}{2 (3 - \gamma)}} \nu^\frac{1}{2} g_\sigma \|_A$ can be easily dominated by the coercivity of $\L$ and $\mathbf{D}$. The degree of freedom for the constants $a > 0$ and $\mathbf{Z} \geq 0$ will play an essential role in controlling the quantity $\delta^{- \frac{1}{2}} l^{ - ( a - \frac{1 - \gamma}{2 (3 - \gamma)} ) } \| ( l^{- \mathbf{Z}} \delta x + l)^a (\delta x + l)^{- \frac{1 - \gamma}{2 (3 - \gamma)}} \nu^\frac{1}{2} g_\sigma \|_A$. The idea is to pull the boundary integral $ \int_{v_3 < 0} |v_3| \mathfrak{w} (v) g_\sigma (0, v) \d v $ to the interior of the equation by the Euler factor $e^{- (\delta x + l)}$. When $|v - \u| \geq 2$, the integral can be bounded by the second quantity with small factor $\delta^{- \frac{1}{2} } l^{-50}$, where the smallness comes from the interaction of the $x$- and $v$-variables, see the estimate of $U_{11}$ in \eqref{U11} later. If $|v - \u| \leq 2$, the integral can be bounded by the first quantity with small factor $\delta^{- \frac{1}{2}} l^{ - ( a - \frac{1 - \gamma}{2 (3 - \gamma)} ) }$. The key is to use the relation $\mathbf{1}_{|v - \u| \leq 2} |v - \u| \leq 2 ( l^{- \mathbf{Z} - 1} \delta x + 1 )^a$, which also guarantees the degree of freedom of the parameters $\mathbf{Z} \geq 0$ and $a > 0$.
    
    By utilizing the above inequality, we can establish the Boundary energy lemma (see Lemma \ref{Lmm-BEL} below). Namely, the boundary energy $\mathcal{E}_{\!B\!C}$ with a certain weight $w_* (v)$ has a lower bound
    \begin{equation*}{\footnotesize
    	\begin{aligned}
    		\mathcal{E}_{\!B\!C} \geq & [1 - (1 - \alpha_*)^2] \int_{ v_3 < 0 } |v_3| w_* (v) g^2_\sigma (0, v) \d v \\
    & - C_0 \alpha_* \delta^{- 1} l^{ - 2 ( a - \frac{1 - \gamma}{2 ( 3 - \gamma ) } ) } \| ( l^{- \mathbf{Z}} \delta x + l)^a (\delta x + l)^{- \frac{1 - \gamma}{2 (3 - \gamma)}} \nu^\frac{1}{2} g_\sigma \|_A^2 - \alpha_* \times ( \textrm{some controllable quantities} ) 
    	\end{aligned}}
    \end{equation*}
    for any $a > 0$ and $\mathbf{Z} \geq 0$. Then the weighted $L^2$ estimates implies
    \begin{equation*}
      \begin{aligned}
        \mathscr{E}_2^A (g_\sigma) \lesssim \delta^{- 1} l^{ - 2 ( a - \frac{1 - \gamma}{2 ( 3 - \gamma ) } ) } \| ( l^{- \mathbf{Z}} \delta x + l)^a (\delta x + l)^{- \frac{1 - \gamma}{2 (3 - \gamma)}} \nu^\frac{1}{2} g_\sigma \|_A^2 + \cdots \,.
      \end{aligned}
    \end{equation*}
    
    \underline{\em (III.2) Spatial-velocity indices iteration approach.} In  the quantity 
    \begin{equation}\label{QD}
      \begin{aligned}
        \delta^{- 1} l^{ - 2 ( a - \frac{1 - \gamma}{2 ( 3 - \gamma ) } ) } \| ( l^{- \mathbf{Z}} \delta x + l)^a (\delta x + l)^{- \frac{1 - \gamma}{2 (3 - \gamma)}} \nu^\frac{1}{2} g_\sigma \|_A^2 \,,
      \end{aligned}
    \end{equation}
    there is an additional spatial polynomial weight $ ( l^{- \mathbf{Z}} \delta x + l)^a $ with positive power $a > 0$. In order to overcome this difficulty, we develop a so-called {\em spatial-velocity indices iteration approach}. The ideas is to shift the spatial polynomial weight to the velocity one in the $L^2_{x,v}$ framework. As in Lemma \ref{Lmm-SVII} below, together with the Nondissipative boundary lemma, the following spatial-velocity indices iteration form is found:
    \begin{equation}\label{SVIIF}{\small
    	\begin{aligned}
    		\psi_{\mathbf{Z}; \mathtt{q}, \beta_\sharp } \lesssim & l^{ - \frac{2}{3 - \gamma} \mathbf{Z} } \psi_{\mathbf{Z}; \mathtt{q} - \frac{1}{3 - \gamma}, \beta_\sharp + \beta_\gamma + \frac{1}{2} } + \alpha_* l^{ 2 \mathtt{q} - (2 a - \frac{1 - \gamma}{3 - \gamma}) } \psi_{\mathbf{Z}; a, 0 } + \textrm{ some controllable terms} \,,
    	\end{aligned}}
    \end{equation}
    for any $a > 0$ and $\mathbf{Z} \geq 0$, where $ \psi_{\mathbf{Z}; \mathtt{q}, \beta_\sharp } $ is defined in \eqref{psi-y}, i.e.,
    $$ \psi_{ \mathbf{Z}; \mathtt{q}, \beta_\sharp} : = \delta \hbar \| ( l^{- \mathbf{Z}} \delta x + l )^{ \mathtt{q} } (\delta x + l)^{ - \frac{1 - \gamma}{2 (3 - \gamma)} } \P w_{\beta_\sharp, 0} g_\sigma \|^2_A + \| ( l^{- \mathbf{Z}} \delta x + l )^\mathtt{q} \nu^\frac{1}{2} \P^\perp w_{\beta_\sharp, 0} g_\sigma \|^2_A \,. $$ 
    Here, compared with Lemma \ref{Lmm-SVII}, we omit the constant factors about the parameters $\delta, \hbar$, due to the factors on the parameter $l$ is dominant. The above iteration form means that once an $x$-polynomial weight $( l^{- \mathbf{Z}} \delta x + l )^{ \frac{1}{3 - \gamma} }$ is reduced, a $|v|$-polynomial weight $ (1 + |v|)^{ \beta_\gamma + \frac{1}{2} } $ will be increased. In the quantity \eqref{QD}, we take $a = \frac{15}{8 (3 - \gamma)} > 0$ and $\mathbf{Z} = 6$ (it will be always taken as 6 in our proof). Then the quantity $\mathscr{E}_2^A (g_\sigma)$ can be bounded by
    \begin{equation*}
      \begin{aligned}
        \mathscr{E}_2^A (g_\sigma) \lesssim l^{- \frac{125}{24 (3 - \gamma)}} \big[ l^\frac{59 - 24 \gamma}{24 (3 - \gamma)} \psi_{ 6; \frac{15}{8 (3 - \gamma)}, 0} \big] + \cdots  \lesssim  l^{- \frac{3 (3 + \gamma)}{32 (3 - \gamma)}} \Xi^A [ \tfrac{59 - 24 \gamma}{24 (3 - \gamma)}; \tfrac{15}{8 (3 - \gamma)}, 0 ] (g_\sigma) + \cdots
      \end{aligned}
    \end{equation*}
    for $- 3 < \gamma \leq 1$, where we have used the relation $l^a \psi_{6; b, c } = \Xi^A [a; b, c] (g_\sigma) $ defined in \eqref{Xi-A}.

    \underline{\em (III.3) Interleaved iteration process.} Together with the spatial-velocity indices iteration form \eqref{SVIIF}, we will apply the so-called {\em interleaved iteration process} to dominate the quantity $\Xi^A [ \tfrac{59 - 24 \gamma}{24 (3 - \gamma)}; \tfrac{15}{8 (3 - \gamma)}, 0 ] (g_\sigma)$. The process can be intuitively expressed by the following Figure \ref{Fig-IIP}.
    \begin{figure}[h]
      \begin{center}
        \begin{tikzpicture}
          \draw (0,0)--(1.6,0)--(1.6,-0.7)--(0,-0.7)--cycle;
          \draw (0.8, -0.35) node{\small $\mathscr{E}_2^A (g_\sigma)$};
          \draw (-1.2,2.5)--(2.8,2.5)--(2.8,3.2)--(-1.2,3.2)--cycle;
          \draw (0.8,2.85) node{\small $ \Xi^A [ \frac{59 - 24 \gamma}{24 (3 - \gamma)}; \frac{15}{8 (3 - \gamma)}, 0 ] (g_\sigma) $};
          \draw (0.2,-2.7)--(-4.8,-2.7)--(-4.8,-3.4)--(0.2,-3.4)--cycle;
          \draw (-2.3,-3.05) node{\scriptsize $ \Xi^A [ \frac{189 - 52 \gamma}{36 (3 - \gamma)}; - \frac{1}{8 (3 - \gamma)}, 2 (\beta_\gamma + \frac{1}{2}) ] (g_\sigma) $ };
          \draw (2,-2.7)--(6.4,-2.7)--(6.4,-3.4)--(2,-3.4)--cycle;
          \draw (4.2,-3.05) node{\scriptsize $ \Xi^A [ \frac{297 - 40 \gamma}{36 (3 - \gamma)}; \frac{3}{4 (3 - \gamma)}, \beta_\gamma + \frac{1}{2} ] (g_\sigma) $};
          \draw (-2.6,1.2)--(-7,1.2)--(-7,0.5)--(-2.6,0.5)--cycle;
          \draw (-4.8,0.85) node{\scriptsize $ \Xi^A [ \frac{9 + \gamma}{3 (3 - \gamma)}; \frac{7}{8 (3 - \gamma)}, \beta_\gamma + \frac{1}{2} ] (g_\sigma) $};
          \draw (4.2,0.5)--(4.2,1.2)--(8.2,1.2)--(8.2,0.5)--cycle;
          \draw (6.2,0.85) node{\scriptsize $ \Xi^A [ \frac{107 - 24 \gamma}{12 (3 - \gamma)}; \frac{7}{4 (3 - \gamma)}, 0 ] (g_\sigma) $};
          \draw[thin,<-] (1,2.47)--(1,0);
          \draw (1,1.23) node[right]{\tiny $l^{- \frac{3 (3 + \gamma)}{32 (3 - \gamma)}}$};
          \draw[thin,<-,color=red] (0.6,0.03)--(0.6,2.5);
          \draw (0.6,1.24) node[left=-1pt,color=red]{\tiny $l^{- \frac{3 (3 + \gamma)}{32 (3 - \gamma)}}$};
          \draw[thin,->,color=red] (2.8,2.85)--(6.2,1.23);
          \draw (4.5,2.04) node[left=1pt,color=red]{\tiny $l^{- \frac{3 (3 + \gamma)}{32 (3 - \gamma)}}$};
          \draw[thin,->,color=red] (-1.2,2.85)--(-4.8,1.23);
          \draw (-3,2.04) node[left=1pt,color=red]{\tiny $l^{- \frac{3 (3 + \gamma)}{32 (3 - \gamma)}}$};
          \draw[thin,->,color=blue] (-4,1.2)--(-1.2,2.47);
          \draw (-2.6,1.8) node[right=2pt,color=blue]{\tiny $\epsilon_*^2$};
          \draw[thin,->,color=blue] (-2.6,0.5)--(-0.03,0);
          \draw (-0.6,0) node[left,color=blue]{\tiny $l^{- \frac{3 (3 + \gamma)}{32 (3 - \gamma)}}$};
          \draw[thin,->,color=blue] (-5.9,0.5)--(-5.9,-0.5)--(-3.7,-0.5)--(-3.7,0.47);
          \draw (-4.8,-0.5) node[above,color=blue]{\tiny $C_{\epsilon_*} l^{- \frac{3 (3 + \gamma)}{32 (3 - \gamma)}}$};
          \draw[thin,->,color=purple] (-2.3,-2.7)--(0,-0.73);
          \draw (-1.15,-1.7) node[left=-1pt,color=purple]{\tiny $l^{- \frac{3 (3 + \gamma)}{32 (3 - \gamma)}}$};
          \draw[thin,->,color=purple] (0.2,-3.05)--(1.97,-3.05);
          \draw (1.1,-3.05) node[above=-1.5pt,color=purple]{\tiny $l^{- \frac{3 (3 + \gamma)}{32 (3 - \gamma)}}$};
          \draw[thin,->,color=green] (2.4,-2.7)--(1.6,-0.73);
          \draw (2,-1.7) node[right=-1pt,color=green]{\tiny $l^{- \frac{3 (3 + \gamma)}{32 (3 - \gamma)}}$};
          \draw[thin,->,color=green] (3.4,-2.7)--(5.7,0.47);
          \draw (4.55,-1.11) node[left=-1pt,color=green]{\tiny $\epsilon_*^2$};
          \draw[thin,->,color=green] (6.2,-2.7)--(6.2,-2)--(4.5,-2)--(4.5,-2.67);
          \draw (5.4,-2.7) node[above,color=green]{\tiny $C_{\epsilon_*} l^{- \frac{3 (3 + \gamma)}{32 (3 - \gamma)}}$};
          \draw[thin,->,color=orange] (6.2,0.5)--(3.9,-2.67);
          \draw (5.1,-0.9) node[right=-1pt,color=orange]{\tiny $l^{- \frac{3 (3 + \gamma)}{32 (3 - \gamma)}} + \epsilon_*^2$};
          \draw[thin,->,color=orange] (4.2,0.5)--(1.6,0.03);
          \draw (2.9,0.25) node[above,color=orange]{\tiny $l^{- \frac{3 (3 + \gamma)}{32 (3 - \gamma)}}$};
          \draw[thin,->,color=orange] (5.2,0.5)--(0.2,-2.67);
          \draw (4,-0.4) node[left=-1pt,color=orange]{\tiny $C_{\epsilon_*} l^{- \frac{3 (3 + \gamma)}{32 (3 - \gamma)}}$};
          \draw (-4.5,-3.8)--(-5,-3.8)--(-5,-4.2)--(-4.5,-4.2)--cycle;
          \draw (-4.75,-4) node{$B_1$};
          \draw (-4.5,-4.8)--(-5,-4.8)--(-5,-5.2)--(-4.5,-5.2)--cycle;
          \draw (-4.75,-5) node{$B_n$};
          \draw (-7,-4.3)--(-7,-4.7)--(-6.5,-4.7)--(-6.5,-4.3)--cycle;
          \draw (-6.75,-4.5) node{$A$};
          \draw (-4.75,-4.42) node{$\vdots$};
          \draw[thin,->] (-6.5,-4.4)--(-5.03,-4);
          \draw (-5.7,-3.99) node{$\kappa_1$};
          \draw[thin,->] (-6.5,-4.6)--(-5.03,-5);
          \draw (-5.7,-4.7) node{$\kappa_n$};
          \draw (0.7, -4.42) node{means $A \lesssim \kappa_1 A_1 + \cdots + \kappa_n A_n +$ some controllable terms.};
        \end{tikzpicture}
      \end{center}
      \caption{Interleaved iteration process.}\label{Fig-IIP}
    \end{figure}
    The process displayed in Figure \ref{Fig-IIP} illustrates the estimates \eqref{U0-d}-\eqref{RHS-d} later. The quantity $ \Xi^A [ \frac{59 - 24 \gamma}{24 (3 - \gamma)}; \frac{15}{8 (3 - \gamma)}, 0 ] (g_\sigma) $ is dominated by directly employing the spatial-velocity indices iteration form \eqref{SVIIF} and the bound $\psi_{6; - \frac{1}{8 (3 - \gamma)}, 2 (\beta_\gamma + \frac{1}{2})}$, see \eqref{U1-d}. When controlling the quantity $ \Xi^A [ \frac{9 + \gamma}{3 (3 - \gamma)}; \frac{7}{8 (3 - \gamma)}, \beta_\gamma + \frac{1}{2} ] (g_\sigma) $ in \eqref{U2-d}, the spatial-velocity indices iteration form \eqref{SVIIF} and the following interpolation inequality are utilized:
    \begin{equation*}{\small
      \begin{aligned}
        l^\frac{15 - 4 \gamma}{12 (3 - \gamma)} ( l^{-6} \delta x + l )^\frac{13}{8 (3 - \gamma)} \leq \epsilon_* l^\frac{59 - 24 \gamma}{48 (3 - \gamma)} ( l^{-6} \delta x + l )^\frac{15}{8 (3 - \gamma)} + C_{\epsilon_*} l^{- \frac{3}{16 (3 - \gamma)}} l^\frac{9 + \gamma}{3 (3 - \gamma)} ( l^{- 6} \delta x + l)^\frac{7}{8 (3 - \gamma)} \,.
      \end{aligned}}
    \end{equation*}
    Similarly, by applying the bound $\psi_{6; - \frac{1}{4 (3 - \gamma)}, 2 (\beta_\gamma + \frac{1}{2})} \leq \mathscr{E}_2^A (g_\sigma)$, the spatial-velocity indices iteration form \eqref{SVIIF}, and the interpolation inequality
    \begin{equation*}{\small
      \begin{aligned}
        l^\frac{279 - 40 \gamma}{72 (3 - \gamma)} ( l^{-6} \delta x + l )^\frac{3}{2 (3 - \gamma)} \leq \epsilon_* l^\frac{107 - 24 \gamma}{24 (3 - \gamma)} ( l^{-6} \delta x + l )^\frac{7}{4 (3 - \gamma)} + C_{\epsilon_*} l^{- \frac{2 (3 + \gamma)}{3 (3 - \gamma)}} l^\frac{297 - 40 \gamma}{72 (3 - \gamma)} ( l^{- 6} \delta x + l)^\frac{3}{4 (3 - \gamma)} \,,
      \end{aligned}}
    \end{equation*}
    the estimate \eqref{U4-d} of the quantity $ \Xi^A [ \frac{297 - 40 \gamma}{36 (3 - \gamma)}; \frac{3}{4 (3 - \gamma)}, \beta_\gamma + \frac{1}{2} ] (g_\sigma) $ can be obtained. Moreover, the estimate \eqref{U3-d} of $ \Xi^A [ \frac{107 - 24 \gamma}{12 (3 - \gamma)}; \frac{7}{4 (3 - \gamma)}, 0 ] (g_\sigma) $ can be derived from the spatial-velocity indices iteration form \eqref{SVIIF} and the interpolation inequality
    \begin{equation*}{\small
      \begin{aligned}
        l^\frac{149 - 36 \gamma}{24 (3 - \gamma)} ( l^{-6} \delta x + l )^\frac{1}{2 (3 - \gamma)} \leq \epsilon_* l^\frac{297 - 40 \gamma}{72 (3 - \gamma)} ( l^{-6} \delta x + l )^\frac{3}{4 (3 - \gamma)} + C_{\epsilon_*} l^{- \frac{2 (3 + \gamma)}{9 (3 - \gamma)}} l^\frac{189 - 52 \gamma}{72 (3 - \gamma)} ( l^{- 6} \delta x + l)^{- \frac{1}{8 (3 - \gamma)}} \,.
      \end{aligned}}
    \end{equation*}
    Finally, by the bound $\psi_{6; - \frac{7}{8 (3 - \gamma)}, 3 (\beta_\gamma + \frac{1}{2})} \leq \mathscr{E}_2^A (g_\sigma)$ for $\beta \geq 3 (\beta_\gamma + \frac{1}{2})$ and the spatial-velocity indices iteration form \eqref{SVIIF}, the estimate \eqref{U5-d} of $ \Xi^A [ \frac{189 - 52 \gamma}{36 (3 - \gamma)}; - \frac{1}{8 (3 - \gamma)}, 2 (\beta_\gamma + \frac{1}{2}) ] (g_\sigma) $ is gained. Once taking $\epsilon_* > 0$ sufficiently small and then choosing $l \geq 1$ large enough, one can obtain the uniform bound of the quantity $\mathscr{E}_2^A (g_\sigma) + \mathscr{E}_{\mathtt{NBE}}^A (g_\sigma)$, hence, the closed uniform weighted $L^2_{x,v}$ estimates in Lemma \ref{Lmm-L2xv-closed} holds.

    {\bf (IV) Designing the uniform norms of \eqref{A1}.} Now we illustrate the process of deriving the uniform bounds of the problem \eqref{A3-lambda} below (equivalently \eqref{A1}). Due to the complication of deriving the uniform bounds, the following sketch map will be initially drawn for the sake of readers' intuition (see Figure \ref{Fig1} below).

    \begin{figure}[h]
    	\begin{center}
    		\begin{tikzpicture}
    			\draw (0.8,0)--(4,0)--(4,-0.8)--(0.8,-0.8)--cycle;
    			\draw (0.75,-0.4) node[right]{\footnotesize $\textcolor{red}{ \mathscr{E}^A_{\infty} (g_\sigma) } = \LL g_\sigma \RR_{A; m, \beta, \vartheta}$};
    			\draw[->] (2.5,-0.9)--(2.5,-1.9);
    			\draw (2.4,-1.4) node[right]{\footnotesize Lemma \ref{Lmm-Y-bnd}};
    			\draw (0.7,-2)--(4.1,-2)--(4.1,-2.8)--(0.7,-2.8)--cycle;
    			\draw (0.65,-2.4) node[right]{\footnotesize $\| z_{\alpha'} \sigma_x^\frac{m}{2} w_{- \gamma, \vartheta} g_\sigma \|_{L^\infty_x L^2_v}$};
    			\draw[->] (2.5,-2.9)--(2.5,-3.9);
    			\draw (2.4,-3.4) node[right]{\footnotesize Lemma \ref{Lmm-Linfty-L2}};
    			\draw (0.1,-4)--(4.9,-4)--(4.9,-5.2)--(0.1,-5.2)--cycle;
    			\draw (0,-4.3) node[right]{{\footnotesize $ \textcolor{red}{ \mathscr{E}^A_{\mathtt{cro}} (g_\sigma) } = \| \nu^\frac{1}{2} z_{- \alpha} \sigma_x^\frac{m}{2} w_{-\gamma, \vartheta} g_\sigma \|_A $ }};
    			\draw (0.45,-4.9) node[right]{{\footnotesize $ + \| \nu^{- \frac{1}{2} } z_{- \alpha} \sigma_x^\frac{m}{2} z_1 w_{-\gamma, \vartheta} \partial_x g_\sigma \|_A $}};
                \draw[->] (5,-4.6)--(6.9,-4.6);
    			\draw (6.3,-4.6) node[above]{\footnotesize Lemma \ref{Lmm-L2xv-alpha}};
    			\draw (6.2,-3) node[below]{\footnotesize Lemma \ref{Lmm-NDBL}};
                \draw (6,-3.8) node{$+$};
    			\draw (8,-5.2)--(12.95,-5.2)--(12.95,-4)--(8,-4)--cycle;
                \draw (7.9,-4.6) node[right]{\footnotesize$ \| (\delta x + l)^{- \frac{m (1 - \gamma)}{2 (3 - \gamma)} } \nu^\frac{1}{2} w_{ - \gamma + \beta_\gamma, \vartheta } g_\sigma  \|_A $};
    			\draw (8,-1)--(11.2,-1)--(11.2,0)--(8,0)--cycle;
                \draw (8,-0.5) node[right]{\footnotesize $[ \mathscr{E}_2^A (g_\sigma) + \mathscr{E}_{\mathtt{NBE}}^A (g_\sigma) ]^\frac{1}{2}$};
                \draw[->] (9.6,-3.9)--(9.6,-1.1);
    			\draw (9.6,-2.5) node[right]{\footnotesize \eqref{X2'}};
    			\draw[->] (4.5,-3.9)--(7.9,-0.5);
                \draw (7.6,-1.7) node[left,color=blue]{\tiny $(\delta \hbar)^{- \frac{15}{2}} l^{ - \frac{3 (3 + \gamma)}{64 (3 - \gamma)} }$};
    		\end{tikzpicture}
    	\end{center}
        \caption{Derivation of uniform bounds for the connection auxiliary equation \eqref{A1}. Here we denote by $g_\sigma = e^{\hbar \sigma} g$.}\label{Fig1}
    \end{figure}
    Based on the Figure \ref{Fig1}, we now illustrate the main ideas. We mainly want to control the weighted $L^\infty_{x,v}$ quantity $\mathscr{E}^A_\infty (g_\sigma)$. However, the operator $K$ is not compact in the weighted $L^\infty_{x,v}$ spaces, which fails to obtain a closed estimate in the $L^\infty_{x,v}$ framework. By applying the property of $K$ in Lemma \ref{Lmm-Kh-2-infty}, the quantity $\mathscr{E}^A_\infty (g_\sigma)$ can be bounded by the norm $ \| z_{\alpha'} \sigma_x^\frac{m}{2} w_{- \gamma, \vartheta} g_\sigma \|_{L^\infty_x L^2_v} $, see Lemma \ref{Lmm-Y-bnd}. In this step, the $L^\infty_{x,v}$ bounds for the operators $Y_A, Z, U$ defined in \eqref{YAn-f}-\eqref{Rn-f}-\eqref{U-f} are important, see Lemma \ref{Lmm-ARU}. By Lemma \ref{Lmm-Linfty-L2}, the quantity $ \| z_{\alpha'} \sigma_x^\frac{m}{2} w_{- \gamma, \vartheta} g_\sigma \|_{L^\infty_x L^2_v} $ is thereby dominated by $ \mathscr{E}^A_{\mathtt{cro}} (g_\sigma) $. It is actually the Sobolev type interpolation in one dimensional space with different weights. Due to the structure of the equation, the singular weight $z_{- \alpha} (v) $ is unavoidable.
    	
    By Lemma \ref{Lmm-L2xv-alpha}, the quantity $ \mathscr{E}^A_{\mathtt{cro}} (g_\sigma) $ can be bounded by 
    $$ \| (\delta x + l)^{ - \frac{m ( 1 - \gamma ) }{ 2 ( 3 - \gamma ) } } \nu^\frac{1}{2} w_{ - \gamma + \beta_\gamma - \mathbf{1}_{m < 0} m (1 - \gamma) /2 , \vartheta } g_\sigma \|_A + \alpha_* (\delta \hbar)^{- \frac{15}{2}} l^{- \frac{3 (3 + \gamma)}{64 (3 - \gamma)}} \big[ \mathscr{E}_2^A (g_\sigma) + \mathscr{E}_{\mathtt{NBE}}^A (g_\sigma) \big]^\frac{1}{2} \,,$$
    where the first term can be further dominated by $ [ \mathscr{E}_2^A (g_\sigma) + \mathscr{E}_{\mathtt{NBE}}^A (g_\sigma) ]^\frac{1}{2} $ if $m \geq 1$, see \eqref{X2'}. Note that the quantity $ \mathscr{E}^A_{\mathtt{cro}} (g_\sigma) $ is considered in the $L^2_{x,v}$ framework. The Nondissipative boundary lemma (Lemma \ref{Lmm-NDBL}) is therefore required to deal with the weighted boundary energy as stated in Part (III) above. Then the second quantity above is obtained. The main goal of this step is to control the $(x,v)$-mixed polynomial type weight $\sigma_x^\frac{m}{2} (x,v)$ and the singular weight $z_{- \alpha} (v)$ involved in $ \mathscr{E}^A_{\mathtt{cro}} (g_\sigma) $. The $(x,v)$-mixed weight $\sigma_x^\frac{m}{2} (x,v)$ can be controlled by $x$-polynomial and $|v|$-polynomial weights. The difficult is to deal with the singular weight $z_{- \alpha} (v)$ associated with the operator $K$. Thanks to Lemma \ref{Lmm-Kh-L2}, the weight $z_{- \alpha} (v)$ is successfully removed, in which the key is to obtain the estimate $\mathsf{h} (v_*) = \int_{\R^3} z_{- \alpha}^2 (v) |v-v_*|^\gamma \M^\frac{3}{8} (v) \d v \lesssim (1 + |v_*|)^\gamma$, see \eqref{h-bnd} below. At the end, the quantity $\mathscr{E}_2^A (g_\sigma) + \mathscr{E}_{\mathtt{NBE}}^A (g_\sigma)$ has been successfully controlled as in Part (III). Therefore, we obtain the uniform a priori estimates for \eqref{A1} in Lemma \ref{Lmm-APE-A3}.

    {\bf (V) Existence of the linear damped problem \eqref{KL-Damped}.} Based on the uniform a priori weighted $L^\infty_{x,v} \cap L^2_{x,v}$ estimates on the connection auxiliary problem \eqref{A1} in Lemma \ref{Lmm-APE-A3}, we sketch the proof of existence and uniqueness of the mild solution to the linear damped problem \eqref{KL-Damped} (equivalently \eqref{KL-Damped-f}).

    Initially, together with the uniform a priori $L^2_{x,v}$ estimates on the connection auxiliary problem \eqref{A1} Lemma \ref{Lmm-L2xv-closed}, the existence and uniqueness of weak solution to the approximate problem \eqref{A1} can be prove by employing the well-known Hahn-Banach Theorem and Lax-Milgram Theorem. Inspired by Lemma II.2 of \cite{DiPerna-Lions-1989-AM}, one knows that the previous weak solution to \eqref{A1} is exactly the mild solution with the form \eqref{A3-3}.

    Then we can justify the existence and uniqueness of the mild solution to \eqref{KL-Damped} by using Lemma \ref{Lmm-APE-A3}. By taking $\varphi_A (v) = 0$ and assuming $ \mathscr{A}^\infty ( e^{\hbar \sigma} h ) < \infty $, the solution $g^A (x,v)$ constructed in Lemma \ref{Lmm-APE-A3} obeys the uniform-in-$A$ bound $\mathscr{E}^A ( e^{ \hbar \sigma } g^A ) \lesssim \mathscr{A}^\infty ( e^{\hbar \sigma} h )$. Note that $g^A$ is defined on $(x,v) \in (0,A) \times \R^3$. We extend $g^A (x,v)$ to $(x,v) \in \R_+ \times \R^3$ as $\tilde{g}^A (x,v) = \mathbf{1}_{x \in  (0, A)} g^A (x,v)$, which, together the bound of $g^A$, is uniformly bounded in the space $\mathbb{B}^\hbar_\infty$ defined in \eqref{Bh-1}. Moreover, we can prove that $\tilde{g}^A (x,v)$ is a Cauchy sequence in $\mathbb{B}^{\hbar'}_\infty$ for any fixed $\hbar' \in (0, \hbar)$. Then we can show that the limit of $ \tilde{g}^A (x,v) $ is the unique solution to \eqref{KL-Damped}. The result on Theorem \ref{Thm-Linear} is thereby obtained.
    
    {\bf (VI) Remove the artificial damping: Freezing Point Method.}
    
    In order to prove the existence of mild solution to the linear problem \eqref{KL}, i.e., to prove Theorem \ref{Thm-Linear}, we should remove the artificial damping $\mathbf{D} f_*$, which means that
    \begin{equation*}
      \begin{aligned}
        \mathfrak{f} (x) = \int_{\R^3} v_3 \psi_3^* f_* (x,v) \d v = 0 \,, \quad \mathfrak{F} (x) = \int_{\R^3} v_3
			\left(
			\begin{array}{c}
				\widehat{\mathbb{B}}_3 \\
				\widehat{\mathbb{A}}_{13}\\
				\widehat{\mathbb{A}}_{23}
			\end{array}
			\right) f_* (x,v) \d v = 0
      \end{aligned}
    \end{equation*}
    for all $x \geq 0$. Here $\mathfrak{f} (x)$ is the coefficient of positive eigensubspace and $\mathfrak{F} (x)$ is that of zero eigensubspace. Note that $\mathfrak{f} (x)$ obeys the first order ODE $\frac{\d}{\d x} \mathfrak{f} (x) = - \ss_+ (\delta x + l)^{- \frac{1 - \gamma}{3 - \gamma}} \mathfrak{f} (x)$, which means that $\mathfrak{f} (x) = 0$ for all $x \geq 0$ if and only if $\mathfrak{f} (0) = 0$. This argument on removing the damping effect in the positive eigensubspace is similar as in Chen-Liu-Yang's work \cite{Chen-Liu-Yang-2004-AA}. Moreover, the coefficient $\mathfrak{F} (x)$ of zero eigensubspace enjoys the second order ODE system 
    \begin{equation*}
		\begin{aligned}
			\tfrac{\d^2}{\d x^2} \mathfrak{F} (x) = ( \delta x + l )^{ - \frac{1 - \gamma}{3 - \gamma} } \mathrm{diag} (\lambda_0, \lambda_1, \lambda_2) \mathfrak{F} (x) \,, \ \ \lim_{x \to + \infty} \mathfrak{F} (x) = 0_3 \,.
		\end{aligned}
	\end{equation*}
    If $\gamma = 1$, the above equation is a constant coefficient second order linear ODE system, so that the solution can be explicitly expressed, see Golse's work \cite{Golse-2008-BIMA}. Then it can be directly concluded that $\mathfrak{F} (x) = 0$ for all $x \geq 0$ if and only if $\mathfrak{F} (0) = 0$. However, it is a nonconstant coefficient second order linear ODE system for $- 3 < \gamma < 1$, whose solution cannot be explicitly written down. For general $- 3 < \gamma \leq 1$, we also need to prove that $\mathfrak{F} (x) = 0$ for all $x \geq 0$ if and only if $\mathfrak{F} (0) = 0$. We employ the {\em Freezing Point Method} (developed in elliptic theory) to fix our issue. We consider the components equations of above ODE system: for $i = 0, 1, 2$,
    \begin{equation*}
		\begin{aligned}
			\tfrac{\d^2}{\d x^2} \mathfrak{F}_i (x) = \lambda_i ( \delta x + l )^{ - \frac{1 - \gamma}{3 - \gamma} } \mathfrak{F}_i (x) \,, \ \ \lim_{x \to + \infty} \mathfrak{F}_i (x) = 0 \,.
		\end{aligned}
	\end{equation*}
    For any fixed $x_0 \geq 0$ and $- 3 < \gamma \leq 1$, we define 
    $$ \boldsymbol{\theta}_i (x_0) = \lambda_i^\frac{1}{2} ( \delta x_0 + l )^{ - \frac{1 - \gamma}{2 ( 3 - \gamma ) } } > 0 \,, \ \mathfrak{g}_i (x) = \lambda_i ( \delta x + l )^{ - \frac{1 - \gamma}{3 - \gamma} } \mathfrak{F}_i (x) - \lambda_i ( \delta x_0 + l )^{ - \frac{1 - \gamma}{3 - \gamma} } \mathfrak{F}_i (x) \,, $$
    which satisfies $\mathfrak{g}_i (x_0) = 0$. Then we can rewrite the component equation as the form
    \begin{equation*}
		\begin{aligned}
			\tfrac{\d^2}{\d x^2} \mathfrak{F}_i (x) = \boldsymbol{\theta}_i^2 (x_0) \mathfrak{F}_i (x) + \mathfrak{g}_i (x) \,, \ \lim_{x \to + \infty} \mathfrak{F}_i (x) = 0 \,.
		\end{aligned}
	\end{equation*}
    By the standard process of solving the second order linear ODE with constant coefficients and applying the far-field condition $\lim_{x \to + \infty} \mathfrak{F}_i (x) = 0$, $ \mathfrak{F}_i (x)$ can be expressed by
    \begin{equation*}
		\begin{aligned}
			\mathfrak{F}_i (x) = \Big[ C_1 (x_0) - \tfrac{1}{2 \boldsymbol{\theta}_i (x_0)} \int_{x_0}^x e^{ \boldsymbol{\theta}_i (x_0) y } \mathfrak{g}_i (y) \d y \Big] e^{- \boldsymbol{\theta}_i (x_0) x} 
		\end{aligned}
	\end{equation*}
	for some constant $C(x_0) \in \R$, which means that $\mathfrak{F}_i (x_0) = C_1 (x_0) e^{- \boldsymbol{\theta}_i (x_0) x_0}$. Moreover, a direct computation shows that $\tfrac{\d}{\d x} \mathfrak{F}_i (x_0) = - \boldsymbol{\theta}_i (x_0) C_1 (x_0) e^{- \boldsymbol{\theta}_i (x_0) x_0}$. By the arbitrariness of $x_0$, one has
\begin{equation*}
  \begin{aligned}
    \tfrac{\d}{\d x} \mathfrak{F}_i (x) = - \boldsymbol{\theta}_i (x) \mathfrak{F}_i (x) = - \lambda_i (\delta x + l)^{ - \frac{1 - \gamma}{3 - \gamma} } \mathfrak{F}_i (x) \,,
  \end{aligned}
\end{equation*}
which means that $\mathfrak{F} (x) = 0$ for all $x \geq 0$ if and only if $\mathfrak{F} (0) = 0$.

As a result, we have proved that the artificial damping $\mathbf{D} f_* (x,v) = 0$ for all $x \geq 0$ if and only if the restriction \eqref{Solvb-1} hold, which can re-characterize the vanishing sources set $\mathbb{VSS}_{\alpha_*}$ defined in \eqref{ASS} as an equivalent form \eqref{ASS-equv} in Remark \ref{Rmk-L-ASS}. Moreover, due to the linear independence of $\psi_3^*$, $\widehat{\mathbb{B}}_3$, $\widehat{\mathbb{A}}_{13}$ and $\widehat{\mathbb{A}}_{23}$, we know that $ \mathbb{VSS}_{\alpha_*} \cap \mathfrak{X}_\gamma^\infty $ is the subspace of $\mathfrak{X}_\gamma^\infty$ with codimension 4, which finish the conclusion of Theorem \ref{Thm-Linear}.

    {\bf (VII) Nonlinear problem \eqref{KL-NL}.} At the end, we focus on the nonlinear problem \eqref{KL-NL}, which can be equivalently represented by \eqref{KL-NL-f}. The one of keys of studying the nonlinear problem is to obtain nonlinear estimate \eqref{Star-1}, i.e.,
    \begin{equation*}
    	\begin{aligned}
    		\mathscr{E}^\infty ( (\delta x + l)^\frac{1 - \gamma}{3 - \gamma} e^{ \hbar \sigma } \Gamma ( f, g ) ) \lesssim \mathscr{E}^\infty ( e^{ \hbar \sigma } f ) \mathscr{E}^\infty ( e^{ \hbar \sigma } g )
    	\end{aligned}
    \end{equation*}
    given in Lemma \ref{Lmm-Gamma}. 
    
    We first decompose the solution $f$ to the problem \eqref{KL-NL-f} as two part $f_1 = (\I - \P^0) f$ and $f_2 = \P^0 f$, whose subject to the equations \eqref{KL-NL-f-decmp}. We then consider the nonlinear damping problem \eqref{KL-NL-f-damp}, hence, adding the artificial damping $\mathbf{D} f_1$ in the $f_1$-equation of \eqref{KL-NL-f-damp}. By employing the linear theory constructed in Theorem \ref{Thm-Linear}, the iterative scheme \eqref{Iter-fi} is contraction under the small data assumption \eqref{X-gamma-0} (i.e., $\varsigma_0 > 0$ sufficiently small). Then the limit of the iterative scheme \eqref{Iter-fi} uniquely solves the nonlinear problem \eqref{KL-NL-f-damp}, which means that the solution operator $\mathcal{I}_\gamma$ in Remark \ref{Rmk-NL-ASS} is well-defined. Following the same arguments as in the linear theory, $\mathbf{D} f_1 (x,v) = 0$ for all $x \geq 0$ if and only if \eqref{Solvb-NL-1}, which is then equivalent to \eqref{Solvb-NL}. Consequently, the vanishing sources set $\widetilde{\mathbb{VSS}}_{\alpha_*}$ for the nonlinear problem \eqref{KL-NL} defined in \eqref{ASS-NL} can be re-characterized as in Remark \ref{Rmk-NL-ASS}. Thanks to the linear independence of $\psi_3^*$, $\widehat{\mathbb{B}}_3$, $\widehat{\mathbb{A}}_{13}$ and $\widehat{\mathbb{A}}_{23}$, we know that $ \widetilde{\mathbb{VSS}}_{\alpha_*} \cap \mathfrak{X}_\gamma^{\varsigma_0} $ for sufficiently small $\varsigma_0 > 0$ is the subspace of $\mathfrak{X}_\gamma^{\varsigma_0}$ with codimension 4. The result on Theorem \ref{Thm-Nonlinear} is therefore constructed.

\subsection{Historical remarks}

The known results on Knudsen layer equation with Maxwell type boundary condition has been listed in Subsection \ref{Subsec:MI} above, hence, \cite{Coron-Golse-Sulem-1988-CPAM,Golse-Perthame-Sulem-1988-ARMA,HJW-2024-preprint,HJW-2023-AMASES,HW-2022-SIMA}, which all considered the hard sphere model ($\gamma = 1$). For the incoming data boundary conditions with angular cutoff collisional kernel cases, there have been many results. The incoming data involves two cases: fixed boundary and moving boundary. For the fixed boundary case, Bardos-Caflisch-Nicolaenko \cite{Bardos-Caflisch-Nicolaenko-1986-CPAM} proved the exponential decay $e^{- c x}$ in $L^\infty ( \d x; L^2 ( (1 + |v|) \d v ) )$ for hard sphere model $\gamma = 1$, and Golse-Poupaud \cite{Golse-Poupaud-1989-MMAS} then proved superalgebraic $O (x^{ - \infty })$ in $L^\infty (\d x; L^2 ( |v_3| \d v ))$ space for $- 2 \leq \gamma < 1$. The results on moving boundary case are listed as follows. Coron-Golse-Sulem \cite{Coron-Golse-Sulem-1988-CPAM} studied the exponential decay $ e^{ - cx } $ in $L^\infty ( e^{cx} \d x; L^2 ( |v_3 + u| \d v ) )$ space for $\gamma = 1$, both degenerate and nondegenerate moving velocities. Ukai-Yang-Yu \cite{Ukai-Yang-Yu-2003-CMP} investigated the exponential decay $e^{- cx}$ in $x$ and algebraic decay $(1 + |v|)^{- \beta}$ in $v$ in $L^\infty_{x,v}$ space for $\gamma = 1$ and nondegenerate moving velocity. Then Golse \cite{Golse-2008-BIMA} proved the same result for the degenerate moving velocity as in \cite{Ukai-Yang-Yu-2003-CMP}. Chen-Liu-Yang \cite{Chen-Liu-Yang-2004-AA} justified the exponential decay $e^{ - c x^\frac{2}{3 - \gamma} }$ in $x$ and algebraic decay $(1 + |v|)^{- \beta}$ in $v$ in $L^\infty_{x,v}$ space for $0 < \gamma \leq 1$ and nondegenerate moving velocity. Wang-Yang-Yang \cite{Wang-Yang-Yang-2007-JMP} proved the same decay results of \cite{Chen-Liu-Yang-2004-AA} for $- 2 < \gamma \leq 0$ and nondegenerate moving velocity. Yang \cite{Yang-2011-JSP} verified the superalgebraic decay $O(x^{-\infty})$ and $O(|v|^{-\infty})$ in $L^\infty_{x,v}$ for $- 3 < \gamma \leq 1$, both degenerate and nondegenerate moving velocities. There also were some nonlinear stability results on the Knudsen boundary layer equation with incoming data, see \cite{Ukai-Yang-Yu-2004-CMP,Wang-Yang-Yang-2006-JMP,Yang-2008-JMAA}. We remark that there is no any result about the Knudsen boundary layer problems with noncutoff angular for all types of boundary conditions.

\subsection{Organization of current paper}

In the next section, we give some preliminaries will be frequently used later. In Section \ref{Sec:WS-CA}, we prove the existence and uniqueness of the weak solution to the connection auxiliary equation \eqref{A1}. The key point is to closed the nondissipative boundary condition in the $L^2_{x,v}$ framework. Section \ref{Sec:UECA} aims at deriving  the uniform weighted $L^\infty_{x,v}$ estimates for \eqref{A1}. In Section \ref{Sec:EKLe}, the existence and uniqueness of the linear problem \eqref{KL} is proved, i.e., Theorem \ref{Thm-Linear}. In Section \ref{Sec:ENL}, we justify the existence and uniqueness of the nonlinear problem \eqref{KL-NL}, hence, Theorem \ref{Thm-Nonlinear}. The $L^\infty_{x,v}$ bounds for the operators $Y_A, Z, U$ are studied in Section \ref{Sec:YZU-proof}, i.e., proving Lemma \ref{Lmm-ARU}. We then study the properties of the operator $K$, namely, to prove Lemma \ref{Lmm-K-Oprt}-\ref{Lmm-Kh-2-infty}-\ref{Lmm-Kh-L2} in Section \ref{Sec:K}. In Section \ref{Sec:WMDM}, the properties of the artificial damping operator $\mathbf{D}$ are studied, i.e., verifying Lemma \ref{Lmm-Dh}-\ref{Lmm-D-XY}-\ref{Lmm-Dh-L2}-\ref{Lmm-Dh-L2-L2}.

\section{Preliminaries}\label{Sec:Prel}

\subsection{Properties of the $(x,v)$-mixed weights $\sigma (x,v)$}

From the works \cite{Chen-Liu-Yang-2004-AA,Wang-Yang-Yang-2006-JMP,Wang-Yang-Yang-2007-JMP,Yang-2008-JMAA}, the weihgt $\sigma (x,v)$ admits the following properties.

\begin{lemma}\label{Lmm-sigma}
	For $-3 < \gamma \leq 1$, some large constant $l > 0$ and small constant $\delta > 0$, there are some positive constants $c, c_1, c_2$ such that
	\begin{equation*}{\small
		\begin{aligned}
			& \sigma (x,v) \geq c (\delta x + l)^\frac{2}{3 - \gamma} \,, \ |\sigma (x,v) - \sigma (x, v_*)| \leq c \big| |v - \u|^2 - |v_* - \u|^2 \big| \,, \\
			& 0 < c_1 \min \big\{ (\delta x + l)^{- \frac{1 - \gamma}{3 - \gamma}} \,, (1 + |v - \u|)^{- 1 + \gamma} \big\} \leq \sigma_x (x,v) \leq c_2 (\delta x + l)^{- \frac{1 - \gamma}{3 - \gamma}} \leq c_2 l^{- \frac{1 - \gamma}{3 - \gamma}} \,, \\
			& |\sigma_x (x,v) v_3 | \leq c \nu (v) \,, \ |\sigma_{xx} (x, v) v_3| \leq \delta \sigma_x \nu (v) \,.
		\end{aligned}}
	\end{equation*}
\end{lemma}

\subsection{The operators $Y_A$, $Z$ and $U$}

We introduce the functions
\begin{equation}\label{kappa}{\small
	\begin{aligned}
		& \kappa (x,v) = \int_0^x \big[ - \tfrac{m \sigma_{xx} (y,v) }{2 \sigma_x (y,v)} - \hbar \sigma_x (y, v) + \tfrac{\nu (v)}{v_3} \big] \d y \,.
	\end{aligned}}
\end{equation}
Moreover, denote by
\begin{equation}\label{M-sigma}{\small
	\begin{aligned}
		\M_\sigma (v) = \M(v) e^{- 2 \hbar \sigma (0, v)} \,.
	\end{aligned}}
\end{equation}
We then define the following linear operators
\begin{equation}\label{YAn-f}{\scriptsize
	Y_A (f) =
	\left\{
	    \begin{aligned}
	    	(1 - \alpha_*) e^{{\kappa} (A, R_0 v) - {\kappa} (x,v)} f (A, R_0 v) \qquad & \\
	    	+ \alpha_* \tfrac{M_w (v)}{\sqrt{\M_\sigma (v)}} \int_{v_3' < 0} (- v_3') \tfrac{ \sigma_x^\frac{m}{2} (0, v) }{ \sigma_x^\frac{m}{2} (0, v') } e^{{\kappa} (A, v') - {\kappa} (x,v)} f(A, v') \sqrt{\M_\sigma (v') } \d v' \,, \quad & \textrm{if } v_3 > 0 \,, \\[10pt]
	    	e^{{\kappa} (A, v) - {\kappa} (x,v)} f (A, v) \,, \quad & \textrm{if } v_3 < 0 \,,
	    \end{aligned}
	\right.}
\end{equation}
and
\begin{equation}\label{Rn-f}{\scriptsize
	Z (f) =
	\left\{
	    \begin{aligned}
	    	(1 - \alpha_*) \int_0^A e^{ - [ {\kappa} (x,v) - {\kappa} (x', R_0 v) ] } \tfrac{1}{v_3} f (x', R_0 v) \d x' \, \qquad \quad & \\
	    	- \alpha_* \tfrac{M_w (v)}{\sqrt{\M_\sigma (v)}} \int_{v_3' < 0} \int_0^A (- v_3') \tfrac{ \sigma_x^\frac{m}{2} (0, v) }{ \sigma_x^\frac{m}{2} (0, v') } e^{- [ {\kappa} (x,v) - {\kappa} (x', v') ] } \tfrac{1}{v_3'} f (x', v') \sqrt{\M_\sigma (v') } \d x' \d v' \,, \quad & \textrm{if  } v_3 > 0 \,, \\[10pt]
	    	0 \,, \quad & \textrm{if  } v_3 < 0 \,,
	    \end{aligned}
	\right.}
\end{equation}
and
\begin{equation}\label{U-f}{\small
	U (f) =
	\left\{
	    \begin{aligned}
	    	\int_0^x e^{- [ {\kappa} (x,v) - {\kappa} (x', v) ]} \tfrac{1}{v_3} f (x' , v) \d x' \,, \quad & \textrm{if  } v_3 > 0 \,, \\
	    	- \int_x^A  e^{- [ {\kappa} (x,v) - {\kappa} (x', v) ]} \tfrac{1}{v_3} f (x' , v) \d x' \,, \quad & \textrm{if  } v_3 < 0 \,.
	    \end{aligned}
	\right.}
\end{equation}
These operators will play an essential role in estimating the weighted $L^\infty_{x,v}$ norms of the Knudsen layer equations \eqref{KL-Damped}. More precisely, they obey the following estimates.

\begin{lemma}\label{Lmm-ARU}
	Let $- 3 < \gamma \leq 1$, $l \geq 1$, $m, \beta \in \mathbb{R}$, $A > 0$ and $\vartheta > 0$. Then, for sufficiently small $\hbar, \delta > 0$, there are some positive constant $C > 0$, independent of $l$, $\delta$, $\hbar$ and $A$, such that
	\begin{enumerate}
    \item $Y_A (f)$ subjects to the estimate
    \begin{equation}\label{YAn-bnd}
    	\begin{aligned}
    		\| Y_A (f) \|_{L^\infty_{x,v}} \leq C \| f (A, \cdot) \|_{L^\infty_v} \,;
    	\end{aligned}
    \end{equation}
	
	\item $Z (f)$ enjoys the bound
	\begin{equation}\label{Rn-bnd}
		\begin{aligned}
			\| Z (f) \|_{L^\infty_{x,v}} \leq C \| \nu^{-1} f \|_{L^\infty_{x,v}} \,;
		\end{aligned}
	\end{equation}

    \item $U (f)$ obeys the bound
    \begin{equation}\label{U-bnd}
    	\begin{aligned}
    		& \| U (f) \|_{L^\infty_{x,v}} \leq C \| \nu^{-1} f \|_{L^\infty_{x,v}} \,.
    	\end{aligned}
    \end{equation}
    \end{enumerate}
\end{lemma}

The proof of Lemma \ref{Lmm-ARU} will be given in Section \ref{Sec:YZU-proof} later.

\subsection{Properties of the operator $K$}\label{Subsec:K}

In this subsection, the goal is to derive some useful properties of the operator $K$ defined in \eqref{K-K1-K2}-\eqref{K1}-\eqref{K2}. It then is focused on the following decay results of the operator $K$.

\begin{lemma}\label{Lmm-K-Oprt}
	Let $- 3 < \gamma \leq 1$, $A > 0$, $m, \beta \in \R$, $\hbar, \vartheta \geq 0$ with $\delta_0 : = \tfrac{1}{4} - T (c \hbar + \vartheta) > 0 $, where $c > 0$ is given in Lemma \ref{Lmm-sigma}. Then there is a constant $C > 0$, independent of $A, \hbar$, such that
	\begin{equation}\label{K-bnd}
		\begin{aligned}
			\| \sigma_x^\frac{m}{2} e^{\hbar \sigma} K f \|_{A; \beta, \vartheta} \leq C \| \sigma_x^\frac{m}{2} e^{\hbar \sigma} f \|_{A; \beta + \gamma - 1, \vartheta} \,.
		\end{aligned}
	\end{equation}
    where the norm $\| \cdot \|_{A; \beta, \vartheta}$ is defined in \eqref{XY-space}.
\end{lemma}

The proof of Lemma \ref{Lmm-K-Oprt} will be given in Subsection \ref{Subsec:K-infty} later.

Next we prove the boundedness of the operator $K_\hbar$ from $Y_{m, 0, \vartheta}^\infty (A) \cap L^\infty_x L^2_v$ to $Y_{m, - \gamma, \vartheta}^\infty (A)$. Here the operator $K_\hbar$ is defined as
\begin{equation}\label{Kh}
	\begin{aligned}
		K_\hbar g_\sigma = e^{\hbar \sigma} K ( e^{- \hbar \sigma} g_\sigma ) \,.
	\end{aligned}
\end{equation}
More precisely, the following results hold.

\begin{lemma}\label{Lmm-Kh-2-infty}
	Let $- 3 < \gamma \leq 1$, $A > 0 $, $m \in \R$, $0 \leq \alpha' < \frac{1}{2}$ and $\hbar, \vartheta \geq 0$ sufficiently small. Then for any $\eta_1 > 0$, there is a $C_{\eta_1} > 0$, independent of $A, \hbar$, such that
	\begin{equation}
		\begin{aligned}
			\LL \nu^{-1} K_\hbar g \RR_{A; m, 0, \vartheta} \leq \eta_1 \LL g \RR_{A; m, 0, \vartheta} + C_{\eta_1} \| \sigma_x^\frac{m}{2} z_{\alpha'} w_{- \gamma, \vartheta} g \|_{L^\infty_x L^2_v} \,.
		\end{aligned}
	\end{equation}
    Here the weight $z_{\alpha'}$ is given in \eqref{z-alpha}.
\end{lemma}

The proof of Lemma \ref{Lmm-Kh-2-infty} will be given in Subsection \ref{Subsec:K-2infty} later.

Next we show the boundedness of the operator $K_\hbar$ in the weighted $L^2_v$ space. More precisely, the following conclusions hold.

\begin{lemma}\label{Lmm-Kh-L2}
	Let $- 3 < \gamma \leq 1$, $m, \beta \in \R$ and $\hbar, \vartheta \geq 0$ be sufficiently small. Define two nonempty sets
	\begin{equation*}
		\begin{aligned}
			& \mathcal{S}_\gamma : = \{ b_0 \in \R | b_0 < 2 \,, 0 \leq b_0 \leq 1 - \gamma \,, b_0 + \gamma + 1 > 0 \} \,, \\
			& \mathcal{T}_\gamma : = \{ b_1 \in \R | b_1 < 3 \,, 0 \leq b_1 \leq 1 - \gamma \,, b_1 + \gamma > 0 \} \,.
		\end{aligned}
	\end{equation*}
	Denote by $\mu_\gamma : = \min \{ \frac{1}{2}, \frac{\gamma + 3}{2}, \frac{b_0 + \gamma + 1}{2} , \frac{b_1 + \gamma}{2} \} > 0$ with $b_0 \in \mathcal{S}_\gamma$ and $b_1 \in \mathcal{T}_\gamma$. Let $0 \leq \alpha < \mu_\gamma$. Then there is a positive constant $C > 0$, independent of $\hbar$, such that
	\begin{equation}\label{Kh-L2-Bnd}
		\begin{aligned}
			\int_{\R^3} |\nu^{- \frac{1}{2}} z_{- \alpha} \sigma_x^\frac{m}{2} w_{\beta, \vartheta} K_\hbar g (x,v)|^2 \d v \leq C \int_{\R^3} |\nu^\frac{1}{2} \sigma_x^\frac{m}{2} w_{\beta, \vartheta} g (x,v)|^2 \d v \,.
		\end{aligned}
	\end{equation}
\end{lemma}

The proof of Lemma \ref{Lmm-Kh-L2} will be given in Subsection \ref{Subsec:K-L2} later.

\subsection{Properties of artificial damping operator $\mathbf{D}$}\label{Subsec:D}

In this subsection, the majority is to study the boundedness of the artificial damping operator $\mathbf{D}$ over various weighted $L^2_{x,v}$ or $L^\infty_{x,v}$ spaces, which will play an essential role in closing the uniform bounds of the connection auxiliary problem \eqref{A1}. For convenience of later use, we introduce the scaled artificial damping operator $\mathbf{D}_\hbar$ by
\begin{equation}\label{Dh}
	\begin{aligned}
		\mathbf{D}_\hbar g = e^{ \hbar \sigma } \mathbf{D} e^{ - \hbar \sigma } g \,,
	\end{aligned}
\end{equation}
where the weight function $\sigma (x,v)$ is given in \eqref{sigma}, and small $\hbar > 0$ is to be determined later.

We first state the coercivity results of the scaled artificial damping operator $\mathbf{D}_\hbar$.

\begin{lemma}[Coercivity of $\mathbf{D}$]\label{Lmm-Dh}
	Let $- 3 < \gamma \leq 1$, $0 < \delta < 1$, $\beta \in \R$. Set $\hbar > 0, \vartheta \geq 0$ be both small enough. Assume that $l > O (1) ( \ln \frac{1}{\delta} )^\frac{3 - \gamma}{2}$, $\ss_0 , \ss_+ \geq O (1) \delta \hbar$, where $O (1) > 1$ is large enough and independent of $\delta, \hbar, l$. Then there are $\mu_0 , C > 0$, independent of $\delta, \hbar, l$, such that
	\begin{equation}\label{Dh-Coercivity}{\small
		\begin{aligned}
			\int_{\R^3} w_{\beta, \vartheta}^2 g ( \mathbf{D}_\hbar g - \hbar \sigma_x v_3 g) \d v \geq & \mu_0 \delta \hbar (\delta x + l)^{- \frac{1 - \gamma}{3 - \gamma}} \int_{\R^3} |\P w_{\beta, \vartheta} g|^2 \d v - C \hbar^\frac{1}{2} \int_{\R^3} \nu (v) |\P^\perp w_{\beta, \vartheta} g|^2 \d v \,.
		\end{aligned}}
	\end{equation}
\end{lemma}

The proof of Lemma \ref{Lmm-Dh} will be given in Subsection \ref{Subsec:D-Coercivity} later.

\begin{lemma}[Weighted $L^\infty_{x,v}$ estimate of $\mathbf{D}$]\label{Lmm-D-XY}
	Let $- 3 < \gamma \leq 1$, $m, \beta, \tilde{N} \in \R$, and $\hbar, \vartheta \geq 0$ with $s_0 : = \tfrac{c'}{2} - c \hbar - \vartheta > 0$, where $c > 0$ is given in Lemma \ref{Lmm-sigma} and $c' > 0$ is mentioned in \eqref{Proj-3}. Then there is a constant $C > 0$ such that
	\begin{equation}\label{D-inf-bnd}
		\begin{aligned}
			\| \sigma_x^\frac{m}{2} e^{\hbar \sigma} \mathbf{D} g \|_{\beta, \vartheta} \leq C \delta \hbar \| \sigma_x^\frac{m}{2} e^{\hbar \sigma} g \|_{\beta - \tilde{N}, \vartheta} \,,
		\end{aligned}
	\end{equation}
	where the norm $\| \cdot \|_{\beta, \vartheta}$ is defined in \eqref{XY-space}, and the constants $\ss_+, \ss_0 > 0$ involved in the operator $\mathbf{D}$ are given Lemma \ref{Lmm-Dh}.
\end{lemma}

The proof of Lemma \ref{Lmm-Dh} will be given in Subsection \ref{Subsec:D-L5} later.

\begin{lemma}[Boundedness of $\mathbf{D}$ from weighted $L^\infty_x L^2_v$ to $L^\infty_{x,v}$]\label{Lmm-Dh-L2}
	Let $- 3 < \gamma \leq 1$, $A > 0$, $m \in \R$, $0 \leq \alpha' < \frac{1}{2}$ and $\hbar , \vartheta \geq 0$ sufficiently small. Then there is a constant $C > 0$ such that
	\begin{equation}
		\begin{aligned}
			\LL \nu^{-1} \mathbf{D}_\hbar g \RR_{A; m, 0, \vartheta} \leq C \| \sigma_x^\frac{m}{2} z_{\alpha'} w_{- \gamma, \vartheta} g \|_{L^\infty_x L^2_v} \,.
		\end{aligned}
	\end{equation}
	Here the weight $z_{\alpha'}$ is given in \eqref{z-alpha}.
\end{lemma}

The proof of Lemma \ref{Lmm-Dh-L2} will be given in Subsection \ref{Subsec:D-L5L2} later.

\begin{lemma}[Weighted $L^2_v$ boundedness of $\mathbf{D}$]\label{Lmm-Dh-L2-L2}
	Let $- 3 < \gamma \leq 1$, $0 \leq \alpha < \frac{1}{2}$, $m, \beta \in \R$ and sufficiently small $\hbar, \vartheta \geq 0$. Then there is a positive constant $C > 0$ such that
	\begin{equation}\label{Dh-L2-Bnd}
		\begin{aligned}
			\int_{\R^3} |\nu^{- \frac{1}{2}} z_{- \alpha} \sigma_x^\frac{m}{2} w_{\beta, \vartheta} \mathbf{D}_\hbar g (x,v)|^2 \d v \leq C \int_{\R^3} | \nu^\frac{1}{2} \sigma_x^\frac{m}{2} w_{\beta, \vartheta} g (x, v) |^2 \d v \,.
		\end{aligned}
	\end{equation}
\end{lemma}

The proof of Lemma \ref{Lmm-Dh-L2-L2} will be given in Subsection \ref{Subsec:D-L2L2} later.


\section{Weak solution to connection auxiliary problem \eqref{A1}}\label{Sec:WS-CA}

In this section, we mainly prove the existence and uniqueness of weak solution to the connection auxiliary problem \eqref{A1} in the weighted $L^2_{x,v}$ space. Moreover, the energy bound is uniform in $A \geq 1$. Let $g_\sigma = e^{\hbar \sigma} g$. The system \eqref{A1} can be equivalently expressed as
\begin{equation}\label{A3-lambda}
	\left\{
	\begin{aligned}
		& v_3 \partial_x g_\sigma + [ - \hbar \sigma_x v_3 + \nu (v) ] g_\sigma - K_\hbar g_\sigma + \mathbf{D}_\hbar g = h_\sigma \,, \\
		& g_\sigma (0, v) |_{v_3 > 0} = (1 - \alpha_*) g_\sigma (0, R_0 v) + \alpha_* \tfrac{M_w (v)}{\sqrt{\M_\sigma (v)}} \int_{v_3' < 0} (- v_3') g_\sigma (0, v') \sqrt{\M_\sigma (v') } \d v' \,, \\
		& g_\sigma (A, v) |_{v_3 < 0} = \varphi_{A, \sigma} (v) \,,
	\end{aligned}
	\right.
\end{equation}
where $K_\hbar$ is defined in \eqref{Kh}, $\mathbf{D}_\hbar$ is given in \eqref{Dh} and $\M_\sigma (v)$ is given in \eqref{M-sigma}, i.e., $ \M_\sigma (v) = \tfrac{\M (v)}{ e^{2 \hbar \sigma (0, v)} } $. Here $h_\sigma = e^{\hbar \sigma} h$.

As stated in Subsection \ref{Subsec:MI} before, we consider the nondissipative Maxwell reflection boundary condition in this paper. As a result, we will first deal with the boundary integral while carrying the weighted $L^2_{x,v}$ estimates.

\subsection{Nondissipative boundary energy}

In this subsection, we first establish the following key lemma to deal with the nondissipative boundary condition. The nondissipation of the boundary energy comes from the form $\int_{v_3 < 0} |v_3| \mathfrak{w} (v) g_\sigma (0, v) \d v$, whose signs are indefinite. The main ideas are as follows. Together with the structure of equation, the integral $\int_{v_3 < 0} |v_3| \mathfrak{w} (v) g_\sigma (0, v) \d v$ can be pulled to some interior integrals by using the Euler factor $e^{ - (\delta x + l) }$. The interior integrals can be dominated by two parts. One is quantity coming from the coercivity of the linear Boltzmann collision operator $\L$ with a small factor. The small factor is due to the interaction of spatial and velocity variables. The other one is the $L^2_{x,v}$ norm with additional spatial weight $(l^{- \mathbf{Z}} \delta x + l)^a$ for any $a > 0$ and $\mathbf{Z} \geq 0$. Thanks to the degree of freedom for the parameters $a > 0$ and $\mathbf{Z} \geq 0$, the weighted $L^2_{x,v}$ norm can be controlled by seeking more estimates.

\begin{lemma}[Nondissipative boundary lemma]\label{Lmm-NDBL}
	Let $- 3 < \gamma \leq 1$, $A \geq 1$, $0 < \delta, \hbar < 1$, $l \geq 1$, $0 \leq \lambda \leq 1$. Let $\mathfrak{w} (v) $ be of exponential decay form $e^{- b |v - \u|^2}$ for some $b > 0$. Assume that $g_\sigma (x,v)$ is a solution to \eqref{A3-lambda}. Then for any $ a > 0$ and $\mathbf{Z} \geq 0$, one has
	\begin{equation}\label{NDBL-inq}{\small
		\begin{aligned}
			\Big| \int_{v_3 < 0} |v_3| \mathfrak{w} (v) g_\sigma (0, v) \d v \Big| \leq C \digamma_{\delta, l,a,\mathbf{Z}} (g_\sigma) + C \digamma_{\delta, l,a, \mathbf{Z}} (\nu^{-1} h_\sigma) + \Big| \int_{ v_3 < 0 } \mathfrak{w} (v) v_3 \varphi_{A, \sigma} (v) \d v \Big|
		\end{aligned}}
	\end{equation}
	for some constant $C > 0$ independent of $ a$, $\lambda$, $A$, $\delta$, $\hbar$, $l$. Here $\digamma_{\delta, l,a, \mathbf{Z}} (\cdot)$ is introduced by
	\begin{equation}\label{F-delta-l-a}{\footnotesize
		\begin{aligned}
			\digamma_{\delta, l,a, \mathbf{Z}} (f) = \delta^{- \frac{1}{2}} l^{- 50} \| (\delta x + l)^{ - \frac{1 - \gamma}{2 (3 - \gamma)}} \nu^\frac{1}{2} f \|_A + \delta^{- \frac{1}{2}} l^{ - ( a - \frac{1 - \gamma}{2 (3 - \gamma)} ) } \| ( l^{- \mathbf{Z}} \delta x + l)^a (\delta x + l)^{- \frac{1 - \gamma}{2 (3 - \gamma)}} \nu^\frac{1}{2} f \|_A \,.
		\end{aligned}}
	\end{equation}
\end{lemma}

\begin{proof}
	Note that $g_\sigma$ solves \eqref{A3-lambda}, i.e.,
	\begin{equation*}
		\begin{aligned}
			\partial_x ( v_3 g_\sigma ) = [ \hbar \sigma_x v_3 - \nu (v) ] g_\sigma + ( K_\hbar - \mathbf{D}_\hbar ) g_\sigma + h_\sigma \,.
		\end{aligned}
	\end{equation*}
	We multiply the above equation by the factor $ e^{ - (\delta x + l) }$. It thereby holds
	\begin{equation}\label{ab-1}
		\begin{aligned}
			\partial_x \big( e^{ - (\delta x + l) } v_3 g_\sigma \big) = & - \delta e^{ - (\delta x + l) } v_3 g_\sigma + [ \hbar \sigma_x v_3 - \nu (v) ] e^{ - (\delta x + l) } g_\sigma \\
			& + e^{ - (\delta x + l) } ( K_\hbar - \mathbf{D}_\hbar ) g_\sigma + e^{ - (\delta x + l) } h_\sigma \,.
		\end{aligned}
	\end{equation}
	We integrate the equation \eqref{ab-1} over $ ( x, v ) \in (0, A) \times \{ v_3 < 0 \}$. Together with the boundary condition $g_\sigma (A, v) |_{v_3 < 0} = \varphi_{A, \sigma} (v)$, one has
{\small
	\begin{align}\label{ab-2}
		\no & - e^{- l } \int_{v_3 < 0} \mathfrak{w} (v) v_3 g_\sigma (0, v) \d v = \underbrace{ - \delta \int_0^A \int_{v_3 < 0} e^{ - (\delta x + l) } \mathfrak{w} (v) v_3 g_\sigma (x,v) \d v \d x }_{ := \mathfrak{A}_1 } \\
		\no & + \underbrace{ \int_0^A \int_{v_3 < 0} [ \hbar \sigma_x v_3 - \nu (v) ] e^{ - (\delta x + l) } \mathfrak{w} (v) g_\sigma (x,v) \d v \d x }_{ := \mathfrak{A}_2 } + \underbrace{ \int_0^A \int_{v_3 < 0} e^{ - (\delta x + l) } \mathfrak{w} (v) ( K_\hbar - \mathbf{D}_\hbar ) g_\sigma (x,v) \d v \d x }_{ := \mathfrak{A}_3 } \\
		& + \underbrace{ \int_0^A \int_{v_3 < 0} e^{ - (\delta x + l) } \mathfrak{w} (v) h_\sigma (x,v) \d x \d v }_{ := \mathfrak{A}_4 } \ \underbrace{ - e^{ - (\delta A + l) } \int_{v_3 < 0} \mathfrak{w} (v) v_3 \varphi_{A, \sigma} (v) \d v }_{ := \mathfrak{A}_5 } \,,
	\end{align}}
	where the weight function $\mathfrak{w} (v)$ is given in Lemma \ref{Lmm-NDBL}.
	
	{\bf Step 1. Control of the quantity $ \mathfrak{A}_1 $.}
	
	Note that for any $a_* > 0$ and $0 < \eps < 1$,
	\begin{equation*}{\small
		\begin{aligned}
			|\mathfrak{A}_1| \leq &  C \delta \int_0^A \int_{v_3 < 0} (\delta x + l)^{a_* + \frac{1 - \gamma}{2 (3 - \gamma)}} e^{ - (\delta x + l) } |v_3| \nu^{ - \frac{1}{2} } e^{ - \frac{b}{2} |v - \u|^2 } \\
			& \qquad \qquad \qquad \qquad \qquad \qquad \times | (\delta x + l)^{ - a_* - \frac{1 - \gamma}{2 (3 - \gamma)}} e^{ - \frac{b}{2} |v - \u|^2 } \nu^\frac{1}{2} g_\sigma (x,v) | \d v \d x \\
			= & \underbrace{ C \delta \int_0^A \int_{v_3 < 0} ( \cdots ) \mathbf{1}_{|v - \u| \geq 2} \d v \d x }_{U_{11}} + \underbrace{ C \delta \int_0^A \int_{v_3 < 0} ( \cdots ) \mathbf{1}_{|v - \u| \leq 2} \d v \d x }_{U_{12}} \,.
		\end{aligned}}
	\end{equation*}

	For the quantity $U_{11}$, the H\"older inequality implies
	\begin{equation*}{\small
		\begin{aligned}
			|U_{11}| \leq & C \delta \big( \int_0^A (\delta x + l)^{2 a_* + \frac{1 - \gamma}{ 3 - \gamma }} e^{ - 2 (\delta x + l) } \d x \big)^\frac{1}{2} \big( \int_{v_3 < 0} \nu^{- 1} |v_3|^2 e^{- b |v - \u|^2} \mathbf{1}_{ |v - \u| \geq 2} \d v \big)^\frac{1}{2} \\
			& \times \| \mathbf{1}_{ |v - \u| \geq 2} (\delta x + l)^{ - a_* - \frac{1 - \gamma}{2 (3 - \gamma)}} e^{ - \frac{b}{2} |v - \u|^2 } \nu^\frac{1}{2} g_\sigma \|_A \,.
		\end{aligned}}
	\end{equation*}
	Note that $\big( \int_{v_3 < 0} \nu^{- 1} |v_3|^2 e^{- b |v - \u|^2} \mathbf{1}_{ |v - \u| \geq 2} \d v \big)^\frac{1}{2} \leq \big( \int_{v_3 < 0} \nu^{- 1} |v_3|^2 e^{- b |v - \u|^2} \d v \big)^\frac{1}{2} \leq C$ and $ \big( \int_0^A $ $ (\delta x + l)^{2 a_* + \frac{1 - \gamma}{ 3 - \gamma }} e^{ - 2 (\delta x + l) } \d x \big)^\frac{1}{2} \leq C \delta^{ - \frac{1}{2} } l^{a_* + \frac{1 - \gamma}{2 (3 - \gamma)}} e^{- l} $. Then one obtains
	\begin{equation*}
		\begin{aligned}
			|U_{11}| \leq & C \delta \delta^{ - \frac{1}{2} } l^{a_* + \frac{1 - \gamma}{2 (3 - \gamma)}} e^{- l} \| \mathbf{1}_{ |v - \u| \geq 2} (\delta x + l)^{ - a_* - \frac{1 - \gamma}{2 (3 - \gamma)}} e^{ - \frac{b}{2} |v - \u|^2 } \nu^\frac{1}{2} g_\sigma \|_A \\
			\leq & C \delta \delta^{ - \frac{1}{2} } l^{a_* + \frac{1 - \gamma}{2 (3 - \gamma)}} e^{- l} \big( \| \mathbf{1}_{ 2 \leq |v - \u| \leq (\delta x + l)^{\kappa_0} } (\delta x + l)^{ - a_* - \frac{1 - \gamma}{2 (3 - \gamma)}} e^{ - \frac{b}{2} |v - \u|^2 } \nu^\frac{1}{2} g_\sigma \|_A \\
			& + \| \mathbf{1}_{ |v - \u| > (\delta x + l)^{\kappa_0}} (\delta x + l)^{ - a_* - \frac{1 - \gamma}{2 (3 - \gamma)}} e^{ - \frac{b}{2} |v - \u|^2 } \nu^\frac{1}{2} g_\sigma \|_A \big) \,,
		\end{aligned}
	\end{equation*}
	where $\kappa_0 > 0$ is to be determined. If $2 \leq |v - \u| \leq (\delta x + l)^{\kappa_0}$, there holds
	\begin{equation*}
		\begin{aligned}
			(\delta x + l)^{ - a_* } e^{ - \frac{b}{2} |v - \u|^2 } \leq 2^{- \frac{a_*}{\kappa_0} } \,.
		\end{aligned}
	\end{equation*}
	Take a sufficiently small $\kappa_0 > 0$ such that $2^{- \frac{a_*}{\kappa_0} } \leq l^{ - 50 - a_* - \frac{1 - \gamma}{2 (3 - \gamma)} }$, and fix this $\kappa_0 > 0$. It therefore follows that
	\begin{equation*}
		\begin{aligned}
			& \| \mathbf{1}_{ 2 \leq |v - \u| \leq (\delta x + l)^{\kappa_0} } (\delta x + l)^{ - a_* - \frac{1 - \gamma}{2 (3 - \gamma)}} e^{ - \frac{b}{2} |v - \u|^2 } \nu^\frac{1}{2} g_\sigma \|_A \\
			\leq & l^{ - 50 - a_* - \frac{1 - \gamma}{2 (3 - \gamma)} } \| (\delta x + l)^{ - \frac{1 - \gamma}{2 (3 - \gamma)}} e^{ - \frac{b}{2} |v - \u|^2 } \nu^\frac{1}{2} g_\sigma \|_A \,.
		\end{aligned}
	\end{equation*}
	If $|v - \u| > (\delta x + l)^{\kappa_0}$, one has
	\begin{equation*}
		\begin{aligned}
			(\delta x + l)^{ - a_* } e^{ - \frac{b}{2} |v - \u|^2 } \leq l^{- a_*} e^{ - \frac{b}{2} (\delta x + l)^{2 \kappa_0} } \leq l^{- a_*} e^{ - \frac{b}{2} l^{2 \kappa_0} } \leq C l^{- a_* - \frac{1 - \gamma}{2 (3 - \gamma)} - 50}
		\end{aligned}
	\end{equation*}
	for some harmless constant $C > 0$. Then
	\begin{equation*}{\small
		\begin{aligned}
			\| \mathbf{1}_{ |v - \u| > (\delta x + l)^{\kappa_0}} (\delta x + l)^{ - a_* - \frac{1 - \gamma}{2 (3 - \gamma)}} e^{ - \frac{b}{2} |v - \u|^2 } \nu^\frac{1}{2} g_\sigma \|_A \leq C l^{- a_* - \frac{1 - \gamma}{2 (3 - \gamma)} - 50} \| (\delta x + l)^{ - \frac{1 - \gamma}{2 (3 - \gamma)}} \nu^\frac{1}{2} g_\sigma \|_A \,.
		\end{aligned}}
	\end{equation*}
	As a result, $U_{11}$ can be bounded by
	\begin{equation}\label{U11}
		\begin{aligned}
			|U_{11}| \leq C \delta \delta^{- \frac{1}{2}} l^{- 50} e^{- l} \| (\delta x + l)^{ - \frac{1 - \gamma}{2 (3 - \gamma)}} \nu^\frac{1}{2} g_\sigma \|_A
		\end{aligned}
	\end{equation}
	
	For the quantity $U_{12}$, the H\"older inequality reduces to
	\begin{equation*}
		\begin{aligned}
			|U_{12}| \leq & C \delta \big( \int_0^A (\delta x + l)^{ \frac{1 - \gamma}{ 3 - \gamma }} e^{ - 2 (\delta x + l) } \d x \big)^\frac{1}{2} \big( \int_{v_3 < 0} \nu^{- 1} |v_3|^2 |v - \u|^{-1} e^{- 2 b |v - \u|^2} \mathbf{1}_{ |v - \u| \leq 2} \d v \big)^\frac{1}{2} \\
			& \times \| \mathbf{1}_{ |v - \u| \leq 2} (\delta x + l)^{ - \frac{1 - \gamma}{2 (3 - \gamma)}} |v - \u| \nu^\frac{1}{2} g_\sigma \|_A
		\end{aligned}
	\end{equation*}
	It is easy to see that $\big( \int_0^A (\delta x + l)^{ \frac{1 - \gamma}{ 3 - \gamma }} e^{ - 2 (\delta x + l) } \d x \big)^\frac{1}{2} \leq C l^\frac{1 - \gamma}{2 (3 - \gamma)} e^{- l}$. Moreover, by using $|v_3|^3 \leq |v - \u|^2$,
	\begin{equation*}{\small
		\begin{aligned}
			\big( \int_{v_3 < 0} \nu^{- 1} |v_3|^2 |v - \u|^{-1} e^{- 2 b |v - \u|^2} \mathbf{1}_{ |v - \u| \leq 2} \d v \big)^\frac{1}{2} \leq & \big( \int_{\R^3} \nu^{- 1} |v - \u| e^{- 2 b |v - \u|^2} \mathbf{1}_{ |v - \u| \leq 2} \d v \big)^\frac{1}{2} \leq C \,.
		\end{aligned}}
	\end{equation*}
	Then
	\begin{equation*}
		\begin{aligned}
			|U_{12}| \leq C \delta \delta^{- \frac{1}{2}} l^\frac{1 - \gamma}{2 (3 - \gamma)} e^{- l} \| \mathbf{1}_{ |v - \u| \leq 2} (\delta x + l)^{ - \frac{1 - \gamma}{2 (3 - \gamma)}} |v - \u| \nu^\frac{1}{2} g_\sigma \|_A \,.
		\end{aligned}
	\end{equation*}
	For any $a > 0$ and $\mathbf{Z} \geq 0$, one knows that $ \mathbf{1}_{ |v - \u| \leq 2} |v - \u| \leq 2 \leq 2 ( l^{- \mathbf{Z} -1} \delta x + 1 )^a $, which means that
	\begin{equation*}
		\begin{aligned}
			\| \mathbf{1}_{ |v - \u| \leq 2} (\delta x + l)^{ - \frac{1 - \gamma}{2 (3 - \gamma)}} |v - \u| \nu^\frac{1}{2} g_\sigma \|_A \leq 2 l^{-a} \| ( l^{- \mathbf{Z}} \delta x + l)^a (\delta x + l)^{ - \frac{1 - \gamma}{2 (3 - \gamma)}} \nu^\frac{1}{2} g_\sigma \|_A
		\end{aligned}
	\end{equation*}
	for any $a > 0$ and $\mathbf{Z} \geq 0$. The quantity $U_{12}$ is then bounded by
	\begin{equation*}
		\begin{aligned}
			|U_{12}| \leq C \delta \delta^{- \frac{1}{2}} l^{ - ( a - \frac{1 - \gamma}{2 (3 - \gamma)} ) } e^{- l} \| ( l^{- \mathbf{Z}} \delta x + l)^a (\delta x + l)^{- \frac{1 - \gamma}{2 (3 - \gamma)}} \nu^\frac{1}{2} g_\sigma \|_A \,.
		\end{aligned}
	\end{equation*}
	Collecting the bounds of $U_{11}$ and $U_{12}$ above, one then gains
	\begin{equation}\label{Afrak-1}
		\begin{aligned}
			| \mathfrak{A}_1 | \leq C \delta \digamma_{\delta, l,a, \mathbf{Z}} (g_\sigma)
		\end{aligned}
	\end{equation}
	for any $a > 0$ and $\mathbf{Z} \geq 0$, where the functional $ \digamma_{\delta, l,a, \mathbf{Z}} (g_\sigma) $ is defined in \eqref{F-delta-l-a}.
	
	{\bf Step 2. Control of the quantity $ \mathfrak{A}_2 $.} Lemma \ref{Lmm-sigma} indicates that $|v_3| \sigma_x (x, v) \leq c \nu (v)$, which means that $| \hbar \sigma_x v_3 - \nu (v) | \leq C \nu (v)$. Then the similar arguments of estimating the quantity $\mathfrak{A}_1$ in \eqref{Afrak-1} follow that
	\begin{equation}\label{Afrak-2}
		\begin{aligned}
			| \mathfrak{A}_2 | \leq C \digamma_{\delta, l,a, \mathbf{Z}} (g_\sigma)
		\end{aligned}
	\end{equation}
	for any $a > 0$ and $\mathbf{Z} \geq 0$.
	
	{\bf Step 3. Control of the quantity $ \mathfrak{A}_3 $.} Following the similar arguments of $| \mathfrak{A}_2 |$ in \eqref{Afrak-2}, one has
	\begin{equation*}
		\begin{aligned}
			| \mathfrak{A}_3 | \leq C \digamma_{\delta, l,a, \mathbf{Z}} ( \nu^{ - 1 } ( K_\hbar - \mathbf{D}_\hbar ) g_\sigma )
		\end{aligned}
	\end{equation*}
	for any $a > 0$ and $\mathbf{Z} \geq 0$. By employing Lemma \ref{Lmm-Kh-L2} and Lemma \ref{Lmm-Dh-L2-L2} with $\alpha = m = \beta = \vartheta = 0$, one obtains $ \| (a x + b)^{ \mathfrak{z} } \nu^{ - \frac{1}{2} } ( K_\hbar - \mathbf{D}_\hbar ) g_\sigma \|_A \leq C \| (a x + b)^{ \mathfrak{z} } \nu^\frac{1}{2} g_\sigma \|_A $ for any $\mathfrak{z} \in \R$, which means that $\digamma_{\delta, l,a, \mathbf{Z}} ( \nu^{ - 1 } ( K_\hbar - \mathbf{D}_\hbar ) g_\sigma ) \leq C \digamma_{\delta, l,a, \mathbf{Z}} ( g_\sigma )$. Consequently,
	\begin{equation}\label{Afrak-3}
		\begin{aligned}
			| \mathfrak{A}_3 | \leq C \digamma_{\delta, l,a, \mathbf{Z}} ( g_\sigma )
		\end{aligned}
	\end{equation}
	for any $a > 0$ and $\mathbf{Z} \geq 0$.
	
	{\bf Step 4. Control of the quantity $ \mathfrak{A}_4 $.} It infers from the similar arguments of $|\mathfrak{A}_3|$ in \eqref{Afrak-3} that
	\begin{equation}\label{Afrak-4}
		\begin{aligned}
			| \mathfrak{A}_4 | \leq C \digamma_{\delta, l,a, \mathbf{Z}} ( \nu^{-1} h_\sigma )
		\end{aligned}
	\end{equation}
	for any $ a > 0 $ and $\mathbf{Z} \geq 0$.
	
	{\bf Step 5. Control of the quantity $ \mathfrak{A}_5 $.} Note that
	\begin{equation*}
		\begin{aligned}
			| \mathfrak{A}_5 | = e^{ - l } e^{ - ( \delta A + l ) + l } \Big| \int_{ v_3 < 0 } \mathfrak{w} (v) v_3 \varphi_{A, \sigma} (v) \d v \Big| \,.
		\end{aligned}
	\end{equation*}
	Observe that $ e^{ - ( \delta A + l ) + l } \leq 1 $. Then
	\begin{equation}\label{Afrak-5}
		\begin{aligned}
			| \mathfrak{A}_5 | \leq e^{ - l } \Big| \int_{ v_3 < 0 } \mathfrak{w} (v) v_3 \varphi_{A, \sigma} (v) \d v \Big| \,.
		\end{aligned}
	\end{equation}
	
	Collecting the all above estimates on $|\mathfrak{A}_i|$ ($1 \leq i \leq 5$) in \eqref{Afrak-1}-\eqref{Afrak-2}-\eqref{Afrak-3}-\eqref{Afrak-4}-\eqref{Afrak-5}, one knows that the inequality \eqref{NDBL-inq} holds. Consequently, the proof of Lemma \ref{Lmm-NDBL} is finished.	
\end{proof}

Based on the Nondissipative boundary lemma (Lemma \ref{Lmm-NDBL}) above, we will deal with the boundary energy with the form $ - \int_{\R^3} v_3 w_* (v) g_\sigma^2 (0,v) \d v $, where the weight $w_* (v)$ is to be determined later.

\begin{lemma}[Boundary energy lemma]\label{Lmm-BEL}
	Let $- 3 < \gamma \leq 1$, $l, A \geq 1$, $ a > 0$, $\mathbf{Z} \geq 0$ and $0 \leq \alpha_* \leq 1$. Assume that \eqref{Assmp-T} holds, i.e., $0 < T_w < 2 T$. We further assume that the weight $w_* (v)$ satisfies
	\begin{equation}\label{BEL-1}
		\begin{aligned}
			w_* (R_0 v) = w_* (v) \,, \quad 0 \leq w_* (v) \lesssim e^{ \frac{1}{8} ( \frac{1}{2 T_w} - \frac{1}{4 T} ) |v - \u|^2 } \,,
		\end{aligned}
	\end{equation}
	and the small $\hbar > 0$ is such that
	\begin{equation}\label{BEL-2}
		\begin{aligned}
			e^{\hbar \sigma (0, v)} \lesssim e^{ C_{T, T_w} |v - \u|^2 } \,.
		\end{aligned}
	\end{equation}
	where $C_{T, T_w} = \min \{ \tfrac{1}{8 T}, \frac{1}{8} ( \frac{1}{2 T_w} - \frac{1}{4 T} ) \} > 0$. Let $g_\sigma (x,v)$ be a solution to \eqref{A3-lambda}. Then there is a constant $C_0 > 0$, independent of $l, \delta, A, \lambda, a , \hbar$, such that
	\begin{equation}\label{BEL-bnd}{\small
		\begin{aligned}
			- \int_{\R^3} v_3 w_* (v) g_\sigma^2 (0,v) \d v \geq & [ 1 - (1 - \alpha_*)^2 ] \int_{v_3 < 0} | v_3 | w_* (v) g_\sigma^2 (0,v) \d v \\
			- C_0 \alpha_* & ( \digamma_{\delta, l,a, \mathbf{Z}} (g_\sigma) + \digamma_{\delta, l,a, \mathbf{Z}} (\nu^{-1} h_\sigma ) )^2  - C_0 \alpha_* \| |v_3|^\frac{1}{2} \varphi_{A, \sigma} \|^2_{L^2_{\Sigma_-^A}} \,,
		\end{aligned}}
	\end{equation}
	where the functional $\digamma_{\delta, l,a, \mathbf{Z}} (\cdot)$ is mentioned as in \eqref{F-delta-l-a}.
\end{lemma}

\begin{proof}
	Observe that the solution $g_\sigma (x,v)$ satisfies the boundary condition
	\begin{equation*}{\small
		\begin{aligned}
			g_\sigma (0, v) |_{v_3 > 0} = (1 - \alpha_*) g_\sigma (0, R_0 v) + \alpha_* \tfrac{M_w (v)}{\sqrt{\M_\sigma (v)}} \int_{v_3' < 0} (- v_3') g_\sigma (0, v') \sqrt{\M_\sigma (v') } \d v' \,,
		\end{aligned}}
	\end{equation*}
	where $\M_\sigma$ is defined in \eqref{M-sigma}. It then follows from a direct calculation that
{\small
	\begin{align}\label{BE-0}
		\no - \int_{\R^3} v_3 w_* (v) & g_\sigma^2 (0,v) \d v = \int_{v_3 < 0} | v_3 | w_* (v) g_\sigma^2 (0,v) \d v - \int_{v_3 > 0} |v_3| w_* (v) g_\sigma^2 (0,v) \d v \\
		\no = & [ 1 - (1 - \alpha_*)^2 ] \int_{v_3 < 0} | v_3 | w_* (v) g_\sigma^2 (0,v) \d v \\
		\no & - \alpha_*^2 \int_{v_3 < 0} |v_3| w_* (v) \tfrac{M_w^2 (v)}{\M_\sigma (v)} \d v \Big( \int_{v_3 < 0} | v_3 | g_\sigma (0, v) \sqrt{\M_\sigma (v)} \d v \Big)^2 \\
		& - 2 \alpha_* (1 - \alpha_*) \int_{v_3 < 0} |v_3| w_* (v) \tfrac{M_w (v)}{\sqrt{\M_\sigma (v)}} g_\sigma (0, v) \d v \int_{v_3 < 0} |v_3| g_\sigma (0, v) \sqrt{\M_\sigma (v)} \d v \,,
	\end{align}}
	where we have utilized the facts
	\begin{equation*}
		\begin{aligned}
			& \int_{v_3 > 0} | v_3 | w_* (v) g_\sigma^2 (0, R_0 v) \d v = \int_{v_3 < 0} | v_3 | w_* (v) g_\sigma^2 (0,v) \d v \,, \\
			& \int_{v_3 > 0} |v_3| w_* (v) \tfrac{M_w (v)}{\sqrt{\M_\sigma (v)}} g_\sigma (0, R_0 v) \d v = \int_{v_3 < 0} |v_3| w_* (v) \tfrac{M_w (v)}{\sqrt{\M_\sigma (v)}} g_\sigma (0, v) \d v \,, \\
			& \int_{v_3 > 0} |v_3| w_* (v) \tfrac{M_w^2 (v)}{\M_\sigma (v)} \d v = \int_{v_3 < 0} |v_3| w_* (v) \tfrac{M_w^2 (v)}{\M_\sigma (v)} \d v \,.
		\end{aligned}
	\end{equation*}
	It is easy to see that
	\begin{equation*}
		\begin{aligned}
			\tfrac{M_w (v)}{\sqrt{\M (v)}} = & \tfrac{\sqrt{2 \pi} T^\frac{3}{2}}{\rho T_w^2} \exp \{ - ( \tfrac{1}{2 T_w } - \tfrac{1}{4 T}) |v - \u|^2 - \tfrac{1}{T_w} (\u - u_w) \cdot (v - \u) - \tfrac{|\u - u_w|^2}{2 T_w} \} \\
			\leq & C \exp \{ - \tfrac{1}{2} ( \tfrac{1}{2 T_w } - \tfrac{1}{4 T}) |v - \u|^2 \} \,.
		\end{aligned}
	\end{equation*}
	Together with the assumptions \eqref{BEL-1}-\eqref{BEL-2} and the definition of $\M_\sigma (v)$ in \eqref{M-sigma}, one obtains
	\begin{equation*}
		\begin{aligned}
			& w_* (v) \tfrac{M_w (v)}{\sqrt{\M_\sigma (v)}} \leq C \exp \{ - \tfrac{1}{4} ( \tfrac{1}{2 T_w } - \tfrac{1}{4 T}) |v - \u|^2 \} \,, \\
			& w_* (v) \tfrac{M_w^2 (v)}{\M_\sigma (v)} \leq C \exp \{ - \tfrac{5}{8} ( \tfrac{1}{2 T_w } - \tfrac{1}{4 T}) |v - \u|^2 \} \,, \ \sqrt{\M_\sigma (v)} \leq C \exp \{ - \tfrac{1}{8 T} |v - \u|^2 \} \,,
		\end{aligned}
	\end{equation*}
	who all exponentially decay at $|v| \to + \infty$ under the assumption $0 < T_w < 2 T$.
	
	We now take $\mathfrak{w} (v) = \sqrt{\M_\sigma (v)}$ or $w_* (v) \tfrac{M_w (v)}{\sqrt{\M_\sigma (v)}}$ in Lemma \ref{Lmm-NDBL}. Then the inequality \eqref{NDBL-inq} tells us
{\small
	\begin{align}\label{BE-1}
		\Big| \int_{v_3 < 0} |v_3| \mathfrak{w} (v) g_\sigma (0, v) \d v \Big| \leq & C ( \digamma_{\delta, l,a, \mathbf{Z}} (g_\sigma) + \digamma_{\delta, l,a, \mathbf{Z}} (\nu^{-1} h_\sigma ) ) + \Big| \int_{v_3 < 0} \mathfrak{w} (v) v_3 \varphi_{A, \sigma} (v)  \d v \Big|
	\end{align}}
	for any $ a > 0$, $0 < \delta < 1$ and $l \geq 1$. Moreover, the H\"older inequality implies
	\begin{equation}\label{BE-5}
		\begin{aligned}
			\Big| \int_{v_3 < 0} \mathfrak{w} (v) v_3 \varphi_{A, \sigma} (v)  \d v \Big| \leq C \| |v_3|^\frac{1}{2} \varphi_{A, \sigma} \|_{L^2_{\Sigma_-^A}}
		\end{aligned}
	\end{equation}
	for $\mathfrak{w} (v) = \sqrt{\M_\sigma (v)}$ or $w_* (v) \frac{M_w (v)}{\sqrt{\M_\sigma (v)}}$. From substituting the bound \eqref{BE-5} into \eqref{BE-1}, one has
	\begin{align}\label{BE-6}
		\Big| \int_{v_3 < 0} |v_3| \mathfrak{w} (v) g_\sigma (0, v) \d v \Big| \leq C ( \digamma_{\delta, l,a, \mathbf{Z}} (g_\sigma) + \digamma_{\delta, l,a, \mathbf{Z}} (\nu^{-1} h_\sigma ) ) + C \| |v_3|^\frac{1}{2} \varphi_{A, \sigma} \|_{L^2_{\Sigma_-^A}}
	\end{align}
	for $\mathfrak{w} (v) = \sqrt{\M_\sigma (v)}$ or $w_* (v) \frac{M_w (v)}{\sqrt{\M_\sigma (v)}}$ and any $a > 0$. Furthermore, one has
	\begin{equation}\label{BE-7}
		\begin{aligned}
			\int_{v_3 < 0} |v_3| w_* (v) \tfrac{M_w^2 (v)}{\M_\sigma (v)} \d v \leq C \int_{v_3 < 0} \exp \{ - \tfrac{5}{8} ( \tfrac{1}{2 T_w } - \tfrac{1}{4 T}) |v - \u|^2 \} \d v \leq C \,.
		\end{aligned}
	\end{equation}
	As a consequence, the equality \eqref{BE-0} and the bounds \eqref{BE-6}-\eqref{BE-7} conclude the estimate \eqref{BEL-bnd}. The proof of Lemma \ref{Lmm-BEL} is therefore completed.
\end{proof}


\subsection{$w_{\beta, \vartheta}$-weighted $L^2_{x,v}$ estimates}

In this subsection, the majority is to dominate the quantity $ [ \mathscr{E}^A_2 ( g_\sigma ) ]^2 $ defined in \eqref{E2-lambda} by employing the weighted $L^2_{x,v}$ energy method. In order to study the uniqueness of the solution, we also should consider the difference ${\vartriangle} g_\sigma = g_{\sigma 2} - g_{\sigma 1}$, where $g_{\sigma i}$ is the solution to \eqref{A3-lambda} associated with the source term $h_{\sigma i}$ for $i = 1, 2$. Then the difference ${\vartriangle} g_\sigma $ subjects to
\begin{equation}\label{L-lambda-Dif}{\footnotesize
	\left\{	
	\begin{aligned}
		& v_3 \partial_x ( {\vartriangle} g_\sigma ) + [ - \hbar \sigma_x v_3 + \nu (v) ] ( {\vartriangle} g_\sigma ) - ( K_\hbar - \mathbf{D}_\hbar ) ( {\vartriangle} g_\sigma ) = {\vartriangle} h_\sigma \,, \\
		& ( {\vartriangle} g_\sigma ) (0, v) |_{v_3 > 0} = (1 - \alpha_*) ( {\vartriangle} g_\sigma ) (0, R_0 v) + \alpha_* \tfrac{M_w (v)}{\sqrt{\M_\sigma (v)}} \int_{v_3' < 0} (- v_3') ( {\vartriangle} g_\sigma ) (0, v') \sqrt{\M_\sigma (v') } \d v' \,, \\
		& ( {\vartriangle} g_\sigma ) (A, v) |_{v_3 < 0} = 0 \,,
	\end{aligned}
	\right.}
\end{equation}
where ${\vartriangle} h_\sigma : = h_{\sigma 2} - h_{\sigma 1}$. More precisely, the following lemma holds.

\begin{lemma}\label{Lmm-L2xv}
	Let $- 3 < \gamma \leq 1$, $0 \leq \alpha_* \leq 1$, $A \geq 1$, $a > 0$, $\mathbf{Z} \geq 0$ and $\beta \in \R$. Moreover, the parameters $\delta, l, \hbar, \vartheta$ are given in Lemma \ref{Lmm-Dh} with further constraint $\hbar \leq o (1) \delta$, where $0 < o (1) \ll 1$ is independent of $\delta, l, \hbar $ and $A$. Assume that $g_\sigma (x,v)$ is a solution to \eqref{A3-lambda}. Then there is a constant $C > 0$, independent of $A, \delta, l, a, \mathbf{Z}$ and $\hbar$, such that
	\begin{equation}\label{L2-g-sigma-Bnd}
		\begin{aligned}
			& [1 - (1 - \alpha_*)^2] \| |v_3|^\frac{1}{2} w_{\beta, \vartheta} g_\sigma \|^2_{L^2_{\Sigma_+^0}} + \| |v_3|^\frac{1}{2} w_{\beta, \vartheta} g_\sigma \|^2_{L^2_{\Sigma_+^A}} + \mathscr{E}^A_2 ( g_\sigma ) \\
			\leq & C \mathscr{A}^A_2 (h_\sigma) + C \mathscr{B}_2 (\varphi_{A, \sigma}) + C \alpha_* \big( \digamma_{\delta, l,a, \mathbf{Z}}^2 (g_\sigma) + \digamma_{\delta, l,a, \mathbf{Z}}^2 (\nu^{-1} h_\sigma) \big) \,,
		\end{aligned}
	\end{equation}
	where the functionals $ \mathscr{E}^A_2 ( \cdot ) $, $ \mathscr{A}^A_2 ( \cdot ) $, $ \mathscr{B}_2 ( \cdot ) $ and $\digamma_{\delta, l,a, \mathbf{Z}} (\cdot)$ are defined in \eqref{E2-lambda}, \eqref{As-def}, \eqref{B-def} and \eqref{F-delta-l-a}, respectively. Moreover, the difference ${\vartriangle} g_\sigma$ enjoys the bound
	\begin{equation}\label{L2-g-sigma-Bnd-Dif}
		\begin{aligned}
			& [1 - (1 - \alpha_*)^2] \| |v_3|^\frac{1}{2} w_{\beta, \vartheta} {\vartriangle} g_\sigma \|^2_{L^2_{\Sigma_+}} + \| |v_3|^\frac{1}{2} w_{\beta, \vartheta} {\vartriangle} g_\sigma \|^2_{L^2_{\Sigma_+^A}} + \mathscr{E}^A_2 ( {\vartriangle} g_\sigma ) \\
			\leq & C \mathscr{A}^A_2 ( {\vartriangle} h_\sigma) + C \alpha_* \big( \digamma_{\delta, l,a, \mathbf{Z}}^2 ( {\vartriangle} g_\sigma) + \digamma_{\delta, l,a, \mathbf{Z}}^2 (\nu^{-1} {\vartriangle} h_\sigma) \big) \,.
		\end{aligned}
	\end{equation}
\end{lemma}

\begin{proof}[Proof of Lemma \ref{Lmm-L2xv}]
	Multiplying \eqref{A3-lambda} by $w_{\beta, \vartheta}^2 g_\sigma$ and integrating by parts over $(x,v) \in \Omega_A \times \R^3$, we have
	\begin{equation*}
		\begin{aligned}
			\iint_{\Omega_A \times \R^3} v_3 \partial_x g_\sigma \cdot w_{\beta, \vartheta}^2 g_\sigma \d v \d x + \mathscr{W} (g_\sigma) = \iint_{\Omega_A \times \R^3} h_\sigma \cdot w_{\beta, \vartheta}^2 g_\sigma \d v \d x \,,
		\end{aligned}
	\end{equation*}
	where
	\begin{equation*}
		\begin{aligned}
			\mathscr{W} (g_\sigma) = & \iint_{\Omega_A \times \R^3} [ - \hbar \sigma_x v_3 + \nu (v) ] g_\sigma \cdot w_{\beta, \vartheta}^2 g_\sigma \d v \d x - \iint_{\Omega_A \times \R^3} ( K_\hbar - \mathbf{D}_\hbar ) g_\sigma \cdot w_{\beta, \vartheta}^2 g_\sigma \d v \d x \,.
		\end{aligned}
	\end{equation*}
	
	Note that
	\begin{equation*}
		\begin{aligned}
			\iint_{\Omega_A \times \R^3} v_3 \partial_x g_\sigma \cdot w_{\beta, \vartheta}^2 g_\sigma \d v \d x = \tfrac{1}{2} \int_{\R^3} v_3 w_{\beta, \vartheta}^2 g_\sigma^2 (A, v) \d v - \tfrac{1}{2} \int_{\R^3} v_3 w_{\beta, \vartheta}^2 g_\sigma^2 (0, v) \d v \,.
		\end{aligned}
	\end{equation*}
	The boundary condition in \eqref{A3-lambda} reduces to
	\begin{equation*}
		\begin{aligned}
			\tfrac{1}{2} \int_{\R^3} v_3 w_{\beta, \vartheta}^2 g_\sigma^2 (A, v) \d v = \tfrac{1}{2} \| |v_3|^\frac{1}{2} w_{\beta, \vartheta} g_\sigma \|^2_{L^2_{\Sigma_+^A}} - \tfrac{1}{2} \| |v_3|^\frac{1}{2} w_{\beta, \vartheta} \varphi_{A, \sigma} \|^2_{L^2_{\Sigma_-^A}} \,.
		\end{aligned}
	\end{equation*}
	Moreover, it follows from Lemma \ref{Lmm-BEL} that
	\begin{align*}
		- \tfrac{1}{2} \int_{\R^3} v_3 w_{\beta, \vartheta}^2 g_\sigma^2 (0, v) \d v \geq & \tfrac{1}{2} [ 1 - (1 - \alpha_*)^2 ] \| |v_3|^\frac{1}{2} w_{\beta, \vartheta} g_\sigma \|^2_{L^2_{\Sigma_+^0}} - C \alpha_* \| |v_3|^\frac{1}{2} \varphi_{A, \sigma} \|^2_{L^2_{\Sigma_-^A}} \\
		& - C \alpha_* \big( \digamma_{\delta, l,a, \mathbf{Z}} (g_\sigma) + \digamma_{\delta, l,a, \mathbf{Z}} (\nu^{-1} h_\sigma) \big)^2 \,.
	\end{align*}
	It is easy to see that $ \| |v_3|^\frac{1}{2} \varphi_{A, \sigma} \|^2_{L^2_{\Sigma_-^A}} \leq C \mathscr{B}_2 ( \varphi_{A, \sigma} ) $. As a consequence, one has
	\begin{align}\label{BC-L2}
		\no & \iint_{\Omega_A \times \R^3} v_3 \partial_x g_\sigma \cdot w_{\beta, \vartheta}^2 g_\sigma \d v \d x \\
		\no \geq & \tfrac{1}{2} [ 1 - (1 - \alpha_*)^2 ] \| |v_3|^\frac{1}{2} w_{\beta, \vartheta} g_\sigma \|^2_{L^2_{\Sigma_+^0}} + \tfrac{1}{2} \| |v_3|^\frac{1}{2} w_{\beta, \vartheta} g_\sigma \|^2_{L^2_{\Sigma_+^A}} - C \mathscr{B}_2 ( \varphi_{A, \sigma} ) \\
		& - C \alpha_* \big( \digamma_{\delta, l,a, \mathbf{Z}} (g_\sigma) + \digamma_{\delta, l,a, \mathbf{Z}} (\nu^{-1} h_\sigma) \big)^2 \,.
	\end{align}
	
	Recalling $\L_\hbar g_\sigma = \nu (v) g_\sigma - K_\hbar g_\sigma$, there holds
	\begin{equation*}
		\begin{aligned}
			\mathscr{W} (g_\sigma) = \iint_{\Omega_A \times \R^3} \L_\hbar g_\sigma \cdot w_{\beta, \vartheta}^2 g_\sigma \d v \d x + \iint_{\Omega_A \times \R^3} ( \mathbf{D}_\hbar g_\sigma - \hbar \sigma_x v_3 g_\sigma ) w_{\beta, \vartheta}^2 g_\sigma \d v \d x \,.
		\end{aligned}
	\end{equation*}
	From using Lemma \ref{Lmm-Dh}, it follows that
	\begin{equation}\label{Dh-bnd}{\footnotesize
		\begin{aligned}
			\iint_{\Omega_A \times \R^3} ( \mathbf{D}_\hbar g_\sigma - \hbar \sigma_x v_3 g_\sigma ) w_{\beta, \vartheta}^2 g_\sigma \d v \d x \geq & \mu_0 \delta \hbar \| (\delta x + l)^{- \frac{1 - \gamma}{2 (3 - \gamma)}} \P w_{\beta, \vartheta} g_\sigma \|^2_A - \lambda C \hbar^\frac{1}{2} \| \nu^\frac{1}{2} \P^\perp w_{\beta, \vartheta} g_\sigma \|^2_A \,.
		\end{aligned}}
	\end{equation}
	Moreover, by Lemma 2.2 of \cite{Chen-Liu-Yang-2004-AA} (or Lemma 2.4 of \cite{Wang-Yang-Yang-2007-JMP}) and the similar arguments in Corollary 1 of \cite{Strain-Guo-2008-ARMA}, it infers that
	\begin{equation}\label{Lh-bnd}
		\begin{aligned}
			\iint_{\Omega_A \times \R^3} \L_\hbar g_\sigma \cdot w_{\beta, \vartheta}^2 g_\sigma \d v \d x \geq \lambda \mu_2 \| \nu^\frac{1}{2} \P^\perp w_{\beta, \vartheta} g_\sigma \|^2_A - C \hbar^2 \| (\delta x + l)^{- \frac{1 - \gamma}{2 (3 - \gamma)}} \P w_{\beta, \vartheta} g_\sigma \|^2_A
		\end{aligned}
	\end{equation}
	for $\mu_2 > 0$. Take small $\hbar \geq 0$ such that $\mu_2 - C \hbar^\frac{1}{2} \geq \frac{1}{2} \mu_2$ and $\mu_0 \delta - C \hbar \geq \frac{1}{2} \mu_0 \delta$. Then
	\begin{equation}\label{Wg-bnd}
		\begin{aligned}
			\mathscr{W} (g_\sigma) \geq & \tfrac{1}{2} \mu_2 \| \nu^\frac{1}{2} \P^\perp w_{\beta, \vartheta} g_\sigma \|^2_A + \tfrac{1}{2} \mu_0 \delta \hbar \| (\delta x + l)^{- \frac{1 - \gamma}{2 (3 - \gamma)}} \P w_{\beta, \vartheta} g_\sigma \|^2_A \geq c_0 \mathscr{E}_2^A (g_\sigma) \,,
		\end{aligned}
	\end{equation}
	where the functional $\mathscr{E}_2^A ( \cdot )$ is given in \eqref{E2-lambda}.
	
	We then control the quantity $\iint_{\Omega_A \times \R^3} h_\sigma \cdot w_{\beta, \vartheta}^2 g_\sigma \d v \d x$. By the H\"older inequality one has
	{\small
		\begin{align}\label{h-bnd-2}
			\no \iint_{\Omega_A \times \R^3} h_\sigma \cdot w_{\beta, \vartheta}^2 g_\sigma \d v \d x = & \iint_{\Omega_A \times \R^3} \P w_{\beta, \vartheta} h_\sigma \cdot \P w_{\beta, \vartheta} g_\sigma \d v \d x + \iint_{\Omega_A \times \R^3} \P^\perp w_{\beta, \vartheta} h_\sigma \cdot \P^\perp w_{\beta, \vartheta} g_\sigma \d v \d x \\
			\no \leq & \tfrac{1}{2} c_0 \delta \hbar \| (\delta x + l)^{- \frac{1 - \gamma}{2 (3 - \gamma)}} \P w_{\beta, \vartheta} g_\sigma \|^2_A + C (\delta \hbar)^{-1} \| (\delta x + l)^{ \frac{1 - \gamma}{2 (3 - \gamma)}} \P w_{\beta, \vartheta} h_\sigma \|^2_A \\
			& + \tfrac{1}{2} c_0 \| \nu^\frac{1}{2} \P^\perp w_{\beta, \vartheta} g_\sigma \|^2_A + C \| \nu^{- \frac{1}{2}} \P^\perp w_{\beta, \vartheta} h_\sigma \|^2_A \\
			\no \leq & \tfrac{1}{2} c_0 \mathscr{E}_2^A (g_\sigma) + C \mathscr{A}_2^A (h_\sigma) \,.
		\end{align}}
	
	Collecting the all above estimates \eqref{BC-L2}, \eqref{Wg-bnd} and \eqref{h-bnd-2}, one  concludes the bound \eqref{L2-g-sigma-Bnd}. Moreover, as similar arguments as in \eqref{L2-g-sigma-Bnd}, the difference ${\vartriangle} g_\sigma : = g_{\sigma 2} - g_{\sigma 1}$ enjoys the estimate \eqref{L2-g-sigma-Bnd-Dif}. Therefore, the proof of Lemma \ref{Lmm-L2xv} is completed.
\end{proof}

We remark that the weighted $L^2_{x,v}$ estimate \eqref{L2-g-sigma-Bnd} in Lemma \ref{Lmm-L2xv} is not closed due to the quantity $\alpha_* \digamma_{\delta, l,a, \mathbf{Z}}^2 (g_\sigma)$. Recalling the definition of $\digamma_{\delta, l,a, \mathbf{Z}} (g_\sigma)$ in \eqref{F-delta-l-a}, one has
\begin{equation*}
	\begin{aligned}
		\digamma_{\delta, l,a, \mathbf{Z}}^2 (g_\sigma) \lesssim & \delta^{- 1} l^{- 100} \| (\delta x + l)^{ - \frac{1 - \gamma}{2 (3 - \gamma)}} \nu^\frac{1}{2} g_\sigma \|_A^2 \\
		& + \delta^{- 1} l^{ - ( 2 a - \frac{1 - \gamma}{ 3 - \gamma } ) } \| ( l^{- \mathbf{Z}} \delta x + l)^a (\delta x + l)^{\- \frac{1 - \gamma}{2 (3 - \gamma)}} \nu^\frac{1}{2} g_\sigma \|_A^2
	\end{aligned}
\end{equation*}
for any $a > 0$ and $\mathbf{Z} \geq 0$. By \eqref{E2-lambda}, it is easy to see that $\| (\delta x + l)^{ - \frac{1 - \gamma}{2 (3 - \gamma)}} \nu^\frac{1}{2} g_\sigma \|_A^2 \lesssim (\delta \hbar)^{-1} \mathscr{E}^A_2 ( g_\sigma )$, which means that
\begin{equation*}
	\begin{aligned}
		\delta^{- 1} l^{- 100} \| (\delta x + l)^{ - \frac{1 - \gamma}{2 (3 - \gamma)}} \nu^\frac{1}{2} g_\sigma \|_A^2 \leq (\delta \hbar)^{-1} \delta^{- 1} l^{- 100} \mathscr{E}^A_2 ( g_\sigma ) \,.
	\end{aligned}
\end{equation*}
Consequently, one has
\begin{equation}\label{F-delta-l-a-bnd}{\small
	\begin{aligned}
		\digamma_{\delta, l,a, \mathbf{Z}}^2 (g_\sigma) \lesssim (\delta \hbar)^{-1} \delta^{- 1} l^{- 100} \mathscr{E}^A_2 ( g_\sigma ) + \delta^{- 1} l^{ - ( 2 a - \frac{1 - \gamma}{ 3 - \gamma } ) } \| (l^{- \mathbf{Z}} \delta x + l )^a (\delta x + l)^{ - \frac{1 - \gamma}{2 (3 - \gamma)}} \nu^\frac{1}{2} g_\sigma \|_A^2 \,.
	\end{aligned}}
\end{equation}
Then, together with \eqref{L2-g-sigma-Bnd}, one gains the following corollary.
\begin{corollary}\label{Coro-L2}
	Under the same assumptions in Lemma \ref{Lmm-L2xv}, there holds
	\begin{align}\label{Sum-1st-1}
		\mathscr{E}^A_2 ( g_\sigma ) \lesssim & \alpha_* (\delta \hbar)^{-1} \delta^{- 1} l^{ - ( 2 a - \frac{1 - \gamma}{ 3 - \gamma } ) } \| (l^{- \mathbf{Z}} \delta x + l )^a (\delta x + l)^{- \frac{1 - \gamma}{2 (3 - \gamma)}} \nu^\frac{1}{2} g_\sigma \|_A^2 \\
		\no & + \alpha_* (\delta \hbar)^{-1} \delta^{- 1} l^{- 100} \mathscr{E}^A_2 ( g_\sigma ) + \alpha_* \digamma_{\delta, l,a, \mathbf{Z}}^2 (\nu^{-1} h_\sigma) + \mathscr{A}^A_2 (h_\sigma) + \mathscr{B}_2 (\varphi_{A, \sigma})
	\end{align}
	for any $\mathbf{Z} \geq 0$ and $a > 0$.
\end{corollary}

Remark that the quantity $\alpha_* (\delta \hbar)^{-1} \delta^{- 1} l^{- 100} \mathscr{E}^A_2 ( g_\sigma ) $ can be absorbed by the quantity $ \mathscr{E}^A_2 ( g_\sigma ) $ in the left-hand side of \eqref{Sum-1st-1}, provided that $(\delta \hbar)^{-1} \delta^{- 1} l^{- 100}$ is sufficiently small. Therefore, the quantity
$$\alpha_* (\delta \hbar)^{-1} \delta^{- 1} l^{ - ( 2 a - \frac{1 - \gamma}{ 3 - \gamma } ) } \| ( l^{- \mathbf{Z}} \delta x + l)^{ a } (\delta x + l)^{ - \frac{1 - \gamma}{2 (3 - \gamma)}} \nu^\frac{1}{2} g_\sigma \|_A^2$$
in the right-hand side of \eqref{Sum-1st-1} is the only one that should be further dominated by developing the so-called {\em Spatial-velocity indices iteration approach}, as given in next subsection.


\subsection{Spatial-velocity indices iteration approach}

The main goal of this subsection is to control the quantity $\alpha_* (\delta \hbar)^{-1} \delta^{- 1} l^{ - ( 2 a - \frac{1 - \gamma}{ 3 - \gamma } ) } \| ( l^{- \mathbf{Z}} \delta x + l)^{ a } (\delta x + l)^{ - \frac{1 - \gamma}{2 (3 - \gamma)}} \nu^\frac{1}{2} g_\sigma \|_A^2$ in the right-hand side of \eqref{Sum-1st-1}. One emphasizes that the small factor $l^{ - ( 2 a - \frac{1 - \gamma}{ 3 - \gamma } ) }$ is dominant when $a > \frac{1 - \gamma}{ 2 ( 3 - \gamma ) }$, while $\alpha_* (\delta \hbar)^{-1} \delta^{- 1}$ is subdominant. In other words, the smallness of the factor $l^{ - ( 2 a - \frac{1 - \gamma}{ 3 - \gamma } ) }$ is based on sacrificing a positive power of the spatial polynomial weight $( l^{- \mathbf{Z}} \delta x + l)^{ a } (\delta x + l)^{ - \frac{1 - \gamma}{2 (3 - \gamma)}}$. In order to deal with this weight, we try to develop a so-called {\em spatial-velocity indices iteration approach} to shift the power of spatial polynomial weight $( l^{- \mathbf{Z}} \delta x + l)^{ a } $ to a controllable velocity one, while the factor $ (\delta x + l)^{ - \frac{1 - \gamma}{2 (3 - \gamma)}} $ coincides with the weight in the functional $\mathscr{E}_2^A (g_\sigma
)$ (see \eqref{E2-lambda}) resulting from the coercivity of the operators $\L$ and $\mathbf{D}$.

For any $\mathtt{q}, \beta_\sharp \in \R$, $\mathbf{Z} \geq 0$ and $a > 0$, we define the spatial-velocity indices functions as
\begin{equation}\label{psi-y}{\small
	\begin{aligned}
		\psi_{ \mathbf{Z}; \mathtt{q}, \beta_\sharp} : = & \delta \hbar \| ( l^{- \mathbf{Z}} \delta x + l )^{ \mathtt{q} } (\delta x + l)^{ - \frac{1 - \gamma}{2 (3 - \gamma)} } \P w_{\beta_\sharp, 0} g_\sigma \|^2_A + \| ( l^{- \mathbf{Z}} \delta x + l )^\mathtt{q} \nu^\frac{1}{2} \P^\perp w_{\beta_\sharp, 0} g_\sigma \|^2_A \,, \\
		y_{a, \mathbf{Z};\mathtt{q}, \beta_\sharp} : = & \alpha_* l^{2 \mathtt{q}} \digamma_{\delta, l,a, \mathbf{Z}}^2 (\nu^{-1} h_\sigma ) + (\delta \hbar)^{-1} \| ( l^{- \mathbf{Z}} \delta x + l)^{ \mathtt{q} } (\delta x + l)^{ \frac{1 - \gamma}{2 (3 - \gamma)} } \P w_{\beta_\sharp, 0} h_\sigma \|^2_A \\
		& + \alpha_* l^{2 \mathtt{q} } \| |v_3|^\frac{1}{2} ( l^{- \mathbf{Z}} \delta A + l )^\mathtt{q} w_{\beta_\sharp, 0} \varphi_{A, \sigma} \|^2_{L^2_{\Sigma_-^A}} + \| ( l^{- \mathbf{Z}} \delta x + l )^\mathtt{q} \nu^{ - \frac{1}{2} } \P^\perp w_{\beta_\sharp, 0} h_\sigma \|^2_A \,,
	\end{aligned}}
\end{equation}
where the functional $\digamma_{\delta, l,a, \mathbf{Z}} (\cdot )$ is defined in \eqref{F-delta-l-a}. Here $\mathtt{q}$ is the spatial polynomial weighted index, and $\beta_\sharp$ stands for the velocity polynomial weighted index.

First, we establish the following spatial-velocity indices iteration lemma.

\begin{lemma}[Spatial-velocity indices iteration form]\label{Lmm-SVII}
	Let $- 3 < \gamma \leq 1$, $0 \leq \alpha_* \leq 1$, $\beta_\sharp, \mathtt{q} \in \R$, $A \geq 1$, $a > 0$, $\mathbf{Z} \geq 0$. The parameters $\delta, \hbar, l$ are given in Lemma \ref{Lmm-Dh}. Then there is a constant $C > 0$, independent of $A, \delta, l, a, \mathbf{Z} $ and $\hbar$, such that
	\begin{equation}\label{It0}
		\begin{aligned}
			\psi_{ \mathbf{Z}; \mathtt{q}, \beta_\sharp } \leq & C |\mathtt{q}| l^{ - \frac{2}{3 - \gamma} \mathbf{Z} } \delta (\delta \hbar)^{-1} \psi_{ \mathbf{Z}; \mathtt{q} - \frac{1}{3 - \gamma}, \beta_\sharp + \beta_\gamma + \frac{1}{2} } + C \alpha_* \delta^{- 1} l^{2 \mathtt{q}} l^{ - ( 2 a - \frac{1 - \gamma}{ 3 - \gamma } ) } (\delta \hbar)^{-1} \psi_{\mathbf{Z}; a, 0} \\
			& + C (\delta \hbar)^{-2} \delta^{- 1} l^{- 100} \mathscr{E}^A_2 ( g_\sigma ) + C y_{a, \mathbf{Z}; \mathtt{q}, \beta_\sharp } \,,
		\end{aligned}
	\end{equation}
	where $\beta_\gamma$ is given in \eqref{beta-gamma}, and the functional $\mathscr{E}^A_2 ( \cdot )$ is defined in \eqref{E2-lambda}.
\end{lemma}

\begin{proof}
	From multiplying \eqref{A3-lambda} by $ ( l^{- \mathbf{Z}} \delta x + l )^{2 \mathtt{q} } w_{\beta_\sharp, 0}^2 g_\sigma $, it follows that
{\small
	\begin{align}\label{Q0}
		\no & \iint_{\Omega_A \times \R^3} \partial_x \big( \tfrac{1}{2} v_3 ( l^{- \mathbf{Z}} \delta x + l )^{2 \mathtt{q} } w_{\beta_\sharp, 0}^2 g_\sigma^2 \big) \d v \d x - \mathtt{q} l^{- \mathbf{Z}} \delta \iint_{\Omega_A \times \R^3} v_3 ( l^{- \mathbf{Z}} \delta x + l )^{ 2 \mathtt{q} - 1 } w_{\beta_\sharp, 0}^2 g_\sigma^2 \d v \d x \\
		& + \iint_{\Omega_A \times \R^3} ( \nu (v) - K_\hbar ) g_\sigma \cdot ( l^{- \mathbf{Z}} \delta x + l )^{2 \mathtt{q} } w_{\beta_\sharp, 0}^2 g_\sigma \d v \d x \\
		\no & + \iint_{\Omega_A \times \R^3} [ \mathbf{D}_\hbar - \hbar \sigma_x v_3 ] g_\sigma ( l^{- \mathbf{Z}} \delta x + l )^{2 \mathtt{q} } w_{\beta_\sharp, 0}^2 g_\sigma \d v \d x = \iint_{\Omega_A \times \R^3} h_\sigma \cdot ( l^{- \mathbf{Z}} \delta x + l )^{2 \mathtt{q} } w_{\beta_\sharp, 0}^2 g_\sigma \d v \d x \,.
	\end{align}}
	Observe that
{\small
		\begin{align*}
			& \iint_{\Omega_A \times \R^3} \partial_x \big( \tfrac{1}{2} v_3 ( l^{- \mathbf{Z}} \delta x + l )^{2 \mathtt{q}} w_{\beta_\sharp, 0}^2 g_\sigma^2 \big) \d v \d x \\
			= & \tfrac{1}{2} \int_{\R^3} v_3 ( l^{- \mathbf{Z}} \delta A + l )^{2 \mathtt{q}} w_{\beta_\sharp, 0}^2 g_\sigma^2 (A, v) \d v - \tfrac{1}{2} l^{2 \mathtt{q}} \int_{\R^3} v_3 w_{\beta_\sharp, 0}^2 g_\sigma^2 (0, v) \d v \,.
		\end{align*}}
	It follows from the boundary condition in \eqref{A3-lambda} that
	\begin{equation*}
		\begin{aligned}
			& \tfrac{1}{2} \int_{\R^3} v_3 ( l^{- \mathbf{Z}} \delta A + l )^{2 \mathtt{q}} w_{\beta_\sharp, 0}^2 g_\sigma^2 (A, v) \d v \\
			= & \tfrac{1}{2} \| |v_3|^\frac{1}{2} ( l^{- \mathbf{Z}} \delta A + l )^{\mathtt{q}} w_{\beta_\sharp, 0} g_\sigma \|^2_{L^2_{\Sigma_+^A}} - \tfrac{1}{2} \| |v_3|^\frac{1}{2} ( l^{- \mathbf{Z}} \delta A + l )^{\mathtt{q}} w_{\beta_\sharp, 0} \varphi_{A, \sigma} \|^2_{L^2_{\Sigma_-^A}} \,.
		\end{aligned}
	\end{equation*}
	Then Lemma \ref{Lmm-BEL} indicates that
		\begin{align*}
			& - \tfrac{1}{2} l^{2 \mathtt{q}} \int_{\R^3} v_3 w_{\beta_\sharp, 0}^2 g_\sigma^2 (0, v) \d v \\
			\geq & [1 - (1 - \alpha_*)^2] l^{2 \mathtt{q}} \| |v_3|^\frac{1}{2} w_{\beta_\sharp, 0} g_\sigma \|^2_{L^2_{\Sigma_+^0}} - C \alpha_* l^{2 \mathtt{q}} \| |v_3|^\frac{1}{2} ( l^{- \mathbf{Z}} \delta A + l )^\mathtt{q} w_{\beta_\sharp, 0} \varphi_{A, \sigma} \|^2_{L^2_{\Sigma_-^A}} \\
			& - C_0 \alpha_* l^{2 \mathtt{q}} ( \digamma_{\delta, l,a, \mathbf{Z}} (g_\sigma) + \digamma_{\delta, l,a, \mathbf{Z}} (\nu^{-1} h_\sigma ) )^2 \,.
		\end{align*}
	In summary, one has
	\begin{align}\label{Q2}
		\no & \iint_{\Omega_A \times \R^3} \partial_x \big( \tfrac{1}{2} v_3 (l^{- \mathbf{Z}} \delta x + l)^{2 \mathtt{q}} w_{\beta_\sharp, 0}^2 g_\sigma^2 \big) \d v \d x \\
		\geq & \tfrac{1}{2} [ 1 - (1 - \alpha_*)^2 ] l^{2 \mathtt{q}} \| |v_3|^\frac{1}{2} w_{\beta_\sharp, 0} g_\sigma \|^2_{ L^2_{\Sigma_+^0} } + \tfrac{1}{2} \| |v_3|^\frac{1}{2} ( l^{- \mathbf{Z}} \delta A + l )^\mathtt{q} w_{\beta_\sharp, 0} g_\sigma \|^2_{L^2_{\Sigma_+^A}} \\
		\no & - C \alpha_* l^{2 \mathtt{q} } ( \digamma_{\delta, l,a, \mathbf{Z}}^2 (g_\sigma) + \digamma_{\delta, l,a, \mathbf{Z}}^2 (\nu^{-1} h_\sigma ) ) - C \alpha_* l^{2 \mathtt{q}} \| |v_3|^\frac{1}{2} ( l^{- \mathbf{Z}} \delta A + l )^\mathtt{q} w_{\beta_\sharp, 0} \varphi_{A, \sigma} \|^2_{L^2_{\Sigma_-^A}} \,.
	\end{align}
	Moreover, the similar arguments in \eqref{Lh-bnd} show that
	\begin{align*}
		\iint_{\Omega_A \times \R^3} ( \nu (v) - K_\hbar ) g_\sigma \cdot (l^{- \mathbf{Z}} \delta x + l)^{2 \mathtt{q}} w_{\beta_\sharp, 0}^2 g_\sigma \d v \d x \geq \mu_2 \| ( l^{- \mathbf{Z}} \delta x + l)^{\mathtt{q}} \nu^\frac{1}{2} \P^\perp w_{\beta_\sharp, 0} g_\sigma \|^2_A \\
		- C \hbar^2 \| ( l^{- \mathbf{Z}} \delta x + l)^{ \mathtt{q} } (\delta x + l)^{ - \frac{1 - \gamma}{2 (3 - \gamma)} } \P w_{\beta_\sharp, 0} g_\sigma \|^2_A \,.
	\end{align*}
	As similarly in \eqref{Dh-bnd}, one further obtains
	\begin{align*}
		& \iint_{\Omega_A \times \R^3} [ \mathbf{D}_\hbar - \hbar \sigma_x v_3 ] g_\sigma ( l^{- \mathbf{Z}} \delta x + l)^{2 \mathtt{q}} w_{\beta_\sharp, 0}^2 g_\sigma \d v \d x \\
		\geq & \mu_0 \delta \hbar \| ( l^{- \mathbf{Z}} \delta x + l)^{ \mathtt{q} } (\delta x + l)^{ - \frac{1 - \gamma}{2 (3 - \gamma)} } \P w_{\beta_\sharp, 0} g_\sigma \|^2_A - C \hbar^\frac{1}{2} \| ( l^{- \mathbf{Z}} \delta x + l)^{\mathtt{q}} \nu^\frac{1}{2} \P^\perp w_{\beta_\sharp, 0} g_\sigma \|^2_A \,.
	\end{align*}
	By taking small $\hbar \geq 0$ such that $\mu_2 - C \hbar^\frac{1}{2} \geq \frac{1}{2} \mu_2$ and $\mu_0 \delta - C \hbar \geq \frac{1}{2} \mu_0 \delta$, one gains
	\begin{align}\label{Q4}
		\no & \iint_{\Omega_A \times \R^3} ( \nu (v) - K_\hbar ) g_\sigma \cdot ( l^{- \mathbf{Z}} \delta x + l)^{2 \mathtt{q}} w_{\beta_\sharp, 0}^2 g_\sigma \d v \d x \\
		& + \iint_{\Omega_A \times \R^3} [ \mathbf{D}_\hbar - \hbar \sigma_x v_3 ] g_\sigma ( l^{- \mathbf{Z}} \delta x + l)^{2 \mathtt{q}} w_{\beta_\sharp, 0}^2 g_\sigma \d v \d x \\
		\no \geq & c_0 \delta \hbar l^{ - \frac{1 - \gamma}{3 - \gamma} } \| ( l^{- \mathbf{Z}} \delta x + l)^{ \mathtt{q} } (\delta x + l)^{ - \frac{1 - \gamma}{2 (3 - \gamma)} } \P w_{\beta_\sharp, 0} g_\sigma \|^2_A + c_0 \| ( l^{- \mathbf{Z}} \delta x + l)^{\mathtt{q}} \nu^\frac{1}{2} \P^\perp w_{\beta_\sharp, 0} g_\sigma \|^2_A \,,
	\end{align}
	where $c_0 = \min \{ \frac{1}{2} \mu_0, \frac{1}{2} \mu_2 \} > 0$.
	
	Furthermore, it follow from similar arguments in \eqref{h-bnd-2} that
	\begin{align}\label{Q5}
		\no & | \iint_{\Omega_A \times \R^3} h_\sigma \cdot ( l^{- \mathbf{Z}} \delta x + l )^{2 \mathtt{q}} w_{\beta_\sharp, 0}^2 g_\sigma \d v \d x | \\
		\leq & \tfrac{c_0}{2} \big( \delta \hbar \| ( l^{- \mathbf{Z}} \delta x + l)^{ \mathtt{q} } (\delta x + l)^{ - \frac{1 - \gamma}{2 (3 - \gamma)} } \P w_{\beta_\sharp, 0} g_\sigma \|^2_A + \| ( l^{- \mathbf{Z}} \delta x + l )^\mathtt{q} \P^\perp w_{\beta_\sharp, 0} g_\sigma \|^2_A \big) \\
		\no & + C \big( (\delta \hbar)^{-1} \| ( l^{- \mathbf{Z}} \delta x + l)^{ \mathtt{q} } (\delta x + l)^{ \frac{1 - \gamma}{2 (3 - \gamma)} } \P w_{\beta_\sharp, 0} h_\sigma \|^2_A + \| ( l^{- \mathbf{Z}} \delta x + l )^\mathtt{q} \P^\perp w_{\beta_\sharp, 0} h_\sigma \|^2_A \big) \,.
	\end{align}
	
	Note that $ |v_3| \lesssim \tfrac{ \nu (v) | v - \u | }{ ( 1 + | v - \u | )^\gamma } \lesssim \nu (v) ( 1 + |v - \u| )^{ 2 \beta_\gamma + 1 } $, where $\beta_\gamma$ is given in \eqref{beta-gamma}. Moreover, $(\delta x + l)^\frac{1 - \gamma}{3 - \gamma} \leq l^{\frac{1 - \gamma}{3 - \gamma} \mathbf{Z} } ( l^{- \mathbf{Z}} \delta x + l )^\frac{1 - \gamma}{3 - \gamma}$ for $l \geq 1$, $\mathbf{Z} \geq 0$ and $- 3 < \gamma \leq 1$. Then
{\small
	\begin{align}\label{Q3}
		\no & | \mathtt{q} l^{- \mathbf{Z}} \delta \iint_{\Omega_A \times \R^3} v_3 ( l^{- \mathbf{Z}} \delta x + l )^{ 2 \mathtt{q} - 1 } w_{\beta_\sharp, 0}^2 g_\sigma^2 \d v \d x | \\
		\leq & | \mathtt{q} | l^{- \mathbf{Z}} \delta \iint_{\Omega_A \times \R^3} | v_3 | ( l^{- \mathbf{Z}} \delta x + l )^{ 2 \mathtt{q} - 1 } w_{\beta_\sharp, 0}^2 g_\sigma^2 \d v \d x \\
		\no \leq & C | \mathtt{q} | l^{- \mathbf{Z}} \delta \iint_{\Omega_A \times \R^3} \nu (v) ( 1 + |v - \u| )^{ 2 \beta_\gamma + 1 } ( l^{- \mathbf{Z}} \delta x + l )^{ 2 \mathtt{q} - 1 } (\delta x + l)^{ \frac{1 - \gamma}{3 - \gamma} - \frac{1 - \gamma}{3 - \gamma} } w_{\beta_\sharp, 0}^2 g_\sigma^2 \d v \d x \\
		\no \leq & C | \mathtt{q} | l^{- \frac{2}{3 - \gamma} \mathbf{Z}} \delta \| ( l^{- \mathbf{Z}} \delta x + l )^{ \mathtt{q} - \frac{1}{3 - \gamma} } (\delta x + l)^{ - \frac{1 - \gamma}{2 (3 - \gamma)} } \nu^\frac{1}{2} w_{\beta_\sharp + \beta_\gamma + \frac{1}{2}, 0} g_\sigma \|^2_A \\
		\no \leq & C | \mathtt{q} | l^{ - \frac{2}{3 - \gamma} \mathbf{Z} } \delta (\delta \hbar)^{-1} \psi_{\mathtt{q} - \frac{1}{3 - \gamma}, \beta_\sharp + \beta_\gamma + \frac{1}{2}} \,,
	\end{align}}
	where the quantity $\psi_{\mathtt{q} - \frac{1}{3 - \gamma}, \beta_\sharp + \beta_\gamma + \frac{1}{2}}$ is defined in \eqref{psi-y}.

	Then collecting the estimates \eqref{Q0}, \eqref{Q2}, \eqref{Q4}, \eqref{Q5} and \eqref{Q3} can conclude that
	\begin{align}\label{Q-bnd}
		\no & [ 1 - (1 - \alpha_*)^2 ] l^{2 \mathtt{q}} \| |v_3|^\frac{1}{2} w_{\beta_\sharp, 0} g_\sigma \|^2_{ L^2_{\Sigma_+^0} } + \| |v_3|^\frac{1}{2} ( l^{- \mathbf{Z}} \delta A + l )^\mathtt{q} w_{\beta_\sharp, 0} g_\sigma \|^2_{ L^2_{\Sigma_+^A} } \\
		\no & + \| ( l^{- \mathbf{Z}} \delta x + l )^\mathtt{q} \nu^\frac{1}{2} \P^\perp w_{\beta_\sharp, 0} g_\sigma \|^2_A + \delta \hbar \| ( l^{- \mathbf{Z}} \delta x + l)^{ \mathtt{q} } (\delta x + l)^{ - \frac{1 - \gamma}{2 (3 - \gamma)} } \P w_{\beta_\sharp, 0} g_\sigma \|^2_A \\
		\no \leq & C | \mathtt{q} | l^{ - \frac{2}{3 - \gamma} \mathbf{Z} } \delta (\delta \hbar)^{-1} \psi_{\mathtt{q} - \frac{1}{3 - \gamma}, \beta_\sharp + \beta_\gamma + \frac{1}{2}} + C_0 \alpha_* l^{2 \mathtt{q}} ( \digamma_{\delta, l,a, \mathbf{Z}}^2 (g_\sigma) + \digamma_{\delta, l,a, \mathbf{Z}}^2 (\nu^{-1} h_\sigma ) ) \\
		\no & + C (\delta \hbar)^{-1} \| ( l^{- \mathbf{Z}} \delta x + l)^{ \mathtt{q} } (\delta x + l)^{ \frac{1 - \gamma}{2 (3 - \gamma)} } \P w_{\beta_\sharp, 0} h_\sigma \|^2_A \\
		& + C \alpha_* l^{2 \mathtt{q} } \| |v_3|^\frac{1}{2} ( l^{- \mathbf{Z}} \delta A + l )^\mathtt{q} w_{\beta_\sharp, 0} \varphi_{A, \sigma} \|^2_{L^2_{\Sigma_-^A}} + C \| ( l^{- \mathbf{Z}} \delta x + l )^\mathtt{q} \nu^{ - \frac{1}{2} } \P^\perp w_{\beta_\sharp, 0} h_\sigma \|^2_A \,.
	\end{align}
	Together with \eqref{psi-y}, the bounds \eqref{Q-bnd} and \eqref{F-delta-l-a-bnd} completes the proof of Lemma \ref{Lmm-SVII}.
\end{proof}

Now, based on the spatial-velocity indices iteration form \eqref{It0} in Lemma \ref{Lmm-SVII}, we will dominate the uncontrolled quantity $\alpha_* (\delta \hbar)^{-1} \delta^{- 1} l^{ - ( 2 a - \frac{1 - \gamma}{ 3 - \gamma } ) } \| ( l^{- \mathbf{Z}} \delta x + l)^a (\delta x + l)^{ - \frac{1 - \gamma}{2 (3 - \gamma)}} \nu^\frac{1}{2} g_\sigma \|_A^2$ in the right-hand side of \eqref{Sum-1st-1}. Namely, we will close the estimates in Lemma \ref{Lmm-APE-A3}.

\begin{lemma}[Uniform and closed weighted $L^2_{x,v}$ estimates]\label{Lmm-L2xv-closed}
	Let $- 3 < \gamma \leq 1$, $0 \leq \alpha_* \leq 1$, $0 < \delta < 1$, $A \geq 1$, $\beta \geq 3 (\beta_\gamma + \frac{1}{2})$, $0 < \vartheta \leq \vartheta_0$, $0 < \hbar \leq o (1) \delta$ and $l \geq O (1) (\delta \hbar)^{ - \frac{512 ( 3 - \gamma )}{ 3 ( 3 + \gamma ) } }$, where the sufficiently small $0 < \vartheta_0, \hbar_0, o (1) \ll 1$ and large enough $O (1) \gg 1$ are all independent of $\delta, \hbar, l, A$. Then the problem \eqref{A3-lambda} admits the following uniform weighted $L^2_{x,v}$ bound
	\begin{equation}\label{L2xv-Unf-Bnd}
		\begin{aligned}
			\mathscr{E}_2^A (g_\sigma) + \mathscr{E}_{\mathtt{NBE}}^A (g_\sigma) \leq C_l \big[ \mathscr{A}_2^A (h_\sigma) + \mathscr{A}_{\mathtt{NBE}}^A (h_\sigma) \big] + C_l \big[ \mathscr{B}_2 (\varphi_{A, \sigma}) + \mathscr{B}_{\mathtt{NBE}} (\varphi_{A, \sigma}) \big]
		\end{aligned}
	\end{equation}
    for some constant $C_l > 0$ depending on $l$ but being independent of $\delta, \hbar, A$, where the functionals $ \mathscr{E}_2^A (g_\sigma), \mathscr{E}_{\mathtt{NBE}}^A (g_\sigma) $ are respectively given in \eqref{E-NBE}, \eqref{E2-lambda}, the functionals $ \mathscr{A}_2^A (h_\sigma), \mathscr{A}_{\mathtt{NBE}}^A (h_\sigma) $ are defined in \eqref{As-def}, and the functionals $ \mathscr{B}_2 (\varphi_{A, \sigma}), \mathscr{B}_{\mathtt{NBE}} (\varphi_{A, \sigma}) $ are introduced in \eqref{B-def}.
\end{lemma}

\begin{proof}[Proof of Lemma \ref{Lmm-L2xv-closed}]
	
	In order to avoid the tedious mathematical symbols, we denote by
	\begin{equation}\label{Xi-Theta}
		\begin{aligned}
			\Theta (h_\sigma, \varphi_{A, \sigma}) : = \alpha_* \digamma_{\delta, l,a, \mathbf{Z}}^2 (\nu^{-1} h_\sigma) + \mathscr{A}^A_2 (h_\sigma) + \mathscr{B}_2 (\varphi_{A, \sigma}) \,.
		\end{aligned}
	\end{equation}
	The definition of $\psi_{\mathbf{Z}; \mathtt{q}, \beta_\sharp}$ in \eqref{psi-y} shows that
	\begin{eqnarray*}
		\begin{aligned}
			\| ( l^{- \mathbf{Z}} \delta x + l)^a (\delta x + l)^{ - \frac{1 - \gamma}{2 (3 - \gamma)}} \nu^\frac{1}{2} g_\sigma \|_A^2 \lesssim (\delta \hbar)^{-1} \psi_{\mathbf{Z}; a, 0} \,.
		\end{aligned}
	\end{eqnarray*}
	Then the bound \eqref{Sum-1st-1} in Corollary \ref{Coro-L2} infers that
	\begin{equation}\label{Sum-1st-2}
		\begin{aligned}
			\mathscr{E}^A_2 ( g_\sigma ) \lesssim & \alpha_* (\delta \hbar)^{-2} \delta^{- 1} l^{ - 2 ( a - \frac{1 - \gamma}{2 ( 3 - \gamma ) } - \mathbf{U}_1) } l^{- 2 \mathbf{U}_1} \psi_{\mathbf{Z}; a, 0} \\
			& + \alpha_* (\delta \hbar)^{-1} \delta^{- 1} l^{- 100} \mathscr{E}^A_2 ( g_\sigma ) + \Theta (h_\sigma, \varphi_{A, \sigma})
		\end{aligned}
	\end{equation}
	for $a > 0 $, $ \mathbf{U}_1 \in \R$ and $\mathbf{Z} \geq 0$ to be determined later. By \eqref{It0} in Lemma \ref{Lmm-SVII} and \eqref{Xi-Theta}, one has
	\begin{equation}\label{U1}
		\begin{aligned}
			l^{- 2 \mathbf{U}_1} &  \psi_{\mathbf{Z}; a, 0} \lesssim \delta (\delta \hbar)^{-1} l^{- 2 ( \mathbf{U}_1 + \frac{1}{3 - \gamma} \mathbf{Z} )} \psi_{\mathbf{Z}; a - \frac{1}{3 - \gamma}, \beta_\gamma + \frac{1}{2}} \\
			& + \delta^{-1} (\delta \hbar)^{-1} l^{ - 2 ( \mathbf{U}_1 - ( a - a_1 ) - \frac{1 - \gamma}{2 (3 - \gamma)} ) } \psi_{\mathbf{Z}; a_1, 0} + \delta^{-1} ( \delta \hbar )^{-1} l^{-100} \mathscr{E}^A_2 ( g_\sigma ) + y_{a_1, \mathbf{Z}; a, 0}
		\end{aligned}
	\end{equation}
	for $a_1 > 0$ to be determined later. Moreover,
	\begin{equation}\label{U2}
		\begin{aligned}
			l^{- 2 \mathbf{U}_2} \psi_{\mathbf{Z}; a - \frac{1}{3 - \gamma}, \beta_\gamma + \frac{1}{2}} \lesssim & \delta (\delta \hbar)^{-1} l^{- 2 ( \mathbf{U}_2 + \frac{1}{3 - \gamma} \mathbf{Z} )} \psi_{\mathbf{Z}; a - \frac{2}{3 - \gamma}, 2 ( \beta_\gamma + \frac{1}{2} )} \\
			& + \delta^{-1} (\delta \hbar)^{-1} l^{ - 2 ( \mathbf{U}_2 - ( a - a_2 - \frac{1}{3 - \gamma} ) - \frac{1 - \gamma}{2 (3 - \gamma)} ) } \psi_{\mathbf{Z}; a_2, 0} \\
			& + \delta^{-1} ( \delta \hbar )^{-1} l^{-100} \mathscr{E}^A_2 ( g_\sigma ) + y_{a_2, \mathbf{Z}; a - \frac{1}{3 - \gamma}, \beta_\gamma + \frac{1}{2}}
		\end{aligned}
	\end{equation}
	for $a_2 > 0 $ and $ \mathbf{U}_2 \in \R$ to be determined. Furthermore, Lemma \ref{Lmm-SVII} gives us
	\begin{equation}\label{U3}{\small
		\begin{aligned}
			l^{- 2 \mathbf{U}_3} & \psi_{\mathbf{Z}; a_1, 0} \lesssim \delta (\delta \hbar)^{-1} l^{- 2 ( \mathbf{U}_3 + \frac{1}{3 - \gamma} \mathbf{Z} )} \psi_{\mathbf{Z}; a_1 - \frac{1}{3 - \gamma}, \beta_\gamma + \frac{1}{2}} \\
			& + \delta^{-1} (\delta \hbar)^{-1} l^{ - 2 ( \mathbf{U}_3 - ( a_1 - a_3 ) - \frac{1 - \gamma}{2 (3 - \gamma)} ) } \psi_{\mathbf{Z}; a_3, 0} + \delta^{-1} ( \delta \hbar )^{-1} l^{-100} \mathscr{E}^A_2 ( g_\sigma ) + y_{a_3, \mathbf{Z}; a_1, 0} \,,
		\end{aligned}}
	\end{equation}
	and
	\begin{equation}\label{U4}
		\begin{aligned}
			l^{- 2 \mathbf{U}_4} \psi_{\mathbf{Z}; a_1 - \frac{1}{3 - \gamma}, \beta_\gamma + \frac{1}{2}} \lesssim & \delta (\delta \hbar)^{-1} l^{- 2 ( \mathbf{U}_4 + \frac{1}{3 - \gamma} \mathbf{Z} )} \psi_{\mathbf{Z}; a_1 - \frac{2}{3 - \gamma}, 2 ( \beta_\gamma + \frac{1}{2} )} \\
			& + \delta^{-1} (\delta \hbar)^{-1} l^{ - 2 ( \mathbf{U}_4 - ( a_1 - a_4 - \frac{1}{3 - \gamma} ) - \frac{1 - \gamma}{2 (3 - \gamma)} ) } \psi_{\mathbf{Z}; a_4, 0} \\
			& + \delta^{-1} ( \delta \hbar )^{-1} l^{-100} \mathscr{E}^A_2 ( g_\sigma ) + y_{a_4, \mathbf{Z}; a_1 - \frac{1}{3 - \gamma}, \beta_\gamma + \frac{1}{2}} \,,
		\end{aligned}
	\end{equation}
	and
	\begin{equation}\label{U5}
		\begin{aligned}
			l^{- 2 \mathbf{U}_5 } \psi_{\mathbf{Z}; a - \frac{2}{3 - \gamma}, 2 ( \beta_\gamma + \frac{1}{2} )} \lesssim \delta (\delta \hbar)^{-1} l^{- 2 ( \mathbf{U}_5 + \frac{1}{3 - \gamma} \mathbf{Z} )} \psi_{\mathbf{Z}; a - \frac{3}{3 - \gamma}, 3 ( \beta_\gamma + \frac{1}{2} )} \\
			+ \delta^{-1} (\delta \hbar)^{-1} l^{ - 2 ( \mathbf{U}_5 - ( a - a_5) - \frac{1 - \gamma}{2 (3 - \gamma)} + \frac{2}{3 - \gamma} ) } \psi_{\mathbf{Z}; a_5, 0} \\
			+ \delta^{-1} ( \delta \hbar )^{-1} l^{-100} \mathscr{E}^A_2 ( g_\sigma ) + y_{a_5, \mathbf{Z}; a - \frac{2}{3 - \gamma}, 2 ( \beta_\gamma + \frac{1}{2} ) } \,,
		\end{aligned}
	\end{equation}
	where the parameters $a_3, a_4, a_5 > 0 $ and $ \mathbf{U}_3, \mathbf{U}_4, \mathbf{U}_5 \in \R$ are all to be determined later.
	
	Now we determinate the all above parameters $a, a_1, a_2, a_3, a_4, a_5 > 0 $, $ \mathbf{U}_1, \mathbf{U}_2, \mathbf{U}_3, \mathbf{U}_4, \mathbf{U}_5 \in \R $ and $\mathbf{Z} \geq 0$ so that the estimate are closed.
	\begin{itemize}
		\item[(i)] Take
		\begin{equation}\label{Para-1}
			\begin{aligned}
				\mathcal{J}_1 : = a - \tfrac{1 - \gamma}{2 ( 3 - \gamma ) } - \mathbf{U}_1 > 0
			\end{aligned}
		\end{equation}
		such that the quantity $\alpha_* (\delta \hbar)^{-2} \delta^{- 1} l^{ - 2 ( a - \frac{1 - \gamma}{2 ( 3 - \gamma ) } - \mathbf{U}_1) } l^{- 2 \mathbf{U}_1} \psi_{\mathbf{Z}; a, 0}$ in the right-hand side of \eqref{Sum-1st-2} can be absorbed by $\alpha_* (\delta \hbar)^{-2} \delta^{- 1} l^{ - 2 ( a - \frac{1 - \gamma}{2 ( 3 - \gamma ) } - \mathbf{U}_1) } l^{- 2 \mathbf{U}_1} \psi_{\mathbf{Z}; a, 0}$ in the left-hand side of \eqref{U1}.
		\item[(ii)] Take
		\begin{equation}\label{Para-2}
			\begin{aligned}
				\mathcal{J}_2 : = \mathbf{U}_1 + \tfrac{1}{3 - \gamma} \mathbf{Z} - \mathbf{U}_2 > 0
			\end{aligned}
		\end{equation}
		such that the quantity $\delta (\delta \hbar)^{-1} l^{- 2 ( \mathbf{U}_1 + \frac{1}{3 - \gamma} \mathbf{Z} )} \psi_{\mathbf{Z}; a - \frac{1}{3 - \gamma}, \beta_\gamma + \frac{1}{2}}$ in the right-hand side of \eqref{U1} can be bounded by $l^{- 2 \mathbf{U}_2} \psi_{\mathbf{Z}; a - \frac{1}{3 - \gamma}, \beta_\gamma + \frac{1}{2}}$ in the left-hand side of \eqref{U2}.
		\item[(iii)] Take
		\begin{equation}\label{Para-3}
			\begin{aligned}
				\mathcal{J}_3 : = \mathbf{U}_1 - ( a - a_1 ) - \tfrac{1 - \gamma}{2 (3 - \gamma)} - \mathbf{U}_3 > 0
			\end{aligned}
		\end{equation}
		such that the quantity $\delta^{-1} (\delta \hbar)^{-1} l^{ - 2 ( \mathbf{U}_1 - ( a - a_1 ) - \frac{1 - \gamma}{2 (3 - \gamma)} ) } \psi_{\mathbf{Z}; a_1, 0}$ in the right-hand side of \eqref{U1} can be dominated by $l^{- 2 \mathbf{U}_3} \psi_{\mathbf{Z}; a_1, 0}$ in the left-hand side of \eqref{U3}.
		\item[(iv)] Take
		\begin{equation}\label{Para-4}
			\begin{aligned}
				a < \tfrac{2}{3 - \gamma} \,, \ \mathcal{J}_4 : = \mathbf{U}_2 + \tfrac{1}{3 - \gamma} \mathbf{Z} > 0
			\end{aligned}
		\end{equation}
		such that the quantity $\delta (\delta \hbar)^{-1} l^{- 2 \mathbf{U}_2} \psi_{\mathbf{Z}; a - \frac{2}{3 - \gamma}, 2 ( \beta_\gamma + \frac{1}{2} )}$ in the right-hand side of \eqref{U2} can be controlled by $(\delta \hbar)^{-1} \mathscr{E}^A_2 ( g_\sigma ) $ in the left-hand side of \eqref{Sum-1st-2}.
		\item[(v)] Denote by $W_\gamma (a, a_2, \mathbf{U}_1, \mathbf{U}_2) = ( \mathbf{U}_2 - ( a - a_2 - \frac{1}{3 - \gamma} ) - \frac{1 - \gamma}{2 (3 - \gamma)} ) - \mathbf{U}_1 [ 1 - (3 - \gamma) (a - a_2) ]$. Take
		\begin{equation}\label{Para-5}
			\begin{aligned}
				a - \tfrac{1}{3 - \gamma} < a_2 < a \,, \ \mathcal{J}_5 : = \tfrac{1}{(3 - \gamma) (a - a_2)} W_\gamma (a, a_2, \mathbf{U}_1, \mathbf{U}_2) - \mathbf{U}_2 > 0
			\end{aligned}
		\end{equation}
		such that the quantity $\delta^{-1} (\delta \hbar)^{-1} l^{ - 2 ( \mathbf{U}_2 - ( a - a_2 - \frac{1}{3 - \gamma} ) - \frac{1 - \gamma}{2 (3 - \gamma)} ) } \psi_{\mathbf{Z}; a_2, 0}$ in the right-hand side of \eqref{U2} can be absorbed by the quantities $l^{- 2 \mathbf{U}_1} \psi_{\mathbf{Z}; a, 0}$ in the left-hand side of \eqref{U1} and $l^{- 2 \mathbf{U}_2} \psi_{\mathbf{Z}; a - \frac{1}{3 - \gamma}, \beta_\gamma + \frac{1}{2}}$ in the left-hand side of \eqref{U2}. More precisely, the majority is to control the factor $l^{ - ( \mathbf{U}_2 - ( a - a_2 - \frac{1}{3 - \gamma} ) - \frac{1 - \gamma}{2 (3 - \gamma)} ) } ( l^{ - \mathbf{Z}} \delta x + l)^{a_2}$. A direct calculation shows
		\begin{equation*}{\small
			\begin{aligned}
				& l^{ - ( \mathbf{U}_2 - ( a - a_2 - \frac{1}{3 - \gamma} ) - \frac{1 - \gamma}{2 (3 - \gamma)} ) } ( l^{ - \mathbf{Z}} \delta x + l)^{a_2} \\
				= & \big[ l^{- \mathbf{U}_1} ( l^{ - \mathbf{Z}} \delta x + l)^a \big]^{ 1 - (3 - \gamma) (a - a_2) } \cdot l^{ - W_\gamma (a, a_2, \mathbf{U}_1, \mathbf{U}_2) } [ ( l^{ - \mathbf{Z}} \delta x + l)^{a - \frac{1}{3 - \gamma}} ]^{ (3 - \gamma) (a - a_2) } \\
				\leq & \epsilon_* l^{- \mathbf{U}_1} ( l^{ - \mathbf{Z}} \delta x + l)^a + C_{\epsilon_*} l^{ - \frac{1}{(3 - \gamma) (a - a_2)} W_\gamma (a, a_2, \mathbf{U}_1, \mathbf{U}_2) } ( l^{ - \mathbf{Z}} \delta x + l)^{a - \frac{1}{3 - \gamma}}
			\end{aligned}}
		\end{equation*}
		for any small $\epsilon_* > 0$ to be determined later. Then
		\begin{equation}\label{U-int-1}{\small
			\begin{aligned}
				& \delta^{-1} (\delta \hbar)^{-1} l^{ - 2 ( \mathbf{U}_2 - ( a - a_2 - \frac{1}{3 - \gamma} ) - \frac{1 - \gamma}{2 (3 - \gamma)} ) } \psi_{\mathbf{Z}; a_2, 0} \\
\leq & 2 \epsilon_*^2 l^{- 2 \mathbf{U}_1} \psi_{\mathbf{Z}; a, 0} + 2 C_{\epsilon_*}^2 [ \delta^{-1} (\delta \hbar)^{-1} ]^{ \frac{2}{(3 - \gamma) (a - a_2)} } l^{ - \frac{2}{(3 - \gamma) (a - a_2)} W_\gamma (a, a_2, \mathbf{U}_1, \mathbf{U}_2) } \psi_{\mathbf{Z}; a - \frac{1}{3 - \gamma}, \beta_\gamma + \frac{1}{2}} \,.
			\end{aligned}}
		\end{equation}
		By choosing $\epsilon_* > 0$ sufficiently small, the first term $2 \epsilon_*^2 l^{- 2 \mathbf{U}_1} \psi_{\mathbf{Z}; a, 0}$ can be controlled by $l^{- 2 \mathbf{U}_1} \psi_{\mathbf{Z}; a, 0}$ in the left-hand side of \eqref{U1}. On the other hand, the second term above can be dominated by $l^{- 2 \mathbf{U}_2} \psi_{\mathbf{Z}; a - \frac{1}{3 - \gamma}, \beta_\gamma + \frac{1}{2}}$ in the left-hand side of \eqref{U2}, provided that
		\begin{equation*}
			\begin{aligned}
				\tfrac{1}{(3 - \gamma) (a - a_2)} W_\gamma (a, a_2, \mathbf{U}_1, \mathbf{U}_2) > \mathbf{U}_2 \,.
			\end{aligned}
		\end{equation*}
		\item[(vi)] Take
		\begin{equation}\label{Para-6}
			\begin{aligned}
				\mathcal{J}_6 : = \mathbf{U}_3 + \tfrac{1}{3 - \gamma} \mathbf{Z} - \mathbf{U}_4 > 0
			\end{aligned}
		\end{equation}
		such that the quantity $\delta (\delta \hbar)^{-1} l^{- 2 ( \mathbf{U}_3 + \frac{1}{3 - \gamma} \mathbf{Z} )} \psi_{\mathbf{Z}; a_1 - \frac{1}{3 - \gamma}, \beta_\gamma + \frac{1}{2}}$ in the right-hand side of \eqref{U3} can be controlled by $l^{- 2 \mathbf{U}_4} \psi_{\mathbf{Z}; a_1 - \frac{1}{3 - \gamma}, \beta_\gamma + \frac{1}{2}}$ in the left-hand side of \eqref{U4}.
		\item[(vii)] Take
		\begin{equation}\label{Para-7}
			\begin{aligned}
				a_1 < \tfrac{2}{3 - \gamma} \,, \ \mathcal{J}_7 : = \mathbf{U}_4 + \tfrac{1}{3 - \gamma} \mathbf{Z} > 0
			\end{aligned}
		\end{equation}
		such that the quantity $\delta (\delta \hbar)^{-1} l^{- 2 ( \mathbf{U}_4 + \frac{1}{3 - \gamma} \mathbf{Z} )} \psi_{\mathbf{Z}; a_1 - \frac{2}{3 - \gamma}, 2 ( \beta_\gamma + \frac{1}{2} )}$ in the right-hand side of \eqref{U4} can be absorbed by $(\delta \hbar)^{-1} \mathscr{E}^A_2 ( g_\sigma ) $ in the left-hand side of \eqref{Sum-1st-2}.
		\item[(viii)] As similar in (v), take
		\begin{equation}\label{Para-8}
			\begin{aligned}
				a_1 - \tfrac{1}{3 - \gamma} < a_4 < a_1 \,, \ \mathcal{J}_8 : = \tfrac{1}{(3 - \gamma) (a - a_2)} W_\gamma (a_1, a_4, \mathbf{U}_3, \mathbf{U}_4) - \mathbf{U}_4 > 0
			\end{aligned}
		\end{equation}
		such that the quantity $\delta^{-1} (\delta \hbar)^{-1} l^{ - 2 ( \mathbf{U}_4 - ( a_1 - a_4 - \frac{1}{3 - \gamma} ) - \frac{1 - \gamma}{2 (3 - \gamma)} ) } \psi_{\mathbf{Z}; a_4, 0}$ in the right-hand side of \eqref{U4} can be dominated by $l^{- 2 \mathbf{U}_3} \psi_{\mathbf{Z}; a_1, 0}$ in the left-hand side of \eqref{U3} and $l^{- 2 \mathbf{U}_4} \psi_{\mathbf{Z}; a_1 - \frac{1}{3 - \gamma}, \beta_\gamma + \frac{1}{2}}$ in the left-hand side of \eqref{U4}. More precisely, the key inequality is
		\begin{equation}\label{U-int-4}{\small
			\begin{aligned}
				& \delta^{-1} (\delta \hbar)^{-1} l^{ - 2 ( \mathbf{U}_4 - ( a_1 - a_4 - \frac{1}{3 - \gamma} ) - \frac{1 - \gamma}{2 (3 - \gamma)} ) } \psi_{\mathbf{Z}; a_4, 0} \leq 2 \epsilon_*^2 l^{- 2 \mathbf{U}_3} \psi_{\mathbf{Z}; a_1, 0} \\
				& \qquad \qquad + 2 C_{\epsilon_*}^2 [ \delta^{-1} (\delta \hbar)^{-1} ]^{ \frac{2}{(3 - \gamma) (a_1 - a_4)} } l^{ - \frac{2}{(3 - \gamma) (a_1 - a_4)} W_\gamma (a_1, a_4, \mathbf{U}_3, \mathbf{U}_4) } \psi_{\mathbf{Z}; a_1 - \frac{1}{3 - \gamma}, \beta_\gamma + \frac{1}{2}} \,.
			\end{aligned}}
		\end{equation}
		\item[(ix)] Take
		\begin{equation}\label{Para-9}{\small
			\begin{aligned}
				& a_3 - a + \tfrac{2}{3 - \gamma} \leq a_1 - \tfrac{1}{3 - \gamma} \,, \\
				& \mathcal{J}_9 : = \mathbf{U}_4 + \mathbf{U}_5 - [ a_1 - \tfrac{1}{3 - \gamma} - ( a_3 - a + \tfrac{2}{3 - \gamma} ) ] - [ \mathbf{U}_3 - (a_1 - a_3) - \tfrac{1 - \gamma}{2 (3 - \gamma)} ] < 0
			\end{aligned}}
		\end{equation}
		such that the quantity $\delta^{-1} (\delta \hbar)^{-1} l^{ - 2 ( \mathbf{U}_3 - ( a_1 - a_3 ) - \frac{1 - \gamma}{2 (3 - \gamma)} ) } \psi_{\mathbf{Z}; a_3, 0}$ in the right-hand side of \eqref{U3} can be bounded by $l^{- 2 \mathbf{U}_5 } \psi_{\mathbf{Z}; a - \frac{2}{3 - \gamma}, 2 ( \beta_\gamma + \frac{1}{2} )}$ in the left-hand side of \eqref{U5} and $l^{- 2 \mathbf{U}_4} \psi_{\mathbf{Z}; a_1 - \frac{1}{3 - \gamma}, \beta_\gamma + \frac{1}{2}}$ in the left-hand side of \eqref{U4}. More precisely, the core is to dominate the factor $l^{ - ( \mathbf{U}_3 - ( a_1 - a_3 ) - \frac{1 - \gamma}{2 (3 - \gamma)} ) } ( l^{ - \mathbf{Z}} \delta x + l )^{a_3} $. Note that
		\begin{equation*}
			\begin{aligned}
				& l^{ - ( \mathbf{U}_3 - ( a_1 - a_3 ) - \frac{1 - \gamma}{2 (3 - \gamma)} ) } ( l^{ - \mathbf{Z}} \delta x + l )^{a_3} \\
				\leq & l^{ \widetilde{W}_\gamma (a, a_1, a_3, \mathbf{U}_3, \mathbf{U}_4, \mathbf{U}_5) } \big[ l^{ - \mathbf{U}_4 } ( l^{ - \mathbf{Z} } \delta x + l )^{a_1 - \frac{1}{3 - \gamma}} \big] \big[ l^{ - \mathbf{U}_5 } ( l^{- \mathbf{Z}} \delta x + l )^{a - \frac{2}{3 - \gamma}} \big] \,,
			\end{aligned}
		\end{equation*}
		where
		\begin{equation*}{\small
			\begin{aligned}
				\widetilde{W}_\gamma (a, a_1, a_3, \mathbf{U}_3, \mathbf{U}_4, \mathbf{U}_5) = \mathbf{U}_4 + \mathbf{U}_5 - [ a_1 - \tfrac{1}{3 - \gamma} - ( a_3 - a + \tfrac{2}{3 - \gamma} ) ] \\
				- [ \mathbf{U}_3 - (a_1 - a_3) - \tfrac{1 - \gamma}{2 (3 - \gamma)} ] \,.
			\end{aligned}}
		\end{equation*}
		One then gains
		\begin{equation}\label{U-int-2}
			\begin{aligned}
				\delta^{-1} (\delta \hbar)^{-1} & l^{ - 2 ( \mathbf{U}_3 - ( a_1 - a_3 ) - \frac{1 - \gamma}{2 (3 - \gamma)} ) } \psi_{\mathbf{Z}; a_3, 0} \leq \epsilon_*^2 l^{- 2 \mathbf{U}_4} \psi_{\mathbf{Z}; a_1 - \frac{1}{3 - \gamma}, \beta_\gamma + \frac{1}{2}} \\
				& + C_{\epsilon_*}^2 [ \delta^{-1} (\delta \hbar)^{-1} ]^2 l^{ 2 \widetilde{W}_\gamma (a, a_1, a_3, \mathbf{U}_3, \mathbf{U}_4, \mathbf{U}_5) } l^{- 2 \mathbf{U}_5 } \psi_{\mathbf{Z}; a - \frac{2}{3 - \gamma}, 2 ( \beta_\gamma + \frac{1}{2} )} \,.
			\end{aligned}
		\end{equation}
		By choosing $\epsilon_* > 0$ sufficiently small such that the first term $\epsilon_*^2 l^{- 2 \mathbf{U}_4} \psi_{\mathbf{Z}; a_1 - \frac{1}{3 - \gamma}, \beta_\gamma + \frac{1}{2}}$ above can be absorbed by $ l^{- 2 \mathbf{U}_4} \psi_{\mathbf{Z}; a_1 - \frac{1}{3 - \gamma}, \beta_\gamma + \frac{1}{2}}$ in the left-hand side of \eqref{U4}. On the other hand, the last term above can be bounded by $ l^{- 2 \mathbf{U}_5 } \psi_{\mathbf{Z}; a - \frac{2}{3 - \gamma}, 2 ( \beta_\gamma + \frac{1}{2} )} $ in the left-hand side of \eqref{U5} under the condition
		\begin{equation*}
			\begin{aligned}
				\widetilde{W}_\gamma (a, a_1, a_3, \mathbf{U}_3, \mathbf{U}_4, \mathbf{U}_5) < 0 \,,
			\end{aligned}
		\end{equation*}
		which means the second inequality in \eqref{Para-9}.
		\item[(x)] Take
		\begin{equation}\label{Para-10}
			\begin{aligned}
				a_5 \leq a_1 - \tfrac{1}{3 - \gamma} \,, \ \mathcal{J}_{10} : = \mathbf{U}_5 + \tfrac{1 + \gamma}{2 (3 - \gamma)} - \mathbf{U}_4 > 0
			\end{aligned}
		\end{equation}
		such that the quantity $ \delta^{-1} (\delta \hbar)^{-1} l^{ - 2 ( \mathbf{U}_5 - ( a - a_5) - \frac{1 - \gamma}{2 (3 - \gamma)} + \frac{2}{3 - \gamma} ) } \psi_{\mathbf{Z}; a_5, 0} $ in the right-hand side of \eqref{U5} can be bounded by the left-hand side of \eqref{U4}. Indeed, under the first condition in \eqref{Para-10}, one has $\psi_{\mathbf{Z}; a_5, 0} \leq l^{ - 2 ( a_1 - \tfrac{1}{3 - \gamma} - a_5 ) } \psi_{\mathbf{Z}; a_1 - \frac{1}{3 - \gamma}, \beta_\gamma + \frac{1}{2}}$, which means that
		\begin{equation}\label{U-int-3}
			\begin{aligned}
				& \delta^{-1} (\delta \hbar)^{-1} l^{ - 2 ( \mathbf{U}_5 - ( a - a_5) - \frac{1 - \gamma}{2 (3 - \gamma)} + \frac{2}{3 - \gamma} ) } \psi_{\mathbf{Z}; a_5, 0} \\
				\leq & \delta^{-1} (\delta \hbar)^{-1} l^{ - 2 ( \mathbf{U}_5 + \tfrac{1 + \gamma}{2 (3 - \gamma)} ) } \psi_{\mathbf{Z}; a_1 - \frac{1}{3 - \gamma}, \beta_\gamma + \frac{1}{2}} \,.
			\end{aligned}
		\end{equation}
		It can be further bounded by the left-hand side of \eqref{U4} under the second condition in \eqref{Para-10}.
		
		\item[(xi)] Take
		\begin{equation}\label{Para-11}
			\begin{aligned}
				\mathcal{J}_{11} : = \mathbf{U}_5 + \tfrac{1}{3 - \gamma} \mathbf{Z} > 0
			\end{aligned}
		\end{equation}
		such that the quantity $\delta (\delta \hbar)^{-1} l^{- 2 ( \mathbf{U}_5 + \frac{1}{3 - \gamma} \mathbf{Z} )} \psi_{\mathbf{Z}; a - \frac{3}{3 - \gamma}, 3 ( \beta_\gamma + \frac{1}{2} )}$ can be bounded by the quantity $(\delta \hbar)^{-1} \mathscr{E}^A_2 ( g_\sigma ) $ in the left-hand side of \eqref{Sum-1st-2}.
	\end{itemize}
	
	The next goal is to take the proper parameters $a, a_1, a_2, a_3, a_4 > 0 $, $ \mathbf{U}_1, \mathbf{U}_2, \mathbf{U}_3, \mathbf{U}_4, \mathbf{U}_5 \in \R$ and $\mathbf{Z} \geq 0$ such that the required nine inequalities \eqref{Para-1}-\eqref{Para-9} above hold.
	
	For $- 3 < \gamma \leq 1$, we now take
	\begin{equation}\label{Choice-para}
		\begin{aligned}
			& a = \tfrac{15}{8 ( 3 - \gamma )} \,, \ a_1 = \tfrac{7}{4 ( 3 - \gamma )} \,, \ a_2 = \tfrac{13}{8 ( 3 - \gamma )} \,, \ a_3 = \tfrac{1}{2 ( 3 - \gamma )} \,, \ a_4 = \tfrac{3}{2 ( 3 - \gamma )} \,, \ a_5 = \tfrac{1}{3 ( 3 - \gamma )} \,, \\
			& \mathbf{U}_1 = - \tfrac{59 - 24 \gamma}{48 ( 3 - \gamma )} \,, \ \mathbf{U}_2 = - \tfrac{9 + \gamma}{6 ( 3 - \gamma )} \,, \ \mathbf{U}_3 = - \tfrac{107 - 24 \gamma}{24 ( 3 - \gamma )} \,, \ \mathbf{U}_4 = - \tfrac{297 - 40 \gamma}{72 ( 3 - \gamma )} \,, \\
			& \mathbf{U}_5 = - \tfrac{189 - 52 \gamma}{72 ( 3 - \gamma )} \,, \ \mathbf{Z} = 6 \,.
		\end{aligned}
	\end{equation}
	The choice of $a$ and $a_i$ ($1 \leq i \leq 4$) in \eqref{Choice-para} obviously guarantees themselves relations in \eqref{Para-1}-\eqref{Para-11} above. We only need to verify the sign of the numbers $\mathcal{J}_i$ for $1 \leq i \leq 11$ under the choice \eqref{Choice-para}. Straightforward calculations tell us
	\begin{equation}\label{J-sign}
		\begin{aligned}
			& \mathcal{J}_1 = \tfrac{125}{48 (3 - \gamma)} > 0 \,, \ \mathcal{J}_2 = \tfrac{301 + 32 \gamma}{48 (3 - \gamma)} > 0 \,, \ \mathcal{J}_3 = \tfrac{125}{48 (3 - \gamma)} > 0 \,, \ \mathcal{J}_4 = \tfrac{27 - \gamma}{6 (3 - \gamma)} > 0 \,, \\
			& \mathcal{J}_5 = \tfrac{3}{16 (3 - \gamma)} > 0 \,, \ \mathcal{J}_6 = \tfrac{51 + 4 \gamma}{9 (3 - \gamma)} > 0 \,, \ \mathcal{J}_7 = \tfrac{135 + 40 \gamma}{72 (3 - \gamma)} > 0 \,, \ \mathcal{J}_8 = \tfrac{2 (3 + \gamma)}{3 (3 - \gamma)} > 0 \,, \\
			& \mathcal{J}_9 = - \tfrac{2 (3 + \gamma)}{9 (3 - \gamma)} < 0 \,, \ \mathcal{J}_{10} = \tfrac{2 (3 + \gamma)}{9 (3 - \gamma)} > 0 \,, \ \mathcal{J}_{11} = \tfrac{243 + 52 \gamma}{72 (3 - \gamma)} > 0 \,.
		\end{aligned}
	\end{equation}
    It is easy to see that for any $- 3 < \gamma \leq 1$,
    \begin{equation}\label{J-lowerbnd}
    	\begin{aligned}
    		\mathcal{J}_i \geq \tfrac{3}{16 ( 3 - \gamma )} \geq \tfrac{3 (3 + \gamma)}{64 (3 - \gamma)} > 0 \ (i \neq 8,9, 10) \,, \ \mathcal{J}_8, - \mathcal{J}_9, \mathcal{J}_{10} \geq \tfrac{3 (3 + \gamma)}{64 (3 - \gamma)} > 0 \,.
    	\end{aligned}
    \end{equation}
    With the above choices \eqref{Choice-para}, the quantities $\mathcal{J}_i$ for $1 \leq i \leq 11$ in \eqref{J-sign} and the positive lower bound of $\mathcal{J}_i$ in \eqref{J-lowerbnd}, the inequalities \eqref{Sum-1st-2}-\eqref{U1}-\eqref{U2}-\eqref{U3}-\eqref{U4}-\eqref{U5} infer that
    \begin{equation}\label{U0-d}
    	\begin{aligned}
    		\mathscr{E}^A_2 ( g_\sigma ) \lesssim & \alpha_* (\delta \hbar)^{-2} \delta^{- 1} l^{ - \frac{3 (3 + \gamma)}{32 (3 - \gamma)} } l^{ \frac{59 - 24 \gamma}{ 24 (3 - \gamma) } } \psi_{6; \frac{15}{8 (3 - \gamma)}, 0} \\
    		& + \alpha_* (\delta \hbar)^{-1} \delta^{- 1} l^{- 100} \mathscr{E}^A_2 ( g_\sigma ) + \Theta (h_\sigma, \varphi_{A, \sigma}) \,,
    	\end{aligned}
    \end{equation}
    and
    \begin{equation}\label{U1-d}
    	\begin{aligned}
    		l^{ \frac{59 - 24 \gamma}{ 24 (3 - \gamma) } } \psi_{6; \frac{15}{8 (3 - \gamma)}, 0} \lesssim & \delta (\delta \hbar)^{-1} l^{ - \frac{3 (3 + \gamma)}{32 (3 - \gamma)} }  l^{ \frac{9 + \gamma}{3 (3 - \gamma)} } \psi_{6; \frac{7}{8 ( 3 - \gamma ) }, \beta_\gamma + \frac{1}{2}} \\
    		& + \delta^{-1} (\delta \hbar)^{-1} l^{ - \frac{3 (3 + \gamma)}{32 (3 - \gamma)} } l^{ \frac{107 - 24 \gamma}{ 12 ( 3 - \gamma ) } } \psi_{6; \frac{7}{4 (3 - \gamma)}, 0} \\
    		& + \delta^{-1} ( \delta \hbar )^{-1} l^{-100} \mathscr{E}^A_2 ( g_\sigma ) + y_{\frac{7}{4 (3 - \gamma)}, 6; \frac{15}{8 (3 - \gamma)}, 0} \,,
    	\end{aligned}
    \end{equation}
    and
    \begin{equation}\label{U2-d}
    	\begin{aligned}
    		l^{ \frac{9 + \gamma}{3 (3 - \gamma)} } \psi_{6; \frac{7}{8 ( 3 - \gamma ) }, \beta_\gamma + \frac{1}{2}} \lesssim & \delta (\delta \hbar)^{-1} l^{ - \frac{3 (3 + \gamma)}{32 (3 - \gamma)} } \psi_{6; - \frac{1}{8 ( 3 - \gamma ) }, 2 ( \beta_\gamma + \frac{1}{2} )} + 2 \epsilon_*^2 l^{ \frac{59 - 24 \gamma}{ 24 (3 - \gamma) } } \psi_{6; \frac{15}{8 (3 - \gamma)}, 0} \\
    		& + 2 C_{\epsilon_*}^2 \delta^{-8} (\delta \hbar)^{-8} l^{ - \frac{3 (3 + \gamma)}{32 (3 - \gamma)} } l^{ \frac{9 + \gamma}{3 (3 - \gamma)} } \psi_{6; \frac{7}{8 ( 3 - \gamma ) }, \beta_\gamma + \frac{1}{2}} \\
    		& + \delta^{-1} ( \delta \hbar )^{-1} l^{-100} \mathscr{E}^A_2 ( g_\sigma ) + y_{\frac{13}{8 (3 - \gamma)}, 6; \frac{7}{ 8 ( 3 - \gamma ) }, \beta_\gamma + \frac{1}{2}}
    	\end{aligned}
    \end{equation}
    for any small $\epsilon_* > 0$ to be determined later, where the inequality \eqref{U-int-1} is used, and
    \begin{equation}\label{U3-d}
    	\begin{aligned}
    		l^{ \frac{107 - 24 \gamma}{ 12 ( 3 - \gamma ) } } \psi_{6; \frac{7}{4 (3 - \gamma)}, 0} \lesssim & \delta (\delta \hbar)^{-1} l^{ - \frac{3 (3 + \gamma)}{32 (3 - \gamma)} } l^{ \frac{297 - 40 \gamma}{36 (3 - \gamma) } } \psi_{6; \frac{3}{4 ( 3 - \gamma ) }, \beta_\gamma + \frac{1}{2}} + \epsilon_*^2 l^{ \frac{ 297 - 40 \gamma }{36 (3 - \gamma)} } \psi_{6; \frac{3}{ 4 ( 3 - \gamma ) }, \beta_\gamma + \frac{1}{2}} \\
    		& + C_{\epsilon_*}^2 \delta^{-2} (\delta \hbar)^{-2} l^{ - \frac{3 (3 + \gamma)}{32 (3 - \gamma)} } l^{ \frac{189 - 52 \gamma}{36 (3 - \gamma)} } \psi_{6; - \frac{1}{ 8 ( 3 - \gamma )}, 2 ( \beta_\gamma + \frac{1}{2} )} \\
    		& + \delta^{-1} ( \delta \hbar )^{-1} l^{-100} \mathscr{E}^A_2 ( g_\sigma ) + y_{\frac{1}{2 (3 - \gamma)}, 6; \frac{7}{4 (3 - \gamma)}, 0}
    	\end{aligned}
    \end{equation}
    for any small $\epsilon_* > 0$ to be determined later, where the inequality \eqref{U-int-2} is utilized, and
    \begin{equation}\label{U4-d}
    	\begin{aligned}
    		l^{ \frac{ 297 - 40 \gamma }{36 (3 - \gamma)} } \psi_{6; \frac{3}{ 4 ( 3 - \gamma ) }, \beta_\gamma + \frac{1}{2}} \lesssim & \delta (\delta \hbar)^{-1} l^{ - \frac{3 (3 + \gamma)}{32 (3 - \gamma)} } \psi_{6; - \frac{1}{ 4 ( 3 - \gamma )}, 2 ( \beta_\gamma + \frac{1}{2} )} + 2 \epsilon_*^2 l^{ \frac{107 - 24 \gamma}{ 12 ( 3 - \gamma ) } } \psi_{6; \frac{7}{4 (3 - \gamma)}, 0} \\
    		& + 2 C_{\epsilon_*}^2 \delta^{-8} (\delta \hbar)^{-8} l^{ - \frac{3 (3 + \gamma)}{32 (3 - \gamma)} } l^{ \frac{ 297 - 40 \gamma }{36 (3 - \gamma)} } \psi_{6; \frac{3}{ 4 ( 3 - \gamma ) }, \beta_\gamma + \frac{1}{2}} \\
    		& + \delta^{-1} ( \delta \hbar )^{-1} l^{-100} \mathscr{E}^A_2 ( g_\sigma ) + y_{\frac{3}{2 (3 - \gamma)}, 6; \frac{3}{ 4 ( 3 - \gamma ) }, \beta_\gamma + \frac{1}{2}}
    	\end{aligned}
    \end{equation}
    for any small $\epsilon_* > 0$ to be determined later, where the inequality \eqref{U-int-4} is employed, and
    \begin{equation}\label{U5-d}
    	\begin{aligned}
    		l^{ \frac{189 - 52 \gamma}{36 (3 - \gamma)} } \psi_{6; - \frac{1}{ 8 ( 3 - \gamma ) }, 2 ( \beta_\gamma + \frac{1}{2} )} \lesssim & \delta (\delta \hbar)^{-1} l^{ - \frac{3 (3 + \gamma)}{32 (3 - \gamma)} } \psi_{6; - \frac{7}{8 ( 3 - \gamma ) }, 3 ( \beta_\gamma + \frac{1}{2} )} \\
    		& + \delta^{-1} (\delta \hbar)^{-1} l^{ - \frac{3 (3 + \gamma)}{32 (3 - \gamma)} } l^{ \frac{ 297 - 40 \gamma }{36 (3 - \gamma)} } \psi_{6; \frac{3}{ 4 ( 3 - \gamma ) }, \beta_\gamma + \frac{1}{2}} \\
    		& + \delta^{-1} ( \delta \hbar )^{-1} l^{-100} \mathscr{E}^A_2 ( g_\sigma ) + y_{ \frac{1}{3 (3 - \gamma)} , 6; - \frac{1}{ 8 ( 3 - \gamma ) }, 2 ( \beta_\gamma + \frac{1}{2} ) } \,,
    	\end{aligned}
    \end{equation}
    where the inequality \eqref{U-int-3} is used.

    Notice that if $\beta \geq 3 (\beta_\gamma + \frac{1}{2})$ involved in the functional $\mathscr{E}_2^A (\cdot)$ in \eqref{E2-lambda}, one has
    \begin{equation}\label{RHS-d}{\small
    	\begin{aligned}
    		\big(\mathrm{RHS\ of\ \eqref{U2-d}} \big)_{1st} + \big(\mathrm{RHS\ of\ \eqref{U4-d}} \big)_{1st} + \big(\mathrm{RHS\ of\ \eqref{U5-d}} \big)_{1st} \lesssim \delta (\delta \hbar)^{-1} l^{ - \frac{3 (3 + \gamma)}{32 (3 - \gamma)} } \mathscr{E}_2^A (g_\sigma) \,,
    	\end{aligned}}
    \end{equation}
    where the symbol $\big(\mathrm{RHS\ of\ (X) } \big)_{1st}$ means the first term in the right-hand side of the inequality (X). Recalling the definition of the functional $\mathscr{E}_{\mathtt{NBE}}^A (\cdot)$ in \eqref{E-NBE} and summing up for the inequalities \eqref{U0-d}-\eqref{U1-d}-\eqref{U2-d}-\eqref{U3-d}-\eqref{U4-d}-\eqref{U5-d}, one gains
    \begin{equation}\label{Bnd-L2-1}
    	\begin{aligned}
    		\mathscr{E}_2^A (g_\sigma) + \mathscr{E}_{\mathtt{NBE}}^A (g_\sigma) \lesssim & \big[ \epsilon_*^2 + ( 1 + C_{\epsilon_*}^2 ) ( \delta \hbar )^{- 16} l^{ - \frac{3 (3 + \gamma)}{32 (3 - \gamma)} } \big] \big[ \mathscr{E}_2^A (g_\sigma) + \mathscr{E}_{\mathtt{NBE}}^A (g_\sigma) \big] \\
    		& + y_{\frac{1}{2 (3 - \gamma)}, 6; \frac{7}{4 (3 - \gamma)}, 0} + y_{\frac{7}{4 (3 - \gamma)}, 6; \frac{15}{8 (3 - \gamma)}, 0} + y_{\frac{13}{8 (3 - \gamma)}, 6; \frac{7}{ 8 ( 3 - \gamma ) }, \beta_\gamma + \frac{1}{2}} \\
    		& +  \Theta ( h_\sigma, \varphi_{A, \sigma} ) + y_{\frac{3}{2 (3 - \gamma)}, 6; \frac{3}{ 4 ( 3 - \gamma ) }, \beta_\gamma + \frac{1}{2}} + y_{ \frac{1}{3 (3 - \gamma)} , 6; - \frac{1}{ 8 ( 3 - \gamma ) }, 2 ( \beta_\gamma + \frac{1}{2} ) }
    	\end{aligned}
    \end{equation}
    for any $\epsilon_* > 0$, where the facts $0 < \delta, \hbar < 1$ and $l^{- 100} \leq l^{- \frac{3 (3 + \gamma)}{32 (3 - \gamma)}}$ for all $- 3 < \gamma \leq 1$ and $l \geq 1$ have been utilized. Initially taking $\epsilon_* > 0$ and then choosing $l \geq 1$ large enough such that
    \begin{equation}\label{l-bnd}
    	\begin{aligned}
    		l \geq O(1) (\delta \hbar)^{- \frac{512 (3 - \gamma)}{3 (3 + \gamma)}} \,,
    	\end{aligned}
    \end{equation}
    where $O (1) \gg 1$ is sufficiently large and independent of $\delta, \hbar$. Remark that $l \geq 1$ in Lemma \ref{Lmm-Dh} is assumed by $ l > O (1) ( \ln \frac{1}{\delta} )^\frac{3 - \gamma}{2} $ for $O (1) \gg 1$ is sufficiently large and independent of $\delta, \hbar$, which can be naturally implied by \eqref{l-bnd}. Recalling the definition of $\digamma_{\delta, l,a, \mathbf{Z}} (\nu^{-1} h_\sigma)$ in \eqref{F-delta-l-a} and $y_{a, \mathbf{Z}; \mathtt{q}, \beta_\sharp}$ in \eqref{psi-y}, one easily obtains
    \begin{equation}\label{Bnd-L2-2}
    	\begin{aligned}
    		& y_{\frac{1}{2 (3 - \gamma)}, 6; \frac{7}{4 (3 - \gamma)}, 0} + y_{\frac{7}{4 (3 - \gamma)}, 6; \frac{15}{8 (3 - \gamma)}, 0} + y_{\frac{13}{8 (3 - \gamma)}, 6; \frac{7}{ 8 ( 3 - \gamma ) }, \beta_\gamma + \frac{1}{2}} \\
    		& + \Theta (h_\sigma, \varphi_{A, \sigma}) + y_{\frac{3}{2 (3 - \gamma)}, 6; \frac{3}{ 4 ( 3 - \gamma ) }, \beta_\gamma + \frac{1}{2}} + y_{ \frac{1}{3 (3 - \gamma)} , 6; - \frac{1}{ 8 ( 3 - \gamma ) }, 2 ( \beta_\gamma + \frac{1}{2} ) } \\
    		\leq & C_l \mathscr{A}_{\mathtt{NBE}}^A (h_\sigma) + C_l \mathscr{B}_{\mathtt{NBE}} (\varphi_{A, \sigma}) + \mathscr{A}^A_2 (h_\sigma) + \mathscr{B}_2 (\varphi_{A, \sigma}) \,,
    	\end{aligned}
    \end{equation}
    where the functionals $\mathscr{A}_{\mathtt{NBE}}^A (h_\sigma)$ and $\mathscr{B}_{\mathtt{NBE}} (\varphi_{A, \sigma})$ are defined in \eqref{As-def} and \eqref{B-def}, respectively.

    As a result, under \eqref{l-bnd}, the inequalities \eqref{Bnd-L2-1} and \eqref{Bnd-L2-2} complete the bound \eqref{L2xv-Unf-Bnd}, and the proof of Lemma \ref{Lmm-L2xv-closed} is then finished.
\end{proof}


\subsection{Existence and uniqueness of weak solution to the problem \eqref{A1} with $\varphi_A = 0$}

In this subsection, based on the uniform a priori estimate in Lemma \ref{Lmm-L2xv-closed}, we will prove the existence and uniqueness of the weak solution to the problem \eqref{A1} (equivalently, the equation \eqref{A3-lambda}) with the boundary source term $\varphi_A = 0$ (equivalently, $\varphi_{A, \sigma} = 0$) by employing the well-known Hahn-Banach Theorem and Lax-Milgram Theorem. The existence result is stated as follows.

\begin{lemma}\label{Lmm-Ext-L2xv}
	Under the same assumptions on Lemma \ref{Lmm-L2xv-closed}, we assume that the source term $h_\sigma$ in \eqref{A3-lambda} satisfies
	\begin{equation*}
		\begin{aligned}
			\mathscr{A}_2^A (h_\sigma) + \mathscr{A}_{\mathtt{NBE}}^A (h_\sigma) < \infty
		\end{aligned}
	\end{equation*}
    and the boundary source term $\varphi_{A, \sigma} = 0$. Then the problem \eqref{A3-lambda} (equivalently \eqref{A1}) admits a unique weak solution $g_\sigma (x, v)$ enjoying the bound \eqref{L2xv-Unf-Bnd} in Lemma \ref{Lmm-L2xv-closed}.
\end{lemma}

\begin{proof}[Proof of Lemma \ref{Lmm-Ext-L2xv}]
We first introduce the Hilbert space
\begin{equation*}
	\begin{aligned}
		\mathbb{X} : = \{ f = f (x,v) | \mathscr{E}_2^A (f) + \mathscr{E}^A_{\mathtt{NBE}} (f) < \infty \}
	\end{aligned}
\end{equation*}
with the inner product
{\small
	\begin{align*}
		\big( f, g\big)_{\mathbb{X}} : = & \delta \hbar \int_0^A \int_{\R^3} \P ( w_{\beta, \vartheta} f ) \P ( w_{\beta, \vartheta} g ) (\delta x + l)^{ - \frac{1 - \gamma}{3 - \gamma} } \d v \d x \\
		& + \int_0^A \int_{\R^3} \P^\perp ( w_{\beta, \vartheta} f ) \P^\perp ( w_{\beta, \vartheta} g ) \nu \d v \d x \\
		& + \sum_{ [ \mathbf{a}; \mathbf{b}, \mathbf{c} ] \in \mathbbm{Index} } l^\mathbf{a} \delta \hbar \int_0^A \int_{\R^3} \P ( w_{\mathbf{c}, 0} f ) \P ( w_{\mathbf{c}, 0} g ) ( l^{- 6} \delta x + l )^{2 \mathbf{b}} ( \delta x + l )^{ - \frac{1 - \gamma}{3 - \gamma} } \d v \d x \\
		& + \sum_{ [ \mathbf{a}; \mathbf{b}, \mathbf{c} ] \in \mathbbm{Index} } l^\mathbf{a} \delta \hbar \int_0^A \int_{\R^3} \P^\perp ( w_{\mathbf{c}, 0} f ) \P^\perp ( w_{\mathbf{c}, 0} g ) \nu \d v \d x \,,
	\end{align*}}
where the index set
\begin{equation*}
	\begin{aligned}
		\mathbbm{Index} = \Big\{ [ \tfrac{59 - 24 \gamma}{24 (3 - \gamma)}; \tfrac{15}{8 (3 - \gamma)}, 0 ] \,, [ \tfrac{9 + \gamma}{3 (3 - \gamma)}; \tfrac{7}{8 (3 - \gamma)}, \beta_\gamma + \tfrac{1}{2} ] \,, [ \tfrac{ 107 - 24 \gamma }{12 (3  - \gamma)}; \tfrac{7}{4 (3 - \gamma)}, 0 ] \,, \\
		[ \tfrac{ 297 - 40 \gamma }{36 (3 - \gamma)}; \tfrac{3}{4 (3 - \gamma)}, \beta_\gamma + \tfrac{1}{2} ] \,, [ \tfrac{189 - 52 \gamma}{36 ( 3 - \gamma )}; - \tfrac{1}{8 (3 - \gamma)}, 2 \beta_\gamma + 1 ] \Big\} \,.
	\end{aligned}
\end{equation*}
We further define the space
\begin{equation*}
	\begin{aligned}
		\mathbb{Y} : = \big\{ h = h (x,v) | \mathscr{A}_2^A (h) + \mathscr{A}_{\mathtt{NBE}}^A (h) < \infty \big\} \,.
	\end{aligned}
\end{equation*}
It is easy to see that $\mathbb{Y}$ is a subspace of $\mathbb{X}$.

For any $\phi \in C_0^\infty ( (0,A) \times \R^3 )$ with $\phi (0, v) |_{v_3 < 0} = 0$, define
\begin{equation*}
	\begin{aligned}
		\chi = - v_3 \partial_x \phi + [ - \hbar \sigma_x v_3 + \nu (v) ] \phi - K_\hbar^* \phi + \mathbf{D}_\hbar^* \phi \,,
	\end{aligned}
\end{equation*}
where $K_\hbar^*$ and $\mathbf{D}_\hbar^*$ are the duality of $K_\hbar$ and $\mathbf{D}_\hbar$, respectively. Define the linear functional $\mathbb{L}_\chi : \mathbb{Y} \to \R$ with $h \mapsto \mathbb{L}_\chi (h) = ( h, \chi )_{ \mathbb{X} } $. By the similar arguments in \eqref{h-bnd-2} or \eqref{Q5}, one knows that $| \mathbb{L}_\chi (h) | \leq C \| h \|_{\mathbb{Y}} \| \chi \|_{\mathbb{X}}$. Note that $\mathbb{Y} \subseteq \mathbb{X}$. The Hahn-Banach Theorem tells us that the linear functional $ \mathbb{L}_\chi $ can be extended to the space $\mathbb{X}$, denoted by $\widetilde{\mathbb{L}}_\chi$, satisfying $ \widetilde{\mathbb{L}}_\chi |_{\mathbb{Y}} = \mathbb{L}_\chi $. Moreover, Lemma \ref{Lmm-L2xv-closed} shows that $( \phi, \chi )_{\mathbb{X}} \geq c_0 \| \phi \|^2_{\mathbb{X}}$ for any $\phi \in C_0^\infty ( (0,A) \times \R^3 )$ with $\phi (0, v) |_{v_3 < 0} = 0$. Then the existence result follows from the Lax-Milgram Theorem.

Next we verify the uniqueness. We focus on the difference problem \eqref{L-lambda-Dif}. Following the same arguments in Lemma \ref{Lmm-L2xv-closed}, one gains
\begin{equation*}
	\begin{aligned}
		\mathscr{E}_2^A ({\vartriangle} g_\sigma) + \mathscr{E}_{\mathtt{NBE}}^A ({\vartriangle} g_\sigma) \leq C_l \big[ \mathscr{A}_2^A ({\vartriangle} h_\sigma) + \mathscr{A}_{\mathtt{NBE}}^A ({\vartriangle} h_\sigma) \big] \,,
	\end{aligned}
\end{equation*}
where ${\vartriangle} g_\sigma = g_{\sigma 2} - g_{\sigma 1}$ and ${\vartriangle} h_\sigma = h_{\sigma 2} - h_{\sigma 1} $. Then the uniqueness holds immediately.
\end{proof}


\section{Weighted $L^\infty_{x,v}$ solution to connection auxiliary equation \eqref{A1}}\label{Sec:UECA}

In this section, based on the weak solution to \eqref{A1} constructed in Lemma \ref{Lmm-Ext-L2xv}, we mainly derive the weighted $L^\infty_{x,v}$ bounds uniformly in $A \geq 1$ for the wean solution to the connection auxiliary problem \eqref{A1} constructed in Lemma \ref{Lmm-Ext-L2xv}. The main idea is to employ the trajectory approach. It is one of the cores in current work. We remark that merely the weighted $L^\infty_{x,v}$ can not be closed, so that we shall close the estimates together with the weighted $L^2_{x,v}$ bound derived in Lemma \eqref{Lmm-L2xv-closed}.

\subsection{Statements of the uniform results}

We mainly focus on the equivalent form \eqref{A3-lambda} of the equation \eqref{A1}, i.e.,
\begin{equation}\label{L-lambda}
	\left\{
	\begin{aligned}
		& \mathscr{L}_\hbar g_\sigma : = v_3 \partial_x g_\sigma + [ - \hbar \sigma_x v_3 + \nu (v) ] g_\sigma - K_\hbar g_\sigma + \mathbf{D}_\hbar g = h_\sigma \,, \\
		& g_\sigma (0, v) |_{v_3 > 0} = (1 - \alpha_*) g_\sigma (0, R_0 v) + \alpha_* \tfrac{M_w (v)}{\sqrt{\M_\sigma (v)}} \int_{v_3' < 0} (- v_3') g_\sigma (0, v') \sqrt{\M_\sigma (v') } \d v' \,, \\
		& g_\sigma (A, v) |_{v_3 < 0} = \varphi_{A, \sigma} (v) \,,
	\end{aligned}
	\right.
\end{equation}
For notation convenience, let $g = \mathscr{L}_\hbar^{-1} (h)$ be the solution to the problem \eqref{L-lambda}. In Lemma \ref{Lmm-Ext-L2xv}, we consider merely the case $\varphi_{A, \sigma} = 0$. Here we study the general case $\varphi_{A, \sigma}$, which will be used to prove the uniqueness of the problem \eqref{KL-Damped}. First, we given the following a priori estimates for the operator $\mathscr{L}_\hbar^{-1} (h)$ uniformly in $A \geq 1$. Note that
\begin{equation*}
	\begin{aligned}
		\partial_x (\sigma_x^\frac{m}{2} g_\sigma ) + \partial_x {\kappa} (x,v) (\sigma_x^\frac{m}{2} g_\sigma ) = \tfrac{1}{v_3} \underbrace{ ( \sigma_x^\frac{m}{2} K_\hbar g_\sigma - \sigma_x^\frac{m}{2} \mathbf{D}_\hbar g_\sigma + \sigma_x^\frac{m}{2} h_\sigma ) }_{: = H (x,v)} \,,
	\end{aligned}
\end{equation*}
which means
\begin{equation}\label{A3-1}
	\begin{aligned}
		\partial_x \big[ e^{{\kappa} (x,v)} \sigma_x^\frac{m}{2} g_\sigma (x,v) \big] = e^{{\kappa} (x,v)} \tfrac{1}{v_3} H (x,v) \,.
	\end{aligned}
\end{equation}
Here $\kappa (x,v)$ is introduced in \eqref{kappa}.

If $v_3 < 0$, by integrating \eqref{A3-1} from $x$ to $A$, it follows that
\begin{equation}\label{v3-n}
	\begin{aligned}
		\sigma_x^\frac{m}{2} g_\sigma (x,v) = e^{{\kappa} (A, v) - {\kappa} (x,v)} \sigma_x^\frac{m}{2} (A, v) \varphi_{A, \sigma} (v) - \int_x^A e^{- [{\kappa} (x,v) - {\kappa} (x', v)]} \tfrac{1}{v_3} H (x', v) \d x' \,.
	\end{aligned}
\end{equation}
If $v_3 > 0$, integrating \eqref{A3-1} from 0 to $x$ and together with the specular reflection boundary condition in \eqref{A3-lambda}, one has
\begin{equation}\label{v3-p-1}
	\begin{aligned}
		\sigma_x^\frac{m}{2} g_\sigma (x,v) = e^{- {\kappa} (x,v)} \sigma_x^\frac{m}{2} (0, v) g_\sigma (0, v) + \int_0^x e^{- [{\kappa} (x,v) - {\kappa} (x', v)]} \tfrac{1}{v_3} H (x', v) \d x' \,.
	\end{aligned}
\end{equation}
By the boundary conditions in \eqref{A3-lambda}, one has
	\begin{align}\label{v3-p-2}
		\no & e^{- {\kappa} (x,v)} \sigma_x^\frac{m}{2} (0, v) g_\sigma (0, v) \\
		\no = & (1 - \alpha_*) e^{- {\kappa} (x,v)} \sigma_x^\frac{m}{2} (0, R_0 v) g_\sigma (0, R_0 v) \\
		& + \alpha_* e^{- {\kappa} (x,v)} \tfrac{M_w (v)}{\sqrt{\M_\sigma (v)}} \int_{v_3' < 0} (- v_3') \tfrac{ \sigma_x^\frac{m}{2} (0, v) }{ \sigma_x^\frac{m}{2} (0, v') } \sigma_x^\frac{m}{2} (0, v') g_\sigma (0, v') \sqrt{\M_\sigma (v') } \d v' \,.
	\end{align}
Here we have utilized the fact $ \sigma_x (0, R_0 v) = \sigma_x (0, v) $ by \eqref{sigma-w-inv}. Note that \eqref{A3-1} indicates
\begin{equation}\label{A3-2}
	\begin{aligned}
		\partial_x \big[ e^{{\kappa} (x, R_0 v)} \sigma_x^\frac{m}{2} g_\sigma (x, R_0 v) \big] = e^{{\kappa} (x, R_0 v)} \tfrac{1}{(R_0 v)_3} H (x, R_0 v) \,.
	\end{aligned}
\end{equation}
Due to $(R_0 v)_3 = - v_3 < 0$, together with the boundary condition in \eqref{A3-lambda}, integrating \eqref{A3-2} from 0 to $A$ implies
	\begin{align*}
		& e^{{\kappa} (A, R_0 v)} \sigma_x^\frac{m}{2} (A, R_0 v) \varphi_{A, \sigma} (R_0 v) - \sigma_x^\frac{m}{2} (0, R_0 v) g_\sigma (0, R_0 v) = \int_0^A e^{{\kappa} (x', R_0 v)} \tfrac{1}{- v_3} H (x', R_0 v) \d x' \,,
	\end{align*}
which means that
\begin{equation}\label{v3-p-3}
	\begin{aligned}
		e^{- {\kappa} (x,v)} \sigma_x^\frac{m}{2} (0, R_0 v) g_\sigma (0, R_0 v) = & e^{{\kappa} (A, R_0 v) - {\kappa} (x,v)} \sigma_x^\frac{m}{2} (A, R_0 v) \varphi_{A, \sigma} (R_0 v) \\
		& + \int_0^A e^{ - [ {\kappa} (x,v) - {\kappa} (x', R_0 v) ] } \tfrac{1}{ v_3} H (x', R_0 v) \d x' \,.
	\end{aligned}
\end{equation}
Furthermore, \eqref{v3-n} shows that for $v_3' < 0$,
\begin{equation*}
	\begin{aligned}
		\sigma_x^\frac{m}{2} (0, v') g_\sigma (0, v') = e^{{\kappa} (A, v')} \sigma_x^\frac{m}{2} (A, v') \varphi_{A, \sigma} (v') - \int_0^A e^{ {\kappa} (x', v') } \tfrac{1}{v_3'} H (x', v') \d x' \,.
	\end{aligned}
\end{equation*}
It thereby holds
\begin{equation}\label{v3-p-4}{\small
	\begin{aligned}
		& e^{- {\kappa} (x,v)} \tfrac{M_w (v)}{\sqrt{\M_\sigma (v)}} \int_{v_3' < 0} (- v_3') \tfrac{ \sigma_x^\frac{m}{2} (0, v) }{ \sigma_x^\frac{m}{2} (0, v') } \sigma_x^\frac{m}{2} (0, v') g_\sigma (0, v') \sqrt{\M_\sigma (v') } \d v' \\
		= & \tfrac{M_w (v)}{\sqrt{\M_\sigma (v)}} \int_{v_3' < 0} (- v_3') \tfrac{ \sigma_x^\frac{m}{2} (0, v) }{ \sigma_x^\frac{m}{2} (0, v') } e^{ {\kappa} (A, v') - {\kappa} (x,v)} \sigma_x^\frac{m}{2} (A, v') \varphi_{A, \sigma} (v') \sqrt{\M_\sigma (v') } \d v' \\
		& - \tfrac{M_w (v)}{\sqrt{\M_\sigma (v)}} \int_{v_3' < 0} \int_0^A (- v_3') \tfrac{ \sigma_x^\frac{m}{2} (0, v) }{ \sigma_x^\frac{m}{2} (0, v') } e^{- [ {\kappa} (x,v) - {\kappa} (x', v') ] } \tfrac{1}{v_3'} H (x', v') \sqrt{\M_\sigma (v') } \d x' \d v' \,.
	\end{aligned}}
\end{equation}
It is thus derived from collecting the equations \eqref{v3-p-1}, \eqref{v3-p-2}, \eqref{v3-p-3} and \eqref{v3-p-4} that for $v_3 > 0$,
{\small
	\begin{align}\label{v3-p}
		\no \sigma_x^\frac{m}{2} g_\sigma (x,v) = & (1 - \alpha_*) e^{ {\kappa} (A, R_0 v) - {\kappa} (x,v)} \sigma_x^\frac{m}{2} (A, R_0 v) \varphi_{A, \sigma} (R_0 v) \\
		\no & + (1 - \alpha_* ) \int_0^A e^{- [ {\kappa} (x,v) - {\kappa} (x', R_0 v)]} \tfrac{1}{ v_3} H (x', R_0 v) \d x' \\
		\no & + \alpha_* \tfrac{M_w (v)}{\sqrt{\M_\sigma (v)}} \int_{v_3' < 0} (- v_3') \tfrac{ \sigma_x^\frac{m}{2} (0, v) }{ \sigma_x^\frac{m}{2} (0, v') } e^{ {\kappa} (A, v') - {\kappa} (x,v)} \\
		\no & \qquad \qquad \qquad \qquad \qquad \quad \qquad \qquad \times \sigma_x^\frac{m}{2} (A, v') \varphi_{A, \sigma} (v') \sqrt{\M_\sigma (v') } \d v' \\
		\no & - \alpha_* \tfrac{M_w (v)}{\sqrt{\M_\sigma (v)}} \int_{v_3' < 0} \int_0^A (- v_3') \tfrac{ \sigma_x^\frac{m}{2} (0, v) }{ \sigma_x^\frac{m}{2} (0, v') } e^{- [ {\kappa} (x,v) - {\kappa} (x', v') ] } \tfrac{1}{v_3'} H (x', v') \sqrt{\M_\sigma (v') } \d x' \d v' \\
		& + \int_0^x e^{- [ {\kappa} (x,v) - {\kappa} (x', v)]} \tfrac{1}{v_3} H (x', v) \d x' \,.
	\end{align}}
Summarily, the equations \eqref{v3-n} and \eqref{v3-p} indicate that
	\begin{align}\label{A3-3}
		\no \sigma_x^\frac{m}{2} g_\sigma = & Y_A (\sigma_x^\frac{m}{2} \varphi_{A, \sigma}) + Z ( \sigma_x^\frac{m}{2} K_\hbar g_\sigma - \sigma_x^\frac{m}{2} \mathbf{D}_\hbar g_\sigma + \sigma_x^\frac{m}{2} h_\sigma) \\
		& + U ( \sigma_x^\frac{m}{2} K_\hbar g_\sigma - \sigma_x^\frac{m}{2} \mathbf{D}_\hbar g_\sigma + \sigma_x^\frac{m}{2} h_\sigma) \,,
	\end{align}
where the operators $Y_A (\cdot)$, $Z (\cdot)$ and $U (\cdot)$ are introduced in \eqref{YAn-f}, \eqref{Rn-f} and \eqref{U-f}, respectively.

Then we can establish the uniform bounds for the problem \eqref{A1} (or equivalently \eqref{A3-lambda}, also \eqref{L-lambda}) in the following lemma.

\begin{lemma}[Uniform weighted $L^\infty_{x,v}$ bounds for \eqref{L-lambda}]\label{Lmm-APE-A3}
	Let $- 3 < \gamma \leq 1$, $0 \leq \alpha_* \leq 1$, $m \geq 1$, the integer $\beta_* \geq 0$ and $0 < \alpha < \mu_\gamma$, where $\mu_\gamma > 0$ is given in Lemma \ref{Lmm-Kh-L2}. The parameters $\delta, \hbar, l, \vartheta, \beta$ are given in Lemma \ref{Lmm-Ext-L2xv}. Assume that the source term $h_\sigma$ and the boundary source term $\varphi_{A, \sigma}$ of the system \eqref{A3-lambda} satisfy
	\begin{equation}\label{Assmp-hsigma}
		\begin{aligned}
			\mathscr{A}^A (h_\sigma) \,, \mathscr{B} ( \varphi_{A, \sigma} ) < \infty \,.
		\end{aligned}
	\end{equation}
	Then the problem \eqref{A1} admits a mild solution $g (x,v) = e^{ - \hbar \sigma (x,v) } g_\sigma (x,v)$, where $ g_\sigma (x,v) $ subjects to equation \eqref{A3-lambda}, such that $g_\sigma = g_\sigma (x,v)$ enjoys the following bounds:
	\begin{equation}\label{Apriori-bnd}
		\begin{aligned}
			\mathscr{E}^A \big( g_\sigma \big) \leq C_l \big( \mathscr{A}^A (h_\sigma) + \mathscr{B} ( \varphi_{A, \sigma} ) \big)
		\end{aligned}
	\end{equation}
	for a constant $C_l > 0$ independent of $A$, $\hbar$ and $\delta$.
	Moreover, let $g_{\sigma i} = g_{\sigma i} (x,v) $ be the solutions with respect to the source terms $h_{\sigma i }$ $(i = 1,2)$, where $ \mathscr{A}^A (h_{\sigma i}) < \infty $ for $ i = 1,2 $. Then $g_{\sigma 2 } - g_{\sigma 1 } $ satisfies
	\begin{equation}\label{Apriori-bnd-diff}
		\begin{aligned}
			\mathscr{E}^A \big( g_{\sigma 2 } - g_{\sigma 1 } \big) \leq C \mathscr{A}^A (h_{\sigma 2 } - h_{\sigma 1 } ) \,.
		\end{aligned}
	\end{equation}
	Here the functionals $\mathscr{E}^A (\cdot)$, $\mathscr{A}^A (\cdot)$ and $\mathscr{B} (\cdot)$ are defined in \eqref{Eg-lambda} and \eqref{Ah-lambda}.
\end{lemma}

\begin{remark}
	As the similar arguments in Lemma II.2 of \cite{DiPerna-Lions-1989-AM}, one knows that the weak solution to \eqref{A1} constructed in Lemma \ref{Lmm-Ext-L2xv} is equivalent to the mild solution to \eqref{A1} with the form \eqref{A3-3}.
\end{remark}

The proof of Lemma \ref{Lmm-APE-A3} will be completed in Subsection \ref{Subsec:Uniform-Est} later.


\subsection{Weighted $L^\infty_{x,v}$ estimates in $Y^\infty_{m, \beta_*, \vartheta} (A)$ space}

In this subsection, we aim at controlling the norm $\LL g_\sigma \RR_{A; m, \beta_*, \vartheta}$ with respect to the space $Y^\infty_{m, \beta_*, \vartheta} (A)$. Moreover, we will also investigate the continuous dependence of $g_\sigma $ with respect to the source term $h_\sigma$. Set $g_{\sigma i} = \mathscr{L}_{ \hbar}^{-1} (h_{\sigma i})$ ($i = 1,2$) for sufficiently small $\hbar \geq 0$. More precisely, the result is stated as follows.

\begin{lemma}\label{Lmm-Y-bnd}
	Let the parameters $\gamma, \alpha_*, A$ be given in Lemma \ref{Lmm-Ext-L2xv} and the parameters $l, \delta, \hbar, \vartheta$ be sufficiently small satisfying the assumptions given in Lemmas \ref{Lmm-ARU}-\ref{Lmm-K-Oprt}-\ref{Lmm-Kh-2-infty}-\ref{Lmm-D-XY}-\ref{Lmm-Dh-L2}. The integer $\beta_* \geq 0$, $m \in \R$, $0 \leq \frac{1}{2} - \mu_\gamma < \alpha' < \frac{1}{2}$, where the constant $\mu_\gamma > 0$ is given in Lemma \ref{Lmm-Kh-L2}. Then there is a constant $C > 0$, independent of $l$, $A$, $\delta$ and $\hbar$, such that
	\begin{equation}\label{LA-g-sigma}
		\begin{aligned}
			& \mathscr{E}_\infty^A ( g_\sigma ) + \LL g_\sigma \RR_{m, \beta_*, \vartheta, \Sigma} \leq C \| z_{\alpha'} \sigma_x^\frac{m}{2} w_{- \gamma, \vartheta} g_\sigma \|_{L^\infty_x L^2_v} + C \big( \mathscr{A}^A_\infty (h_\sigma) + \mathscr{B}_\infty ( \varphi_{A, \sigma} ) \big) \,,
		\end{aligned}
	\end{equation}
	where the functionals $ \mathscr{E}_\infty^A ( \cdot ) $, $ \mathscr{A}^A_\infty ( \cdot ) $ and $ \mathscr{B}_\infty ( \cdot ) $ are defined in \eqref{E-infty}, \eqref{As-def} and \eqref{B-def}, respectively. Moreover, ${\vartriangle} g_{\sigma} (x,v)$ satisfying \eqref{L-lambda-Dif} enjoys the bound
	\begin{equation}\label{LA-g-sigma-Dif}
		\begin{aligned}
			\mathscr{E}_\infty^A ( {\vartriangle} g_\sigma ) + \LL {\vartriangle} g_\sigma \RR_{m, \beta_*, \vartheta, \Sigma} \leq C \| z_{\alpha'} \sigma_x^\frac{m}{2} w_{- \gamma, \vartheta} {\vartriangle} g_\sigma \|_{L^\infty_x L^2_v} + C \mathscr{A}^A_\infty ( {\vartriangle} h_\sigma ) \,.
		\end{aligned}
	\end{equation}
\end{lemma}

\begin{proof}
	By Lemma \ref{Lmm-ARU}, the equation \eqref{A3-3} shows
		\begin{align}\label{Y-bnd-0}
			\no & \LL g_\sigma \RR_{A; m, \beta_*, \vartheta} + \LL g_\sigma \RR_{m, \beta_*, \vartheta, \Sigma} = \| \sigma_x^\frac{m}{2} w_{\beta_*, \vartheta} g_\sigma \|_{ L^\infty_{x,v} } + \| \sigma_x^\frac{m}{2} w_{\beta_*, \vartheta} g_\sigma \|_{L^\infty_\Sigma} \\
			\no \leq & \| w_{\beta_*, \vartheta} Y_A (\sigma_x^\frac{m}{2} \varphi_{A, \sigma} ) \|_{L^\infty_{x,v}} + \sum_{ \Xi \in \{ Z, U \} } \| w_{\beta_*, \vartheta} \Xi ( \sigma_x^\frac{m}{2} ( K_\hbar - \mathbf{D}_\hbar ) g_\sigma + \sigma_x^\frac{m}{2} h_\sigma ) \|_{L^\infty_{x,v}} \\
			\leq & \mathscr{B}_\infty ( \varphi_{A, \sigma} ) + C \| \nu^{-1} ( \sigma_x^\frac{m}{2} ( K_\hbar - \mathbf{D}_\hbar ) g_\sigma + \sigma_x^\frac{m}{2} h_\sigma ) \|_{A; \beta_*, \vartheta} \\
			\no \leq & \mathscr{B}_\infty ( \varphi_{A, \sigma} ) + C \big( \LL \nu^{-1} ( K_\hbar - \mathbf{D}_\hbar ) g_\sigma \RR_{A; m, \beta_*, \vartheta} + \mathscr{A}^A_\infty (h_\sigma) \big) \,.
		\end{align}
	By Lemma \ref{Lmm-K-Oprt} and Lemma \ref{Lmm-D-XY}, one has $ \LL \nu^{-1} ( K_\hbar - \mathbf{D}_\hbar ) g_\sigma \RR_{A; m, \beta_*, \vartheta} \leq C \LL g_\sigma \RR_{A; m, \beta_* - 1, \vartheta} $. It therefore infers that
	\begin{equation*}
		\begin{aligned}
			\LL g_\sigma \RR_{A; m, \beta_*, \vartheta} + \LL g_\sigma \RR_{m, \beta_*, \vartheta, \Sigma} \leq C \LL g_\sigma \RR_{A; m, \beta_* - 1, \vartheta} + C \mathscr{A}^A_\infty (h_\sigma) + \mathscr{B}_\infty ( \varphi_{A, \sigma} ) \,.
		\end{aligned}
	\end{equation*}
	Inductively, for any given integer $\beta_* \geq 0$, it follows that
	\begin{equation}\label{Y-bnd-1}
		\begin{aligned}
			\LL g_\sigma \RR_{A; m, \beta_*, \vartheta} + \LL g_\sigma \RR_{m, \beta_*, \vartheta, \Sigma} \leq C \LL g_\sigma \RR_{A; m, 0, \vartheta} + C \big( \mathscr{A}^A_\infty (h_\sigma) + \mathscr{B}_\infty ( \varphi_{A, \sigma} ) \big) \,,
		\end{aligned}
	\end{equation}
	where we have utilized $ \LL \nu^{-1} h_\sigma \RR_{A; m, i, \vartheta} \leq C \LL \nu^{-1} h_\sigma \RR_{A; m, \beta_*, \vartheta} $ for $0 \leq i \leq \beta_*$.
	
	It remains to dominate the quantity $\LL g_\sigma \RR_{A; m, 0, \vartheta}$. Together with \eqref{A3-3}, the similar arguments in \eqref{Y-bnd-0} indicate that
	\begin{equation*}{\small
		\begin{aligned}
			\LL g_\sigma \RR_{A; m, 0, \vartheta} + \LL g_\sigma \RR_{m, 0, \vartheta, \Sigma} \leq & C \LL \nu^{-1} ( K_\hbar - \mathbf{D}_\hbar ) g_\sigma \RR_{A; m, 0, \vartheta} + C \big( \LL \nu^{-1} h_\sigma \RR_{A; m, 0, \vartheta} + \| \sigma_x^\frac{m}{2} (A, \cdot) w_{0, \vartheta} \varphi_{A, \sigma} \|_{L^\infty_v} \big) \,.
		\end{aligned}}
	\end{equation*}
	Lemmas \ref{Lmm-Kh-2-infty} and Lemma \ref{Lmm-Dh-L2} show that
	\begin{equation*}
		\begin{aligned}
			\LL \nu^{-1} ( K_\hbar - \mathbf{D}_\hbar ) g_\sigma \RR_{A; m, 0, \vartheta} \leq \eta_1 \LL g_\sigma \RR_{A; m, 0, \vartheta} + C_{\eta_1} \| z_{\alpha'} \sigma_x^\frac{m}{2} w_{- \gamma, \vartheta} g_\sigma \|_{L^\infty_x L^2_v}
		\end{aligned}
	\end{equation*}
	for any small $\eta_1 > 0$. Taking $\eta_1 > 0$ such that $C \eta_1 \leq \frac{1}{2}$, it follows that
	\begin{equation}\label{Y-bnd-2}
		\begin{aligned}
			\LL g_\sigma \RR_{A; m, 0, \vartheta} + \LL g_\sigma \RR_{m, 0, \vartheta, \Sigma} \leq C \| z_{\alpha'} \sigma_x^\frac{m}{2} w_{- \gamma, \vartheta} g_\sigma \|_{L^\infty_x L^2_v} + C \big( \mathscr{A}^A_\infty (h_\sigma) + \mathscr{B}_\infty ( \varphi_{A, \sigma} ) \big) \,.
		\end{aligned}
	\end{equation}
	Then \eqref{Y-bnd-1} and \eqref{Y-bnd-2} imply the bound \eqref{LA-g-sigma}.
	
	By the virtue of the similar arguments in \eqref{LA-g-sigma}, ${\vartriangle} g_\sigma$ subjecting to the equation \eqref{L-lambda-Dif} satisfies the bound \eqref{LA-g-sigma-Dif}. Then the proof of Lemma \ref{Lmm-Y-bnd} is finished.
\end{proof}

\subsection{Estimate for $L^\infty_x L^2_v$ norm with weight $ z_{\alpha'} \sigma_x^\frac{m}{2} w_{- \gamma, \vartheta} $}

In this subsection, we will control the norm $\| z_{\alpha'} \sigma_x^\frac{m}{2} w_{- \gamma, \vartheta} g_\sigma \|_{L^\infty_x L^2_v}$ appeared in the right hand side of \eqref{LA-g-sigma}, Lemma \ref{Lmm-Y-bnd}. More precisely, the follow lemma holds.

\begin{lemma}\label{Lmm-Linfty-L2}
	Under the same assumptions in Lemma \ref{Lmm-Y-bnd}, let $\alpha = \frac{1}{2} - \alpha'$, where $\alpha'$ is given in Lemma \ref{Lmm-Y-bnd}. Then there is a constant $C > 0$, independent of $l$, $A$, $\delta$ and $\hbar$, such that
	\begin{equation}\label{Y-bnd-4}
		\begin{aligned}
			\| z_{\alpha'} \sigma_x^\frac{m}{2} w_{- \gamma, \vartheta} g_\sigma \|_{L^\infty_x L^2_v} \leq C \mathscr{E}_{\mathtt{cro}}^A (g_\sigma) \,,
		\end{aligned}
	\end{equation}
	where the functional $\mathscr{E}_{\mathtt{cro}}^A ( \cdot )$ is given in \eqref{E-cro}. Moreover, one similarly has
	\begin{equation}\label{Y-bnd-4-Dif}
		\begin{aligned}
			\| z_{\alpha'} \sigma_x^\frac{m}{2} w_{- \gamma, \vartheta} {\vartriangle} g_\sigma \|_{L^\infty_x L^2_v} \leq C \mathscr{E}_{\mathtt{cro}}^A ( {\vartriangle} g_\sigma ) \,.
		\end{aligned}
	\end{equation}
\end{lemma}

\begin{proof}
	Denote by
	\begin{equation*}
		\begin{aligned}
			& \phi (x) = \int_{\R^3} | z_{\alpha'} \sigma_x^\frac{m}{2} w_{- \gamma, \vartheta} g_\sigma (x,v)|^2 \d v \,, \bar{\phi} (x,v) = \tfrac{\partial}{\partial x} | z_{\alpha'} \sigma_x^\frac{m}{2} w_{- \gamma, \vartheta} g_\sigma |^2 (x, v) \,.
		\end{aligned}
	\end{equation*}
	Then for any $x, y \in \bar{\Omega}_A$,
	\begin{equation}\label{phi}
		\begin{aligned}
			\phi (x) - \phi (y) = \int_y^x \int_{\R^3} \bar{\phi} (x', v) \d v \d x' \,.
		\end{aligned}
	\end{equation}
	
	We claim that
	\begin{equation}\label{Claim-phi}
		\begin{aligned}
			\| \bar{\phi} \|_{L^1_{x,v}} \leq C \mathscr{E}_{\mathtt{cro}}^A (g_\sigma) \,.
		\end{aligned}
	\end{equation}
	Indeed, a direct computation yields
{\small
		\begin{align}\label{M1-M2}
			\no \| \bar{\phi} \|_{L^1_{x,v}} \leq & \underbrace{ 2 \int_0^A \int_{\R^3} |\sigma_x^\frac{m}{2} z_{\alpha'} w_{- \gamma, \vartheta} g_\sigma| \cdot |\sigma_x^\frac{m}{2} z_{\alpha'} w_{- \gamma, \vartheta} \partial_x g_\sigma| \d v \d x }_{: = \scorpio_1 } \\
			& + \underbrace{ m \int_0^A \int_{\R^3} |\sigma_x^\frac{m}{2} z_{\alpha'} w_{- \gamma, \vartheta} g_\sigma| \cdot |\sigma_x^{ \frac{m}{2} - 1 } \sigma_{xx} z_{\alpha'} w_{- \gamma, \vartheta} g_\sigma| \d v \d x }_{: = \scorpio_2} \,.
		\end{align}}
	By the virtue of $|z_1 \sigma_{xx}| \leq |v_3 \sigma_{xx}| \leq \delta \nu (v) \sigma_x$ derived from Lemma \ref{Lmm-sigma}, we have
{\small
		\begin{align}\label{M2-bnd}
			\scorpio_2 \leq \delta |m| \int_0^A \int_{\R^3} \nu (v) \sigma_x^m (x,v) z_{2 \alpha' - 1} w_{- \gamma, \vartheta}^2 (v) g_\sigma^2 (x,v) \d v \d x = & \delta |m| \| \nu^\frac{1}{2} z_{- \alpha} \sigma_x^\frac{m}{2} w_{- \gamma, \vartheta} g_\sigma \|^2_A \,,
		\end{align}}
	where the fact $z_{2 \alpha' - 1} = z_{- 2 \alpha} = z^2_{- \alpha}$ with $\alpha = \frac{1}{2} - \alpha'$ has been used.
	
	Furthermore, the relation $z_{2 \alpha' - 1} = z^2_{- \alpha}$ also implies
	\begin{equation}\label{M1-bnd}{\small
		\begin{aligned}
			\scorpio_1 = & 2 \int_0^A \int_{\R^3} |\nu^\frac{1}{2} z_{- \alpha} \sigma_x^\frac{m}{2} w_{- \gamma, \vartheta} g_\sigma| \cdot |\nu^{- \frac{1}{2}} z_{- \alpha} \sigma_x^\frac{m}{2} z_1 w_{- \gamma, \vartheta} \partial_x g_\sigma| \d v \d x \\
			\leq & 2 \| \nu^\frac{1}{2} z_{- \alpha} \sigma_x^\frac{m}{2} w_{- \gamma, \vartheta} g_\sigma \|_A \| \nu^{- \frac{1}{2}} z_{- \alpha} \sigma_x^\frac{m}{2} z_1 w_{- \gamma, \vartheta} \partial_x g_\sigma \|_A \\
			\leq & \| \nu^\frac{1}{2} z_{- \alpha} \sigma_x^\frac{m}{2} w_{- \gamma, \vartheta} g_\sigma \|^2_A + \| \nu^{- \frac{1}{2}} z_{- \alpha} \sigma_x^\frac{m}{2} z_1 w_{- \gamma, \vartheta} \partial_x g_\sigma \|^2_A \,.
		\end{aligned}}
	\end{equation}
	By using $w_{- \gamma, \vartheta} \leq w_{- \gamma + \beta_\gamma, \vartheta}$, one can conclude the claim \eqref{Claim-phi} from \eqref{M1-M2}-\eqref{M2-bnd}-\eqref{M1-bnd}.
	
	Since the claim \eqref{Claim-phi} holds with the finite value in the right hand side of \eqref{Claim-phi}, the relation \eqref{phi} tells us that $\phi \in C (\bar{\Omega}_A)$. Let $ M_\phi = \max_{x \in \bar{\Omega}_A} \phi (x) \geq 0 $ and $ m_\phi = \min_{x \in \bar{\Omega}_A} \phi (x) \geq 0 $. Then there are two points $x_M$, $x_m \in \bar{\Omega}_A$ such that
	\begin{equation*}
		\begin{aligned}
			0 \leq M_\phi - m_\phi = \phi (x_M) - \phi (x_m) = \int_{x_m}^{x_M} \int_{\R^3} \bar{\phi} (x,v) \d v \d x \leq \| \bar{\phi} \|_{L^1_{x,v}} \,.
		\end{aligned}
	\end{equation*}
	If $M_\phi \geq \frac{3}{2} m_\phi$, one has $ M_\phi \leq \tfrac{2}{3} M_\phi + \| \bar{\phi} \|_{L^1_{x,v}} $, which means that
	\begin{equation}\label{Mphi-1}
		\begin{aligned}
			M_\phi \leq 3 \| \bar{\phi} \|_{L^1_{x,v}} \,.
		\end{aligned}
	\end{equation}
	If $M_\phi < \tfrac{3}{2} m_\phi$, one has
	\begin{equation}\label{Mphi-2}{\small
		\begin{aligned}
			M_\phi < & \tfrac{3}{2} m_\phi \leq \tfrac{3}{2 A} \int_0^A \phi (x) \d x = \tfrac{3}{2 A} \| \nu^{- \frac{1}{2}} z_{\alpha} z_{\alpha'} \nu^\frac{1}{2} z_{- \alpha} \sigma_x^\frac{m}{2} w_{- \gamma, \vartheta} g_\sigma \|^2_A \\
			\leq & C \| \nu^\frac{1}{2} z_{- \alpha} \sigma_x^\frac{m}{2} w_{- \gamma + \beta_\gamma, \vartheta} g_\sigma \|^2_A \,,
		\end{aligned}}
	\end{equation}
	where the last inequality is derived from $A^{-1} \leq 1$ for $A \geq 1$ and $\nu^{- \frac{1}{2}} z_{\alpha} z_{\alpha'} w_{- \gamma, \vartheta} \leq C w_{- \gamma + \beta_\gamma, \vartheta}$. We remark that the index $\beta_\gamma$ given in \eqref{beta-gamma} is required here. Consequently, \eqref{Claim-phi}, \eqref{Mphi-1} and \eqref{Mphi-2} indicate that
	\begin{equation*}
		\begin{aligned}
			\| z_{\alpha'} \sigma_x^\frac{m}{2} w_{- \gamma, \vartheta} g_\sigma \|_{L^\infty_x L^2_v} = M_\phi^\frac{1}{2} \leq C \mathscr{E}_{\mathtt{cro}}^A (g_\sigma) \,.
		\end{aligned}
	\end{equation*}
	Namely, the bound \eqref{Y-bnd-4} holds. Furthermore, the estimate of the bound \eqref{Y-bnd-4-Dif} is similar to that of \eqref{Y-bnd-4}. Thus the proof of Lemma \ref{Lmm-Linfty-L2} is completed.
\end{proof}

\subsection{Estimate for $L^2_{x,v}$ with weight $\nu^\frac{1}{2} \sigma_x^\frac{m}{2} z_{- \alpha} w_{- \gamma + \beta_\gamma, \vartheta}$}

In this subsection, we will dominate the quantity $ \mathscr{E}_{\mathtt{cro}}^A (g_\sigma) = \| \nu^{- \frac{1}{2}} z_{- \alpha} \sigma_x^\frac{m}{2} z_1 w_{- \gamma + \beta_\gamma, \vartheta} \partial_x g_\sigma \|_A $$+ \| \nu^\frac{1}{2} z_{- \alpha} \sigma_x^\frac{m}{2} w_{- \gamma + \beta_\gamma, \vartheta} g_\sigma \|_A $ in the right hand side of \eqref{Y-bnd-4}.

\begin{lemma}\label{Lmm-L2xv-alpha}
	Under the same assumptions in Lemma \ref{Lmm-Linfty-L2}, there is a constant $C > 0$, independent of $A$, $l$, $ \delta$ and $\hbar$, such that
	\begin{equation}\label{C-L2}{\small
		\begin{aligned}
			& [1 - (1 - \alpha_*)^2] \| |v_3|^\frac{1}{2} \sigma_x^\frac{m}{2} z_{- \alpha} w_{- \gamma + \beta_\gamma, \vartheta} g_\sigma \|^2_{L^2_{\Sigma_+^0}} + \| |v_3|^\frac{1}{2} \sigma_x^\frac{m}{2} z_{- \alpha} w_{- \gamma + \beta_\gamma, \vartheta} g_\sigma \|^2_{L^2_{\Sigma_+^A}} + [ \mathscr{E}_{\mathtt{cro}}^A (g_\sigma) ]^2 \\
			\leq & C \| (\delta x + l)^{- \frac{m ( 1 - \gamma ) }{ 2 (3 - \gamma) } } \nu^\frac{1}{2} w_{- \gamma + \beta_\gamma - \mathbf{1}_{m<0} m (1 - \gamma) / 2, \vartheta} g_\sigma \|^2_A + C [ \mathscr{A}^A_{\mathtt{cro}} (h_\sigma) ]^2 + C [ \mathscr{B}_{\mathtt{cro}} ( \varphi_{A, \sigma} ) ]^2 \\
			& + C \alpha_* (\delta \hbar)^{-15} l^{ - \frac{3 (3 + \gamma)}{32 (3 - \gamma)} } \big[ \mathscr{E}^A_2 (g_\sigma) + \mathscr{E}_{\mathtt{NBE}}^A (g_\sigma) \big] + C_l \alpha_* \mathscr{A}_{\mathtt{NBE}}^A (h_\sigma)
		\end{aligned}}
	\end{equation}
	for a constant $C_l > 0$ depending on $l$ but being independent of $A, \delta, \hbar$, where the functionals $ \mathscr{E}_{\mathtt{cro}}^A (\cdot) $, $ \mathscr{E}_2^A (\cdot) $, $ \mathscr{E}_{\mathtt{NBE}}^A (\cdot) $, $ \mathscr{A}_{\mathtt{NBE}}^A (\cdot) $, $ \mathscr{A}^A_{\mathtt{cro}} (\cdot) $ and $ \mathscr{B}_{\mathtt{cro}} ( \cdot ) $ are introduced in \eqref{E-cro}, \eqref{E2-lambda}, \eqref{E-NBE}, \eqref{As-def}, \eqref{As-def} and \eqref{B-def}, respectively. Similarly, there holds
{\small
		\begin{align}\label{L2w-g-sigma-Bnd-Dif}
			\no & [1 - (1 - \alpha_*)^2] \| |v_3|^\frac{1}{2} \sigma_x^\frac{m}{2} z_{- \alpha} w_{- \gamma + \beta_\gamma, \vartheta} {\vartriangle} g_\sigma \|^2_{L^2_{\Sigma_+^0}} + \| |v_3|^\frac{1}{2} \sigma_x^\frac{m}{2} z_{- \alpha} w_{- \gamma + \beta_\gamma, \vartheta} {\vartriangle} g_\sigma \|^2_{L^2_{\Sigma_+^A}} + [ \mathscr{E}_{\mathtt{cro}}^A ({\vartriangle} g_\sigma) ]^2 \\
			\no \leq & C \| (\delta x + l)^{ - \frac{m ( 1 - \gamma ) }{ 2 (3 - \gamma) } } \nu^\frac{1}{2} w_{- \gamma + \beta_\gamma - \mathbf{1}_{m < 0} m (1 - \gamma) / 2, \vartheta} {\vartriangle} g_\sigma \|^2_A + C [ \mathscr{A}^A_{\mathtt{cro}} ({\vartriangle} h_\sigma) ]^2 \\
			& + C \alpha_* (\delta \hbar)^{-15} l^{ - \frac{3 (3 + \gamma)}{32 (3 - \gamma)} } \big[ \mathscr{E}^A_2 ({\vartriangle} g_\sigma) + \mathscr{E}_{\mathtt{NBE}}^A ({\vartriangle} g_\sigma) \big] + C_l \alpha_* \mathscr{A}_{\mathtt{NBE}}^A ({\vartriangle} h_\sigma) \,.
		\end{align}}
\end{lemma}

\begin{proof}
	For $0 \leq \alpha < \mu_\gamma$ with $\mu_\gamma > 0$ given in Lemma \ref{Lmm-Kh-L2}, multiplying \eqref{A3-lambda} by $\sigma_x^m z_{- \alpha}^2 w_{- \gamma + \beta_\gamma, \vartheta}^2 g_\sigma$ and integrating by parts over $(x,v) \in \Omega_A \times \R^3$, one has
{\small
		\begin{align*}
			& \iint_{\Omega_A \times \R^3} \partial_x \big( \tfrac{1}{2} v_3 z_{- \alpha}^2 \sigma_x^m w_{- \gamma + \beta_\gamma, \vartheta}^2 g_\sigma^2 \big) \d v \d x - \tfrac{1}{2} \iint_{\Omega_A \times \R^3} m v_3 z_{- \alpha}^2 \sigma_{xx} \sigma_x^{m-1} w_{- \gamma + \beta_\gamma, \vartheta}^2 g_\sigma^2 \d v \d x \\
			& + \iint_{\Omega_A \times \R^3} [ - \hbar \sigma_x v_3 + \nu (v) ] \sigma_x^m z_{- \alpha}^2 w_{- \gamma + \beta_\gamma, \vartheta}^2 g_\sigma^2 \d v \d x \\
			& - \iint_{\Omega_A \times \R^3} ( K_\hbar - \mathbf{D}_\hbar ) g_\sigma \cdot \sigma_x^m z_{- \alpha}^2 w_{- \gamma + \beta_\gamma, \vartheta}^2 g_\sigma^2 \d v \d x = \iint_{\Omega_A \times \R^3} h_\sigma \cdot \sigma_x^m z_{- \alpha}^2 w_{- \gamma + \beta_\gamma, \vartheta}^2 g_\sigma^2 \d v \d x \,.
		\end{align*}}
	Notice that by Lemma \ref{Lmm-sigma},
	\begin{equation}\label{N1}
		\begin{aligned}
			\iint_{\Omega_A \times \R^3} [ - \hbar \sigma_x v_3 + \nu (v) ] \sigma_x^m z_{- \alpha}^2 w_{- \gamma + \beta_\gamma, \vartheta}^2 g_\sigma^2 \d v \d x \geq c_\hbar \| \nu^\frac{1}{2} z_{- \alpha} \sigma_x^\frac{m}{2} w_{- \gamma + \beta_\gamma, \vartheta} g_\sigma \|^2_A \,,
		\end{aligned}
	\end{equation}
	where $c_\hbar = 1 - c \hbar > \frac{1}{2}$ for sufficiently small $\hbar \geq 0$. Observe that
		\begin{align*}
			& \iint_{\Omega_A \times \R^3} \partial_x \big( \tfrac{1}{2} v_3 z_{- \alpha}^2 \sigma_x^m w_{- \gamma + \beta_\gamma, \vartheta}^2 g_\sigma^2 \big) \d v \d x \\
			= & \tfrac{1}{2} \int_{\R^3} v_3 z_{- \alpha}^2 \sigma_x^m w_{- \gamma + \beta_\gamma, \vartheta}^2 g_\sigma^2 (A, v) \d v - \tfrac{1}{2} \int_{\R^3} v_3 z_{- \alpha}^2 \sigma_x^m (0, v) w_{- \gamma + \beta_\gamma, \vartheta}^2 (v) g_\sigma^2 (0, v) \d v \,.
		\end{align*}
	By the boundary conditions in \eqref{A3-lambda}, it infers that
{\small
		\begin{align*}
			& \int_{\R^3} v_3 z_{- \alpha}^2 \sigma_x^m w_{- \gamma + \beta_\gamma, \vartheta}^2 g_\sigma^2 (A, v) \d v \\
			= & \int_{v_3 > 0} v_3 z_{- \alpha}^2 \sigma_x^m w_{- \gamma + \beta_\gamma, \vartheta}^2 g_\sigma^2 (A, v) \d v + \int_{v_3 < 0} v_3 z_{- \alpha}^2 \sigma_x^m w_{- \gamma + \beta_\gamma, \vartheta}^2 g_\sigma^2 (A, v) \d v \\
			= & \int_{v_3 > 0} | v_3 | z_{- \alpha}^2 \sigma_x^m w_{- \gamma + \beta_\gamma, \vartheta}^2 g_\sigma^2 (A, v) \d v - \int_{v_3 < 0} |v_3| z_{- \alpha}^2 \sigma_x^m w_{- \gamma + \beta_\gamma, \vartheta}^2 \varphi_{A, \sigma}^2 (A, v) \d v \\
			= & \| |v_3|^\frac{1}{2} z_{- \alpha} \sigma_x^\frac{m}{2} w_{- \gamma + \beta_\gamma, \vartheta} g_\sigma \|^2_{L^2_{\Sigma_+^A}} - \| |v_3|^\frac{1}{2} z_{- \alpha} \sigma_x^\frac{m}{2} w_{- \gamma + \beta_\gamma, \vartheta} \varphi_{A, \sigma} \|^2_{L^2_{\Sigma_-^A}} \,.
		\end{align*}}
	Moreover, we choose $w_* (v) = z_{- \alpha}^2 \sigma_x^m (0, v) w^2_{- \gamma + \beta_\gamma, \vartheta} (v)$ in Lemma \ref{Lmm-BEL}. Then the estimate \eqref{BEL-bnd} thereby reduces to
{\small
		\begin{align*}
			& - \int_{\R^3} v_3 z_{- \alpha}^2 \sigma_x^m (0, v) w_{- \gamma + \beta_\gamma, \vartheta}^2 (v) g_\sigma^2 (0, v) \d v \\
			\geq & [ 1 - (1 - \alpha_*)^2 ] \int_{v_3 < 0} | v_3 | z_{- \alpha}^2 \sigma_x^m (0, v) w_{- \gamma + \beta_\gamma, \vartheta}^2 (v) g_\sigma^2 (0,v) \d v \\
			& - C_0 \alpha_* \big( \digamma_{\delta, l,a, \mathbf{Z}} (g_\sigma) + \digamma_{\delta, l,a, \mathbf{Z}} (\nu^{-1} h_\sigma) \big)^2 - C_0 \alpha_* \| |v_3|^\frac{1}{2} \varphi_{A, \sigma} \|^2_{L^2_{\Sigma_-^A}} \,,
		\end{align*}}
	where the functional $\digamma_{\delta, l,a, \mathbf{Z}} (\cdot)$ is defined in \eqref{F-delta-l-a}. Observe that $ 1 \leq \sqrt{C_1} z_{- \alpha} \sigma_x^\frac{m}{2} w_{- \gamma + \beta_\gamma, \vartheta} $ for some universal constant $C_1 > 0$. As a result, one has
{\small
		\begin{align}\label{N2}
			\no & \iint_{\Omega_A \times \R^3} \partial_x \big( \tfrac{1}{2} v_3 z_{- \alpha}^2 \sigma_x^m w_{- \gamma + \beta_\gamma, \vartheta}^2 g_\sigma^2 \big) \d v \d x \\
			\geq & \tfrac{1}{2} [ 1 - (1 - \alpha_*)^2 ] \| |v_3|^\frac{1}{2} z_{- \alpha} \sigma_x^\frac{m}{2} w_{- \gamma + \beta_\gamma, \vartheta} g_\sigma \|^2_{L^2_{\Sigma_+^0}} + \tfrac{1}{2} \| |v_3|^\frac{1}{2} z_{- \alpha} \sigma_x^\frac{m}{2} w_{- \gamma + \beta_\gamma, \vartheta} g_\sigma \|^2_{L^2_{\Sigma_+^A}} \\
			\no & - \tfrac{1}{2} C_0 \alpha_* \big( \digamma_{\delta, l,a, \mathbf{Z}} (g_\sigma) + \digamma_{\delta, l,a, \mathbf{Z}} (\nu^{-1} h_\sigma) \big)^2 - \tfrac{1}{2} ( 1 + C_0  C_1 \alpha_* ) \| |v_3|^\frac{1}{2} z_{- \alpha} \sigma_x^\frac{m}{2} w_{- \gamma + \beta_\gamma, \vartheta} \varphi_{A, \sigma} \|^2_{L^2_{\Sigma_-^A}} \,.
		\end{align}}
	By Lemma \ref{Lmm-sigma}, $|v_3 \sigma_{xx}| \leq \delta \sigma_x \nu (v)$. Then it holds
		\begin{align*}
			& \big| \tfrac{1}{2} \iint_{\Omega_A \times \R^3} m v_3 z_{- \alpha}^2 \sigma_{xx} \sigma_x^{m-1} w_{- \gamma + \beta_\gamma, \vartheta}^2 g_\sigma^2 \d v \d x \big| \\
			\leq & C \delta \| \nu^\frac{1}{2} z_{- \alpha} \sigma_x^\frac{m}{2} w_{- \gamma + \beta_\gamma, \vartheta} g_\sigma \|^2_A \leq \tfrac{c_\hbar}{16} \| \nu^\frac{1}{2} z_{- \alpha} \sigma_x^\frac{m}{2} w_{- \gamma + \beta_\gamma, \vartheta} g_\sigma \|^2_A \,,
		\end{align*}
	where $\delta > 0$ is taken small enough such that $C \delta \leq \frac{c_\hbar}{16}$. Moreover, the H\"older inequality shows
		\begin{align*}
			& \big| \iint_{\Omega_A \times \R^3} h_\sigma \cdot \sigma_x^m z_{- \alpha}^2 w_{- \gamma + \beta_\gamma, \vartheta}^2 g_\sigma^2 \d v \d x \big| \\
			\leq & \tfrac{c_\hbar}{8} \| \nu^\frac{1}{2} z_{- \alpha} \sigma_x^\frac{m}{2} w_{- \gamma + \beta_\gamma, \vartheta} g_\sigma \|^2_A + C \| \nu^{- \frac{1}{2}} z_{- \alpha} \sigma_x^\frac{m}{2} w_{- \gamma + \beta_\gamma, \vartheta} h_\sigma \|^2_A \,,
		\end{align*}
	and
	\begin{equation*}
		\begin{aligned}
			& \big| \iint_{\Omega_A \times \R^3} ( K_\hbar - \mathbf{D}_\hbar ) g_\sigma \cdot \sigma_x^m z_{- \alpha}^2 w_{- \gamma + \beta_\gamma, \vartheta}^2 g_\sigma^2 \d v \d x \big| \\
			\leq & \tfrac{c_\hbar}{4} \| \nu^\frac{1}{2} z_{- \alpha} \sigma_x^\frac{m}{2} w_{- \gamma + \beta_\gamma, \vartheta} g_\sigma \|^2_A + C \| \nu^{- \frac{1}{2}} z_{- \alpha} \sigma_x^\frac{m}{2} w_{- \gamma + \beta_\gamma, \vartheta} ( K_\hbar - \mathbf{D}_\hbar ) g_\sigma \|^2_A \,.
		\end{aligned}
	\end{equation*}
	We thereby establish
		\begin{align}\label{M1}
			\no & [1 - (1 - \alpha_*)^2] \| |v_3|^\frac{1}{2} \sigma_x^\frac{m}{2} z_{- \alpha} w_{- \gamma + \beta_\gamma, \vartheta} g_\sigma \|^2_{L^2_{\Sigma_+}} \\
			\no & + \| |v_3|^\frac{1}{2} \sigma_x^\frac{m}{2} z_{- \alpha} w_{- \gamma + \beta_\gamma, \vartheta} g_\sigma \|^2_{ L^2_{ \Sigma_+^A } } + \| \nu^\frac{1}{2} z_{- \alpha} \sigma_x^\frac{m}{2} w_{- \gamma + \beta_\gamma, \vartheta} g_\sigma \|^2_A \\
			\leq & C \| \nu^{- \frac{1}{2}} z_{- \alpha} \sigma_x^\frac{m}{2} x w_{- \gamma + \beta_\gamma, \vartheta} ( K_\hbar - \mathbf{D}_\hbar ) g_\sigma \|^2_A + C [ \mathscr{A}^A_{\mathtt{cro}} (h_\sigma) ]^2 + C [ \mathscr{B}_{\mathtt{cro}} ( \varphi_{A, \sigma} ) ]^2 \\
			\no & + C \alpha_* \big( \digamma_{\delta, l,a, \mathbf{Z}} (g_\sigma) + \digamma_{\delta, l,a, \mathbf{Z}} (\nu^{-1} h_\sigma) \big)^2 \,.
		\end{align}
	
	Lemma \ref{Lmm-Kh-L2} and Lemma \ref{Lmm-Dh-L2-L2} show that for $0 \leq \alpha < \mu_\gamma$ and $- 3 < \gamma \leq 1$,
	\begin{equation}\label{M2}
		\begin{aligned}
			\| \nu^{- \frac{1}{2}} z_{- \alpha} \sigma_x^\frac{m}{2} w_{- \gamma + \beta_\gamma, \vartheta} ( K_\hbar - \mathbf{D}_\hbar ) g_\sigma \|^2_A \lesssim \| \nu^\frac{1}{2} \sigma_x^\frac{m}{2} w_{- \gamma + \beta_\gamma, \vartheta} g_\sigma \|^2_A \,.
		\end{aligned}
	\end{equation}
	It follow from Lemma \ref{Lmm-sigma} that
	\begin{equation*}{\small
		\begin{aligned}
			\sigma_x^\frac{m}{2} (x,v) \lesssim & \max \{ (\delta x + l)^{- \frac{m ( 1 - \gamma ) }{ 2 (3 - \gamma) } } , ( 1 + | v - \u | )^{- \frac{m (1 - \gamma) }{2}} \} \lesssim & (\delta x + l)^{- \frac{m ( 1 - \gamma ) }{ 2 (3 - \gamma) } } ( 1 + | v - \u | )^{- \frac{m (1 - \gamma) }{2}}
		\end{aligned}}
	\end{equation*}
	for $m < 0 $, and
	\begin{equation*}
		\begin{aligned}
			\sigma_x^\frac{m}{2} (x,v) \lesssim (\delta x + l)^{- \frac{m ( 1 - \gamma ) }{ 2 (3 - \gamma) } }
		\end{aligned}
	\end{equation*}
	for $m \geq 0$. It thereby infers
	\begin{equation}\label{KD-bnd}
		\begin{aligned}
			\| \nu^\frac{1}{2} \sigma_x^\frac{m}{2} w_{- \gamma + \beta_\gamma, \vartheta} g_\sigma \|_A \lesssim \| (\delta x + l)^{ - \frac{m ( 1 - \gamma ) }{ 2 (3 - \gamma) } } \nu^\frac{1}{2} w_{- \gamma + \beta_\gamma - \mathbf{1}_{m < 0} m (1 - \gamma) / 2, \vartheta} g_\sigma \|_A \,.
		\end{aligned}
	\end{equation}

    If we take $a = \frac{15}{8 (3 - \gamma)}$ and $\mathbf{Z} = 6$, then the inequality \eqref{F-delta-l-a-bnd} indicates that
    \begin{equation}\label{F1}
    	\begin{aligned}
    		\digamma_{\delta, l,a, \mathbf{Z}}^2 (g_\sigma) \lesssim & (\delta \hbar)^{-2} l^{- 100} \mathscr{E}^A_2 (g_\sigma) \\
    		& + (\delta \hbar)^{-2} l^{ - \frac{125}{24 ( 3 - \gamma )} } l^{ \frac{59 - 24 \gamma}{24 (3 - \gamma)} } \| (l^{-6} \delta x + l)^\frac{15}{8 (3 - \gamma)} (\delta x + l)^{ - \frac{1 - \gamma}{2 ( 3 - \gamma )} } \nu^\frac{1}{2} g_\sigma \|_A^2 \\
    		\lesssim & (\delta \hbar)^{-2} l^{- 100} \mathscr{E}^A_2 (g_\sigma) + (\delta \hbar)^{-2} l^{ - \frac{125}{24 ( 3 - \gamma )} } \mathscr{E}_{\mathtt{NBE}}^A (g_\sigma) \,,
    	\end{aligned}
    \end{equation}
    where the functionals $\mathscr{E}^A_2 (g_\sigma)$ and $\mathscr{E}_{\mathtt{NBE}}^A (g_\sigma)$ are defined in \eqref{E2-lambda} and \eqref{E-NBE}, respectively. Moreover, under the choice $a = \frac{15}{8 (3 - \gamma)}$ and $\mathbf{Z} = 6$, the quantity $ \digamma_{\delta, l,a, \mathbf{Z}}^2 (\nu^{-1} h_\sigma) $ can be easily bounded by
    \begin{equation}\label{F2}
    	\begin{aligned}
    		\digamma_{\delta, l,a, \mathbf{Z}}^2 (\nu^{-1} h_\sigma) \leq C_l \mathscr{A}_{\mathtt{NBE}}^A (h_\sigma)
    	\end{aligned}
    \end{equation}
    for some constant $C_l > 0$ depending on $l$ but being independent of $A, \delta, \hbar$, where the functional $\mathscr{A}_{\mathtt{NBE}}^A (h_\sigma)$ is given in \eqref{As-def}.

	As a result, the bounds \eqref{M1}, \eqref{M2}, \eqref{KD-bnd}, \eqref{F1} and \eqref{F2} indicate that
		\begin{align}\label{L2w-g-sigma-Bnd}
			\no & [1 - (1 - \alpha_*)^2] \| |v_3|^\frac{1}{2} \sigma_x^\frac{m}{2} z_{- \alpha} w_{- \gamma + \beta_\gamma, \vartheta} g_\sigma \|^2_{L^2_{\Sigma_+^0}} \\
			\no & + \| |v_3|^\frac{1}{2} \sigma_x^\frac{m}{2} z_{- \alpha} w_{- \gamma + \beta_\gamma, \vartheta} g_\sigma \|^2_{ L^2_{ \Sigma_+^A } } + \| \nu^\frac{1}{2} z_{- \alpha} \sigma_x^\frac{m}{2} w_{- \gamma + \beta_\gamma, \vartheta} g_\sigma \|^2_A \\
			\no \leq & C \| (\delta x + l)^{ - \frac{m ( 1 - \gamma ) }{ 2 (3 - \gamma) } } \nu^\frac{1}{2} w_{- \gamma + \beta_\gamma - \mathbf{1}_{m < 0} m (1 - \gamma) / 2, \vartheta} g_\sigma \|^2_A + C [ \mathscr{A}^A_{\mathtt{cro}} (h_\sigma) ]^2 + C [ \mathscr{B}_{\mathtt{cro}} ( \varphi_{A, \sigma} ) ]^2 \\
			& + C \alpha_* (\delta \hbar)^{-2} l^{- 100} \mathscr{E}^A_2 (g_\sigma) + C \alpha_* (\delta \hbar)^{-2} l^{ - \frac{125}{24 ( 3 - \gamma )} } \mathscr{E}_{\mathtt{NBE}}^A (g_\sigma) + C_l \alpha_* \mathscr{A}_{\mathtt{NBE}}^A (h_\sigma)
		\end{align}
	for $0 \leq \alpha < \mu_\gamma$. Here the constant $C > 0$ is independent of $A$, $\delta$, $\hbar$ and $l$.
	
	Recalling \eqref{A3-lambda}, one has $ \partial_x g_\sigma = - [ - \hbar \sigma_x + \tfrac{\nu (v)}{v_3} ] g_\sigma + \tfrac{1}{v_3} ( K_\hbar - \mathbf{D}_\hbar ) g_\sigma + \tfrac{1}{v_3} h_\sigma $, which means
	\begin{equation*}
		\begin{aligned}
			& | \nu^{- \frac{1}{2}} z_{- \alpha} \sigma_x^\frac{m}{2} z_1 w_{- \gamma + \beta_\gamma, \vartheta} \partial_x g_\sigma | \\
			\leq & C (\hbar \nu^{-1} z_1 \sigma_x + \tfrac{z_1}{|v_3|}) |\nu^\frac{1}{2} z_{- \alpha} \sigma_x^\frac{m}{2} w_{- \gamma + \beta_\gamma, \vartheta} g_\sigma | \\
			& + C | \tfrac{z_1}{|v_3|} \nu^{- \frac{1}{2}} \sigma_x^\frac{m}{2} z_{- \alpha} w_{- \gamma + \beta_\gamma, \vartheta} ( K_\hbar - \mathbf{D}_\hbar ) g_\sigma | + C | \tfrac{z_1}{|v_3|} \nu^{- \frac{1}{2}} \sigma_x^\frac{m}{2} z_{- \alpha} w_{- \gamma + \beta_\gamma, \vartheta} h_\sigma | \,.
		\end{aligned}
	\end{equation*}
	Observe that $\frac{z_1}{|v_3|} \leq 1$, and $\sigma_x \leq c \frac{\nu (v)}{|v_3|}$ by Lemma \ref{Lmm-sigma}, which imply $ \hbar \nu^{-1} z_1 \sigma_x + \tfrac{z_1}{|v_3|} \leq c \hbar + 1 $. It therefore follows that
	\begin{equation*}
		\begin{aligned}
			\| \nu^{- \frac{1}{2}} z_{- \alpha} \sigma_x^\frac{m}{2} z_1 w_{- \gamma + \beta_\gamma, \vartheta} & \partial_x g_\sigma \|^2_A \leq C \| \nu^\frac{1}{2} z_{- \alpha} \sigma_x^\frac{m}{2} w_{- \gamma + \beta_\gamma, \vartheta} g_\sigma \|^2_A \\
			& + C \| \nu^{- \frac{1}{2}} \sigma_x^\frac{m}{2} z_{- \alpha} w_{- \gamma + \beta_\gamma, \vartheta} ( K_\hbar - \mathbf{D}_\hbar ) g_\sigma \|^2_A + C [ \mathscr{A}^A_{\mathtt{cro}} (h_\sigma) ]^2 \,.
		\end{aligned}
	\end{equation*}
	Lemma \ref{Lmm-Kh-L2} and Lemma \ref{Lmm-Dh-L2-L2} show that
	$$ \int_{\R^3} | \nu^{- \frac{1}{2}} \sigma_x^\frac{m}{2} z_{- \alpha} w_{- \gamma + \beta_\gamma, \vartheta} ( K_\hbar - \mathbf{D}_\hbar ) g_\sigma |^2 \d v \leq C \int_{\R^3} |\nu^\frac{1}{2} z_{- \alpha} \sigma_x^\frac{m}{2} w_{- \gamma + \beta_\gamma, \vartheta} g_\sigma|^2 \d v \,. $$
	Then, together with \eqref{L2w-g-sigma-Bnd}, one has
	\begin{equation}\label{L2w-pxg-sigma-Bnd}
		\begin{aligned}
			& \| \nu^{- \frac{1}{2}} z_{- \alpha} \sigma_x^\frac{m}{2} z_1 w_{- \gamma + \beta_\gamma, \vartheta} \partial_x g_\sigma \|^2_A \\
			\leq & C \| \nu^\frac{1}{2} z_{- \alpha} \sigma_x^\frac{m}{2} w_{- \gamma + \beta_\gamma, \vartheta} g_\sigma \|^2_A + C \| \nu^{- \frac{1}{2}} \sigma_x^\frac{m}{2} z_{- \alpha} w_{- \gamma + \beta_\gamma, \vartheta} h_\sigma \|^2_A \\
			\leq & C \| (\delta x + l)^{ - \frac{m ( 1 - \gamma ) }{ 2 (3 - \gamma) } } \nu^\frac{1}{2} w_{- \gamma + \beta_\gamma - \mathbf{1}_{m<0} m (1 - \gamma) / 2, \vartheta} g_\sigma \|^2_A + C [ \mathscr{A}^A_{\mathtt{cro}} (h_\sigma) ]^2 + C [ \mathscr{B}_{\mathtt{cro}} ( \varphi_{A, \sigma} ) ]^2 \\
			& + C \alpha_* (\delta \hbar)^{-2} l^{- 100} \mathscr{E}^A_2 (g_\sigma) + C \alpha_* (\delta \hbar)^{-2} l^{ - \frac{125}{24 ( 3 - \gamma )} } \mathscr{E}_{\mathtt{NBE}}^A (g_\sigma) + C_l \alpha_* \mathscr{A}_{\mathtt{NBE}}^A (h_\sigma) \,.
		\end{aligned}
	\end{equation}
	Note that $(\delta \hbar)^{-2} l^{- 100} , (\delta \hbar)^{-2} l^{ - \frac{125}{24 ( 3 - \gamma )} } \leq (\delta \hbar)^{-15} l^{ - \frac{3 (3 + \gamma)}{32 (3 - \gamma)} } $ for all $- 3 < \gamma \leq 1$. Then the bounds \eqref{L2w-g-sigma-Bnd} and \eqref{L2w-pxg-sigma-Bnd} conclude the estimate \eqref{C-L2}.

	Furthermore, as the similar arguments in \eqref{L2w-g-sigma-Bnd} and \eqref{L2w-pxg-sigma-Bnd}, one can easily knows that ${\vartriangle} g_\sigma$ obeying the equation \eqref{L-lambda-Dif} satisfies the bound \eqref{L2w-g-sigma-Bnd-Dif}. Consequently, the proof of Lemma \ref{Lmm-L2xv-alpha} is completed.
\end{proof}

\subsection{Close weighted $L^\infty_{x,v}$ bound: Proof of Lemma \ref{Lmm-APE-A3}}\label{Subsec:Uniform-Est}

In this subsection, we will finish the proof of Lemma \ref{Lmm-APE-A3} based on Lemma \ref{Lmm-Y-bnd}, Lemma \ref{Lmm-Linfty-L2}, Lemma \ref{Lmm-L2xv-alpha} and Lemma \ref{Lmm-L2xv-closed}.

\begin{proof}[Proof of Lemma \ref{Lmm-APE-A3}]
First, it follows from the inequalities \eqref{LA-g-sigma} in Lemma \ref{Lmm-Y-bnd} and \eqref{Y-bnd-4} in Lemma \ref{Lmm-Linfty-L2} that
\begin{equation}\label{X1}
	\begin{aligned}
		\mathscr{E}^A_\infty ( g_\sigma ) + \LL g_\sigma \RR_{m, \beta_*, \vartheta, \Sigma} \lesssim \mathscr{E}^A_{\mathtt{cro}} (g_\sigma) + \mathscr{A}_\infty^A ( h_\sigma ) + \mathscr{B}_\infty ( \varphi_{A, \sigma} ) \,.
	\end{aligned}
\end{equation}
Together with \eqref{C-L2} in Lemma \ref{Lmm-L2xv-alpha} and \eqref{X1}, one then gains
\begin{equation}\label{X2}
	\begin{aligned}
		& \mathscr{E}^A_\infty ( g_\sigma ) + \LL g_\sigma \RR_{m, \beta_*, \vartheta, \Sigma} + \mathscr{E}^A_{\mathtt{cro}} (g_\sigma) \\
		\lesssim & \| (\delta x + l)^{- \frac{m ( 1 - \gamma ) }{ 2 (3 - \gamma) } } \nu^\frac{1}{2} w_{- \gamma + \beta_\gamma - \mathbf{1}_{m<0} m (1 - \gamma) / 2, \vartheta} g_\sigma \|_A \\
		& + \mathscr{A}^A_{ \mathtt{cro} } ( h_\sigma ) + \mathscr{B}_{ \mathtt{cro} } ( \varphi_{A, \sigma} ) + \mathscr{A}_\infty^A ( h_\sigma ) + \mathscr{B}_\infty ( \varphi_{A, \sigma} ) \\
		& + \alpha_*^\frac{1}{2} (\delta \hbar)^{- \frac{15}{2}} l^{ - \frac{3 (3 + \gamma)}{64 (3 - \gamma)} } \big[ \mathscr{E}^A_2 (g_\sigma) + \mathscr{E}_{\mathtt{NBE}}^A (g_\sigma) \big]^\frac{1}{2} + C_l \alpha_*^\frac{1}{2} \big[ \mathscr{A}_{\mathtt{NBE}}^A (h_\sigma) \big]^\frac{1}{2} \,.
	\end{aligned}
\end{equation}
Notice that for $m \geq 1$ and $\beta \geq - \gamma + \beta_\gamma$ in the functional $ \mathscr{E}^A_2 (g_\sigma) $,
\begin{equation}\label{X2'}
	\begin{aligned}
		\| (\delta x + l)^{- \frac{m ( 1 - \gamma ) }{ 2 (3 - \gamma) } } \nu^\frac{1}{2} w_{- \gamma + \beta_\gamma - \mathbf{1}_{m<0} m (1 - \gamma) / 2, \vartheta} g_\sigma \|_A \leq (\delta \hbar)^{ - \frac{1}{2} } \big[ \mathscr{E}^A_2 (g_\sigma) \big]^\frac{1}{2} \,.
	\end{aligned}
\end{equation}
Recalling that $\beta \geq 3 (\beta_\gamma + \frac{1}{2})$ is required in Lemma \ref{Lmm-L2xv-closed}, we take $\beta \geq \max \{ 3 (\beta_\gamma + \frac{1}{2}), \beta_\gamma - \gamma \} = 3 (\beta_\gamma + \frac{1}{2})$ for $- 3 < \gamma \leq 1$, where the last equality holds due to the definition of $\beta_\gamma$ in \eqref{beta-gamma}.

Combining with the bounds \eqref{X2}, \eqref{X2'} and \eqref{L2xv-Unf-Bnd} in Lemma \ref{Lmm-L2xv-closed}, one easily sees that
\begin{equation}\label{X3}
	\begin{aligned}
		& \mathscr{E}^A_\infty ( g_\sigma ) + \LL g_\sigma \RR_{m, \beta_*, \vartheta, \Sigma} + \mathscr{E}^A_{\mathtt{cro}} (g_\sigma) + (\delta \hbar)^{ - \frac{1}{2}  } \big[ \mathscr{E}^A_2 (g_\sigma) + \mathscr{E}_{\mathtt{NBE}}^A (g_\sigma) \big]^\frac{1}{2} \\
		\lesssim & + \mathscr{A}^A_{ \mathtt{cro} } ( h_\sigma ) + \mathscr{B}_{ \mathtt{cro} } ( \varphi_{A, \sigma} ) + \mathscr{A}_\infty^A ( h_\sigma ) + \mathscr{B}_\infty ( \varphi_{A, \sigma} ) \\
		& + \alpha_*^\frac{1}{2} (\delta \hbar)^{- 8} l^{ - \frac{3 (3 + \gamma)}{64 (3 - \gamma)} } (\delta \hbar)^{ - \frac{1}{2}  } \big[ \mathscr{E}^A_2 (g_\sigma) + \mathscr{E}_{\mathtt{NBE}}^A (g_\sigma) \big]^\frac{1}{2} \\
		& + C_l \big[ \mathscr{A}_2^A (h_\sigma) + \mathscr{A}_{\mathtt{NBE}}^A (h_\sigma) \big]^\frac{1}{2} + C_l \big[ \mathscr{B}_2 (\varphi_{A, \sigma}) + \mathscr{B}_{\mathtt{NBE}} (\varphi_{A, \sigma}) \big]^\frac{1}{2} \,.
	\end{aligned}
\end{equation}
Under the assumption $ l \geq O (1) (\delta \hbar)^{ - \frac{512 (3 - \gamma)}{3 (3 + \gamma)} }$ in Lemma \ref{Lmm-L2xv-closed} for sufficiently large $O (1) \gg 1$ independent of $A, \delta, \hbar, l$, the coefficient $\alpha_*^\frac{1}{2} (\delta \hbar)^{- 8} l^{ - \frac{3 (3 + \gamma)}{64 (3 - \gamma)} }$ is actually sufficiently small, so that the quantity $\alpha_*^\frac{1}{2} (\delta \hbar)^{- 8} l^{ - \frac{3 (3 + \gamma)}{64 (3 - \gamma)} } (\delta \hbar)^{ - \frac{1}{2}  } \big[ \mathscr{E}^A_2 (g_\sigma) + \mathscr{E}_{\mathtt{NBE}}^A (g_\sigma) \big]^\frac{1}{2}$ in the right-hand side of \eqref{X3} can be absorbed by that in the left-hand side of \eqref{X3}. Then the bound \eqref{Apriori-bnd} holds.

Moreover, by the similar arguments in \eqref{X3}, we can also easily prove the bound \eqref{Apriori-bnd-diff} for the difference ${\vartriangle} g_\sigma$. Therefore, the proof of Lemma \ref{Lmm-APE-A3} is finished.
\end{proof}


\section{Existence of linear problem \eqref{KL}: Proof of Theorem \ref{Thm-Linear}}\label{Sec:EKLe}

In this section, we mainly justify the existence and uniqueness of the linear problem \eqref{KL}. We first prove the existence result of the damped problem \eqref{KL-Damped} based on the uniform-in-$A$ $L^\infty_{x,v}$ solution to the equation \eqref{A1} with $\varphi_A = 0$ constructed in Lemma \ref{Lmm-APE-A3}. Then by characterizing the structure of the vanishing sources set $\mathbb{VSS}_{\alpha_*}$ defined in \eqref{ASS}, one can remove the artificial damping $\mathbf{D}$ in \eqref{KL-Damped}, which means the existence of the equation \eqref{KL-P0}. At the end, together with the ODE problem \eqref{P0f-eq}, the existence of the equation \eqref{KL} can be obtained, hence, proof of Theorem \ref{Thm-Linear}.


\subsection{Limits from \eqref{A1} with $\varphi_A = 0$ to \eqref{KL-Damped}}

In this section, based on Lemma \ref{Lmm-APE-A3}, we will construct the existence of the problem \eqref{KL-Damped} by taking limit $A \to + \infty$ in the equation \eqref{A1} with $\varphi_A = 0$.

For small $\hbar > 0$ given in Lemma \ref{Lmm-APE-A3}, we initially define a Banach space
\begin{equation}\label{Bh-1}
	\begin{aligned}
		\mathbb{B}_\infty^\hbar : = \{ f (x,v) ; \mathscr{E}^\infty ( e^{\hbar \sigma} f ) < \infty \}
	\end{aligned}
\end{equation}
with the norm
\begin{equation}\label{Bh-2}
	\begin{aligned}
		\| f \|_{ \mathbb{B}_\infty^\hbar } = \mathscr{E}^\infty ( e^{\hbar \sigma} f ) \,,
	\end{aligned}
\end{equation}
where the functional $\mathscr{E}^\infty (\cdot ) $ is given in \eqref{Eg-lambda} with $A = \infty$.

Now we state the existence result on the problem \eqref{KL-Damped} as follows.

\begin{lemma}\label{Lmm-KLe}
	Assume that the parameters $\{ \gamma, \alpha_*, \beta, \beta_*, \alpha, m, \delta, \hbar, \vartheta, l, \rho, T, T_w, \u, u_w \}$ satisfy the hypotheses (PH) given in Theorem \ref{Thm-Linear}. Further assume that the source term $h (x,v) $ in \eqref{KL-Damped} satisfies
	\begin{equation}\label{Asmp-h}
		\begin{aligned}
			\mathscr{A}^\infty ( e^{\hbar \sigma} h ) < \infty \,,
		\end{aligned}
	\end{equation}
    where the functional $\mathscr{A}^\infty (\cdot)$ is defined in \eqref{Ah-lambda} with $A = \infty$. Then the problem \eqref{KL-Damped} admits a unique solution $g = g (x,v) \in \mathbb{B}_\infty^\hbar$ such that
    \begin{equation}
    	\begin{aligned}
    		\| g \|_{ \mathbb{B}_\infty^\hbar } \leq C \mathscr{A}^\infty ( e^{\hbar \sigma} h )
    	\end{aligned}
    \end{equation}
    for some constant $C > 0$.
\end{lemma}

\begin{proof}[Proof of Lemma \ref{Lmm-KLe}]
We mainly consider the approximate problem \eqref{A1} with $\varphi_A (v) = 0$. Then
\begin{equation}\label{Asmp-phi}
	\begin{aligned}
		\mathscr{B} ( e^{\hbar \sigma} \varphi_A ) = 0
	\end{aligned}
\end{equation}
for all $A \geq 1$, where the functional $\mathscr{B} (\cdot)$ is defined in \eqref{Ah-lambda}. The conditions \eqref{Asmp-phi} and \eqref{Asmp-h} tell us that for any $1 \leq A < \infty$,
\begin{equation}
	\begin{aligned}
		\mathscr{A}^A ( e^{\hbar \sigma} h ) + \mathscr{B} ( e^{\hbar \sigma} \varphi_A ) \leq \mathfrak{C}_\flat : = \mathscr{A}^\infty ( e^{\hbar \sigma} h ) < \infty
	\end{aligned}
\end{equation}
Then Lemma \ref{Lmm-APE-A3} indicates that the problem \eqref{A1} with $\varphi_A = 0$ admits a unique solution $g^A = g^A (x,v)$, $(x,v) \in (0, A) \times \R^3$ satisfying
\begin{equation}
	\begin{aligned}
		\mathscr{E}^A ( e^{ \hbar \sigma } g^A ) \leq C ( \mathscr{A}^A ( e^{\hbar \sigma} h ) + \mathscr{B} ( e^{\hbar \sigma} \varphi_A ) ) \leq C \mathfrak{C}_\flat \,,
	\end{aligned}
\end{equation}
where $C > 0$ is independent of $A \geq 1$ and the functional $\mathscr{E}^A (\cdot)$ is given in \eqref{Eg-lambda}.

We now extend $ g^A (x,v) $ as follows
\begin{equation}\label{Extnd-g}
	\begin{aligned}
		\tilde{g}^A (x,v) = \mathbf{1}_{x \in (0, A)} g^A (x,v) \,, (x,v) \in (0, \infty) \times \R^3 \,.
	\end{aligned}
\end{equation}
It is easy to see that $\tilde{g}^A (x,v) \in \mathbb{B}_\infty^\hbar$ with
\begin{equation}\label{gA-h-bnd}
	\begin{aligned}
		\| \tilde{g}^A \|_{ \mathbb{B}_\infty^\hbar } = \mathscr{E}^\infty ( e^{\hbar \sigma} \tilde{g}^A ) \leq C \mathfrak{C}_\flat \,.
	\end{aligned}
\end{equation}
Then there is a $g' = g' (x,v) \in \mathbb{B}_\infty^\hbar $ such that
\begin{equation}\label{Cnv-g1}
	\begin{aligned}
		\tilde{g}^A (x,v) \to g' (x,v) \quad \textrm{weakly in } \mathbb{B}_\infty^\hbar
	\end{aligned}
\end{equation}
as $A \to + \infty$ (in the sense of subsequence, if necessary). Moreover, $g' (x,v)$ obeys
\begin{equation}\label{g-prime-bnd}
	\begin{aligned}
		\| g' \|_{ \mathbb{B}_\infty^\hbar } = \mathscr{E}_1^\infty ( e^{\hbar \sigma} g' ) \leq C \mathfrak{C}_\flat \,.
	\end{aligned}
\end{equation}

For any $1 \leq A_1 < A_2 < + \infty$, let $g^{A_i} (x, v)$ $(i = 1,2)$ be the solution to the problem
\begin{equation*}{\small
	\begin{aligned}
		\left\{
		  \begin{aligned}
		  	& v_3 \partial_x g^{A_i} + \L g^{A_i} + \mathbf{D} g^{A_i} = h \,, \quad x \in (0, A_i) \,, \\
		  	& g^{A_i} (0, v) |_{v_3 > 0} = (1 - \alpha_*) g^{A_i} (0, R_0 v) + \alpha_* \tfrac{M_w (v)}{\sqrt{\M (v)}} \int_{v_3' < 0} (- v_3') g^{A_i} (0, v') \sqrt{\M (v')} \d v' \,, \\
		  	& g^{A_i} (A_i, v) |_{v_3 < 0} = 0 \,,
		  \end{aligned}
		\right.
	\end{aligned}}
\end{equation*}
constructed in Lemma \ref{Lmm-APE-A3}. Then $g^{A_i} (x, v)$ $(i = 1,2)$ subject to the estimates
\begin{equation}\label{gAi-bnd}
	\begin{aligned}
		\mathscr{E}^{A_i} ( e^{ \hbar \sigma } g^{A_i} ) \leq C \mathscr{A}^{A_i} (e^{\hbar \sigma} h ) \leq C \mathfrak{C}_\flat \,.
	\end{aligned}
\end{equation}
Let $\tilde{g}^{A_i}$ be the extension of $g^{A_i}$ as given in \eqref{Extnd-g}. Then
\begin{equation}\label{E1-bnd1}
	\begin{aligned}
		( \tilde{g}^{A_1} - \tilde{g}^{A_2} ) (x,v) = \mathbf{1}_{x \in (0, A_1)} ( g^{A_1} - g^{A_2} ) (x,v) - \mathbf{1}_{x \in [A_1, A_2) } g^{A_2} (x,v) \,.
	\end{aligned}
\end{equation}

It is easy to see that $g^{A_1} - g^{A_2}$ obeys
\begin{equation*}{\footnotesize
	\begin{aligned}
		\left\{
		  \begin{aligned}
		  	& v_3 \partial_x ( \tilde{g}^{A_1} - \tilde{g}^{A_2} ) + \L( \tilde{g}^{A_1} - \tilde{g}^{A_2} ) = 0 \,, \ 0 < x < A_1 \,, \\
		  	& ( \tilde{g}^{A_1} - \tilde{g}^{A_2} ) (0, v) |_{v_3 > 0} = (1 - \alpha_*) ( \tilde{g}^{A_1} - \tilde{g}^{A_2} ) (0, R_0 v) + \alpha_* \tfrac{M_w (v)}{\sqrt{\M (v)}} \int_{v_3' < 0} (- v_3') ( \tilde{g}^{A_1} - \tilde{g}^{A_2} ) (0, v') \sqrt{\M (v')} \d v' \,, \\
		  	& ( \tilde{g}^{A_1} - \tilde{g}^{A_2} ) (A_1, v) |_{v_3 < 0} = - g^{A_2} (A_1, v) \,.
		  \end{aligned}
		\right.
	\end{aligned}}
\end{equation*}
It thereby follows from the similar arguments in Lemma \ref{Lmm-APE-A3} that for any fixed $\hbar' \in (0, \hbar)$,
\begin{equation}\label{EA2A1-bnd1}
	\begin{aligned}
		\mathscr{E}^{A_1} ( e^{\hbar' \sigma} ( g^{A_1} - g^{A_2} ) ) \leq C \mathscr{B} ( e^{\hbar' \sigma (A_1, v) } ( - g^{A_2} (A_1, v) ) ) \leq C \mathscr{B} ( e^{\hbar' \sigma (A_1, v) } g^{A_2} (A_1, v) ) \,.
	\end{aligned}
\end{equation}

Lemma \ref{Lmm-sigma} shows that $ \sigma (A_1, v) \geq c (\delta A_1 + l)^\frac{2}{3 - \gamma} $ uniformly in $v \in \R^3$, which indicates that
\begin{equation}\label{B-1}
	\begin{aligned}
		\sup_{v \in \R^3} e^{ - (\hbar - \hbar') \sigma (A_1, v) } \leq e^{ - c (\hbar - \hbar') (\delta A_1 + l)^\frac{2}{3 - \gamma} } \,.
	\end{aligned}
\end{equation}
It will be frequently used later. In the following, we focus on controlling the quantity $\mathscr{B} ( e^{\hbar' \sigma (A_1, v) } g^{A_2} (A_1, v) )$. Recalling \eqref{Ah-lambda}, one has
\begin{equation}\label{Bbnd-0}
	\begin{aligned}
		\mathscr{B} (\cdot) = \mathscr{B}_\infty (\cdot) + \mathscr{B}_{\mathtt{cro}} (\cdot) + \big[ \mathscr{B}_{\mathtt{NBE}} (\cdot) \big]^\frac{1}{2} + \big[ \mathscr{B}_2 (\cdot) \big]^\frac{1}{2} \,.
	\end{aligned}
\end{equation}

\underline{\em Case 1. Control of $ \mathscr{B}_\infty (e^{\hbar' \sigma (A_1, v) } g^{A_2} (A_1, v)) $.}

By the definition of $\mathscr{B}_\infty (\cdot)$ in \eqref{B-def}, one has
	\begin{align}\label{Bbnd-1}
		\no & \mathscr{B}_\infty (e^{\hbar' \sigma (A_1, v) } g^{A_2} (A_1, v)) = \| e^{\hbar' \sigma (A_1, v) } \sigma_x^\frac{m}{2} (A_1, v) w_{\beta, \vartheta} (v) g^{A_2} (A_1, v) \|_{L^\infty_v} \\
		\no \leq & \sup_{v \in \R^3} e^{ - (\hbar - \hbar') \sigma (A_1, v) } \| e^{\hbar \sigma (A_1, v) } \sigma_x^\frac{m}{2} (A_1, v) w_{\beta, \vartheta} (v) g^{A_2} (A_1, v) \|_{L^\infty_v} \\
		\leq & e^{ - c (\hbar - \hbar') (\delta A_1 + l)^\frac{2}{3 - \gamma} } \LL e^{\hbar \sigma} g^{A_2} \RR_{A_2; m, \beta, \vartheta} = e^{ - c (\hbar - \hbar') (\delta A_1 + l)^\frac{2}{3 - \gamma} } \mathscr{E}_\infty^{A_2} ( e^{\hbar \sigma} g^{A_2} ) \,,
	\end{align}
where the inequality \eqref{B-1} is utilized, and $\mathscr{E}_\infty^{A_2} (\cdot)$ is defined in \eqref{E-infty}.

\underline{\em Case 2. Control of $ \big[ \mathscr{B}_2 (e^{\hbar' \sigma (A_1, v) } g^{A_2} (A_1, v)) \big]^\frac{1}{2} $.}

By the definition of $\mathscr{B}_2 (\cdot)$ in \eqref{B-def}, one has
{\small
	\begin{align*}
		& \big[ \mathscr{B}_2 (e^{\hbar' \sigma (A_1, v) } g^{A_2} (A_1, v)) \big]^\frac{1}{2} \\
		= & \| |v_3|^\frac{1}{2} w_{\beta, \vartheta} e^{\hbar' \sigma (A_1, v) } g^{A_2} (A_1, v) \|_{L^2_{\Sigma_-^{A_1}}} =  \Big( \int_{v_3 > 0} |v_3| w_{\beta, \vartheta}^2 e^{2 \hbar' \sigma (A_1, v) } | g^{A_2} (A_1, v) |^2 \d v \Big)^\frac{1}{2} \\
		= & \Big( \int_{v_3 > 0} \Phi (A_1, v) | z_{\alpha'} (v) \sigma_x^\frac{m}{2} (A_1, v) w_{- \gamma, \vartheta} (v) e^{ \hbar \sigma (A_1, v) } g^{A_2} (A_1, v) |^2 \d v \Big)^\frac{1}{2} \,,
	\end{align*}}
where $m \geq 1$, $\alpha' \in (\tfrac{1}{2} - \mu_\gamma, \frac{1}{2})$ is given in Lemma \ref{Lmm-Y-bnd} below, and
\begin{equation*}
	\begin{aligned}
		\Phi (A_1, v) = |v_3| z_{- \alpha'}^2 (v) (1 + |v|)^{2 \beta + 2 \gamma} \sigma_x^{- m} (A_1, v) e^{ - 2 (\hbar - \hbar') \sigma (A_1, v) } \,.
	\end{aligned}
\end{equation*}
Lemma \ref{Lmm-sigma} implies that for $m \geq 1$,
\begin{equation}\label{B-2}
	\begin{aligned}
		\sigma_x^{- m} (A_1, v) \leq C (\delta A_1 + l)^{ \frac{m (1 - \gamma)}{3 - \gamma} } (1 + |v - \u|)^{ m (1 - \gamma) } \,.
	\end{aligned}
\end{equation}
It is also easy to verify that $ \sigma (A_1, v) \geq c' (1 + |v - \u|^2) $ for some constant $c' > 0$, which means that
\begin{equation}\label{B-3}
	\begin{aligned}
		e^{ - (\hbar- \hbar') \sigma (A_1, v) } \leq e^{ - c' (\hbar - \hbar') } e^{ - c' (\hbar - \hbar') |v - \u|^2 } \leq e^{ - c' (\hbar - \hbar') |v - \u|^2 } \,.
	\end{aligned}
\end{equation}
Then the bounds \eqref{B-1}, \eqref{B-2} and \eqref{B-3} show that
\begin{equation*}{\small
	\begin{aligned}
		\Phi (A_1, v) \leq & C |v_3| z_{- \alpha'}^2 (v) (1 + |v|)^{2 \beta + 2 \gamma + m (1 - \gamma)} e^{ - c' (\hbar - \hbar') |v - \u|^2 } (\delta A_1 + l)^{\frac{m (1 - \gamma)}{3 - \gamma} } e^{ - c (\hbar - \hbar') (\delta A_1 + l)^\frac{2}{3 - \gamma} } \\
		\leq & C e^{ - \frac{c}{2} (\hbar - \hbar') (\delta A_1 + l)^\frac{2}{3 - \gamma} }
	\end{aligned}}
\end{equation*}
uniformly in $v \in \R^3$. It therefore follows that
\begin{equation*}
	\begin{aligned}
		& \big[ \mathscr{B}_2 (e^{\hbar' \sigma (A_1, v) } g^{A_2} (A_1, v)) \big]^\frac{1}{2} \\
		\leq & C e^{ - \frac{c}{4} (\hbar - \hbar') (\delta A_1 + l)^\frac{2}{3 - \gamma} } \Big( \int_{v_3 > 0} | z_{\alpha'} (v) \sigma_x^\frac{m}{2} (A_1, v) w_{- \gamma, \vartheta} (v) e^{ \hbar \sigma (A_1, v) } g^{A_2} (A_1, v) |^2 \d v \Big)^\frac{1}{2} \\
		\leq & C e^{ - \frac{c}{4} (\hbar - \hbar') (\delta A_1 + l)^\frac{2}{3 - \gamma} } \| z_{\alpha'} \sigma_x^\frac{m}{2} w_{- \gamma, \vartheta} e^{\hbar \sigma} g^{A_2} \|_{L^\infty_x L^2_v} \,.
	\end{aligned}
\end{equation*}
From Lemma \ref{Lmm-Linfty-L2} below, it is derived that
\begin{equation*}
	\begin{aligned}
		& \| z_{\alpha'} \sigma_x^\frac{m}{2} w_{- \gamma, \vartheta} e^{\hbar \sigma} g^{A_2} \|_{L^\infty_x L^2_v} \\
		\leq & C \| \nu^{- \frac{1}{2}} z_{- \alpha} \sigma_x^\frac{m}{2} z_1 w_{- \gamma + \beta_\gamma, \vartheta} \partial_x ( e^{\hbar \sigma} g^{A_2} ) \|_{A_2} + C \| \nu^{ \frac{1}{2}} z_{- \alpha} \sigma_x^\frac{m}{2} w_{- \gamma + \beta_\gamma, \vartheta} e^{\hbar \sigma} g^{A_2} \|_{A_2} \\
		= & C \mathscr{E}_{\mathtt{cro}}^{A_2} ( e^{\hbar \sigma} g^{A_2} ) \,,
	\end{aligned}
\end{equation*}
where the functional $\mathscr{E}_{\mathtt{cro}}^{A_2} ( \cdot ) $ is defined in \eqref{E-cro}. Then there hold
\begin{equation}\label{Bbnd-2}
	\begin{aligned}
		\big[ \mathscr{B}_2 (e^{\hbar' \sigma (A_1, v) } g^{A_2} (A_1, v)) \big]^\frac{1}{2} \leq C e^{ - \frac{c}{4} (\hbar - \hbar') (\delta A_1 + l)^\frac{2}{3 - \gamma} } \mathscr{E}_{\mathtt{cro}}^{A_2} ( e^{\hbar \sigma} g^{A_2} ) \,.
	\end{aligned}
\end{equation}

\underline{\em Case 3. Control of $ \mathscr{B}_{\mathtt{cro}} (e^{\hbar' \sigma (A_1, v) } g^{A_2} (A_1, v)) + \big[ \mathscr{B}_{\mathtt{NBE}} (e^{\hbar' \sigma (A_1, v) } g^{A_2} (A_1, v)) \big]^\frac{1}{2} $.}

By employing the similar arguments of controlling the quantity $\big[ \mathscr{B}_2 (e^{\hbar' \sigma (A_1, v) } g^{A_2} (A_1, v)) \big]^\frac{1}{2}$, one obtains
\begin{equation}\label{Bbnd-3}
	\begin{aligned}
		& \mathscr{B}_{\mathtt{cro}} (e^{\hbar' \sigma (A_1, v) } g^{A_2} (A_1, v)) + \big[ \mathscr{B}_{\mathtt{NBE}} (e^{\hbar' \sigma (A_1, v) } g^{A_2} (A_1, v)) \big]^\frac{1}{2} \\
		\leq & C e^{ - \frac{c}{4} (\hbar - \hbar') (\delta A_1 + l)^\frac{2}{3 - \gamma} } \| z_{\alpha'} \sigma_x^\frac{m}{2} w_{- \gamma, \vartheta} e^{\hbar \sigma} g^{A_2} \|_{L^\infty_x L^2_v} \\
		\leq & C e^{ - \frac{c}{4} (\hbar - \hbar') (\delta A_1 + l)^\frac{2}{3 - \gamma} } \mathscr{E}_{\mathtt{cro}}^{A_2} ( e^{\hbar \sigma} g^{A_2} ) \,.
	\end{aligned}
\end{equation}

As a consequence, the relations \eqref{Bbnd-0}, \eqref{Bbnd-1}, \eqref{Bbnd-2} and \eqref{Bbnd-3} reduce to
\begin{equation}\label{EA1A2-bnd3}
	\begin{aligned}
		\mathscr{B} ( e^{\hbar' \sigma (A_1, v) } g^{A_2} ( A_1, v) ) \leq C e^{ - \frac{c}{4} (\hbar - \hbar') (\delta A_1 + l)^\frac{2}{3 - \gamma} } ( \mathscr{E}_\infty^{A_2} ( e^{\hbar \sigma} g^{A_2} ) + \mathscr{E}_{\mathtt{cro}}^{A_2} ( e^{\hbar \sigma} g^{A_2} ) ) \\
		\leq C e^{ - \frac{c}{4} (\hbar - \hbar') (\delta A_1 + l)^\frac{2}{3 - \gamma} } \mathscr{E}^{A_2} ( e^{\hbar \sigma} g^{A_2} ) \leq C \mathfrak{C}_\flat e^{ - \frac{c}{4} (\hbar - \hbar') (\delta A_1 + l)^\frac{2}{3 - \gamma} } \,,
	\end{aligned}
\end{equation}
where the last inequality is derived from the bound \eqref{gAi-bnd}. Collecting the estimates \eqref{EA2A1-bnd1} and \eqref{EA1A2-bnd3}, one gains
\begin{equation}\label{E1-bnd2}
	\begin{aligned}
		\mathscr{E}^\infty ( \mathbf{1}_{x \in (0, A_1)} e^{\hbar' \sigma} ( g^{A_1} - g^{A_2} ) ) = \mathscr{E}^{A_1} ( e^{\hbar' \sigma} ( g^{A_1} - g^{A_2} ) ) \leq C \mathfrak{C}_\flat e^{ - \frac{c}{4} (\hbar - \hbar') (\delta A_1 + l)^\frac{2}{3 - \gamma} }
	\end{aligned}
\end{equation}
for any $\hbar' \in (0, \hbar)$.

Note that
\begin{equation*}
	\begin{aligned}
		& \mathscr{E}^\infty ( \mathbf{1}_{x \in [ A_1, A_2 ) } e^{\hbar' \sigma} g^{A_2} ) = \mathscr{E}^{A_2} ( \mathbf{1}_{x \in [ A_1, A_2 ) } e^{\hbar' \sigma} g^{A_2} ) \\
		\leq & \mathscr{E}^{A_2} ( e^{\hbar \sigma} g^{A_2} ) \sup_{x > 0, v \in \R^3} \mathbf{1}_{x \in [ A_1, A_2 ) } e^{ - (\hbar - \hbar') \sigma (x,v) } \leq C \mathfrak{C}_\flat \sup_{x > 0, v \in \R^3} \mathbf{1}_{x \in [ A_1, A_2 ) } e^{ - (\hbar - \hbar') \sigma (x,v) } \,,
	\end{aligned}
\end{equation*}
where the last inequality is deduced from \eqref{gAi-bnd}. Since $\sigma (x,v) \geq c (\delta x + l)^\frac{2}{3 - \gamma}$ by Lemma \ref{Lmm-sigma}, one has $ \mathbf{1}_{x \in [ A_1, A_2 ) } e^{ - (\hbar - \hbar') \sigma (x,v) } \leq \mathbf{1}_{x \in [ A_1, A_2 ) } e^{ - c (\hbar - \hbar') (\delta x + l)^\frac{2}{3 - \gamma} } \leq e^{ - c (\hbar - \hbar') (\delta A_1 + l)^\frac{2}{3 - \gamma} } $ uniformly in $x$ and $v$. Then
\begin{equation}\label{E1-bnd3}
	\begin{aligned}
		\mathscr{E}^\infty ( \mathbf{1}_{x \in [ A_1, A_2 ) } e^{\hbar' \sigma} g^{A_2} ) \leq C \mathfrak{C}_\flat e^{ - c (\hbar - \hbar') (\delta A_1 + l)^\frac{2}{3 - \gamma} } \,.
	\end{aligned}
\end{equation}
Therefore, the decomposition \eqref{E1-bnd1} and the estimates \eqref{E1-bnd2}-\eqref{E1-bnd3} imply that for any fixed $1 \leq A_1 < A_2 < \infty$ and $\hbar' \in (0, \hbar)$,
\begin{equation}\label{E-bnd1}
	\begin{aligned}
		\mathscr{E}^\infty ( e^{\hbar' \sigma} ( \tilde{g}^{A_1} - \tilde{g}^{A_2} ) ) \leq & \mathscr{E}^\infty ( \mathbf{1}_{x \in (0, A_1)} e^{\hbar' \sigma} ( g^{A_1} - g^{A_2} ) ) + \mathscr{E}^\infty ( \mathbf{1}_{x \in [ A_1, A_2 ) } e^{\hbar' \sigma} g^{A_2} ) \\
		\leq & C \mathfrak{C}_\flat e^{ - \frac{c}{4} (\hbar - \hbar') (\delta A_1 + l)^\frac{2}{3 - \gamma} } \to 0
	\end{aligned}
\end{equation}
as $A_1 \to + \infty$. Moreover, by \eqref{gA-h-bnd},
\begin{equation}\label{E-bnd2}
	\begin{aligned}
		\mathscr{E}^\infty ( e^{\hbar' \sigma} \tilde{g}^A ) \leq \| \tilde{g}^A \|_{ \mathbb{B}_\infty^\hbar } = \mathscr{E}^\infty ( e^{\hbar \sigma} \tilde{g}^A ) \leq C \mathfrak{C}_\flat
	\end{aligned}
\end{equation}
for any $\hbar' \in (0, \hbar)$.

Denote by $\mathbb{B}_\infty^{\hbar'}$ be a Banach space defined as the same way of $\mathbb{B}_\infty^\hbar$ in \eqref{Bh-1}-\eqref{Bh-2}. It is easy to see that $\mathbb{B}_\infty^\hbar \subseteq \mathbb{B}_\infty^{\hbar'}$, i.e., $\| g \|_{ \mathbb{B}_\infty^{\hbar'} } \leq \| g \|_{ \mathbb{B}_\infty^\hbar }$. Then the estimates \eqref{E-bnd1} and \eqref{E-bnd2} tell us that $\{ \tilde{g}^A \}_{A > 1}$ is a bounded Cauchy sequence in $\mathbb{B}_\infty^{\hbar'}$. As a result, there is a unique $g = g (x,v) \in \mathbb{B}_\infty^{\hbar'}$ such that
\begin{equation}\label{Cnv-g2}
	\begin{aligned}
		\tilde{g}^A (x,v) \to g (x,v) \quad \textrm{strongly in } \mathbb{B}_\infty^{\hbar'}
	\end{aligned}
\end{equation}
as $A \to + \infty$. Combining with the convergences \eqref{Cnv-g1} and \eqref{Cnv-g2}, the uniqueness of the limit shows that $ g (x,v) = g' (x,v) $. Moreover, the bound \eqref{g-prime-bnd} infers that $g (x,v)$ satisfies the bound
\begin{equation}\label{g-bnd}
	\begin{aligned}
		\| g \|_{\mathbb{B}_\infty^\hbar} \leq C \mathfrak{C}_\flat = C \mathscr{A}^\infty ( e^{\hbar \sigma} h ) \,.
	\end{aligned}
\end{equation}
Taking limit $A \to + \infty$ in the mild solution form of \eqref{A1} with $\varphi_A = 0$, i.e.,
\begin{equation*}
	\begin{aligned}
		& v_3 \partial_x \tilde{g}^A + \L \tilde{g}^A + \mathbf{D} \tilde{g}^A = h \,, \ 0 < x < A \,, \\
		& \tilde{g}^A (0, v) |_{v_3 > 0} = (1 - \alpha_*) \tilde{g}^A (0, R_0 v) + \alpha_* \tfrac{M_w (v)}{\sqrt{\M (v)}} \int_{v_3' < 0} (- v_3') \tilde{g}^A (0, v') \sqrt{\M (v')} \d v' \,,
	\end{aligned}
\end{equation*}
we easily know that the limit $g (x,v)$ of $\tilde{g}^A (x,v)$ subjects to
\begin{equation}\label{g-equ}
	\begin{aligned}
		& v_3 \partial_x g + \L g + \mathbf{D} g = h \,, x > 0 \,, \\
		& g (0, v) |_{v_3 > 0} = (1 - \alpha_*) g (0, R_0 v) + \alpha_* \tfrac{M_w (v)}{\sqrt{\M (v)}} \int_{v_3' < 0} (- v_3') g (0, v') \sqrt{\M (v')} \d v'
	\end{aligned}
\end{equation}
in the sense of mild solution. Furthermore, the bounds \eqref{B-2} and \eqref{g-bnd} imply
\begin{equation}\label{g-point-bnd}
	\begin{aligned}
		| g (x, v) | \leq & e^{- \hbar \sigma (x,v)} \sigma_x^{- \frac{m}{2}} (x, v) w^{-1}_{\beta, \vartheta} (v) \| g \|_{ \mathbb{B}_\infty^\hbar } \\
		\leq & C \mathfrak{C}_\flat e^{- c \hbar (\delta x + l)^\frac{2}{3 - \gamma} } \cdot C (\delta x + l)^{ \frac{m (1 - \gamma)}{2( 3 - \gamma )} } (1 + |v - \u|)^{- m (1 - \gamma) } w^{-1}_{\beta, \vartheta} (v) \\
		\leq & C' \mathfrak{C}_\flat e^{- \frac{c}{2} \hbar (\delta x + l)^\frac{2}{3 - \gamma} } \to 0
	\end{aligned}
\end{equation}
as $x \to + \infty$. Hence, $ \lim_{x \to + \infty} g (x,v) = 0 $. Therefore, the function $g (x,v)$ solves the problem \eqref{KL-Damped}.

At the end, we justify the uniqueness. Assume that $g_i (x,v)$ $(i = 1,2)$ are both the solutions to the problem \eqref{KL-Damped} enjoying the bound \eqref{g-bnd}. Then, for any fixed $A \geq 1$, the difference $g_1 - g_2$ obeys
\begin{equation*}{\small
	\begin{aligned}
		\left\{
		  \begin{aligned}
		  	& v_3 \partial_x (g_1 - g_2) + \L (g_1 - g_2) + \mathbf{D} (g_1 - g_2) = 0 \,, \ 0 < x < A \,, \\
		  	& (g_1 - g_2) (0, v) |_{v_3 > 0} = (1 - \alpha_*) (g_1 - g_2) (0, R_0 v) + \alpha_* \tfrac{M_w (v)}{\sqrt{\M (v)}} \int_{v_3' < 0} (- v_3') (g_1 - g_2) (0, v') \sqrt{\M (v')} \d v' \,, \\
		  	& (g_1 - g_2) (A, v) |_{v_3 < 0} = \phi_A (v) : = \mathbf{1}_{v_3 < 0} (g_1 - g_2) (A, v) \,.
		  \end{aligned}
		\right.
	\end{aligned}}
\end{equation*}
Following the similar arguments in Lemma \ref{Lmm-APE-A3}, one knows that
\begin{equation*}
	\begin{aligned}
		\mathscr{E}^A ( e^{ \hbar' \sigma } (g_1 - g_2) ) \leq C \mathscr{B} ( e^{\hbar' \sigma (A, v)} \phi_A (v) ) = C \mathscr{B} ( e^{\hbar' \sigma (A, v)} \mathbf{1}_{v_3 < 0} (g_1 - g_2) (A, v) )
	\end{aligned}
\end{equation*}
for any fixed $\hbar' \in (0, \hbar)$. By employing the same arguments of \eqref{EA1A2-bnd3}, one has
\begin{equation*}{\small
	\begin{aligned}
		\mathscr{B} ( e^{\hbar' \sigma (A, v)} \mathbf{1}_{v_3 < 0} (g_1 - g_2) (A, v) ) \leq & C e^{ - \frac{c}{4} (\hbar - \hbar') (\delta A + l)^\frac{2}{3 - \gamma} } \sum_{i=1,2} \mathscr{E}_1^A ( e^{\hbar \sigma} g_i ) \\
		\leq & C \mathscr{A}^\infty (e^{\hbar \sigma} h) e^{ - \frac{c}{4} (\hbar - \hbar') (\delta A + l)^\frac{2}{3 - \gamma} } \,,
	\end{aligned}}
\end{equation*}
where the last inequality is derived from \eqref{g-bnd}. As a result,
\begin{equation*}{\small
	\begin{aligned}
		\mathscr{E}^\infty ( e^{ \hbar' \sigma } (g_1 - g_2) ) = \lim_{A \to + \infty} \mathscr{E}^A ( e^{ \hbar' \sigma } (g_1 - g_2) ) \leq \lim_{A \to + \infty} \big[ C \mathscr{A}^\infty (e^{\hbar \sigma} h) e^{ - \frac{c}{4} (\hbar - \hbar') (\delta A + l)^\frac{2}{3 - \gamma} } \big] = 0 \,,
	\end{aligned}}
\end{equation*}
which means that $g_1 = g_2$. Therefore, the uniqueness holds. The proof of Lemma \ref{Lmm-KLe} is finished.
\end{proof}


\subsection{Existence and uniqueness of \eqref{KL}: Proof of Theorem \ref{Thm-Linear}}

In this subsection, we mainly prove the existence and uniqueness of the linear problem \eqref{KL}, i.e., justify Theorem \ref{Thm-Linear}. As stated in Subsection \ref{Subsec:OESKL} before, the problem \eqref{KL} can be decomposed as the problems \eqref{P0f-eq} and \eqref{KL-P0}. Because \eqref{P0f-eq} admits a unique explicit solution \eqref{P0f-sol}, we only need to prove the existence and uniqueness of the problem \eqref{KL-P0} by applying the solution to \eqref{KL-Damped} constructed in Lemma \ref{Lmm-KLe}. More precisely, by the equivalence between the problems \eqref{KL-Damped} and \eqref{KL-Damped-f} stated in Subsection \ref{Subsec:OESKL} before, Lemma \ref{Lmm-KLe} actually implies the existence and uniqueness of the equation \eqref{KL-Damped-f}. Then by equivalently characterizing the structure of the vanishing sources set $\mathbb{VSS}_{\alpha_*}$ defined in \eqref{ASS}, one can remove the artificial damping $\mathbf{D} (\I - \P^0) f$ in \eqref{KL-Damped-f}, so that the problem \eqref{KL-Damped-f} can be uniquely solved. Here the operator $\P^0$ is given in \eqref{P0}.

\begin{proof}[Proof of Theorem \ref{Thm-Linear}]
	Lemma \ref{Lmm-KLe} implies that the problem \eqref{KL-Damped} admits a unique mild solution $g (x,v)$ with
	\begin{equation*}
		\begin{aligned}
			\mathscr{E}^\infty ( e^{ \hbar \sigma } g ) \lesssim \mathscr{A}^\infty ( e^{\hbar \sigma} h ) \,.
		\end{aligned}
	\end{equation*}
	As the formal illustrations in Subsection \ref{Subsec:OESKL}, the problems \eqref{KL-Damped} and \eqref{KL-Damped-f} are equivalent under the relations \eqref{gh-def}, i.e.,
	\begin{equation*}
		\begin{aligned}
			g (x,v) = & f_* (x,v) - \Upsilon (x) \widetilde{f}_b (v) \,, \\
			h (x,v) = & ( \mathbb{I} - \mathbb{P} ) S(x,v) - v_3 \partial_x \Upsilon (x) \widetilde{f}_b (v) - \Upsilon (x) ( \L + \mathbf{D} ) \widetilde{f}_b (v) \,.
		\end{aligned}
	\end{equation*}
	Then the problem \eqref{KL-Damped-f} admits a unique solution $f_* (x,v)$ enjoying
{\footnotesize
		\begin{align*}
			\mathscr{E}^\infty ( e^{\hbar \sigma} & f_* ) \lesssim \mathscr{E}^\infty ( e^{\hbar \sigma} g ) + \mathscr{E}^\infty ( e^{\hbar \sigma} \Upsilon (x) \widetilde{f}_b (v) ) \lesssim \mathscr{A}^\infty ( e^{\hbar \sigma} h ) + \mathscr{E}^\infty ( e^{\hbar \sigma} \Upsilon (x) \widetilde{f}_b (v) ) \\
			\lesssim & \mathscr{A}^\infty ( e^{\hbar \sigma} (\mathbb{I} - \mathbb{P}) S ) + \mathscr{A}^\infty ( e^{\hbar \sigma} v_3 \Upsilon' (x) \widetilde{f}_b (v) ) + \mathscr{A}^\infty ( e^{\hbar \sigma} \Upsilon (x) ( \L + \mathbf{D} ) \widetilde{f}_b (v) ) + \mathscr{E}^\infty ( e^{\hbar \sigma} \Upsilon (x) \widetilde{f}_b (v) ) \,.
		\end{align*}}
	By using the properties of $K$ and $\mathbf{D}$ in Subsection \ref{Subsec:K} and Subsection \ref{Subsec:D}, respectively, one can obtain that
	\begin{equation*}
		\begin{aligned}
			\mathscr{A}^\infty ( e^{\hbar \sigma} \Upsilon (x) ( \L + \mathbf{D} ) \widetilde{f}_b (v) ) \lesssim \mathscr{A}^\infty ( e^{\hbar \sigma} \Upsilon (x) \nu \widetilde{f}_b (v) ) \,.
		\end{aligned}
	\end{equation*}
	Note that $\Upsilon (x) = \Upsilon' (x) = 0$ for $ x \geq 2 $. Then one easily has
	\begin{equation*}
		\begin{aligned}
			e^{\hbar \sigma} \Upsilon (x) + e^{\hbar \sigma} \Upsilon' (x) \lesssim e^{3 \hbar |v - \u|^2} \,.
		\end{aligned}
	\end{equation*}
	As a result, together with the definitions of $\mathscr{A}^{\infty} (\cdot )$ and $\mathscr{E}^\infty (\cdot)$ in \eqref{Ah-lambda} and \eqref{Eg-lambda}, respectively, one derives that
	\begin{equation*}{\small
		\begin{aligned}
			& \mathscr{A}^\infty ( e^{\hbar \sigma} v_3 \Upsilon' (x) \widetilde{f}_b (v) ) + \mathscr{A}^\infty ( e^{\hbar \sigma} v_3 \Upsilon' (x) \widetilde{f}_b (v) ) \\
			& + \mathscr{A}^\infty ( e^{\hbar \sigma} \Upsilon (x) \nu \widetilde{f}_b (v) ) + \mathscr{E}^\infty ( e^{\hbar \sigma} \Upsilon (x) \widetilde{f}_b (v) ) \\
			\lesssim & \| \widetilde{f}_b \|_{\mathfrak{N}} + \sum_{ \mathbf{c} \in \{ 0, \beta_\gamma + \frac{1}{2}, 2 \beta_\gamma + 1 \} } \big( \| \P w_{1 + \mathbf{c}, 3 \hbar} \widetilde{f}_b \|_{L^2_v} + \| \nu^\frac{1}{2} \P^\perp w_{1 + \mathbf{c}, 3 \hbar} \widetilde{f}_b \|_{L^2_v} \big) \,.
		\end{aligned}}
	\end{equation*}
	It is easy to see that for $\beta \geq 3 (\beta_\gamma + \frac{1}{2})$,
	\begin{equation*}{\small
		\begin{aligned}
			\sum_{ \mathbf{c} \in \{ 0, \beta_\gamma + \frac{1}{2}, 2 \beta_\gamma + 1 \} } \big( \| \P w_{1 + \mathbf{c}, 3 \hbar} \widetilde{f}_b \|_{L^2_v} + \| \nu^\frac{1}{2} \P^\perp w_{1 + \mathbf{c}, 3 \hbar} \widetilde{f}_b \|_{L^2_v} \big) \lesssim \| \widetilde{f}_b \|_{\mathfrak{N}} \,,
		\end{aligned}}
	\end{equation*}
	where the norm $\| \widetilde{f}_b \|_{\mathfrak{N}}$ is defined in \eqref{fb-N}. Collecting the above estimates, one knows that the mild solution $f_*$ to the problem \eqref{KL-Damped-f} enjoys the bound
	\begin{equation}\label{f*-bnd}
		\begin{aligned}
			\mathscr{E}^\infty ( e^{\hbar \sigma} f_* ) \lesssim \mathscr{A}^\infty ( e^{\hbar \sigma} (\mathbb{I} - \mathbb{P}) S ) + \| \widetilde{f}_b \|_{\mathfrak{N}} \,.
		\end{aligned}
	\end{equation}

Define
	\begin{equation}\label{f}
		\begin{aligned}
			f (x, v) = \P^0 f (x,v) + f_* (x,v)
		\end{aligned}
	\end{equation}
where $\P^0 f (x,v)$ is given in \eqref{P0f-sol}. Then $f (x,v)$ is the unique solution to the damped problem \eqref{dd}. By the explicit expression of $\P^0 f$ in \eqref{P0f-sol}, one has
	\begin{equation}
		\begin{aligned}
			e^{ \hbar \sigma (x,v) } \P^0 f (x,v) = - e^{ \hbar \sigma (x,v) } \int_x^{+ \infty} e^{ - \hbar \sigma (x', v) } \tfrac{1}{v_3} e^{ \hbar \sigma (x', v) } \mathbb{P} S (x', v) \d x' \,.
		\end{aligned}
	\end{equation}
	If $(x,v)$ satisfies $\delta x + l \leq 2 (1 + |v - \u|)^{3 - \gamma}$, by the definition of $\sigma (x,v)$ in \eqref{sigma}, the weight $e^{\hbar \sigma (x,v)}$ can be bounded by $C e^{c \hbar |v - \u|^2}$. Due to the smallness of $\hbar > 0$, the weight $e^{\hbar \sigma (x,v)}$ can be absorbed by the exponential decay factors on $v$ variable contained in $ \tfrac{1}{v_3} \mathbb{P} S (x', v) $. Then
	\begin{equation*}{\small
		\begin{aligned}
			| \mathbf{1}_{ \delta x + l \leq 2 (1 + |v - \u|)^{3 - \gamma} } e^{ \hbar \sigma (x,v) } \P^0 f (x,v) | \lesssim \int_x^{+ \infty} e^{- c (\delta x' + l)^\frac{2}{3 - \gamma}} \d x' \LL e^{\hbar \sigma} S \RR_{\infty; m, \beta_*, \vartheta} \lesssim \LL e^{\hbar \sigma} S \RR_{\infty; m, \beta_*, \vartheta}
		\end{aligned}}
	\end{equation*}
	If $\delta x + l \geq 2 (1 + |v - \u|)^{3 - \gamma}$, one has $e^{\hbar \sigma (x,v)} = e^{ 5 (\delta x + l)^\frac{2}{3 - \gamma} }$ and $e^{\hbar \sigma (x',v)} = e^{ 5 (\delta x' + l)^\frac{2}{3 - \gamma} }$ for $x' \geq x$. However, if we still dominate it by the quantity $\LL e^{\hbar \sigma} S \RR_{\infty; m, \beta_*, \vartheta}$, one then has
	\begin{equation*}
		\begin{aligned}
			| \mathbf{1}_{ \delta x + l \geq 2 (1 + |v - \u|)^{3 - \gamma} } e^{ \hbar \sigma (x,v) } \P^0 f (x,v) | \lesssim e^{ 5 (\delta x + l)^\frac{2}{3 - \gamma}} \int_x^{+ \infty} e^{- 5 (\delta x' + l)^\frac{2}{3 - \gamma}} \d x' \LL e^{\hbar \sigma} S \RR_{\infty; m, \beta_*, \vartheta} \,.
		\end{aligned}
	\end{equation*}
	A direct analysis implies $\lim_{x \to + \infty} \tfrac{ e^{ 5 (\delta x + l)^\frac{2}{3 - \gamma}} \int_x^{+ \infty} e^{- 5 (\delta x' + l)^\frac{2}{3 - \gamma}} \d x' }{ \frac{3 - \gamma}{10 \delta} (\delta x + l)^\frac{1 - \gamma}{3 - \gamma} } = 1$, which means that the function $e^{ 5 (\delta x + l)^\frac{2}{3 - \gamma}} \int_x^{+ \infty} e^{- 5 (\delta x' + l)^\frac{2}{3 - \gamma}} \d x'$ grow with the rate $ \frac{3 - \gamma}{10 \delta} (\delta x + l)^\frac{1 - \gamma}{3 - \gamma} $ as $x \to + \infty$. In other words, the quantity $| \mathbf{1}_{ \delta x + l \geq 2 (1 + |v - \u|)^{3 - \gamma} } e^{ \hbar \sigma (x,v) } \P^0 f (x,v) |$ can only be bounded by
\begin{equation*}
    \begin{aligned}
        | \mathbf{1}_{ \delta x + l \geq 2 (1 + |v - \u|)^{3 - \gamma} } e^{ \hbar \sigma (x,v) } \P^0 f (x,v) | \lesssim \LL (\delta x + l)^\frac{1 - \gamma}{3 - \gamma} e^{\hbar \sigma} \mathbb{P} S \RR_{\infty; m, \beta_*, \vartheta} \,.
    \end{aligned}
\end{equation*}
In summary, one has
	\begin{equation*}
		\begin{aligned}
			\mathscr{E}^\infty_\infty ( e^{ \hbar \sigma } \P^0 f ) = \mathscr{E}^\infty_\infty ( e^{ \hbar \sigma } \int_x^{+ \infty} \tfrac{1}{v_3} \mathbb{P} S (x', v) \d x' ) \lesssim \LL (\delta x + l)^\frac{1 - \gamma}{3 - \gamma} e^{\hbar \sigma} \mathbb{P} S \RR_{\infty; m, \beta_*, \vartheta} \,.
		\end{aligned}
	\end{equation*}
From the similar arguments above, one can also derive that
\begin{equation*}
    \begin{aligned}
\mathscr{E}_X^\infty (e^{ \hbar \sigma } \P^0 f) \lesssim \mathscr{E}_X^\infty ((\delta x + l)^\frac{1 - \gamma}{3 - \gamma} e^{\hbar \sigma} S)
    \end{aligned}
\end{equation*}
for $X = \mathtt{cro}, \mathtt{NBE}$ and $2$. Recalling the definition of $\mathscr{E}^\infty (\cdot)$ in \eqref{Eg-lambda}, one gains
\begin{equation}\label{P0f-bnd}
    \begin{aligned}
        \mathscr{E}^\infty (e^{ \hbar \sigma } \P^0 f) \lesssim \mathscr{E}^\infty ((\delta x + l)^\frac{1 - \gamma}{3 - \gamma} e^{\hbar \sigma} S) \,.
    \end{aligned}
\end{equation}
Moreover, together with the similar arguments in \eqref{P0f-bnd}, it also can be derived from the definition of $\widetilde{f}_b$ in \eqref{fb-tilde} that
	\begin{equation}\label{fb-tilde-bnd}
		\begin{aligned}
			\| \widetilde{f}_b \|_{\mathfrak{N}} \leq \| f_b \|_{\mathfrak{N}} + \mathscr{E}^\infty ( (\delta x + l)^\frac{1 - \gamma}{3 - \gamma} e^{\hbar \sigma} S) \,.
		\end{aligned}
	\end{equation}
Note that by the definition of $\mathscr{A}^\infty ( e^{ \hbar \sigma } (\mathbb{I} - \mathbb{P}) S )$ in \eqref{Ah-lambda},
\begin{equation}\label{A-E-bnd}
    \begin{aligned}
        \mathscr{A}^\infty ( e^{ \hbar \sigma } (\mathbb{I} - \mathbb{P}) S ) \lesssim \mathscr{E}^\infty ( e^{ \hbar \sigma } (\mathbb{I} - \mathbb{P}) S ) \lesssim \mathscr{E}^\infty ( (\delta x + l)^\frac{1 - \gamma}{3 - \gamma} e^{\hbar \sigma} S) \,.
    \end{aligned}
\end{equation}
	Therefore, the relations \eqref{f*-bnd}-\eqref{f}-\eqref{P0f-bnd}-\eqref{fb-tilde-bnd}-\eqref{A-E-bnd} indicate that the unique solution $f (x, v)$ to \eqref{KL} enjoys the bound
	\begin{equation*}
		\begin{aligned}
			\mathscr{E}^\infty ( e^{\hbar \sigma} f ) \lesssim \mathscr{E}^\infty ( (\delta x + l)^\frac{1 - \gamma}{3 - \gamma} e^{\hbar \sigma} S) + \| {f}_b \|_{\mathfrak{N}} \,.
		\end{aligned}
	\end{equation*}	
Namely, if the source term $(S, f_b) \in \mathfrak{X}_\gamma^\infty$ (defined in \eqref{X-gamma-infty}), the solution operator $\mathbb{I}_\gamma$ introduced in Remark \ref{Rmk-L-ASS} is well-defined.

Next we prove the existence of \eqref{KL}. Note that the solution $f_*$ to \eqref{KL-Damped-f} is exactly a solution to \eqref{KL-P0} if and only if $\mathbf{D} f_* (x,v) = 0$ for all $x \geq 0$ and $v \in \R^3$. Hence $\P^+ v_3 f_* = 0$ and $\P^0 f_* = 0$.

	Now we remove the artificial damping $ \mathbf{D} f_* $ in \eqref{KL-Damped-f}, i.e.,
	\begin{equation*}{\small
		\left\{	
		\begin{aligned}
			& v_3 \partial_x f_* + \L f_* + \mathbf{D} f_* = ( \mathbb{I} - \mathbb{P} ) S \,, \\
			& f_* (0,v) |_{v_3 > 0} = (1 - \alpha_*) f_* (0, R_0 v) + \alpha_* \mathcal{D}_w f_* (0, v) + \widetilde{f}_b (v) \,, \ \lim_{x \to + \infty} f_* (x,v) = 0
		\end{aligned}
		\right.}
	\end{equation*}
	in certain equivalent forms at $x = 0$.
	
	Multiplying \eqref{KL-Damped-f} by $\psi_3^* (v)$ and integrating over $v \in \R^3$, one has
	\begin{equation*}
		\begin{aligned}
			\tfrac{\d}{\d x} \int_{\R^3} v_3 \psi_3^* f_* \d v = - \ss_+ (\delta x + l)^{- \frac{1 - \gamma}{3 - \gamma}} \int_{\R^3} v_3 \psi_3^* f_* \d v \,,
		\end{aligned}
	\end{equation*}
	which means that $\tfrac{\d}{\d x} \Big( e^{ \frac{\ss_+ (3 - \gamma)}{2 \delta} ( \delta x + l )^\frac{2}{3 - \gamma} } \P^+ v_3 f_* \Big) = 0$. Consequently,
	\begin{equation}\label{P+-equiv}
		\begin{aligned}
			\P^+ v_3 f_* (x, v) = 0 (\forall x \geq 0) \ \ \textrm{if and only if } \ \P^+ v_3 f_* (0, v) = 0 \,.
		\end{aligned}
	\end{equation}
	
	Multiplying by $( \widehat{\mathbb{B}}_3, \widehat{\mathbb{A}}_{13}, \widehat{\mathbb{A}}_{23} )^\top$ and integrating over $v \in \R^3$, one has
	\begin{equation*}
		\begin{aligned}
			\tfrac{\d}{\d x} \int_{\R^3} v_3
			\left(
			    \begin{array}{c}
			    	\widehat{\mathbb{B}}_3 \\
			    	\widehat{\mathbb{A}}_{13}\\
			    	\widehat{\mathbb{A}}_{23}
			    \end{array}
			\right) f_* (x,v) \d v = - \int_{\R^3}
			\left(
			\begin{array}{c}
				\widehat{\mathbb{B}}_3 \\
				\widehat{\mathbb{A}}_{13}\\
				\widehat{\mathbb{A}}_{23}
			\end{array}
			\right) \L f_* (x,v) \d v \,,
		\end{aligned}
	\end{equation*}
	where we have used the facts $ \widehat{\mathbb{A}}_{13}, \widehat{\mathbb{A}}_{23}, \widehat{\mathbb{B}}_3 \in \mathrm{Null}^\perp (\L) $ and $\int_{\R^3} ( \widehat{\mathbb{B}}_3, \widehat{\mathbb{A}}_{13}, \widehat{\mathbb{A}}_{23} )^\top (\mathbb{I} - \mathbb{P} ) S \d v = 0 $. Here the operator $\mathbb{P}$ is defined in \eqref{P-AB}. Notice that by \eqref{psi*-basis}-\eqref{AB-1}-\eqref{AB-2},
	\begin{equation*}
		\begin{aligned}
			\int_{\R^3}
			\left(
			\begin{array}{c}
				\widehat{\mathbb{B}}_3 \\
				\widehat{\mathbb{A}}_{13}\\
				\widehat{\mathbb{A}}_{23}
			\end{array}
			\right) \L f_* (x,v) \d v = \int_{\R^3}
			\left(
			\begin{array}{c}
				{\mathbb{B}}_3 \\
				{\mathbb{A}}_{13}\\
				{\mathbb{A}}_{23}
			\end{array}
			\right) f_* (x,v) \d v = \int_{\R^3} v_3
			\left(
			\begin{array}{c}
				\tfrac{\sqrt{\rho}}{2 \sqrt{T}} \psi_0^* \\
				\sqrt{\tfrac{\rho}{T}} \psi_1^*\\
				\sqrt{\tfrac{\rho}{T}} \psi_2^*
			\end{array}
			\right) f_* (x,v) \d v \,.
		\end{aligned}
	\end{equation*}
	For notational simplicity, we define
	\begin{equation*}
		\begin{aligned}
			\mathfrak{F} (x) = \int_{\R^3} v_3
			\left(
			\begin{array}{c}
				\widehat{\mathbb{B}}_3 \\
				\widehat{\mathbb{A}}_{13}\\
				\widehat{\mathbb{A}}_{23}
			\end{array}
			\right) f_* (x,v) \d v \,, \ \quad \mathfrak{G} (x) = \int_{\R^3} v_3
			\left(
			\begin{array}{c}
				\psi_0^* \\
				\psi_1^*\\
				\psi_2^*
			\end{array}
			\right) f_* (x,v) \d v \,.
		\end{aligned}
	\end{equation*}
	By the fact $\lim_{x \to + \infty} f_* (x, v) = 0$ one knows that $\lim_{x \to + \infty} \mathfrak{F} (x) = 0$. As a result,
	\begin{equation}\label{AB-eq1}
		\begin{aligned}
			\tfrac{\d}{\d x} \mathfrak{F} (x) = - \mathrm{diag} ( \tfrac{\sqrt{\rho}}{2 \sqrt{T}} , \sqrt{\tfrac{\rho}{T}} , \sqrt{\tfrac{\rho}{T}} ) \mathfrak{G} (x) \,, \ \ \lim_{x \to + \infty} \mathfrak{F} (x) = 0_3 \,,
		\end{aligned}
	\end{equation}
	where $0_n = (\underbrace{0,\cdots,0}_{n})^\top \in \R^n$. We then multiply \eqref{KL-Damped-f} by $(\tfrac{\sqrt{\rho}}{2 \sqrt{T}} \psi_0^*, \sqrt{\tfrac{\rho}{T}} \psi_1^*, \sqrt{\tfrac{\rho}{T}} \psi_2^*)^\top$ and integrate the resultant over $v \in \R^3$. It therefore holds
	\begin{equation}\label{AB-eq2}
		\begin{aligned}
			\tfrac{\d}{\d x} \big[ \mathrm{diag} ( \tfrac{\sqrt{\rho}}{2 \sqrt{T}} , \sqrt{\tfrac{\rho}{T}} , \sqrt{\tfrac{\rho}{T}} ) \mathfrak{G} (x) \big] = - ( \delta x + l )^{ - \frac{1 - \gamma}{3 - \gamma} } \mathrm{diag} (\lambda_0, \lambda_1, \lambda_2) \mathfrak{F} (x) \,,
		\end{aligned}
	\end{equation}
	where
	\begin{equation*}
		\begin{aligned}
			\lambda_0 = \tfrac{\sqrt{\rho}}{2 \sqrt{T}} \ss_0 ( \int_{\R^3} \widehat{\mathbb{B}}_3 \mathbb{B}_3 \d v )^{-1} > 0 \,, \ \lambda_i = \sqrt{\tfrac{\rho}{T}} \ss_0 ( \int_{\R^3} \widehat{\mathbb{A}}_{i3} \mathbb{A}_{i3} \d v )^{-1} > 0 \ (i = 1,2) \,.
		\end{aligned}
	\end{equation*}
	Then the equations \eqref{AB-eq1} and \eqref{AB-eq2} are equivalent to
	\begin{equation}\label{F-eq}
		\begin{aligned}
			\tfrac{\d^2}{\d x^2} \mathfrak{F} (x) = ( \delta x + l )^{ - \frac{1 - \gamma}{3 - \gamma} } \mathrm{diag} (\lambda_0, \lambda_1, \lambda_2) \mathfrak{F} (x) \,, \ \ \lim_{x \to + \infty} \mathfrak{F} (x) = 0_3 \,,
		\end{aligned}
	\end{equation}
	which is a second order ODE system.
	
	We then claim that
	\begin{equation}\label{Claim-FG}
		\begin{aligned}
			\mathfrak{F} (x) = 0 \ (\forall x \geq 0) \ \ \textrm{if and only if } \ \mathfrak{F} (0) = 0_3 \,.
		\end{aligned}
	\end{equation}
	Indeed, we first rewrite the equation \eqref{F-eq} in the components form
	\begin{equation}\label{F-eq-comp}
		\begin{aligned}
			\tfrac{\d^2}{\d x^2} \mathfrak{F}_i (x) = \lambda_i ( \delta x + l )^{ - \frac{1 - \gamma}{3 - \gamma} } \mathfrak{F}_i (x) \,, \ \ \lim_{x \to + \infty} \mathfrak{F}_i (x) = 0
		\end{aligned}
	\end{equation}
	for $i = 0, 1, 2$, where
	\begin{equation*}
		\begin{aligned}
			\mathfrak{F}_0 (x) = \int_{\R^3} v_3 \widehat{\mathbb{B}}_3 f_* (x,v) \d v \,, \ \mathfrak{F}_i (x) = \int_{\R^3} v_3 \widehat{\mathbb{A}}_{i3} f_* (x,v) \d v \ \ (i=1,2) \,.
		\end{aligned}
	\end{equation*}
	Now we employ the so-called {\em Freezing Point Method} to complete the proof of claim \eqref{Claim-FG}. For any fixed $x_0 \geq 0$, define
	\begin{equation*}
		\begin{aligned}
			\boldsymbol{\theta}_i (x_0) = \lambda_i^\frac{1}{2} ( \delta x_0 + l )^{ - \frac{1 - \gamma}{2 ( 3 - \gamma ) } } > 0 \,, \ \mathfrak{g}_i (x) = \lambda_i ( \delta x + l )^{ - \frac{1 - \gamma}{3 - \gamma} } \mathfrak{F}_i (x) - \lambda_i ( \delta x_0 + l )^{ - \frac{1 - \gamma}{3 - \gamma} } \mathfrak{F}_i (x)
		\end{aligned}
	\end{equation*}
	for $i = 0, 1, 2$. Then the component equations \eqref{F-eq-comp} for $\mathfrak{F} (x)$ can be reformulated as
	\begin{equation}\label{F-eq-cons}
		\begin{aligned}
			\tfrac{\d^2}{\d x^2} \mathfrak{F}_i (x) - \boldsymbol{\theta}_i^2 (x_0) \mathfrak{F}_i (x) = \mathfrak{g}_i (x) \,, \ \ \lim_{x \to + \infty} \mathfrak{F}_i (x) = 0
		\end{aligned}
	\end{equation}
	Note that the general solution of the second order linear ODE $\tfrac{\d^2}{\d x^2} \mathfrak{F}_i (x) - \boldsymbol{\theta}_i^2 (x_0) \mathfrak{F}_i (x) = 0$ is $\mathfrak{F}_i (x) = C_1 e^{- \boldsymbol{\theta}_i (x_0) x} + C_2 e^{ \boldsymbol{\theta}_i (x_0) x}$, where $C_1, C_2$ are two arbitrary constants. Then one can set that the solution of \eqref{F-eq-cons} admits the form
	\begin{equation*}
		\begin{aligned}
			\mathfrak{F}_i (x) = C_1 (x) e^{- \boldsymbol{\theta}_i (x_0) x} + C_2 (x) e^{ \boldsymbol{\theta}_i (x_0) x} \,,
		\end{aligned}
	\end{equation*}
	where $C_1 (x)$ and $C_2 (x)$ are two functions to be determined later. Substituting the above form into \eqref{F-eq-cons}, one derives from a direct calculation that
	\begin{equation*}
		\begin{aligned}
			- \boldsymbol{\theta}_i (x_0) e^{ - \boldsymbol{\theta}_i (x_0) x } C_1' (x) + \boldsymbol{\theta}_i (x_0) e^{ \boldsymbol{\theta}_i (x_0) x } C_2' (x) + \big[ e^{ - \boldsymbol{\theta}_i (x_0) x } C_1' (x) + e^{ \boldsymbol{\theta}_i (x_0) x } C_2' (x) \big]' = \mathfrak{g}_i (x) \,.
		\end{aligned}
	\end{equation*}
	One then takes $C_1 (x)$ and $C_2 (x)$ satisfying
	\begin{equation*}
		\begin{aligned}
			& e^{ - \boldsymbol{\theta}_i (x_0) x } C_1' (x) + e^{ \boldsymbol{\theta}_i (x_0) x } C_2' (x) = 0 \,, \\
			& - \boldsymbol{\theta}_i (x_0) e^{ - \boldsymbol{\theta}_i (x_0) x } C_1' (x) + \boldsymbol{\theta}_i (x_0) e^{ \boldsymbol{\theta}_i (x_0) x } C_2' (x) = \mathfrak{g}_i (x) \,,
		\end{aligned}
	\end{equation*}
	which implies that
	\begin{equation*}
		\begin{aligned}
			C_1' (x) = - \tfrac{1}{2 \boldsymbol{\theta}_i (x_0)} e^{ \boldsymbol{\theta}_i (x_0) x } \mathfrak{g}_i (x) \,, \ \ C_2' (x) = \tfrac{1}{2 \boldsymbol{\theta}_i (x_0)} e^{ - \boldsymbol{\theta}_i (x_0) x } \mathfrak{g}_i (x) \,.
		\end{aligned}
	\end{equation*}
	Integrating the above two equations from $x_0$ to $x$, one gains
	\begin{equation*}{\small
		\begin{aligned}
			& C_1 (x) = C_1 (x_0) - \tfrac{1}{2 \boldsymbol{\theta}_i (x_0)} \int_{x_0}^x e^{ \boldsymbol{\theta}_i (x_0) y } \mathfrak{g}_i (y) \d y \,, \ C_2 (x) = C_2 (x_0) + \tfrac{1}{2 \boldsymbol{\theta}_i (x_0)} \int_{x_0}^x e^{ - \boldsymbol{\theta}_i (x_0) y } \mathfrak{g}_i (y) \d y \,.
		\end{aligned}}
	\end{equation*}
	Consequently,
{\footnotesize
		\begin{align*}
			\mathfrak{F}_i (x) = \Big[ C_1 (x_0) - \tfrac{1}{2 \boldsymbol{\theta}_i (x_0)} \int_{x_0}^x e^{ \boldsymbol{\theta}_i (x_0) y } \mathfrak{g}_i (y) \d y \Big] e^{- \boldsymbol{\theta}_i (x_0) x} + \Big[ C_2 (x_0) + \tfrac{1}{2 \boldsymbol{\theta}_i (x_0)} \int_{x_0}^x e^{ - \boldsymbol{\theta}_i (x_0) y } \mathfrak{g}_i (y) \d y \Big] e^{ \boldsymbol{\theta}_i (x_0) x} \,.
		\end{align*}}
	Due to the far-field condition $\lim_{x \to + \infty} \mathfrak{F}_i (x) = 0$, it must hold
	\begin{equation*}
		\begin{aligned}
			\mathfrak{F}_i (x) = \Big[ C_1 (x_0) - \tfrac{1}{2 \boldsymbol{\theta}_i (x_0)} \int_{x_0}^x e^{ \boldsymbol{\theta}_i (x_0) y } \mathfrak{g}_i (y) \d y \Big] e^{- \boldsymbol{\theta}_i (x_0) x} \,.
		\end{aligned}
	\end{equation*}
	which means that
	\begin{equation}\label{O-1}
		\begin{aligned}
			\mathfrak{F}_i (x_0) = C_1 (x_0) e^{- \boldsymbol{\theta}_i (x_0) x_0} \,.
		\end{aligned}
	\end{equation}
	Moreover,
	\begin{equation*}
		\begin{aligned}
			\tfrac{\d}{\d x} \mathfrak{F}_i (x) = - \boldsymbol{\theta}_i (x_0) \Big[ C_1 (x_0) - \tfrac{1}{2 \boldsymbol{\theta}_i (x_0)} \int_{x_0}^x e^{ \boldsymbol{\theta}_i (x_0) y } \mathfrak{g}_i (y) \d y \Big] e^{- \boldsymbol{\theta}_i (x_0) x} - \tfrac{1}{2 \boldsymbol{\theta}_i (x_0)} \mathfrak{g}_i (x) \,.
		\end{aligned}
	\end{equation*}
	Note that $\mathfrak{g}_i (x_0) = 0$. One then has
	\begin{equation}\label{O-2}
		\begin{aligned}
			\tfrac{\d}{\d x} \mathfrak{F}_i (x_0) = - \boldsymbol{\theta}_i (x_0) C_1 (x_0) e^{- \boldsymbol{\theta}_i (x_0) x_0} \,.
		\end{aligned}
	\end{equation}
	Due to the arbitrariness of $x_0 \geq 0$, the relations \eqref{O-1} and \eqref{O-2} give us
	\begin{equation*}
		\begin{aligned}
			\tfrac{\d}{\d x} \mathfrak{F}_i (x) = - \boldsymbol{\theta}_i (x) \mathfrak{F}_i (x) = - \lambda_i (\delta x + l)^{ - \frac{1 - \gamma}{3 - \gamma} } \mathfrak{F}_i (x)
		\end{aligned}
	\end{equation*}
	for all $x \geq 0$, which is equivalent to $\tfrac{\d}{\d x} \Big( e^{ \frac{\lambda_i (3 - \gamma)}{2 \delta} ( \delta x + l )^\frac{2}{3 - \gamma} } \mathfrak{F}_i (x) \Big) = 0$. As a result, $\mathfrak{F}_i (x) = 0$ for all $x \geq 0$ if and only if $\mathfrak{F}_i (0) = 0$. The claim \eqref{Claim-FG} is therefore valid.

	Note that $\P^0 f_* = 0$ is equivalent to $\mathfrak{F} (x) = 0$ for all $x \geq 0$. As a consequence, by \eqref{P+-equiv} and \eqref{Claim-FG}, we have proved that $\mathbf{D} f_* (x, v) = 0$ for all $x \geq 0$ and $v \in \R^3$ if and only if
	\begin{equation}\label{Solvb-1}
		\begin{aligned}
			\int_{\R^3} v_3
			\left(
			\begin{array}{c}
				\psi_3^* \\
				\widehat{\mathbb{B}}_3 \\
				\widehat{\mathbb{A}}_{13}\\
				\widehat{\mathbb{A}}_{23}
			\end{array}
			\right) f_* (0,v) \d v = 0_4 \,.
		\end{aligned}
	\end{equation}
By using the relations \eqref{f} and \eqref{P0f-sol}, the solvability conditions \eqref{Solvb-1} can be equivalently expressed by the form \eqref{SBC} in Remark \ref{Rmk-L-ASS}. Then the vanishing sources set $\mathbb{VSS}_{\alpha_*}$ can be re-characterized by
\begin{equation*}
  \begin{aligned}
    \mathbb{VSS}_{\alpha_*} = \Big\{ & (S, f_b); S \in \mathrm{Null}^\perp (\L), \mathbb{I}_\gamma (S, f_b) = f, \\
    & \int_{\R^3}
			\left(
			\begin{array}{c}
				\psi_3^* \\
				\widehat{\mathbb{B}}_3 \\
				\widehat{\mathbb{A}}_{13}\\
				\widehat{\mathbb{A}}_{23}
			\end{array}
			\right) \Big( v_3 f (0,v) + \int_0^\infty S (z, v) \d z \Big) \d v = 0_4 \Big\} \,.
  \end{aligned}
\end{equation*}
Consequently, if $(S, f_b) \in \mathbb{VSS}_{\alpha_*} \cap \mathfrak{X}_\gamma^\infty $, the problem \eqref{KL} admits a unique solution $f (x,v)$ enjoying the bound \eqref{Bnd-f} in Theorem \ref{Thm-Linear}.

Because the four functions $\psi_3^*$, $\widehat{\mathbb{B}}_3$, $\widehat{\mathbb{A}}_{13}$, $\widehat{\mathbb{A}}_{23}$ are linearly independent, the space $\mathbb{VSS}_{\alpha_*} \cap \mathfrak{X}_\gamma^\infty$ is the subspace of $\mathfrak{X}_\gamma^\infty$ with codimension 4. The proof of Theorem \ref{Thm-Linear} is therefore finished.
\end{proof}


\section{Nonlinear problem \eqref{KL-NL}: Proof of Theorem \ref{Thm-Nonlinear}}\label{Sec:ENL}

In this section, we devote to investigating the nonlinear problem \eqref{KL-NL} near the Maxwellian $\M (v)$ by employing the linear theory constructed in Theorem \ref{Thm-Linear}. It is equivalent to study the equation \eqref{KL-NL-f}. The key point is to dominate the quantity $ \mathscr{E}^\infty ( (\delta  x + l)^\frac{1 - \gamma}{3 - \gamma} e^{ \hbar \sigma } \Gamma ( f, g ) ) $, where the functional $ \mathscr{E}^\infty ( \cdot ) $ is given in \eqref{Eg-lambda}.

Notice that $ \mathscr{E}^\infty ( (\delta  x + l)^\frac{1 - \gamma}{3 - \gamma} e^{ \hbar \sigma } \Gamma ( f, g ) ) $ is composed of the weighted $L^\infty_{x,v}$-norms and $L^2_{x,v}$-norms of $ e^{ \hbar \sigma } \Gamma ( f, g ) $. Concerning the weights involved in $ \mathscr{E}^\infty ( (\delta  x + l)^\frac{1 - \gamma}{3 - \gamma} e^{ \hbar \sigma } \Gamma ( f, g ) ) $, one has
\begin{equation}\label{Gamma-sigma}
	\begin{aligned}
		e^{ \hbar \sigma } \Gamma ( f, g ) = e^{ \hbar \sigma } \Gamma ( e^{ - \hbar \sigma } e^{ \hbar \sigma } f, e^{ - \hbar \sigma } e^{ \hbar \sigma } g ) \,.
	\end{aligned}
\end{equation}
By the properties of $\sigma$ in Lemma \ref{Lmm-sigma}, it holds $ e^{\hbar \sigma} \geq e^{ c (\delta x + l)^\frac{2}{3 - \gamma} } e^{ c |v|^2 } $. As shown in \eqref{Gamma-sigma}, loosely speaking, one of $ e^{ - \hbar \sigma } $ can absorb the factor $e^{ \hbar \sigma }$ out of $\Gamma$, and the other decay factor $ e^{ - \hbar \sigma } \leq e^{ - c (\delta x + l)^\frac{2}{3 - \gamma} } e^{ - c |v|^2 } $ can be used to adjust the required weights.

It is easy to see that $\Gamma (f, g)$ can be pointwise bounded by the $L^\infty_{x,v}$-norms of $f$ and $g$. By the similar arguments in Lemma 3 of \cite{Strain-Guo-2008-ARMA}, the weighted $L^2_{x,v}$-norms of $ e^{ \hbar \sigma } \Gamma ( f, g ) $ can be bounded by the quantity with form $\| w_1 e^{ \hbar \sigma } f \|_{L^2_{x,v}} \| w_2 e^{ \hbar \sigma } g \|_{L^\infty_{x,v}} + \| w_1 e^{ \hbar \sigma } g \|_{L^2_{x,v}} \| w_2 e^{ \hbar \sigma } f \|_{L^\infty_{x,v}} $, where $w_1, w_2$ are some required weights. As a result, we can establish the following lemma.

\begin{lemma}\label{Lmm-Gamma}
	For arbitrary functions $f (x,v)$ and $g (x,v)$, there holds
	\begin{equation}\label{Star-1}
		\begin{aligned}
			\mathscr{E}^\infty ( (\delta  x + l)^\frac{1 - \gamma}{3 - \gamma} e^{ \hbar \sigma } \Gamma ( f, g ) ) \leq C \mathscr{E}^\infty ( e^{ \hbar \sigma } f ) \mathscr{E}^\infty ( e^{ \hbar \sigma } g ) \,,
		\end{aligned}
	\end{equation}
    where the functionals $ \mathscr{E}^\infty ( \cdot ) $ is defined in \eqref{Eg-lambda}.
\end{lemma}

The proof of Lemma \ref{Lmm-Gamma} can be finished by applying the properties of the weight $\sigma (x,v)$ and employing the similar arguments in Lemma 3 of \cite{Strain-Guo-2008-ARMA}. For simplicity, we omit the details here.

Based on Lemma \ref{Lmm-Gamma}, we will study the nonlinear problem \eqref{KL-NL-f} (equivalently \eqref{KL-NL}) by employing the iterative approach, hence, prove Theorem \ref{Thm-Nonlinear}.

\begin{proof}[\bf Proof of Theorem \ref{Thm-Nonlinear}]
	As similar as in investigating the linear problem, the nonlinear problem \eqref{KL-NL-f} can be decomposed as
\begin{equation}\label{KL-NL-f-decmp}{\small
  \left\{
  \begin{aligned}
    & v_3 \partial_x f_1 + \L f_1 = (\mathbb{I} - \mathbb{P}) \Gamma (f, f) + (\mathbb{I} - \mathbb{P}) \widehat{h} \,, \ v_3 \partial_x f_2 = \mathbb{P} \Gamma (f, f) + \mathbb{P} \widehat{h} \,, \\
    & f = f_1 + f_2 \,, \ f_1 = (\I - \P^0) f \,, \ f_2 = \P^0 f \,, \\
    & f_1 (0, v) |_{v_3 > 0} = (1 - \alpha_*) f (0, R_0 v) + \alpha_* \mathcal{D}_w f_1 (0, v) + \widetilde{f}_b (\widehat{f}_b, f_2) \,, \\
    & \lim_{x \to + \infty} f_1 (x,v) = \lim_{x \to + \infty} f_2 (x,v) = 0 \,,
  \end{aligned}
  \right.}
\end{equation}
where the operators $\P^0$ and $\mathbb{P}$ are respectively defined in \eqref{P0} and \eqref{P-AB}, and
\begin{equation}\label{fb-fb-f2}
  \begin{aligned}
    \widetilde{f}_b (\widehat{f}_b, f_2) (v) = \widehat{f}_b (v) - f_2 (0, v) \mathbf{1}_{v_3 > 0} + (1 - \alpha_*) f_2 (0, R_0 v) \mathbf{1}_{v_3 > 0} + \alpha_* \mathcal{D}_2 f_2 (0,v) \mathbf{1}_{v_3 > 0} \,.
  \end{aligned}
\end{equation}
As inspired by the linear theory, we first consider the artificial damped nonlinear problem \eqref{KL-NL-f-damp}, i.e.,
\begin{equation*}{\small
  \left\{
  \begin{aligned}
    & v_3 \partial_x f_1 + \L f_1 + \mathbf{D} f_1 = (\mathbb{I} - \mathbb{P}) \Gamma (f, f) + (\mathbb{I} - \mathbb{P}) \widehat{h} \,, \\
    & v_3 \partial_x f_2 = \mathbb{P} \Gamma (f, f) + \mathbb{P} \widehat{h} \,, f = f_1 + f_2 \,, \\
    & f_1 (0, v) |_{v_3 > 0} = (1 - \alpha_*) f (0, R_0 v) + \alpha_* \mathcal{D}_w f_1 (0, v) + \widetilde{f}_b (\widehat{f}_b, f_2) \,, \\
    & \lim_{x \to + \infty} f_1 (x,v) = \lim_{x \to + \infty} f_2 (x,v) = 0 \,,
  \end{aligned}
  \right.}
\end{equation*}
where the artificial damping operator $\mathbf{D}$ is defined in \eqref{D-operator}.

Note that $f_1$ and $f_2$ are coupled each other in the system \eqref{KL-NL-f-damp}. In order to study the existence and uniqueness of the damped nonlinear problem \eqref{KL-NL-f-damp}, we design the following iteration scheme, which decouples the functions $f_1$ and $f_2$:
	\begin{equation}\label{Iter-fi}{\small
		\left\{
		    \begin{aligned}
		    	& v_3 \partial_x f^{i+1}_1 + \L f^{i+1}_1 + \mathbf{D} f^{i+1}_1 = (\mathbb{I} - \mathbb{P}) \Gamma ( f^i, f^i ) + (\mathbb{I} - \mathbb{P}) \widehat{h} \,, \ x > 0 \,, v \in \R^3 \,, \\
                & v_3 \partial_x f^{i+1}_2 = \mathbb{P} \Gamma ( f^i, f^i ) + \mathbb{P} \widehat{h} \,, \ x > 0 \,, v \in \R^3 \,, \ f^{i+1} = f^{i+1}_1 + f^{i+1}_2 \,, \\
		    	& f^{i+1}_1 (0, v) |_{v_3 > 0} = ( 1 - \alpha_* ) f^{i+1}_1 (0, R_0 v) + \alpha_* \mathcal{D}_w f^{i+1}_1 (0, v) + \widetilde{f}_b (\widehat{f}_b, f^{i+1}_2) \,, \\
		    	& \lim_{x \to + \infty} f^{i+1}_1 (x, v) = \lim_{x \to + \infty} f^{i+1}_2 (x, v) = 0 \,,
		    \end{aligned}
		\right.}
	\end{equation}
which starts from $f^0_1 (x,v) = f^0_2 (x,v) = 0$. 

We first iteratively solve $f^{i+1}_2 (x,v)$, which are a sequence of ODE equations. By \eqref{P0f-bnd}, one has
\begin{equation}\label{f2-i+1-bnd}
    \begin{aligned}
        \mathscr{E}^\infty (e^{ \hbar \sigma } f^{i+1}_2) \lesssim & \mathscr{E}^\infty ((\delta x + l)^\frac{1 - \gamma}{3 - \gamma} e^{\hbar \sigma} ( \Gamma ( f^i, f^i ) + \widehat{h} ) ) \\
        \lesssim & \mathscr{E}^\infty ((\delta x + l)^\frac{1 - \gamma}{3 - \gamma} e^{\hbar \sigma} \Gamma ( f^i, f^i ) ) + \mathscr{E}^\infty ((\delta x + l)^\frac{1 - \gamma}{3 - \gamma} e^{\hbar \sigma} \widehat{h} ) \\
        \lesssim & \big[ \mathscr{E}^\infty ( e^{\hbar \sigma} f^i_1 ) + \mathscr{E}^\infty ( e^{\hbar \sigma} f^i_2 ) \big]^2 + \mathscr{E}^\infty ((\delta x + l)^\frac{1 - \gamma}{3 - \gamma} e^{\hbar \sigma} \widehat{h} ) \,,
    \end{aligned}
\end{equation}
where the last inequality is derived from \eqref{Star-1} in Lemma \ref{Lmm-Gamma}.

We then iteratively solve $f^{i+1}_1 (x,v)$, which subject to the linear problem with the same type of \eqref{KL-Damped-f}. By \eqref{f*-bnd}, one has
\begin{equation}\label{f1-i+1-bnd}
	\begin{aligned}
		\mathscr{E}^\infty ( e^{\hbar \sigma} f^{i+1}_1 ) \lesssim & \mathscr{A}^\infty ( e^{\hbar \sigma} (\mathbb{I} - \mathbb{P}) ( \Gamma ( f^i, f^i ) + \widehat{h} ) ) + \| \widetilde{f}_b (\widehat{f}_b, f^{i+1}_2) \|_{\mathfrak{N}} \\
\lesssim & \big[ \mathscr{E}^\infty ( e^{\hbar \sigma} f^i_1 ) + \mathscr{E}^\infty ( e^{\hbar \sigma} f^i_2 ) \big]^2 + \mathscr{E}^\infty ((\delta x + l)^\frac{1 - \gamma}{3 - \gamma} e^{\hbar \sigma} \widehat{h} ) + \| \widehat{f}_b \|_{\mathfrak{N}} \,,
	\end{aligned}
\end{equation}
where the last inequality is derived from the similar arguments in \eqref{fb-tilde-bnd}, \eqref{A-E-bnd} and \eqref{f2-i+1-bnd}.

Then the bounds \eqref{f2-i+1-bnd} and \eqref{f1-i+1-bnd} indicate that
	\begin{equation}\label{Star-2}
		\begin{aligned}
			\mathscr{E}^\infty ( e^{ \hbar \sigma } f^{i+1}_1 ) + \mathscr{E}^\infty ( e^{ \hbar \sigma } f^{i+1}_2 ) \leq C_0 \big[ \mathscr{E}^\infty ( e^{\hbar \sigma} f^i_1 ) + \mathscr{E}^\infty ( e^{\hbar \sigma} f^i_2 ) \big]^2 + C_0 \varsigma
		\end{aligned}
	\end{equation}
	for some constant $ C_0 > 0$, where the quantity $\varsigma$ is defined in \eqref{X-gamma-0}, i.e., $\varsigma = \mathscr{E}^\infty ((\delta x + l)^\frac{1 - \gamma}{3 - \gamma} e^{\hbar \sigma} \widehat{h} ) + \| \widehat{f}_b \|_{\mathfrak{N}}$. 
	
Now we assert that there is a small $\varsigma_0 > 0$ such that if $ \varsigma \leq \varsigma_0 $, then for all $i \geq 1$,
	\begin{equation}\label{Claim-fi}
		\begin{aligned}
			\mathscr{E}^\infty ( e^{ \hbar \sigma } f^i_1 ) + \mathscr{E}^\infty ( e^{\hbar \sigma} f^i_2 ) \leq 2 C_0 \varsigma \,.
		\end{aligned}
	\end{equation}
Note that $f^0_1 = f^0_2 = 0$. The bound \eqref{Star-2} shows that $ \mathscr{E}^\infty ( e^{ \hbar \sigma } f^1_1 ) + \mathscr{E}^\infty ( e^{ \hbar \sigma } f^1_2 ) \leq C_0 \varsigma < 2 C_0 \varsigma $, i.e., the claim \eqref{Claim-fi} holds for $i = 1$. Assume that the claim \eqref{Claim-fi} holds for $1, 2, \cdots, i$. Then the case $i + 1$ can be carried out by \eqref{Star-2} that $ \mathscr{E}^\infty ( e^{ \hbar \sigma } f^{i+1}_1 ) + \mathscr{E}^\infty ( e^{ \hbar \sigma } f^{i+1}_2 ) \leq C ( 2 C_0 \varsigma )^2 + C_0 \varsigma = ( 4 C C_0 \varsigma + 1  ) C_0 \varsigma $. Take $\varsigma_0 > 0$ small such that $ 4 C C_0 \varsigma_0 \leq 1 $. Then if $\varsigma \leq \varsigma_0$, one has $ \mathscr{E}^\infty ( e^{ \hbar \sigma } f^{i+1}_1 ) + \mathscr{E}^\infty ( e^{ \hbar \sigma } f^{i+1}_2 ) \leq ( 4 C C_0 \varsigma + 1  ) C_0 \varsigma \leq 2 C_0 \varsigma $. Therefore, the Induction Principle concludes the claim \eqref{Claim-fi}, which indicates that $\{ f^i \}_{i \geq 1}$ is bounded in the Banach space $ \mathbb{B}_\hbar^\infty $ defined in \eqref{Bh-1}.
	
	Next we will show that $\{ f^i \}_{i \geq 1}$ is a Cauchy sequence in $ \mathbb{B}_\hbar^\infty $. Observe that $f^{i+1}_1 - f^i_1$ and $f^{i+1}_2 - f^i_2$ subject to
	\begin{equation*}{\small
		\left\{
		\begin{aligned}
			& v_3 \partial_x ( f^{i+1}_1 - f^i_1 ) + \L ( f^{i+1}_1 - f^i_1 ) + \mathbf{D} ( f^{i+1}_1 - f^i_1 ) \\
            & \qquad \qquad \qquad = ( \mathbb{I} - \mathbb{P} ) \Gamma ( f^i - f^{i-1}, f^i ) + ( \mathbb{I} - \mathbb{P} ) \Gamma ( f^{i-1}, f^i - f^{i-1} ) \,, \ x > 0 \,, v \in \R^3 \,, \\
            & v_3 \partial_x ( f^{i+1}_2 - f^i_2 ) = \mathbb{P} \Gamma ( f^i - f^{i-1}, f^i ) + \mathbb{P} \Gamma ( f^{i-1}, f^i - f^{i-1} ) \,, \ x > 0 \,, v \in \R^3 \,, \\
            & f^{i+1} - f^i = ( f^{i+1}_1 - f^i_1 ) + ( f^{i+1}_2 - f^i_2 ) \,, \\
			& ( f^{i+1}_1 - f^i_1 ) (0, v) |_{v_3 > 0} = ( 1 - \alpha_* ) ( f^{i+1}_1 - f^i_1 ) (0, R_0 v) + \alpha_* \mathcal{D}_w ( f^{i+1}_1 - f^i_1 ) (0, v) + \widetilde{f}_b ( 0, f^{i+1}_2 - f^i_2 ) \,, \\
			& \lim_{x \to + \infty} ( f^{i+1}_1 - f^i_1 ) (x, v) = \lim_{x \to + \infty} ( f^{i+1}_2 - f^i_2 ) (x, v) = 0 \,.
		\end{aligned}
		\right.}
	\end{equation*}
By the similar arguments in \eqref{Star-2}, one has
	\begin{align*}
		& \mathscr{E}^\infty ( e^{ \hbar \sigma } ( f^{i+1}_1 - f^i_1 ) ) + \mathscr{E}^\infty ( e^{ \hbar \sigma } ( f^{i+1}_2 - f^i_2 ) ) \\
        \leq & C \big[ \mathscr{E}^\infty ( e^{ \hbar \sigma } f^i ) + \mathscr{E}^\infty ( e^{ \hbar \sigma } f^{i-1} ) \big] \big[ \mathscr{E}^\infty ( e^{ \hbar \sigma } ( f^i_1 - f^{i-1}_2 ) ) + \mathscr{E}^\infty ( e^{ \hbar \sigma } ( f^i_2 - f^{i-1}_2 ) ) \big] \\
        \leq & 4 C C_0 \varsigma \big[ \mathscr{E}^\infty ( e^{ \hbar \sigma } ( f^i_1 - f^{i-1}_2 ) ) + \mathscr{E}^\infty ( e^{ \hbar \sigma } ( f^i_2 - f^{i-1}_2 ) ) \big] \,,
	\end{align*}
where the last inequality is implied by \eqref{Claim-fi}. As a consequence,
\begin{equation}\label{Star-3}
	\begin{aligned}
		\mathscr{E}^\infty ( e^{ \hbar \sigma } ( f^{i+1} - f^i ) ) \leq 4 C C_0^2 \varsigma \mathscr{E}^\infty ( e^{ \hbar \sigma } ( f^i - f^{i-1} ) ) \leq \tfrac{1}{2} \mathscr{E}^\infty ( e^{ \hbar \sigma } ( f^i - f^{i-1} ) )
	\end{aligned}
\end{equation}
by further taking small $\varsigma_0 > 0$ such that $ 4 C C_0^2 \varsigma \leq 4 C C_0^2 \varsigma_0 \leq \frac{1}{2} $.
	
Note that $f^0_1 = f^0_2 = 0$. It follows from iterating \eqref{Star-3} and employing \eqref{Claim-fi} that
\begin{equation}\label{Star-4}{\small
	\begin{aligned}
		& \mathscr{E}^\infty ( e^{ \hbar \sigma } ( f^{i+1}_1 - f^i_1 ) ) + \mathscr{E}^\infty ( e^{ \hbar \sigma } ( f^{i+1}_2 - f^i_2 ) ) \\
        \leq & ( \tfrac{1}{2} )^i \big[ \mathscr{E}^\infty ( e^{ \hbar \sigma } f^1_1 ) + \mathscr{E}^\infty ( e^{ \hbar \sigma } f^1_2 ) \big] \leq 2 C_0 \varsigma ( \tfrac{1}{2} )^i \to 0
	\end{aligned}}
\end{equation}
as $i \to + \infty$. Consequently, \eqref{Claim-fi} and \eqref{Star-4} tell us that $ \{ f^i_1 \}_{i \geq 1} $ and $ \{ f^i_2 \}_{i \geq 1} $ are both bounded Cauchy sequence in $ \mathbb{B}_\hbar^\infty $. Then there is a unique pair of $ (f_1, f_2) (x, v) \in \mathbb{B}_\hbar^\infty $ such that
\begin{equation*}
    \begin{aligned}
    	(f^i_1, f^i_2) (x, v) \to (f_1, f_2) (x,v) \quad \textrm{strongly in } \mathbb{B}_\hbar^\infty \,.
    \end{aligned}
\end{equation*}
Passing the limit $i \to + \infty$ in the mild formation of the iterative scheme \eqref{Iter-fi}, one knows that $(f_1, f_2) (x,v)$ solves the damped nonlinear problem \eqref{KL-NL-f-damp}. Moreover, $(f_1, f_2) ( x,v )$ enjoys the estimate $ \mathscr{E}^\infty ( e^{ \hbar \sigma } f_1 ) + \mathscr{E}^\infty ( e^{ \hbar \sigma } f_2 ) \leq 2 C_0 \varsigma $. Moreover, the uniqueness of \eqref{KL-NL-f-damp} can be obtained by similar arguments in \eqref{Star-3}. As a result, for any $( \widehat{h}, \widehat{f}_b ) \in \mathfrak{X}_\gamma^{\varsigma_0}$, the solution operator $\mathcal{I}^\gamma$ to the damped nonlinear problem \eqref{KL-NL-f-damp} by
\begin{equation}\label{SolOper-NL}
  \begin{aligned}
    \mathcal{I}^\gamma ( \widehat{h}, \widehat{f}_b ) = f
  \end{aligned}
\end{equation}
given in Remark \ref{Rmk-NL-ASS} is well-defined, where the space $\mathfrak{X}_\gamma^{\varsigma_0}$ is defined in \eqref{X-gamma-0}.

Now we prove the existence of the nonlinear problem \eqref{KL-NL} (equivalently \eqref{KL-NL-f-decmp}) by removing the artificial damping term $\mathbf{D} f_1$ in \eqref{KL-NL-f-decmp}. By the similar derivations of \eqref{Solvb-1}, one knows that $\mathbf{D} f_1 (x,v) = 0$ for any $x \geq 0$ and $v \in \R^3$ if and only if 
\begin{equation}\label{Solvb-NL-1}{\small
  \begin{aligned}
    \int_{\R^3} v_3
    \left(
    \begin{array}{c}
      \psi_3^* \\
      \widehat{\mathbb{B}}_3 \\
      \widehat{\mathbb{A}}_{13} \\
      \widehat{\mathbb{A}}_{23}
    \end{array}
    \right) f_1 (0, v) \d v = 0_4 \,.
  \end{aligned}}
\end{equation}
Note that $f_1 = f - f_2$ and $f_2 (0, v) = - \frac{1}{v_3} \int_0^\infty \big[ \mathbb{P} \Gamma (f, f) + \mathbb{P} \widehat{h} \big] (z, v) \d z$. Together with the facts $\int_{\R^3} X \cdot (\mathbb{I} - \mathbb{P} ) [ \Gamma (f, f) + \widehat{h} ] \d v = 0$ for $X = \psi_3^*, \widehat{\mathbb{B}}_3, \widehat{\mathbb{A}}_{13}, \widehat{\mathbb{A}}_{23} $, the the conditions \eqref{Solvb-NL-1} is equivalent to
\begin{equation}\label{Solvb-NL}{\small
  \begin{aligned}
    \int_{\R^3}
    \left(
    \begin{array}{c}
      \psi_3^* \\
      \widehat{\mathbb{B}}_3 \\
      \widehat{\mathbb{A}}_{13} \\
      \widehat{\mathbb{A}}_{23}
    \end{array}
    \right) \Big( v_3 f (0, v) + \int_0^\infty \big[ \Gamma (f, f) + \widehat{h} \big] (z, v) \d z \Big) \d v = 0_4 \,.
  \end{aligned}}
\end{equation}
Then the vanishing sources set
\begin{equation}\label{X-gamma-NL}
  \begin{aligned}
    \widetilde{\mathbb{VSS}}_{\alpha_*} = \big\{ ( \widehat{h}, \widehat{f}_b ); \widehat{h} \in \mathrm{Null}^\perp, \mathcal{I}^\gamma ( \widehat{h}, \widehat{f}_b ) = f \textrm{ satisfies \eqref{Solvb-NL}} \big\} \,.
  \end{aligned}
\end{equation}
Consequently, $\mathcal{I}^\gamma ( \widehat{h}, \widehat{f}_b ) = f$ is the solution to the nonlinear problem \eqref{KL-NL} if the source terms $ ( \widehat{h}, \widehat{f}_b ) \in \widetilde{\mathbb{VSS}}_{\alpha_*} \cap {\mathfrak{X}}_\gamma^{\varsigma_0} $. 

Since the functions $\psi_3^*, \widehat{\mathbb{B}}_3, \widehat{\mathbb{A}}_{13}, \widehat{\mathbb{A}}_{23}$ are linearly independent in $L^2_v$, $\widetilde{\mathbb{VSS}}_{\alpha_*} \cap {\mathfrak{X}}_\gamma^{\varsigma_0}$ is the subspace of ${\mathfrak{X}}_\gamma^{\varsigma_0}$ with codimension 4. The proof of Theorem \ref{Thm-Nonlinear} is therefore finished.
\end{proof}


\section{Bounds for operators $Y_A$, $Z$ and $U$: Proof of Lemma \ref{Lmm-ARU}}\label{Sec:YZU-proof}

In this section, we mainly aim at verifying the proof of Lemma \ref{Lmm-ARU}.

\begin{proof}[Proof of Lemma \ref{Lmm-ARU}]
	The proof of this lemma will be divided into three steps as follows.
	
	{\bf Step 1. $L^\infty_{x,v}$ bound for the operator $Y_A$.}
	
	We now show the inequality \eqref{YAn-bnd}.
	
	If $v_3 > 0$, Lemma \ref{Lmm-sigma} indicates that $0 < v_3 \sigma_x \leq c \nu (v)$ and $|\sigma_{xx} v_3| \leq \delta \sigma_x \nu (v)$, which mean that
	\begin{equation}\label{k-bnd-1}
		\begin{aligned}
			\kappa (x,v) = & \int_0^x \big[ - \tfrac{m \sigma_{xx} (y,v) }{2 \sigma_x (y,v)} - \hbar \sigma_x (y, v) + \tfrac{\nu (v)}{v_3} \big] \d y \\
			\geq & \int_0^x  (1 - c \hbar - \tfrac{|m|}{2} \delta ) \tfrac{\nu (v)}{v_3} \d y = (1 - c \hbar - \tfrac{|m|}{2} \delta) \tfrac{\nu (v)}{v_3} x \geq 0 \,,
		\end{aligned}
	\end{equation}
	provided that $c_{\hbar, \delta} : = 1 - c \hbar - \tfrac{|m|}{2} \delta > 0$. Actually, one can take sufficiently small $\hbar, \delta > 0$ such that $c_{\hbar, \delta} > \frac{1}{2}$. Then one has
	\begin{equation}\label{k-bnd1-prime}
		\begin{aligned}
			e^{- \kappa (x,v)} \leq 1
		\end{aligned}
	\end{equation}
	for $1 - c \hbar - \tfrac{|m|}{2} \delta > 0$. Moreover, $(R_0 v)_3 = - v_3 < 0$, one similarly knows that
	\begin{equation}\label{k-n}
		\begin{aligned}
			\kappa (A, R_0 v) = & \int_0^A \big[ - \tfrac{m \sigma_{xx} (y, R_0 v) }{2 \sigma_x (y, R_0 v)} - \hbar \sigma_x (y, R_0 v) + \tfrac{\nu (R_0 v)}{( R_0 v )_3} \big] \d y \\
			\leq & - \int_0^A  (1 - c \hbar - \tfrac{|m|}{2} \delta ) \tfrac{\nu (v)}{v_3} \d y = - (1 - c \hbar - \tfrac{|m|}{2} \delta) \tfrac{\nu (v)}{v_3} A \leq 0 \,,
		\end{aligned}
	\end{equation}
	which means that $ e^{ \kappa (A, R_0 v) } \leq 1 $. As a result, for $v_3 > 0$,
	\begin{equation}\label{k-bnd1}
		\begin{aligned}
			e^{ {\kappa} (A, R_0 v) - {\kappa} (x,v) } \leq 1 \,.
		\end{aligned}
	\end{equation}
	Here ${\kappa} (x,v)$ is defined in \eqref{kappa}. It thereby holds
	\begin{equation}\label{Y-1}
		\begin{aligned}
			\| \mathbf{1}_{v_3 > 0} e^{ {\kappa} (A, R_0 v) - {\kappa} (x,v) } f (A, R_0 v) \|_{L^\infty_{x,v}} \leq \| f (A, \cdot ) \|_{L^\infty_v} \,.
		\end{aligned}
	\end{equation}
	
	Moreover, by the similar arguments in \eqref{k-bnd1}, one easily obtains $ e^{ {\kappa} (A, v') - {\kappa} (x,v) } \leq 1 $ for $v_3' < 0$ and $v_3 > 0$. Recall the definition of $M_\sigma$ in \eqref{M-sigma}, i.e., $ \M_\sigma (v) = \M (v) e^{ - 2 \hbar \sigma (0,v) } $. Based on Lemma \ref{Lmm-sigma}, it is easy to verify that
	\begin{equation*}
		\begin{aligned}
			|v_3'| \sigma_x^{- \frac{m}{2}} (0, v') \sqrt{\M_\sigma (v')} \leq C \M^\frac{1}{4} (v')
		\end{aligned}
	\end{equation*}
	for some universal constant $C > 0$. Now we deal with the factor $\frac{M_w (v)}{\sqrt{\M_\sigma (v)}} \sigma_x^\frac{m}{2} (0, v)$. Lemma \ref{Lmm-sigma} tells us
	\begin{equation*}
		0 \leq \sigma_x^\frac{m}{2} (0, v) e^{ \hbar \sigma (0, v) } \leq c_1 (1 + |v - \u|)^{(1 - \gamma) |m|} e^{ C \hbar |v - \u|^2 } \,.
	\end{equation*}
	Then
	\begin{equation*}{\footnotesize
		\begin{aligned}
			\tfrac{M_w (v)}{\sqrt{\M_\sigma (v)}} \sigma_x^\frac{m}{2} (0, v) = & \tfrac{\sqrt{2 \pi} T^\frac{3}{2}}{\rho T_w^2} \sigma_x^\frac{m}{2} (0, v) e^{ \hbar \sigma (0, v) } \exp \{ - ( \tfrac{1}{2 T_w } - \tfrac{1}{4 T}) |v - \u|^2 - \tfrac{1}{T_w} (\u - u_w) \cdot (v - \u) - \tfrac{|\u - u_w|^2}{2 T_w} \} \\
			\leq & C (1 + |v - \u|)^{(1 - \gamma) |m|} e^{c_1 \hbar |v - \u|^2} \exp \{ - \tfrac{1}{2} ( \tfrac{1}{2 T_w } - \tfrac{1}{4 T}) |v - \u|^2 \}
		\end{aligned}}
	\end{equation*}
	with the assumption \eqref{Assmp-T}, i,e., $0 < T_w < 2 T$. By taking small $\hbar > 0$ such that $ c_1 \hbar < \tfrac{1}{2} ( \tfrac{1}{2 T_w } - \tfrac{1}{4 T}) $, it thereby infers $ \tfrac{M_w (v)}{\sqrt{\M_\sigma (v)}} \sigma_x^\frac{m}{2} (0, v) \leq C $ uniformly in $v \in \R^3$. We then have shown that
	\begin{equation}\label{Y-2}
		\begin{aligned}
			& \| \mathbf{1}_{v_3 > 0} \tfrac{M_w (v)}{\sqrt{\M_\sigma (v)}} \int_{v_3' < 0} (- v_3') \tfrac{ \sigma_x^\frac{m}{2} (0, v) }{ \sigma_x^\frac{m}{2} (0, v') } e^{ {\kappa} (A, v') - {\kappa} (x,v)} f(A, v') \sqrt{\M_\sigma (v') } \d v' \|_{L^\infty_{x,v}} \\
			\leq & C \| \int_{v_3' < 0} |f (A, v')| \M^\frac{1}{4} (v') \d v' \|_{L^\infty_{x,v}} \leq C \| f (A, \cdot) \|_{L^\infty_v} \,,
		\end{aligned}
	\end{equation}
	where the fact $ \int_{v_3' < 0} \M^\frac{1}{4} (v') \d v' \leq C $ has been used.
	
	If $v_3 < 0$,
	\begin{equation}\label{k-Ax}{\small
		\begin{aligned}
			\kappa (A, v) - \kappa (x,v) = \int_x^A [ - \tfrac{m \sigma_{xx} (y,v)}{2 \sigma_x (y,v)} - \hbar \sigma_x (y,v) - \tfrac{\nu (v)}{|v_3|}] \d y \leq - (1 - c \hbar - \tfrac{m}{2} \delta ) \tfrac{\nu (v)}{|v_3|} (A - x) \leq 0 \,,
		\end{aligned}}
	\end{equation}
	which means that
	\begin{equation}\label{k-bnd2-prime}
		\begin{aligned}
			e^{\kappa (A, v) - \kappa (x,v)} \leq e^{ - (1 - c \hbar - \frac{m}{2} \delta) \tfrac{\nu (v)}{|v_3|} (A - x) } \leq 1 \,.
		\end{aligned}
	\end{equation}	
	Consequently, there holds
	\begin{equation}\label{Y-3}
		\begin{aligned}
			\| \mathbf{1}_{v_3 < 0} e^{ {\kappa} (A, v) - {\kappa} (x,v)} f (A, v) \|_{L^\infty_{x,v}} \leq \| f (A, \cdot ) \|_{L^\infty_v} \,.
		\end{aligned}
	\end{equation}
	Recalling the definition of the operator $Y_A (f)$ in \eqref{YAn-f}, the estimates \eqref{Y-1}, \eqref{Y-2} and \eqref{Y-3} infer that
	\begin{equation*}
		\begin{aligned}
			\| Y_A (f) \|_{L^\infty_{x,v}} \leq C \| f (A, \cdot) \|_{L^\infty_v}
		\end{aligned}
	\end{equation*}
	for some constant $C > 0$ independent of $A$, $\hbar$. Namely the bound \eqref{YAn-bnd} holds.
	
	{\bf Step 2. $L^\infty_{x,v}$ bound for the operator $Z$.}
	
	Now we prove the inequalities \eqref{Rn-bnd}.
	
	If $v_3 > 0$, then \eqref{nu-rf} and $\sigma_x > 0$ derived from Lemma \ref{Lmm-sigma} show that for $v_3 > 0$,
	\begin{equation}
		\begin{aligned}
			\kappa (x', R_0 v) = \int_0^{x'} \big[ - \tfrac{m \sigma_{xx} (y, R_0 v)}{2 \sigma_x (y, R_0 v)} - \hbar \sigma_x (y, R_0 v) + \tfrac{\nu (R_0 v)}{(R_0 v)_3} \big] \d y \leq - (1 - \tfrac{\delta}{2}) \tfrac{\nu (v)}{ v_3} x' < 0 \,.
		\end{aligned}
	\end{equation}
	Together with \eqref{k-bnd-1}, one has $ e^{- \kappa (x,v) + \kappa (x', R_0 v)} \leq e^{ - c_{\hbar, \delta} \frac{\nu(v)}{v_3} x} e^{- (1 - \tfrac{\delta}{2}) \frac{\nu(v)}{v_3} x'} $ for $v_3 > 0$. Then there hold
		\begin{align}\label{Rn-f-bdn}
			\no & (1 - \tfrac{1}{n}) | \int_0^A e^{ - [ {\kappa} (x,v) - {\kappa} (x', R_0 v) ] } \tfrac{1}{v_3} f (x', R_0 v) \d x' | \\
			\leq & (1 - \tfrac{1}{n}) \mathbf{1}_{v_3 > 0} e^{- c_{\hbar, \delta} \frac{\nu(v)}{v_3} x} \int_0^A e^{- (1 - \tfrac{\delta}{2}) \frac{\nu(v)}{v_3} x'} \tfrac{\nu (v)}{v_3} |\nu^{-1} (v) f (x', R_0 v)| \d x'
		\end{align}
	for small $\hbar > 0$. Note that the right-hand side of the bound \eqref{Rn-f-bdn} can be bounded by
	\begin{equation*}{\small
		\begin{aligned}
			& (1 - \tfrac{1}{n}) \int_0^A e^{- (1 - \tfrac{\delta}{2}) \frac{\nu(v)}{v_3} x'} \tfrac{\nu (v)}{v_3} \d x' \|\nu^{-1} (v) f (x, R_0 v) \|_{L^\infty_x} = (1 - \tfrac{\delta}{2})^{-1} (1 - \tfrac{1}{n}) \nu^{-1} (v) \| f (\cdot, R_0 v) \|_{L^\infty_x} \,,
		\end{aligned}}
	\end{equation*}
	which implies
	\begin{equation}\label{Z-1}
		\begin{aligned}
			(1 - \tfrac{1}{n}) \| \int_0^A e^{ - [ {\kappa} (x,v) - {\kappa} (x', R_0 v) ] } \tfrac{1}{v_3} f (x', R_0 v) \d x' \|_{L^\infty_{x,v}} \leq (1 - \tfrac{\delta}{2})^{-1} (1 - \tfrac{1}{n}) \|\nu^{-1} f \|_{L^\infty_{x,v}} \,.
		\end{aligned}
	\end{equation}
	
	Now we control the quantity
	\begin{equation*}
		\begin{aligned}
			\mathfrak{Q} (f) = \tfrac{M_w (v)}{\sqrt{\M_\sigma (v)}} \int_{v_3' < 0} \int_0^A (- v_3') \tfrac{ \sigma_x^\frac{m}{2} (0, v) }{ \sigma_x^\frac{m}{2} (0, v') } e^{- [ {\kappa} (x,v) - {\kappa} (x', v') ] } \tfrac{1}{v_3'} f (x', v') \sqrt{\M_\sigma (v') } \d x' \d v' \,.
		\end{aligned}
	\end{equation*}
	The inequality \eqref{k-bnd-1} implies $ \mathbf{1}_{v_3 > 0} e^{- \kappa (x,v)} \leq \mathbf{1}_{v_3 > 0} e^{ - c_{\hbar, \delta} \frac{\nu (v)}{v_3} x } \leq 1 $. By the similar arguments in \eqref{k-n}, one has $ \mathbf{1}_{v_3' < 0} e^{ \kappa (x', v') } \leq \mathbf{1}_{v_3' < 0} e^{ - c_{\hbar, \delta} \frac{\nu (v') }{|v_3'|} x' } $. It also follows from the similar arguments of \eqref{Y-2} that $ \tfrac{M_w (v)}{\sqrt{\M_\sigma (v)}} \tfrac{ \sigma_x^\frac{m}{2} (0, v) }{ \sigma_x^\frac{m}{2} (0, v') } \sqrt{\M_\sigma (v') } \leq C \M^\frac{1}{4} (v') $ uniformly in $v \in \R^3$ for sufficiently small $\hbar > 0$, where the constant $C > 0$ is independent of $A, n, \hbar, l$ and $\delta$. As a consequence,
{\small
		\begin{align}\label{Z-2}
			\no \| \mathfrak{Q} (f) \|_{L^\infty_{x,v}} \leq & C \| \nu^{-1} f \|_{L^\infty_{x,v}} \int_{v_3' < 0} \int_0^A \tfrac{\nu (v')}{ |v_3'| } e^{ - c_{\hbar, \delta} \frac{\nu (v') }{|v_3'|} x' } |v_3'| \M^\frac{1}{4} (v') \d x' \d v' \\
			\no = & \tfrac{C}{c_{\hbar, \delta}} \| \nu^{-1} f \|_{L^\infty_{x,v}} \int_{v_3' < 0} ( - e^{ - c_{\hbar, \delta} \tfrac{\nu (v')}{|v_3'|} x' }  ) \big|_0^A |v_3'| \M^\frac{1}{4} (v') \d v' \\
			\leq & \tfrac{C}{c_{\hbar, \delta}} \| \nu^{-1} f \|_{L^\infty_{x,v}} \int_{v_3' < 0} |v_3'| \M^\frac{1}{4} (v') \d v' \leq C \| \nu^{-1} f \|_{L^\infty_{x,v}}
		\end{align}}
	with small $\hbar > 0$ and assumption \eqref{Assmp-T}. Therefore, plugging the estimates \eqref{Z-1} and \eqref{Z-2} into the expression \eqref{Rn-f} of the operator $Z (f)$, one gains $ \| Z (f) \|_{L^\infty_{x,v}} \leq C \| \nu^{-1} f \|_{L^\infty_{x,v}} $, hence the bound \eqref{Rn-bnd} holds.
	
	{\bf Step 3. $L^\infty_{x,v}$ bound for the operator $U$.}
	
	We now justify the bound \eqref{U-bnd}. By the definition of ${\kappa} (x,v)$ in \eqref{kappa}, one has
	\begin{equation*}
		\begin{aligned}
			\mathbf{1}_{x' \in (0, x)} \mathbf{1}_{v_3 > 0} e^{ - [ {\kappa} (x,v) - {\kappa} (x', v) ] } = e^{ - \int_{x'}^x [ - \frac{m \sigma_{xx} (y, v)}{2 \sigma_x (y,v)} - \hbar \sigma_x (y, v) + \frac{\nu (v)}{v_3} ] \d y } \,.
		\end{aligned}
	\end{equation*}
	By the similar arguments in \eqref{k-bnd-1}, one has
	\begin{equation*}
		\begin{aligned}
			& \mathbf{1}_{x' \in (0, x)} \mathbf{1}_{v_3 > 0} e^{ - [ {\kappa} (x,v) - {\kappa} (x', v) ] } \leq \mathbf{1}_{x' \in (0, x)} \mathbf{1}_{v_3 > 0}  e^{ - c_{\hbar, \delta} \frac{\nu (v)}{v_3} (x - x') } \,.
		\end{aligned}
	\end{equation*}
	Together with the definition of the operator $U$ in \eqref{U-f}, it infers that
		\begin{align}\label{U-1}
			\no | \mathbf{1}_{v_3 > 0} U (f) | \leq & C \int_0^x \tfrac{\nu (v)}{v_3} e^{ - \frac{1}{2} c_{\hbar, \delta} \frac{\nu (v)}{v_3} (x - x') } | \nu^{-1} f (x' v) | \d x' \\
			\no \leq & C \| \nu^{-1} f \|_{L^\infty_{x,v}} \int_0^x \tfrac{\nu (v)}{v_3} e^{ - c_{\hbar, \delta} \frac{\nu (v)}{v_3} (x - x') } \d x' \\
			= & \tfrac{C}{c_{\hbar, \delta}} \| \nu^{-1} f \|_{L^\infty_{x,v}} ( 1 - e^{ - \frac{1}{2} c_{\hbar, \delta} \frac{\nu (v)}{v_3} x } ) \leq \tfrac{2 C}{c_{\hbar, \delta}} \| \nu^{-1} f \|_{L^\infty_{x,v}} \,.
		\end{align}
	
	We then consider the case $v_3 < 0$. Combining with the definition of ${\kappa} (x,v)$ in \eqref{kappa} and the inequality \eqref{k-Ax}, one derives
		\begin{align*}
			\mathbf{1}_{x' \in (x,A)} \mathbf{1}_{v_3 < 0} e^{ - [ {\kappa} (x,v) - {\kappa} (x', v) ] } = & \mathbf{1}_{x' \in (x,A)} \mathbf{1}_{v_3 < 0} e^{ \int^{x'}_x [ - \frac{m \sigma_{xx} (y, v)}{2 \sigma_x (y,v)} - \hbar \sigma_x (y, v) - \frac{\nu (v)}{|v_3|} ] \d y } \\
			\leq & \mathbf{1}_{x' \in (x,A)} \mathbf{1}_{v_3 < 0} e^{ - c_{\hbar, \delta} \frac{\nu (v)}{|v_3|} (x' - x) } \,.
		\end{align*}
	Then
	\begin{equation}\label{U-2}
		\begin{aligned}
			| \mathbf{1}_{v_3 < 0} U (f) | \leq & \int_x^A e^{ - c_{\hbar, \delta} \frac{\nu (v)}{|v_3|} (x' - x) } \tfrac{\nu (v)}{|v_3|} | \nu^{-1} (v) f (x', v) | \d x' \\
			\leq & \| \nu^{-1} f \|_{L^\infty_{x,v}} \int_x^A e^{ - c_{\hbar, \delta} \frac{\nu (v)}{|v_3|} (x' - x) } \tfrac{\nu (v)}{|v_3|} \d x' \\
			= & \tfrac{1}{c_{\hbar, \delta}} \| \nu^{-1} f \|_{L^\infty_{x,v}} ( 1 - e^{ - c_{\hbar, \delta} \frac{\nu (v)}{|v_3|} (A - x) } ) \leq \tfrac{1}{c_{\hbar, \delta}} \| \nu^{-1} f \|_{L^\infty_{x,v}} \,.
		\end{aligned}
	\end{equation}
	Then the bounds \eqref{U-1} and \eqref{U-2} conclude the estimate \eqref{U-bnd}. The proof of Lemma \ref{Lmm-ARU} is therefore completed.
\end{proof}

\section{Properties of the operator $K$: Proof of Lemma \ref{Lmm-K-Oprt}-\ref{Lmm-Kh-2-infty}-\ref{Lmm-Kh-L2}}\label{Sec:K}

In this section, we aim at investigating the properties of the operator $K$.

\subsection{$L^\infty_{x,v}$ property of $K$: Proof of Lemma \ref{Lmm-K-Oprt}}\label{Subsec:K-infty}

In this subsection, we will prove the $L^\infty_{x,v}$ property of $K$, i.e., proof of Lemma \ref{Lmm-K-Oprt}.

Before proving Lemma \ref{Lmm-K-Oprt}, we need to analyze the factor $\sigma_x$ in the bound \eqref{K-bnd}. By the definition of $\sigma$ in \eqref{sigma}, it infers
\begin{equation}\label{sigma-x}
	\sigma_x (x,v) =
	\left\{
	\begin{array}{ll}
		\delta (1 + |v - \u|)^{- 1 + \gamma} \,, & (x, v) \in \Omega_1 \,, \\
		\delta \big( c_1 (1 + |v - \u|)^{-1 + \gamma} + c_2 (\delta x + l)^{- \frac{1-\gamma}{3-\gamma}} \big) \,, & (x,v) \in \Omega_2 \,, \\
		\tfrac{10 \delta}{3 - \gamma} (\delta x + l)^{- \frac{1-\gamma}{3-\gamma}} \,, & (x,v) \in \Omega_3 \,,
	\end{array}
	\right.
\end{equation}
where $c_1$ and $c_2$ are positive functions depending on $(x,v)$ and $\Upsilon$ with $c_1 + c_2$ admiting a uniform lower bound. Also $\Omega_i$ ($i=1,2,3$) are defined as follows.
\begin{equation}\label{Omega123}
	\begin{aligned}
		& \Omega_1 = \{ (x,v) : \delta x + l \leq (1 + |v-\u|)^{3 - \gamma} \} \,, \\
		& \Omega_2 = \{ (x,v) : (1 + |v-\u|)^{3 - \gamma} < \delta x + l < 2 (1 + |v-\u|)^{3 - \gamma} \} \,, \\
		& \Omega_3 = \{ (x,v) : \delta x + l \geq 2 (1 + |v-\u|)^{3 - \gamma} \} \,.
	\end{aligned}
\end{equation}

\begin{lemma}\label{Lmm-Psigma}
	Let $-3 < \gamma \leq 1$. Denote by
	\begin{equation}
		\begin{aligned}
			P_\sigma = \tfrac{\sigma_x (x,v)}{\sigma_x (x, v_*)} \,, \quad P_v = \tfrac{( 1 + | v - \u | )^{-1 + \gamma}}{( 1 + | v_* - \u | )^{-1 + \gamma}} \,.
		\end{aligned}
	\end{equation}
	Then the following results hold:
{\small
	\begin{figure}[H]
		\begin{center}
			\begin{tikzpicture}
				\draw (-5,0)--(5,0);
				\draw (-5, -1)--(5,-1);
				\draw (-5, -2)--(5,-2);
				\draw (-5, -3)--(5,-3);
				\draw (-5, -4)--(5,-4);
				\draw (-5, 0)--(-5, -4);
				\draw (5, 0)--(5, -4);
				\draw (-3.1, 0)--(-3.1, -4);
				\draw (-0.4, 0)--(-0.4, -4);
				\draw (2.3, 0)--(2.3, -4);
				\draw (-5, -0.35)--(-3.1, -1);
				\draw (-4.35, -0)--(-3.1, -1);
				\draw (-3.5, -0.3) node{\footnotesize $(x,v)$ };
				\draw (-5.1, -0.8) node[right]{\footnotesize $(x, v_*)$};
				\draw (-4.9, -0.25) node[right]{$P_\sigma$};
				\draw (-2.15, -0.6) node[right]{$\Omega_1$};
				\draw (0.55, -0.6) node[right]{$\Omega_2$};
				\draw (3.25, -0.6) node[right]{$\Omega_3$};
				\draw (-4.4, -1.55) node[right]{$\Omega_1$};
				\draw (-4.4, -2.55) node[right]{$\Omega_2$};
				\draw (-4.4, -3.55) node[right]{$\Omega_3$};
				\draw (-2.6, -1.55) node[right]{$P_\sigma = P_v$};
				\draw (-2.6, -2.55) node[right]{$P_\sigma \thicksim P_v$};
				\draw (-2.9, -3.55) node[right]{$P_v \lesssim P_\sigma \lesssim 1$};
				\draw (0.1, -1.55) node[right]{$P_\sigma \thicksim P_v$};
				\draw (0.1, -2.55) node[right]{$P_\sigma \thicksim P_v$};
				\draw (0.1, -3.55) node[right]{$P_\sigma \thicksim 1$};
				\draw (2.45, -1.55) node[right]{$1 \lesssim P_\sigma \lesssim P_v$};
				\draw (2.9, -2.55) node[right]{$P_\sigma \thicksim 1$};
				\draw (2.9, -3.55) node[right]{$P_\sigma = 1$};
			\end{tikzpicture}
		\end{center}
	\end{figure}}
\end{lemma}

\begin{proof}
	Based on the expression $\sigma_x$ in \eqref{sigma-x} and the definitions of $\Omega_1$ ($i = 1,2,3$) in \eqref{Omega123}, the results in Lemma \ref{Lmm-Psigma} will be proved case by case.
	
	{\em Case 1. $(x,v) \in \Omega_1$, $(x,v_*) \in \Omega_1$.}
	
	By \eqref{sigma-x}, one has $ \sigma_x (x,v) = \delta (1 + |v - \u|)^{-1 + \gamma} $ and $ \sigma_x (x,v_*) = \delta (1 + |v_* - \u|)^{-1 + \gamma} $. Then
	\begin{equation*}
		\begin{aligned}
			P_\sigma = \tfrac{\sigma_x (x,v)}{\sigma_x (x, v_*)} = \tfrac{(1 + |v - \u|)^{-1 + \gamma}}{(1 + |v_* - \u|)^{-1 + \gamma}} = P_v \,.
		\end{aligned}
	\end{equation*}
	
	{\em Case 2. $(x,v) \in \Omega_1$, $(x,v_*) \in \Omega_2$.}
	
	Due to $(x,v_*) \in \Omega_2$, Lemma \ref{Lmm-sigma} and the definition of $\Omega_2$ in \eqref{Omega123} show that
	\begin{equation*}
		\begin{aligned}
			\sigma_x (x,v_*) \thicksim (\delta x + l)^{- \frac{1 - \gamma}{3 - \gamma}} \thicksim (1 + |v_* - \u|)^{- 1 + \gamma} \,.
		\end{aligned}
	\end{equation*}
	Together $\sigma_x (x,v) = \delta (1 + |v - \u|)^{-1 + \gamma}$, it follows that
	\begin{equation*}
		\begin{aligned}
			P_\sigma = \tfrac{\sigma_x (x,v)}{\sigma_x (x, v_*)} \thicksim \tfrac{(1 + |v - \u|)^{-1 + \gamma}}{(1 + |v_* - \u|)^{-1 + \gamma}} = P_v \,.
		\end{aligned}
	\end{equation*}
	
	{\em Case 3. $(x,v) \in \Omega_1$, $(x,v_*) \in \Omega_3$.}
	
	By \eqref{sigma-x}, one has $ \sigma_x (x, v_*) = \tfrac{10 \delta}{3 - \gamma} (\delta x + l)^{- \frac{1 - \gamma}{3 - \gamma}} $. Then
	\begin{equation*}
		\begin{aligned}
			\tfrac{\sigma_x (x,v)}{\sigma_x (x, v_*)} = \tfrac{3 - \gamma}{10} (1 + |v - \u|)^{-1 + \gamma} (\delta x + l)^\frac{1 - \gamma}{3 - \gamma} \,.
		\end{aligned}
	\end{equation*}
	In $\Omega_1$, one has $\delta x + l \leq (1 + |v - \u|)^{3 - \gamma}$. Then
	\begin{equation*}
		\begin{aligned}
			\tfrac{\sigma_x (x,v)}{\sigma_x (x, v_*)} \leq \tfrac{3 - \gamma}{10} (1 + |v - \u|)^{-1 + \gamma} (1 + |v - \u|)^{1 - \gamma} = \tfrac{3 - \gamma}{10} \,.
		\end{aligned}
	\end{equation*}
	In $\Omega_3$, one has $\delta x + l \geq 2 (1 + |v_* - \u|)^{3 - \gamma}$. Then
	\begin{equation*}
		\begin{aligned}
			\tfrac{\sigma_x (x,v)}{\sigma_x (x, v_*)} \geq 2^\frac{1 - \gamma}{3 - \gamma} \tfrac{3 - \gamma}{10} (1 + |v - \u|)^{-1 + \gamma} (1 + |v_* - \u|)^{1 - \gamma} = 2^\frac{1 - \gamma}{3 - \gamma} \tfrac{3 - \gamma}{10} P_v \,.
		\end{aligned}
	\end{equation*}
	In summary, $ 2^\frac{1 - \gamma}{3 - \gamma} \tfrac{3 - \gamma}{10} P_v \leq P_\sigma = \tfrac{\sigma_x (x,v)}{\sigma_x (x, v_*)} \leq \tfrac{3 - \gamma}{10} $, i.e., $P_v \lesssim P_\sigma \lesssim 1$.
	
	{\em Case 4. $(x,v) \in \Omega_2$, $(x,v_*) \in \Omega_1$.}
	
	By the similar arguments in {\em Case 2}, there hold
	\begin{equation*}
		\begin{aligned}
			\sigma_x (x,v) \thicksim (\delta x + l)^{- \frac{1 - \gamma}{3 - \gamma}} \thicksim (1 + |v - \u|)^{- 1 + \gamma} \,, \quad \sigma_x (x, v_*) = (1 + |v_* - \u|)^{- 1 + \gamma} \,,
		\end{aligned}
	\end{equation*}
	which implies $P_\sigma \thicksim P_v$.
	
	{\em Case 5. $(x,v) \in \Omega_2$, $(x,v_*) \in \Omega_2$.}
	
	By the arguments in {\em Case 2} and {\em Case 3}, one knows
	\begin{equation*}
		\begin{aligned}
			\sigma_x (x,v) \thicksim (1 + |v - \u|)^{- 1 + \gamma} \,, \quad \sigma_x (x, v_*) \thicksim (1 + |v_* - \u|)^{- 1 + \gamma} \,,
		\end{aligned}
	\end{equation*}
	which means $P_\sigma \thicksim P_v$.
	
	{\em Case 6. $(x,v) \in \Omega_2$, $(x,v_*) \in \Omega_3$.}
	
	In $\Omega_2$, $\sigma_x (x, v) \thicksim (1 + |v - \u|)^{-1 + \gamma} \thicksim (\delta x + l)^{- \frac{1 - \gamma}{3 - \gamma}}$. In $\Omega_3$, $\sigma_x (x, v_*) = \tfrac{10 \delta}{3 - \gamma} (\delta x + l)^{- \frac{1 - \gamma}{3 - \gamma}}$. It thereby infers that $ P_\sigma \thicksim 1 $.
	
	{\em Case 7. $(x,v) \in \Omega_3$, $(x,v_*) \in \Omega_1$.}
	
	For $(x,v) \in \Omega_3$, one has
	\begin{equation*}
		\begin{aligned}
			\sigma_x (x, v) = \tfrac{10 \delta}{3 - \gamma} (\delta x + l)^{- \frac{1 - \gamma}{3 - \gamma}} \leq \tfrac{10 \delta}{3 - \gamma} 2^{- \frac{1 - \gamma}{3 - \gamma}} (1 + |v - \u|)^{-1 + \gamma} \,.
		\end{aligned}
	\end{equation*}
	For $(x,v_*) \in \Omega_1$, it is easy to see that $ \sigma_x (x, v_*) = \delta (1 + |v_* - \u|)^{- 1 + \gamma} \leq \delta (\delta x + l)^{- \frac{1 - \gamma}{3 - \gamma}} $. Then there hold
	\begin{equation*}
		\begin{aligned}
			\tfrac{\sigma_x (x,v)}{\sigma_x (x, v_*)} = \tfrac{10}{3 - \gamma}  \tfrac{(\delta x + l)^{- \frac{1 - \gamma}{3 - \gamma}}}{(1 + |v_* - \u|)^{- 1 + \gamma}} \leq \tfrac{10}{3 - \gamma} 2^{- \frac{1 - \gamma}{3 - \gamma}} \tfrac{(1 + |v - \u|)^{-1 + \gamma}}{(1 + |v_* - \u|)^{- 1 + \gamma}} \,,
		\end{aligned}
	\end{equation*}
	and
	\begin{equation*}
		\begin{aligned}
			\tfrac{\sigma_x (x,v)}{\sigma_x (x, v_*)} = \tfrac{10}{3 - \gamma}  \tfrac{(\delta x + l)^{- \frac{1 - \gamma}{3 - \gamma}}}{(1 + |v_* - \u|)^{- 1 + \gamma}} \geq \tfrac{10}{3 - \gamma} \tfrac{(\delta x + l)^{- \frac{1 - \gamma}{3 - \gamma}}}{(\delta x + l)^{- \frac{1 - \gamma}{3 - \gamma}}} = \tfrac{10}{3 - \gamma} \,.
		\end{aligned}
	\end{equation*}
	As a result, one has $1 \lesssim P_\sigma \lesssim P_v$.
	
	{\em Case 8. $(x,v) \in \Omega_3$, $(x,v_*) \in \Omega_2$.}
	
	For $(x,v) \in \Omega_3$ and $(x,v_*) \in \Omega_2$,
	\begin{equation*}
		\begin{aligned}
			\sigma_x (x, v) = \tfrac{10 \delta}{3 - \gamma} (\delta x + l)^{- \frac{1 - \gamma}{3 - \gamma}} \,, \quad \sigma_x (x.v_*) \thicksim (\delta x + l)^{- \frac{1 - \gamma}{3 - \gamma}} \thicksim (1 + |v_* - \u|)^{- 1 + \gamma} \,,
		\end{aligned}
	\end{equation*}
	which shows that $P_\sigma = \tfrac{\sigma_x (x,v)}{\sigma_x (x, v_*)} \thicksim 1$.
	
	{\em Case 9. $(x,v) \in \Omega_3$, $(x,v_*) \in \Omega_3$.}
	
	In this case, one has $ \sigma_x (x, v) = \sigma_x (x, v_*) = \tfrac{10 \delta}{3 - \gamma} (\delta x + l)^{- \frac{1 - \gamma}{3 - \gamma}} $, which means that $P_\sigma = \tfrac{\sigma_x (x,v)}{\sigma_x (x, v_*)} = 1$.
	
	The proof of Lemma \ref{Lmm-Psigma} is therefore completed.
\end{proof}

\begin{proof}[Proof of Lemma \ref{Lmm-K-Oprt}]
	For the simplicity of notations, we denote by
	\begin{equation}\label{p}
		\begin{aligned}
			\mathfrak{p} (x,v) = \sigma_x^\frac{m}{2} (x,v) e^{\hbar \sigma (x,v)} w_{\beta, \vartheta} (v) \,.
		\end{aligned}
	\end{equation}
	By the definition of the norm $\| \cdot \|_{\beta, \vartheta}$ in \eqref{XY-space}, one knows that
	\begin{equation}
		\begin{aligned}
			\| \sigma_x^\frac{m}{2} e^{\hbar \sigma} K f \|_{A; \beta, \vartheta} = \| \mathfrak{p} K f \|_{L^\infty_{x,v}} \,.
		\end{aligned}
	\end{equation}
	Let $\tilde{f} (x,v) = \mathfrak{p} f (x,v)$ and $K_\mathfrak{p} (\cdot) = \mathfrak{p} K ( \mathfrak{p}^{-1} \cdot )$. Then one has
	\begin{equation}{\small
		\begin{aligned}
			& \| \sigma_x^\frac{m}{2} e^{\hbar \sigma} K f \|_{A; \beta, \vartheta} = \| \mathfrak{p} K f \|_{L^\infty_{x,v}} = \| K_\mathfrak{p} \tilde{f} \|_{L^\infty_{x,v}} \,, \ \| \sigma_x^\frac{m}{2} e^{\hbar \sigma} f \|_{A; \beta + \gamma - 1, \vartheta} = \| (1 + |v|)^{\gamma - 1} \tilde{f} \|_{L^\infty_{x,v}} \,.
		\end{aligned}}
	\end{equation}
	It thereby suffices to prove
	\begin{equation}\label{Kp-bnd}
		\begin{aligned}
			\| K_\mathfrak{p} \tilde{f} \|_{L^\infty_{x,v}} \leq C \| (1 + |v|)^{\gamma - 1} \tilde{f} \|_{L^\infty_{x,v}} \,.
		\end{aligned}
	\end{equation}
	
	Let $\chi (r)$ be a smooth monotone function satisfying $\chi (r) = 0$ for $0 \leq r \leq 1$ and $\chi (r) = 1$ for $r \geq 2$. Then, together with \eqref{K-K1-K2}-\eqref{K1}-\eqref{K2}, the operator $K_\mathfrak{p}$ can be decomposed as follows:
	\begin{equation}\label{Kp-split}
		\begin{aligned}
			K_\mathfrak{p} \tilde{f} (x,v) = - K_{1 \mathfrak{p}} \tilde{f} (x,v) + K_{2 \mathfrak{p}}^{1 - \chi} \tilde{f} (x,v) + K_{2 \mathfrak{p}}^\chi \tilde{f} (x,v) \,,
		\end{aligned}
	\end{equation}
	where
{\footnotesize
		\begin{align}\label{K12p}
			\no K_{1 \mathfrak{p}} \tilde{f} (x,v) = & \int_{\R^3} \tfrac{\mathfrak{p} (x,v)}{\mathfrak{p} (x,v_*)} |v - v_*|^\gamma \M^\frac{1}{2} (v) \M^\frac{1}{2} (v_*) \tilde{f} (x,v_*) \d v_* \,, \\
			K_{2 \mathfrak{p}}^{1 - \chi} \no \tilde{f} (x,v) = & \mathfrak{p} (x,v) \iint_{\mathbb{R}^3 \times \mathbb{S}^2} (1 - \chi) (|v_* - v|) \big[ \M^\frac{1}{2} (v_*') \tfrac{\tilde{f}}{\mathfrak{p}} (x,v') + \M^\frac{1}{2} (v') \tfrac{\tilde{f}}{\mathfrak{p}} (x,v_*') \big] \M^\frac{1}{2} (v_*) b (\omega, v_* - v) \d \omega \d v_* \,, \\
			K_{2 \mathfrak{p}}^\chi \tilde{f} (x,v) = & \mathfrak{p} (x,v) \iint_{\mathbb{R}^3 \times \mathbb{S}^2} \chi (|v_* - v|) \big[ \M^\frac{1}{2} (v_*') \tfrac{\tilde{f}}{\mathfrak{p}} (x,v') + \M^\frac{1}{2} (v') \tfrac{\tilde{f}}{\mathfrak{p}} (x,v_*') \big] \M^\frac{1}{2} (v_*) b (\omega, v_* - v) \d \omega \d v_* \,.
		\end{align}}
	
	{\bf Step 1. Estimate of $\tfrac{\mathfrak{p} (x,v)}{\mathfrak{p} (x,v_*)}$.}
	
	By \eqref{p}, one has $ \tfrac{\mathfrak{p} (x,v)}{\mathfrak{p} (x,v_*)} = \big[ \tfrac{\sigma_x (x, v)}{\sigma_x (x, v_*)} \big]^\frac{m}{2} \tfrac{(1 + |v|)^\beta}{(1 + |v_*|)^\beta} e^{\hbar [ \sigma (x,v) - \sigma (x, v_*) ]} e^{\vartheta ( |v - \u|^2 - |v_* - \u|^2 )} $. By Lemma \ref{Lmm-sigma}, it follows $ | \sigma (x,v) - \sigma (x, v_*) | \leq c \big| |v - \u|^2 - |v_* - \u|^2 \big| $.

	We claim that for any $\beta \in \R$,
	\begin{equation}\label{Claim-beta}
		\begin{aligned}
			\tfrac{(1 + |v|)^\beta}{(1 + |v_*|)^\beta} \leq (1 + |v - v_*|)^{|\beta|} \,.
		\end{aligned}
	\end{equation}
	Indeed, if $\beta \geq 0$,
	\begin{equation*}
		\begin{aligned}
			(1 + |v|)^\beta \leq (1 + |v_*| + |v - v_*|)^\beta \leq (1 + |v_*|)^\beta (1 + |v - v_*|)^\beta \,,
		\end{aligned}
	\end{equation*}
	namely, $ \tfrac{(1 + |v|)^\beta}{(1 + |v_*|)^\beta} \leq (1 + |v - v_*|)^\beta $. If $\beta < 0$,
	\begin{equation*}
		\begin{aligned}
			(1 + |v_*|)^{- \beta} \leq (1 + |v_*|)^{- \beta} (1 + |v - v_*|)^{- \beta} = (1 + |v_*|)^{- \beta} (1 + |v - v_*|)^{|\beta|} \,,
		\end{aligned}
	\end{equation*}
	which means that $ \tfrac{(1 + |v|)^\beta}{(1 + |v_*|)^\beta} \leq (1 + |v - v_*|)^{|\beta|} $. As a result, the claim \eqref{Claim-beta} holds. Moreover, Lemma \ref{Lmm-Psigma} and similar arguments in \eqref{Claim-beta} indicate that for any $m \in \R$,
	\begin{equation}\label{pro-sigma}
		\begin{aligned}
			\big[ \tfrac{\sigma_x (x, v)}{\sigma_x (x, v_*)} \big]^\frac{m}{2} \leq C_{m, \gamma} \Big( 1 + \tfrac{( 1 + | v - \u | )^{- \frac{m (1 - \gamma)}{2}}}{( 1 + | v_* - \u | )^{- \frac{m (1 - \gamma)}{2}}} \Big) \leq C_{m, \gamma} (1 + |v - v_*|)^\frac{|m| (1 - \gamma)}{2} \,.
		\end{aligned}
	\end{equation}
	In summary, one derives that
	\begin{equation}\label{p-pro}
		\begin{aligned}
			\tfrac{\mathfrak{p} (x,v)}{\mathfrak{p} (x,v_*)} \leq C (1 + |v - v_*|)^k \exp \big[ (c \hbar + \vartheta) \big| |v - \u|^2 - |v_* - \u|^2 \big| \big] \,,
		\end{aligned}
	\end{equation}
	where $k = |\beta| + |m| (1 - \gamma) /2 \geq 0$.
	
	{\bf Step 2. Estimate of $K_{1 \mathfrak{p}} \tilde{f} (x,v)$.}
	
	By the definition of $K_{1 \mathfrak{p}} \tilde{f} (x,v)$ in \eqref{K12p} and the bound \eqref{p-pro}, it is easy to see
	\begin{equation}\label{A+A-}
		\begin{aligned}
			| K_{1 \mathfrak{p}} \tilde{f} (x,v)| \leq C \underbrace{ \int_{\R^3} |v_* - v|^\gamma (1 + |v_* - v|)^k [\M (v) \M (v_*)]^{2 \delta_0} | \tilde{f} (x, v_*)| \d v_* }_{: = \mathsf{A}^-} \,,
		\end{aligned}
	\end{equation}
	where $\delta_0 = \tfrac{1}{4} - T (c \hbar + \vartheta) > 0$ by the assumptions in Lemma \ref{Lmm-K-Oprt}. Due to $|v - \u|^2 + |v_* - \u|^2 \geq \tfrac{1}{2} |v - v_*|^2$,
	\begin{equation}\label{MM-bnd}
		\begin{aligned}
			\M (v) \M (v_*) \leq \big[ \tfrac{\rho}{(2 \pi T)^{3/2}} \big]^2 \exp \big[ - \tfrac{|v - v_*|^2}{4 T} \big] \,.
		\end{aligned}
	\end{equation}
	Together with $[\M (v) \M (v_*)]^{\delta_0} \leq C (1 + |v_*|)^{\gamma - 1}$, it infers that
	\begin{equation}\label{A-bnd}
		\begin{aligned}
			\mathsf{A}^- \leq & C \| (1 + |v|)^{\gamma - 1} \tilde{f} \|_{L^\infty_{x,v}} \int_{\R^3} |v_* - v|^\gamma (1 + |v_* - v|)^k \exp \big[ - \tfrac{\delta_0 |v - v_*|^2}{4 T} \big] \d v_* \\
			\leq & C \| (1 + |v|)^{\gamma - 1} \tilde{f} \|_{L^\infty_{x,v}} \,,
		\end{aligned}
	\end{equation}
	where the last inequality is followed from the convergence of the integral
	\begin{equation*}
		\begin{aligned}
			\int_{\R^3} |v_* - v|^\gamma (1 + |v_* - v|)^k \exp \big[ - \tfrac{\delta_0 |v - v_*|^2}{4 T} \big] \d v_* < \infty \,,
		\end{aligned}
	\end{equation*}
	which holds if and only if $\gamma > - 3$.
	
	Therefore, the relations \eqref{A+A-} and \eqref{A-bnd} indicate that
	\begin{equation}\label{K1p-bnd}
		\begin{aligned}
			| K_{1 \mathfrak{p}} \tilde{f} (x,v)| \leq C \| (1 + |v|)^{\gamma - 1} \tilde{f} \|_{L^\infty_{x,v}}
		\end{aligned}
	\end{equation}
	for $\gamma > - 3$, $\beta, m \in \R$ and $\vartheta \geq 0$ sufficiently small.

	{\bf Step 3. Estimate of $K_{2 \mathfrak{p}}^{1 - \chi} \tilde{f} (x,v)$.}
	
	First, we assert that
	\begin{equation}\label{Claim-M}
		\begin{aligned}
			\M (v_*) \leq C \M^{1 - \delta_0} (v) \,, \ \M (v_*') \leq C \M^{1 - \delta_0} (v') \,, \ \M (v') \leq C \M^{1 - \delta_0} (v_*') \,,
		\end{aligned}
	\end{equation}
	provided that $|v_* - v| = |v_*' - v'| \leq 2$, where $\delta_0 \in (0, 1/4)$ is given in Lemma \ref{Lmm-K-Oprt}. Indeed, if $|v_* - v| \leq 2$, it follows that
	\begin{equation*}
		\begin{aligned}
			|v - \u|^2 = & |v - v_* + v_* - \u|^2 = |v_* - \u|^2 + 2 (v - v_*) \cdot (v_* - \u) + |v - v_*|^2 \\
			\leq & (1 + \tfrac{\delta_0}{1 - \delta_0}) |v_* - \u|^2 + (1 + \tfrac{1 - \delta_0}{\delta_0}) |v - v_*|^2 \leq \tfrac{1}{1 - \delta_0} |v_* - \u|^2 + \tfrac{4}{\delta_0} \,.
		\end{aligned}
	\end{equation*}
	Then there hold
	\begin{equation*}
		\begin{aligned}
			\M (v) = \tfrac{\rho}{(2 \pi T)^\frac{3}{2}} \exp \big( - \tfrac{|v - \u|^2}{2 T} \big) \geq \tfrac{\rho}{(2 \pi T)^\frac{3}{2}} \exp \big[ - \tfrac{|v_* - \u|^2}{2 T (1 - \delta_0)} - \tfrac{2}{\delta_0 T} \big] \geq C \M^\frac{1}{1 - \delta_0} (v_*) \,,
		\end{aligned}
	\end{equation*}
	which concludes the first inequality in \eqref{Claim-M}. Moreover, the same arguments imply the last two bounds in \eqref{Claim-M}.
	
	By \eqref{Claim-M} and the definition of $K_{2 \mathfrak{p}}^{1 - \chi} \tilde{f} (x,v)$ in \eqref{K12p}, it is easy to obtain
	\begin{equation*}
		\begin{aligned}
			| K_{2 \mathfrak{p}}^{1 - \chi} \tilde{f} (x,v)| \leq C ( \Pi' + \Pi_*' ) \,,
		\end{aligned}
	\end{equation*}
	where
	\begin{equation*}
		\begin{aligned}
			\Pi' = & \mathfrak{p} (x,v) \iint_{\mathbb{R}^3 \times \mathbb{S}^2} (1 - \chi) (|v_* - v|) \big| \tfrac{\tilde{f}}{\mathfrak{p}} (x, v') \big| \M^\frac{1 - \delta_0}{2} (v') \M^\frac{1}{2} (v_*) b (\omega, v_* - v) \d \omega \d v_* \,, \\
			\Pi_*' = & \mathfrak{p} (x,v) \iint_{\mathbb{R}^3 \times \mathbb{S}^2} (1 - \chi) (|v_* - v|) \big| \tfrac{\tilde{f}}{\mathfrak{p}} (x, v_*') \big| \M^\frac{1 - \delta_0}{2} (v_*') \M^\frac{1}{2} (v_*) b (\omega, v_* - v) \d \omega \d v_* \,.
		\end{aligned}
	\end{equation*}
	Now we focus on dominating the quantity $\Pi'$. By \eqref{p-pro}, one observes that
	\begin{equation*}
		\begin{aligned}
			\tfrac{\mathfrak{p} (x,v)}{\mathfrak{p} (x, v')} \leq & C (1 + |v - v'|)^k \exp \big[ (c \hbar + \vartheta) \big| |v - \u|^2 - |v' - \u|^2 \big| \big] \\
			\leq & C (1 + |v - v'|)^k [ \M (v) \M (v') ]^{- 2 T (c \hbar + \vartheta)} = C (1 + |v - v'|)^k [ \M (v) \M (v') ]^{2 \delta_0 - \frac{1}{2} }\,.
		\end{aligned}
	\end{equation*}
	Then the quantiy $\Pi'$ can be bounded by
{\tiny
	\begin{equation*}
		\begin{aligned}
			\Pi' \leq & C \iint_{\mathbb{R}^3 \times \mathbb{S}^2} (1 - \chi) (|v_* - v|) |\tilde{f} (x, v')| (1 + |v - v'|)^k \M^{2 \delta_0 - \frac{1}{2}} (v) \M^{\frac{1}{2} - \frac{\delta_0}{4 (1 - \delta_0)}} (v_*) \M^{\frac{3}{2} \delta_0} (v') \M^\frac{\delta_0}{4 (1 - \delta_0)} (v_*) b (\omega, v_* - v) \d \omega \d v_* \\
			\leq & C \iint_{\mathbb{R}^3 \times \mathbb{S}^2} (1 - \chi) (|v_* - v|) |\tilde{f} (x, v')| (1 + |v - v'|)^k \M^{\frac{5}{4} \delta_0 } (v) \M^{\frac{3}{2} \delta_0} (v') \M^\frac{\delta_0}{4 (1 - \delta_0)} (v_*) b (\omega, v_* - v) \d \omega \d v_* \,,
		\end{aligned}
	\end{equation*}}
	where the last inequality is followed from the bound $\M^{\frac{1}{2} - \frac{\delta_0}{4 (1 - \delta_0)}} (v_*) \leq C \M^{\frac{1}{2} - \frac{3}{4} \delta_0} (v)$ by \eqref{Claim-M}. It is derived from \eqref{MM-bnd} that $ \M^{\frac{5}{8} \delta_0 } (v) \M^{\frac{5}{8} \delta_0 } (v') \leq C \exp \big( - \tfrac{\delta_0 |v - v'|^2}{20 T} \big) $. Together with $\M^\frac{\delta_0}{4} (v') \leq C (1 + |v'|)^{\gamma - 1}$, the quantity $\Pi'$ can be further bounded by
	\begin{equation*}
		\begin{aligned}
			\Pi' \leq & C \| (1 + |v|)^{\gamma - 1} \tilde{f} \|_{L^\infty_{x,v}} \iint_{\mathbb{R}^3 \times \mathbb{S}^2} (1 - \chi) (|v_* - v|) (1 + |v - v'|)^k \M^{\frac{5}{8} \delta_0} (v) \M^{\frac{5}{8} \delta_0} (v') \\
			& \qquad \qquad \quad \qquad \qquad \qquad \quad \times \exp \big( - \tfrac{\delta_0 |v - v'|^2}{20 T} \big) \M^\frac{\delta_0}{4 (1 - \delta_0)} (v_*) b (\omega, v_* - v) \d \omega \d v_* \\
			\leq & C \| (1 + |v|)^{\gamma - 1} \tilde{f} \|_{L^\infty_{x,v}} \,,
		\end{aligned}
	\end{equation*}
	where we have used the estimates $(1 + |v - v'|)^k \exp \big( - \tfrac{\delta_0 |v - v'|^2}{20 T} \big) \leq C$ and
{\small
	\begin{equation*}
		\begin{aligned}
			\iint_{\mathbb{R}^3 \times \mathbb{S}^2} (1 - \chi) (|v_* - v|)  \M^{\frac{5}{8} \delta_0} (v) \M^{\frac{5}{8} \delta_0} (v')  \M^\frac{\delta_0}{4 (1 - \delta_0)} (v_*) b (\omega, v_* - v) \d \omega \d v_* \leq C \M^{\frac{5}{8} \delta_0} (v) \nu (v) \leq C \,.
		\end{aligned}
	\end{equation*}}
	
	Similarly, the quantity $\Pi_*'$ can be bounded by $ \Pi_*' \leq C \| (1 + |v|)^{\gamma - 1} \tilde{f} \|_{L^\infty_{x,v}} $ for $- 3 < \gamma \leq  1$, $\beta \in \R$ and sufficiently small $\vartheta \geq 0$. Consequently, it concludes that
	\begin{equation}\label{K2p-bnd-l}
		\begin{aligned}
			| K_{2 \mathfrak{p}}^{1 - \chi} \tilde{f} (x,v)| \leq C \| (1 + |v|)^{\gamma - 1} \tilde{f} \|_{L^\infty_{x,v}}
		\end{aligned}
	\end{equation}
	for $- 3 < \gamma \leq  1$, $\beta \in \R$ and sufficiently small $\vartheta \geq 0$.
	
	{\bf Step 4. Estimate of $K_{2 \mathfrak{p}}^{\chi} \tilde{f} (x,v)$.}
	By letting $V = v_* - v$, $V_\shortparallel = (V \cdot \omega) \omega$, $V_\perp = V - V_\shortparallel$ and $\omega = \frac{v' -v}{|v' - v|}$, one has $\d \omega \d v_* = 2 |V_\shortparallel|^{-2} \d V_\perp \d V_\shortparallel$, see Section 2 of \cite{Grad-1963}. Moreover, $v' = v + V_\shortparallel$, $v_*' = v + V_\perp$, $v_* = v + V$. Then, by \eqref{K12p}, one has
	\begin{equation*}
		\begin{aligned}
			K_{2 \mathfrak{p}}^{\chi} \tilde{f} (x,v) = \Lambda_1 + \Lambda_2 \,,
		\end{aligned}
	\end{equation*}
	where
   {\small
	\begin{equation}\label{Lambda12}
		\begin{aligned}
			\Lambda_1 = 2 \mathfrak{p} (x,v) \int_{\R^3} \int_{V_\perp \perp V_\shortparallel} |V_\shortparallel|^{-2} \chi (|V|) \tfrac{\tilde{f}}{\mathfrak{p}} (x, v + V_\shortparallel) \M^\frac{1}{2} (v + V_\perp) \M^\frac{1}{2} (v + V) b (\omega, V) \d V_\perp \d V_\shortparallel \,, \\
			\Lambda_2 = 2 \mathfrak{p} (x,v) \int_{\R^3} \int_{V_\perp \perp V_\shortparallel} |V_\shortparallel|^{-2} \chi (|V|) \tfrac{\tilde{f}}{\mathfrak{p}} (v + V_\perp) \M^\frac{1}{2} (x, v + V_\shortparallel) \M^\frac{1}{2} (v + V) b (\omega, V) \d V_\perp \d V_\shortparallel \,.
		\end{aligned}
	\end{equation}}
	
	\underline{\em Step 4.1. Control of $|\Lambda_1|$.} Set
	\begin{equation}\label{zeta}
		\begin{aligned}
			\zeta = v + \tfrac{1}{2} V_\shortparallel \,, \ \zeta_\shortparallel = [(\zeta - \u) \cdot \omega] \omega \,, \ \zeta_\perp = (\zeta - \u) - \zeta_\shortparallel \,.
		\end{aligned}
	\end{equation}
	A direct calculation implies
	\begin{equation}\label{Trans-1}{\small
		\begin{aligned}
			& - \tfrac{|v + V_\perp - \u|^2 + |v + V - \u|^2}{4 T} = - \tfrac{|v + V_\perp - \u|^2 + |v + V_\perp - \u + V_\shortparallel|^2}{4 T} \\
			= & - \tfrac{|\zeta - \tfrac{1}{2} V_\shortparallel + V_\perp - \u|^2 + |\zeta + V_\perp - \u + \tfrac{1}{2} V_\shortparallel|^2}{4 T} = - \tfrac{|\zeta - \u + V_\perp |^2 + \tfrac{1}{4} | V_\shortparallel |^2}{2 T} = - \tfrac{|\zeta_\shortparallel|^2 + |V_\perp + \zeta_\perp|^2 + \tfrac{1}{4} | V_\shortparallel |^2}{2 T} \,,
		\end{aligned}}
	\end{equation}
	which yields
	\begin{equation}\label{Lambda1-1}
		\begin{aligned}
			\M^\frac{1}{2} (v + V_\perp) \M^\frac{1}{2} (v + V) = \tfrac{\rho}{(2 \pi T)^\frac{3}{2}} \exp \big[ - \tfrac{|V_\shortparallel|^2}{8 T} - \tfrac{|\zeta_\shortparallel|^2}{2 T} - \tfrac{|V_\perp + \zeta_\perp|^2}{2 T} \big] \,.
		\end{aligned}
	\end{equation}
	Moreover, \eqref{b} and \eqref{Cutoff} imply that
	\begin{equation}\label{Lambda1-2}
		\begin{aligned}
			|b (\omega, V)| \leq \tilde{b}_0 \tfrac{|V \cdot \omega|}{|V|} |V|^\gamma = \tfrac{\tilde{b}_0 |V_\shortparallel|}{|V|^{1 - \gamma}} \,.
		\end{aligned}
	\end{equation}
	It also follows from \eqref{p-pro} that $ \tfrac{\mathfrak{p} (x,v)}{\mathfrak{p} (x, v + V_\shortparallel)} \leq C (1 + |V_\shortparallel|)^k \exp \big[ \tfrac{(1 - 4 \delta_0)}{4 T} \big| |v - \u|^2 - |v + V_\shortparallel - \u|^2 \big| \big] $, where $\delta_0 = \tfrac{1}{4} - T (c \hbar + \vartheta) > 0$ has been used. Note that
	\begin{equation}\label{Trans-2}
		\begin{aligned}
			|v - \u|^2 - |v + V_\shortparallel - \u|^2 = - V_\shortparallel \cdot [2 (v - \u) + V_\shortparallel] = - V_\shortparallel \cdot (2 \zeta_\perp + 2 \zeta_\shortparallel) = - 2 V_\shortparallel \cdot \zeta_\shortparallel \,,
		\end{aligned}
	\end{equation}
	which infers
	\begin{equation}\label{Lambda1-3}
		\begin{aligned}
			\tfrac{\mathfrak{p} (x,v)}{\mathfrak{p} (x, v + V_\shortparallel)} \leq C (1 + |V_\shortparallel|)^k \exp \big[ \tfrac{(1 - 4 \delta_0 )}{2 T} | V_\shortparallel \cdot \zeta_\shortparallel | \big] \,.
		\end{aligned}
	\end{equation}
	Plugging \eqref{Lambda1-1}-\eqref{Lambda1-2}-\eqref{Lambda1-3} into the expression of $\Lambda_1$ in \eqref{Lambda12}, one has
	\begin{equation}\label{Lambda1-pm}{\small
		\begin{aligned}
			|\Lambda_1| \leq C \int_{\R^3} \int_{V_\perp \perp V_\shortparallel} |V_\shortparallel|^{-1} (1 + |V_\shortparallel|)^k \exp \big[ - \tfrac{|V_\shortparallel|^2}{8 T} - \tfrac{|\zeta_\shortparallel|^2}{2 T} + \tfrac{(1 - 4 \delta_0)}{2 T} | V_\shortparallel \cdot \zeta_\shortparallel | \big] \\
			\times \exp \big( - \tfrac{|V_\perp + \zeta_\perp|^2}{2 T} \big) \tfrac{\chi (|V|)}{|V|^{1 - \gamma}} | \tilde{f} (x, v + V_\shortparallel)| \d V_\perp \d V_\shortparallel \,.
		\end{aligned}}
	\end{equation}
	
	It is easy to see that by $- \tfrac{|V_\shortparallel|^2}{8 T} - \tfrac{|\zeta_\shortparallel|^2}{2 T} + \tfrac{(1 - 4 \delta_0)}{2 T} | V_\shortparallel \cdot \zeta_\shortparallel | \leq - \tfrac{\delta_0 |V_\shortparallel|^2}{2 T} - \tfrac{2 \delta_0 |\zeta_\shortparallel|^2}{T}$,
{\small
		\begin{align*}
			|\Lambda_1| \leq & C \| (1 + |v|)^{\gamma - 1} \tilde{f} \|_{L^\infty_{x,v}} \int_{\R^3} |V_\shortparallel|^{-1} (1 + |V_\shortparallel|)^k (1 + |v + V_\shortparallel|)^{1 - \gamma} \exp \big[ - \tfrac{\delta_0 |V_\shortparallel|^2}{2 T} - \tfrac{2 \delta_0 |\zeta_\shortparallel|^2}{T} \big] \\
			& \qquad \qquad \qquad \qquad \qquad \qquad \qquad \qquad \ \times \int_{V_\perp \perp V_\shortparallel} \exp \big( - \tfrac{|V_\perp + \zeta_\perp|^2}{2 T} \big) \tfrac{\chi (|V|)}{|V|^{1 - \gamma}} \d V_\perp \d V_\shortparallel \,.
		\end{align*}}
	Recalling that $v + V_\shortparallel = \zeta_\shortparallel + \zeta_\perp + \tfrac{1}{2} V_\shortparallel + \u$, it holds $ |v + V_\shortparallel| \leq C ( 1 + |\zeta_\shortparallel| + |\zeta_\perp| + |V_\shortparallel| ) $, which means $ (1 + |v + V_\shortparallel|)^{1 - \gamma} \leq C (1 + |V_\shortparallel|)^{1 - \gamma} (1 + |\zeta_\shortparallel|)^{1 - \gamma} (1 + |\zeta_\perp|)^{1 - \gamma} $. Observe that $ (1 + |V_\shortparallel|)^{k + 1 - \gamma} (1 + |\zeta_\shortparallel|)^{1 - \gamma} \exp \big[ - \tfrac{\delta |V_\shortparallel|^2}{4 T} - \tfrac{\delta |\zeta_\shortparallel|^2}{T} \big] \leq C $ and $ \tfrac{\chi (|V|)}{|V|^{1 - \gamma}} \leq \tfrac{C}{(1 + |V_\shortparallel|^2 + |V_\perp|^2 )^\frac{1-\gamma}{2}} $. Then the quantity $ \Lambda_1 $ can be further bounded by
	\begin{equation*}{\small
		\begin{aligned}
			|\Lambda_1| \leq C \| (1 + |v|)^{\gamma - 1} \tilde{f} \|_{L^\infty_{x,v}} & \int_{\R^3} |V_\shortparallel|^{-1} \exp \big[ - \tfrac{\delta |V_\shortparallel|^2}{4 T} - \tfrac{\delta |\zeta_\shortparallel|^2}{T} \big] \\
			& \times \underbrace{ \int_{V_\perp \perp V_\shortparallel} \exp \big( - \tfrac{|V_\perp + \zeta_\perp|^2}{2 T} \big) \tfrac{(1 + |\zeta_\perp|)^{1 - \gamma}}{(1 + |V_\shortparallel|^2 + |V_\perp|^2 )^\frac{1-\gamma}{2}} \d V_\perp }_{: = U} \d V_\shortparallel \,.
		\end{aligned}}
	\end{equation*}
	
	It turns to show that the quantity $U$ is uniformly bounded in $V_\shortparallel$ and $\zeta_\perp$. Let
	\begin{equation*}
		\begin{aligned}
			I_1 = \{ V_\perp ; V_\perp \perp V_\shortparallel \,, |V_\perp + \zeta_\perp| > \tfrac{1}{2} |\zeta_\perp| \} \,, \quad I_2 = \{ V_\perp ; V_\perp \perp V_\shortparallel \,, |V_\perp + \zeta_\perp| \leq \tfrac{1}{2} |\zeta_\perp| \} \,.
		\end{aligned}
	\end{equation*}
	Then
	\begin{equation*}{\small
		\begin{aligned}
			U = \int_{I_1} (\cdots) \d V_\perp + \int_{I_2} (\cdots) \d V_\perp : = U_1 + U_2 \,.
		\end{aligned}}
	\end{equation*}
	If $|V_\perp + \zeta_\perp| > \tfrac{1}{2} |\zeta_\perp|$, one has $ \tfrac{(1 + |\zeta_\perp|)^{1 - \gamma}}{(1 + |V_\shortparallel|^2 + |V_\perp|^2 )^\frac{1-\gamma}{2}} \leq C \exp \big( \tfrac{|V_\perp + \zeta_\perp|^2}{4 T} \big) $. Then $U_1$ can be bounded by
	\begin{equation*}{\small
		\begin{aligned}
			U_1 \leq C\int_{V_\perp \perp V_\shortparallel} \exp \big( - \tfrac{|V_\perp + \zeta_\perp|^2}{4 T} \big) \d V_\perp \leq C \,.
		\end{aligned}}
	\end{equation*}
	If $|V_\perp + \zeta_\perp| \leq \tfrac{1}{2} |\zeta_\perp|$, it holds $ |V_\perp| \geq |\zeta_\perp| - |V_\perp + \zeta_\perp| \geq \tfrac{1}{2} |\zeta_\perp| $, which implies $ \tfrac{(1 + |\zeta_\perp|)^{1 - \gamma}}{(1 + |V_\shortparallel|^2 + |V_\perp|^2 )^\frac{1-\gamma}{2}} \leq \tfrac{(1 + |\zeta_\perp|)^{1 - \gamma}}{(1 + \tfrac{1}{4} |\zeta_\perp|^2 )^\frac{1-\gamma}{2}} \leq C $. Then $U_2$ can be dominated by
	\begin{equation*}{\small
		\begin{aligned}
			U_2 \leq C\int_{V_\perp \perp V_\shortparallel} \exp \big( - \tfrac{|V_\perp + \zeta_\perp|^2}{2 T} \big) \d V_\perp \leq C \,.
		\end{aligned}}
	\end{equation*}
	It therefore follows that $U = U_1 + U_2 \leq C$, which implies
{\small
	\begin{equation}\label{Lambda1-bnd}
		\begin{aligned}
			|\Lambda_1| \leq C \| (1 + |v|)^{\gamma - 1} \tilde{f} \|_{L^\infty_{x,v}} \int_{\R^3} |V_\shortparallel|^{-1} \exp \big[ - \tfrac{\delta |V_\shortparallel|^2}{4 T} - \tfrac{\delta |\zeta_\shortparallel|^2}{T} \big] \d V_\shortparallel \leq C \| (1 + |v|)^{\gamma - 1} \tilde{f} \|_{L^\infty_{x,v}}
		\end{aligned}
	\end{equation}}
	for $- 3 < \gamma \leq 1$, $\beta \in \R$ and sufficiently small $\vartheta \geq 0$.

	\underline{\em Step 4.2. Control of $|\Lambda_2|$.} Recalling the definition \eqref{zeta}, a direct computation implies
	\begin{equation}\label{L1}
		\begin{aligned}
			- \tfrac{|v + V_\shortparallel - \u|^2 + |v + V - \u|^2}{4 T} = - \tfrac{|\zeta_\shortparallel + \tfrac{1}{2} V_\shortparallel|^2}{2 T} - \tfrac{|\zeta_\perp|^2}{4 T} - \tfrac{|V_\perp + \zeta_\perp|^2}{4 T} \,,
		\end{aligned}
	\end{equation}
	which yields
	\begin{equation}\label{Lambda2-1}
		\begin{aligned}
			\M^\frac{1}{2} (v + V_\shortparallel) \M^\frac{1}{2} (v + V) = \tfrac{\rho}{(2 \pi T)^\frac{3}{2}} \exp \big[ - \tfrac{|\zeta_\shortparallel + \tfrac{1}{2} V_\shortparallel|^2}{2 T} - \tfrac{|\zeta_\perp|^2}{4 T} - \tfrac{|V_\perp + \zeta_\perp|^2}{4 T} \big] \,.
		\end{aligned}
	\end{equation}
	By \eqref{p-pro}, there holds $ \tfrac{\mathfrak{p} (x,v)}{\mathfrak{p} (x, v + V_\perp)} \leq C (1 + |V_\perp|)^k \exp \big[ \tfrac{1 - 4 \delta_0}{4 T} \big| |v - \u|^2 - |v + V_\perp - \u|^2 \big| \big] $. By \eqref{zeta}, it follows
	\begin{equation}\label{L2}
		\begin{aligned}
			|v - \u|^2 - |v + V_\perp - \u|^2 = - |V_\perp + \zeta_\perp|^2 + |\zeta_\perp|^2 \,.
		\end{aligned}
	\end{equation}
	Moreover, $(1 + |V_\perp|)^k \leq (1 + |\zeta_\perp|)^k (1 + |V_\perp + \zeta_\perp|)^k$. Then
	\begin{equation}\label{Lambda2-2}
		\begin{aligned}
			\tfrac{\mathfrak{p} (x,v)}{\mathfrak{p} (x, v + V_\perp)} \leq C (1 + |\zeta_\perp|)^k (1 + |V_\perp + \zeta_\perp|)^k \exp \big[ \tfrac{1 - 4 \delta_0}{4 T} \big( |V_\perp + \zeta_\perp|^2 + |\zeta_\perp|^2 \big) \big] \,.
		\end{aligned}
	\end{equation}
	Substituting \eqref{Lambda1-2}, \eqref{Lambda2-1} and \eqref{Lambda2-2} into the expression of $\Lambda_2$ in \eqref{Lambda12}, it follows that
		\begin{align}\label{Lambda2-pm}
			\no |\Lambda_2| \leq & C \int_{\R^3} |V_\shortparallel|^{-1} (1 + |\zeta_\perp|)^k \exp \big( - \tfrac{|\zeta_\shortparallel + \tfrac{1}{2} V_\shortparallel|^2}{2 T} - \tfrac{\delta_0 |\zeta_\perp|^2}{T} \big) \\
			& \times \int_{V_\perp \perp V_\shortparallel} (1 + |V_\perp + \zeta_\perp|)^k \tfrac{\chi (|V|)}{|V|^{1 - \gamma}} | \tilde{f} (x, v + V_\perp)| \exp \big( - \tfrac{\delta_0 |V_\perp + \zeta_\perp|^2}{T} \big) \d V_\perp \d V_\shortparallel \,.
		\end{align}
	
	Notice that $ \tfrac{\chi (|V|)}{|V|^{1 - \gamma}} \leq \tfrac{C}{(1 + |V_\shortparallel|^2 + |V_\perp|^2 )^\frac{1-\gamma}{2}} \leq \tfrac{C}{(1 + |V_\shortparallel|)^{1-\gamma}} $ and
	\begin{equation*}
		\begin{aligned}
			| \tilde{f} (x, v + V_\perp)| \leq & \| (1 + |v|)^{\gamma - 1} \tilde{f} \|_{L^\infty_{x,v}} (1 + |v + V_\perp|)^{1 - \gamma} \\
			\leq & C \| (1 + |v|)^{\gamma - 1} \tilde{f} \|_{L^\infty_{x,v}} (1 + |V_\perp + \zeta_\perp|)^{1 - \gamma} (1 + |\zeta_\shortparallel + \tfrac{1}{2} V_\shortparallel|)^{1 - \gamma} (1 + |\zeta_\shortparallel|)^{1 - \gamma} \,,
		\end{aligned}
	\end{equation*}
	where the relation $v + V_\perp = \u + \zeta_\shortparallel + \zeta_\perp - \tfrac{1}{2} V_\shortparallel + V_\perp$ has been utilized. Moreover,
	\begin{equation*}
		\begin{aligned}
			(1 + |\zeta_\perp|)^k (1 + |\zeta_\shortparallel + \tfrac{1}{2} V_\shortparallel|)^{1 - \gamma} \exp \big( - \tfrac{\delta_0 |\zeta_\perp|^2}{T} - \tfrac{|\zeta_\shortparallel + \tfrac{1}{2} V_\shortparallel|^2}{4 T} \big) \leq C \,.
		\end{aligned}
	\end{equation*}
	Then the quantity $ \Lambda_2 $ can be further bounded by
	\begin{equation*}{\small
		\begin{aligned}
			|\Lambda_2| \leq & C \| (1 + |v|)^{\gamma - 1} \tilde{f} \|_{L^\infty_{x,v}} \int_{\R^3} |V_\shortparallel|^{-1} \tfrac{(1 + |\zeta_\shortparallel|)^{1 - \gamma}}{(1 + |V_\shortparallel|)^{1-\gamma}} \exp \big( - \tfrac{|\zeta_\shortparallel + \tfrac{1}{2} V_\shortparallel|^2}{4 T} \big) \\
			& \qquad \qquad \qquad \qquad \times \int_{V_\perp \perp V_\shortparallel} (1 + |V_\perp + \zeta_\perp|)^{k + 1 - \gamma} \exp \big( - \tfrac{\delta_0 |V_\perp + \zeta_\perp|^2}{T} \big) \d V_\perp \d V_\shortparallel \\
			\leq & C \| (1 + |v|)^{\gamma - 1} \tilde{f} \|_{L^\infty_{x,v}} \underbrace{ \int_{\R^3} |V_\shortparallel|^{-1} \tfrac{(1 + |\zeta_\shortparallel|)^{1 - \gamma}}{(1 + |V_\shortparallel|)^{1-\gamma}} \exp \big( - \tfrac{|\zeta_\shortparallel + \tfrac{1}{2} V_\shortparallel|^2}{4 T} \big) \d V_\shortparallel }_{: = W} \,.
		\end{aligned}}
	\end{equation*}
	Set $ I\!I_1 = \{ V_\shortparallel; |\zeta_\shortparallel + \tfrac{1}{2} V_\shortparallel| > \tfrac{1}{2} |\zeta_\shortparallel| \} \,, I\!I_2 = \{ V_\shortparallel; |\zeta_\shortparallel + \tfrac{1}{2} V_\shortparallel| \leq \tfrac{1}{2} |\zeta_\shortparallel| \} $. The quantity $W$ can thereby be decomposed as
	\begin{equation*}
		\begin{aligned}
			W = \int_{I\!I_1} (\cdots) \d V_\shortparallel + \int_{I\!I_2} (\cdots) \d V_\shortparallel : = W_1 + W_2 \,.
		\end{aligned}
	\end{equation*}
	If $|\zeta_\shortparallel + \tfrac{1}{2} V_\shortparallel| > \tfrac{1}{2} |\zeta_\shortparallel|$, one has $ \tfrac{(1 + |\zeta_\shortparallel|)^{1 - \gamma}}{(1 + |V_\shortparallel|)^{1-\gamma}} \exp \big( - \tfrac{|\zeta_\shortparallel + \tfrac{1}{2} V_\shortparallel|^2}{8 T} \big) \leq C $. Then
	\begin{equation}\label{W1}
		\begin{aligned}
			W_1 \leq & C \int_{\R^3} |V_\shortparallel|^{-1} \exp \big( - \tfrac{|\zeta_\shortparallel + \tfrac{1}{2} V_\shortparallel|^2}{8 T} \big) \d V_\shortparallel \\
			\leq & C \int_{|V_\shortparallel| \leq 1} |V_\shortparallel|^{-1} \d V_\shortparallel + C \int_{|V_\shortparallel| > 1} \exp \big( - \tfrac{|\zeta_\shortparallel + \tfrac{1}{2} V_\shortparallel|^2}{8 T} \big) \d V_\shortparallel \leq C \,.
		\end{aligned}
	\end{equation}
	If $|\zeta_\shortparallel + \tfrac{1}{2} V_\shortparallel| \leq \tfrac{1}{2} |\zeta_\shortparallel|$, it follows $ \tfrac{1}{2} |V_\shortparallel| \geq |\zeta_\shortparallel| - |\zeta_\shortparallel + \tfrac{1}{2} V_\shortparallel| \geq \tfrac{1}{2} |\zeta_\shortparallel| $, i.e., $|V_\shortparallel| \geq |\zeta_\shortparallel|$, which infers that $\tfrac{(1 + |\zeta_\shortparallel|)^{1 - \gamma}}{(1 + |V_\shortparallel|)^{1-\gamma}} \leq 1$. As a result,
	\begin{equation}\label{W2}
		\begin{aligned}
			W_2 \leq \int_{\R^3} |V_\shortparallel|^{-1} \exp \big( - \tfrac{|\zeta_\shortparallel + \tfrac{1}{2} V_\shortparallel|^2}{4 T} \big) \d V_\shortparallel \leq C \,.
		\end{aligned}
	\end{equation}
	It thereby holds $W = W_1 + W_2 \leq C$, which shows
	\begin{equation}\label{Lambda2-bnd}
		\begin{aligned}
			|\Lambda_2| \leq C \| (1 + |v|)^{\gamma - 1} \tilde{f} \|_{L^\infty_{x,v}} \,.
		\end{aligned}
	\end{equation}
	
	By \eqref{Lambda1-bnd} and \eqref{Lambda2-bnd}, one establishes that for $- 3 < \gamma \leq 1$,
	\begin{equation}\label{K2p-bnd-h}
		\begin{aligned}
			| K_{2 \mathfrak{p}}^{\chi} \tilde{f} (x,v) | \leq |\Lambda_1| + |\Lambda_2| \leq C \| (1 + |v|)^{\gamma - 1} \tilde{f} \|_{L^\infty_{x,v}} \,.
		\end{aligned}
	\end{equation}
	
	Finally, by plugging the estimates \eqref{K1p-bnd}, \eqref{K2p-bnd-l} and \eqref{K2p-bnd-h} into \eqref{Kp-split}, one knows that the bound \eqref{Kp-bnd} holds. The proof of Lemma \ref{Lmm-K-Oprt} is therefore completed.
\end{proof}

\subsection{$L^\infty_x L^2_v$-$L^\infty_{x,v}$ property of $K$: Proof of Lemma \ref{Lmm-Kh-2-infty}}\label{Subsec:K-2infty}

In this subsection, we will justify the property that the weighted $L^\infty_x L^2_v$ norm can be bounded by the weighted $L^\infty_{x,v}$ norm, i.e., proof of Lemma \ref{Lmm-Kh-2-infty}.

\begin{proof}[Proof of Lemma \ref{Lmm-Kh-2-infty}]
	Let $\chi_\eps (r)$ be a smooth monotone function satisfying $\chi_\eps (r) = 0$ for $0 \leq r \leq \eps$ and $\chi_\eps (r) = 1$ for $r \geq 2 \eps$, where $\eps > 0$ is small to be determined. By \eqref{K-K1-K2}-\eqref{K1}-\eqref{K2}, one has
	\begin{equation}
		\begin{aligned}
			K_\hbar g (x,v) = - K_{\hbar 1} g (x,v) + K_{\hbar 2}^{1 - \chi_\eps} g (x,v) + K_{\hbar 2}^{\chi_\eps} g (x,v) \,,
		\end{aligned}
	\end{equation}
	where for $\Gamma = \chi_\eps$ and $1 - \chi_\eps$,
		\begin{align}\label{Kh-12}
			\no K_{\hbar 1} g (x,v) = & \int_{\R^3} e^{\hbar \sigma (x,v) - \hbar \sigma (x, v_*)} |v - v_*|^\gamma \M^\frac{1}{2} (v) \M^\frac{1}{2} (v_*) g (x, v_*) \d v_* \,, \\
			K_{\hbar 2}^\Gamma g (x,v) = & \iint_{\mathbb{R}^3 \times \mathbb{S}^2} \Gamma (|v - v_*|) e^{\hbar \sigma (x,v) - \hbar \sigma (x, v')} \M^\frac{1}{2} (v_*') \M^\frac{1}{2} (v_*) g (x,v') b (\omega, v_* - v) \d \omega \d v_* \\
			\no & + \iint_{\mathbb{R}^3 \times \mathbb{S}^2} \Gamma (|v - v_*|) e^{\hbar \sigma (x,v) - \hbar \sigma (x, v_*')} \M^\frac{1}{2} (v') \M^\frac{1}{2} (v_*) g (x,v_*') b (\omega, v_* - v) \d \omega \d v_* \,.
		\end{align}
	Observe that $\LL \nu^{-1} K_\hbar g \RR_{A; m, 0, \vartheta} = \| \sigma_x^\frac{m}{2} w_{- \gamma, \vartheta} K_\hbar g \|_{L^\infty_{x,v}} $.
	
	{\bf Step 1. Control the quantity $\LL \nu^{-1} K_{\hbar 1} g \RR_{A; m, 0, \vartheta}$.}
	
	By splitting
{\small
	\begin{equation*}
		\begin{aligned}
			\sigma_x^\frac{m}{2} (x,v) w_{- \gamma, \vartheta} (v) K_{\hbar 1} g (x,v) = & \sigma_x^\frac{m}{2} (x,v) w_{- \gamma, \vartheta} (v) \int_{|v-v_*|\leq \eps} (\cdots) \d v_* \\
			& + \sigma_x^\frac{m}{2} (x,v) w_{- \gamma, \vartheta} (v) \int_{|v-v_*|> \eps} (\cdots) \d v_* = I_1 + I_2 \,.
		\end{aligned}
	\end{equation*}}
	Recalling the bound \eqref{p-pro} with $\beta = 0$, one has
	\begin{equation}\label{p0-bnd}
		\begin{aligned}
			\tfrac{\sigma_x^\frac{m}{2} (x,v) w_{0, \vartheta} (v) e^{\hbar \sigma (x,v)} }{ \sigma_x^\frac{m}{2} (x,v_*) w_{0, \vartheta} (v_*) e^{\hbar \sigma (x,v_*)} } \leq C (1 + |v - v_*|)^\frac{|m| ( 1 - \gamma )}{2} e^{ (c \hbar + \vartheta ) \big| |v - \u|^2 - |v_* - \u|^2 \big| }
		\end{aligned}
	\end{equation}
	for sufficiently small $\hbar, \vartheta \geq 0$. It therefore holds
{\small
	\begin{equation*}
		\begin{aligned}
			|I_1| \leq & C \LL g \RR_{A; m, 0, \vartheta} (1 + |v|)^{- \gamma} \int_{|v-v_*| \leq \eps} |v-v_*|^\gamma (1 + |v - v_*|)^\frac{|m| ( 1 - \gamma )}{2} e^{- [ \tfrac{T}{4} - (c \hbar + \vartheta ) ] (|v- \u|^2 + |v_* - \u|^2) } \d v_* \\
			\leq & C \LL g \RR_{A; m, 0, \vartheta} \int_{|v-v_*| \leq \eps} |v-v_*|^\gamma e^{- [ \tfrac{T}{8} - (c \hbar + \vartheta ) ] (|v- \u|^2 + |v_* - \u|^2) } \d v_* \leq C \eps^{3 + \gamma} \LL g \RR_{A; m, 0, \vartheta} \,,
		\end{aligned}
	\end{equation*}}
	provided that $\hbar, \vartheta \geq 0$ are both sufficiently small and $ \gamma > - 3$.
	
	By Lemma \ref{Lmm-sigma} and \eqref{pro-sigma}, one easily knows that
	\begin{equation*}
		\begin{aligned}
			w_{- \gamma, \vartheta} (v) \sigma_x^\frac{m}{2} (x,v) \sigma_x^{- \frac{m}{2}} (x,v_*) e^{\hbar \sigma (x,v) - \hbar \sigma (x, v_*)} |v - v_*|^\gamma \mathbf{1}_{|v-v_*| > \eps} \M^\frac{1}{4} (v) \M^\frac{1}{4} (v_*) \leq C_\eps
		\end{aligned}
	\end{equation*}
	uniformly in $(x,v,v_*)$, provided that $\hbar, \vartheta \geq 0$ are small enough. Then the quantity $I_2$ can be bounded by
	\begin{equation*}
		\begin{aligned}
			|I_2| \leq & C_\eps \int_{|v-v_*|> \eps} \M^\frac{1}{4} (v) \M^\frac{1}{4} (v_*) \sigma_x^\frac{m}{2} (x,v_*) |g (x,v_*)| \d v_* \\
			\leq & C_\eps \Big( \int_{|v-v_*|> \eps} z_{- \alpha'}^2 (v_*) \M^\frac{1}{2} (v) \M^\frac{1}{2} (v_*) \d v_* \Big)^\frac{1}{2} \| \sigma_x^\frac{m}{2} z_{\alpha'} g \|_{L^\infty_x L^2_v} \\
			\leq & C_\eps \| \sigma_x^\frac{m}{2} z_{\alpha'} g \|_{L^\infty_x L^2_v} \leq C_\eps \| \sigma_x^\frac{m}{2} z_{\alpha'} w_{- \gamma, \vartheta } g \|_{L^\infty_x L^2_v} \,,
		\end{aligned}
	\end{equation*}
	where the last inequality is derived from $0 \leq \alpha' < \frac{1}{2}$. Here $z_{\alpha'}$ is defined in \eqref{z-alpha}. Consequently, it holds
	\begin{equation*}
		\begin{aligned}
			| \sigma_x^\frac{m}{2} (x,v) w_{- \gamma, \vartheta} (v) K_{\hbar 1} g (x,v) | \leq |I_1| + |I_2| \leq & C \eps^{3 + \gamma} \LL g \RR_{A; m, 0, \vartheta} + C_\eps \| \sigma_x^\frac{m}{2} z_{\alpha'} w_{- \gamma, \vartheta} g \|_{L^\infty_x L^2_v} \,,
		\end{aligned}
	\end{equation*}
	which means that
	\begin{equation}\label{Loss-Est}
		\begin{aligned}
			\LL \nu^{-1} K_{\hbar1} g \RR_{A; m, 0, \vartheta} \leq C \eps^{3 + \gamma} \LL g \RR_{A; m, 0, \vartheta} + C_\eps \| \sigma_x^\frac{m}{2} z_{\alpha'} w_{- \gamma, \vartheta} g \|_{L^\infty_x L^2_v}
		\end{aligned}
	\end{equation}
	for $0 \leq \alpha' < \frac{1}{2}$, $- 3 < \gamma \leq 1$, $m \in \R$, sufficiently small $\hbar, \vartheta \geq 0$, and small $\eps > 0$ to be determined.
	
	{\bf Step 2. Control the quantity $\LL \nu^{-1} K_{\hbar 2}^{1 - \chi_\eps} g \RR_{m, 0, \vartheta}$. }
	
	Recalling the definition of $K_{\hbar 2}^{1 - \chi_\eps} g (x,v)$ in \eqref{Kh-12}, we decompose $K_{\hbar 2}^{1 - \chi_\eps} g (x,v) = I\!I_1 + I\!I_2$, where
	\begin{equation*}
		\begin{aligned}
			I\!I_1 = & \iint_{\mathbb{R}^3 \times \mathbb{S}^2} (1 - \chi_\eps) (|v - v_*|) e^{\hbar \sigma (x,v) - \hbar \sigma (x, v')} \M^\frac{1}{2} (v_*') \M^\frac{1}{2} (v_*) g (x,v') b (\omega, v_* - v) \d \omega \d v_* \,, \\
			I\!I_2 = & \iint_{\mathbb{R}^3 \times \mathbb{S}^2} (1 - \chi_\eps) (|v - v_*|) e^{\hbar \sigma (x,v) - \hbar \sigma (x, v_*')} \M^\frac{1}{2} (v') \M^\frac{1}{2} (v_*) g (x,v_*') b (\omega, v_* - v) \d \omega \d v_* \,.
		\end{aligned}
	\end{equation*}
	Due to $|v - v_*| \leq 2 \eps \leq 2$, it follow from \eqref{Claim-M} that $\M (v_*') \leq C \M^{1 - \delta_0}$ and $\M (v_*') \leq C \M^{1 - \delta_0} (v')$ for $\delta_0 \in (0, 1)$. Together with \eqref{p0-bnd} (replacing $v_*$ by $v'$), it infers that for sufficiently small $\hbar, \vartheta \geq 0$,
	\begin{equation}\label{II1-pm}
		\begin{aligned}
			| \sigma_x^\frac{m}{2} (x,v) w_{- \gamma, \vartheta} (v) I\!I_1 | \leq C I\!I_1^- \,,
		\end{aligned}
	\end{equation}
	where
	\begin{equation*}{\small
		\begin{aligned}
			I\!I_1^- = & (1 + |v|)^{- \gamma} \iint_{\mathbb{R}^3 \times \mathbb{S}^2} (1 + |v - v'|)^\frac{|m| ( 1 - \gamma )}{2} \M^\frac{1 - \delta_0}{4} (v) |\sigma_x^\frac{m}{2} w_{0, \vartheta} g (x,v')| \\
			& \times (1 - \chi_\eps) (|v - v_*|) e^{(c \hbar + \vartheta) (|v - \u|^2 + |v' - \u|^2)}  \M^\frac{1 - \delta_0}{4} (v') \M^\frac{1}{2} (v_*) b (\omega, v_* - v) \d \omega \d v_* \,.
		\end{aligned}}
	\end{equation*}
	For the quantity $I\!I_1^-$, one easily has
{\small
	\begin{equation}\label{II1-n}
		\begin{aligned}
			I\!I_1^- \leq & C \LL g \RR_{A; m, 0, \vartheta} (1 + |v|)^{- \gamma} \iint_{\mathbb{R}^3 \times \mathbb{S}^2} (1 - \chi_\eps) (|v - v_*|) (1 + |v - v'|)^\frac{|m| ( 1 - \gamma )}{2} \M^\frac{1 - \delta_0}{4} (v) \\
			& \times e^{(c \hbar + \vartheta) (|v - \u|^2 + |v' - \u|^2)}  \M^\frac{1 - \delta_0}{4} (v') \M^\frac{1}{2} (v_*) b (\omega, v_* - v) \d \omega \d v_* \\
			\leq & C \LL g \RR_{A; m, 0, \vartheta} \int_{\R^3} (1 - \chi_\eps) (|v - v_*|) \M^\frac{1 - \delta_0}{8} (v) \M^\frac{1}{2} (v_*) |v-v_*|^\gamma \d v_* \leq C \eps^{3 + \gamma} \LL g \RR_{A; m, 0, \vartheta} \,.
		\end{aligned}
	\end{equation}}
	Then the bounds \eqref{II1-pm} and \eqref{II1-n} indicate that
	\begin{equation}\label{II1-bnd}
		\begin{aligned}
			| \sigma_x^\frac{m}{2} (x,v) w_{- \gamma, \vartheta} (v) I\!I_1 | \leq C \eps^{3 + \gamma} \LL g \RR_{A; m, 0, \vartheta}
		\end{aligned}
	\end{equation}
	for $- 3 < \gamma \leq 1$ and sufficiently small $\vartheta, \hbar \geq 0$.
	
	Similarly in \eqref{II1-bnd}, there holds
	\begin{equation}\label{II2-bnd}
		\begin{aligned}
			| \sigma_x^\frac{m}{2} (x,v) w_{- \gamma, \vartheta} (v) I\!I_2 | \leq C \eps^{3 + \gamma} \LL g \RR_{A; m, 0, \vartheta}
		\end{aligned}
	\end{equation}
	for $- 3 < \gamma \leq 1$ and sufficiently small $\vartheta, \hbar \geq 0$. Then, by \eqref{II1-bnd} and \eqref{II2-bnd}, one has
	\begin{equation}\label{Lv-Est}
		\begin{aligned}
			\LL \nu^{-1} K_{\hbar 2}^{1 - \chi_\eps} g \RR_{A; m, 0, \vartheta} \leq C \eps^{3 + \gamma} \LL g \RR_{A; m, 0, \vartheta}
		\end{aligned}
	\end{equation}
	for $- 3 < \gamma \leq 1$, $ m \in \R $ and sufficiently small $\vartheta, \hbar \geq 0$.
	
	{\bf Step 3. Control the quantity $\LL \nu^{-1} K_{\hbar 2}^{\chi_\eps} g \RR_{A; m, 0, \vartheta}$. }
	
	Recalling \eqref{Kh-12}, we split $K_{\hbar 2}^{\chi_\eps} g (x,v) = I\!I\!I_1 + I\!I\!I_2$, where
	\begin{equation}\label{III1-III2}
		\begin{aligned}
			I\!I\!I_1 = & \iint_{\mathbb{R}^3 \times \mathbb{S}^2} \chi_\eps (|v - v_*|) e^{\hbar \sigma (x,v) - \hbar \sigma (x, v')} \M^\frac{1}{2} (v_*') \M^\frac{1}{2} (v_*) g (x,v') b (\omega, v_* - v) \d \omega \d v_* \,, \\
			I\!I\!I_2 = & \iint_{\mathbb{R}^3 \times \mathbb{S}^2} \chi_\eps (|v - v_*|) e^{\hbar \sigma (x,v) - \hbar \sigma (x, v_*')} \M^\frac{1}{2} (v') \M^\frac{1}{2} (v_*) g (x,v_*') b (\omega, v_* - v) \d \omega \d v_* \,.
		\end{aligned}
	\end{equation}
	
	\underline{\em Step 3.1. Estimate of $\LL \nu^{-1} I\!I\!I_1 \RR_{A; m, 0, \vartheta}$.} We first divide the quantity $I\!I\!I_1$ into two parts:
	\begin{equation}
		\begin{aligned}
			I\!I\!I_1 = & \underbrace{ \iint_{\mathbb{R}^3 \times \mathbb{S}^2} \mathbf{1}_{\{ |v' - v| \leq \tilde{\eps} |v_* - v| \}} (\cdots) \d \omega \d v_* }_{:= I\!I\!I_{11} } + \underbrace{ \iint_{\mathbb{R}^3 \times \mathbb{S}^2} \mathbf{1}_{\{ |v' - v| > \tilde{\eps} |v_* - v| \}} (v') (\cdots) \d \omega \d v_* }_{:= I\!I\!I_{12} } \,,
		\end{aligned}
	\end{equation}
	where $\tilde{\eps} > 0$ is small enough to be determined.
	
	{\em Step 3.1.1. Estimate the quantity $| \sigma_x^\frac{m}{2} (x,v) w_{- \gamma , \vartheta} (v) I\!I\!I_{11} |$.} If $|v' - v| \leq \tilde{\eps} |v_* - v|$, one has
	\begin{equation*}
		\begin{aligned}
			b (\omega, v_* - v) \leq C |v_* - v|^\gamma \tfrac{|v' - v|}{|v_* - v|} \leq C \tilde{\eps} |v_* - v|^\gamma \,.
		\end{aligned}
	\end{equation*}
	Together with \eqref{p0-bnd} (replacing $v_*$ by $v'$), it is easy to derive that
	\begin{equation}
		\begin{aligned}
			| \sigma_x^\frac{m}{2} (x,v) w_{- \gamma , \vartheta} (v) I\!I\!I_{11} | \leq C \tilde{\eps} \widetilde{I\!I\!I}_{11} \LL g \RR_{A; m, 0, \vartheta} \,,
		\end{aligned}
	\end{equation}
	where
{\footnotesize
	\begin{equation*}
		\begin{aligned}
			\widetilde{I\!I\!I}_{11} = \iint_{\mathbb{R}^3 \times \mathbb{S}^2} \tfrac{\chi_\eps (|v_* - v|)}{|v_* - v|^{- \gamma}} (1 + |v|)^{- \gamma} (1 + |v' - v|)^\frac{|m| ( 1 - \gamma )}{2} e^{ (c \hbar + \vartheta ) \big| |v - \u|^2 - |v' - \u|^2 \big| } e^{ - \tfrac{1}{4 T} ( |v_* - \u|^2 + |v_*' - \u|^2 ) } \d \omega \d v_* \,.
		\end{aligned}
	\end{equation*}}
	Note that {\small $\tfrac{\chi_\eps (|v_* - v|)}{|v_* - v|^{- \gamma}} (1 + |v|)^{- \gamma} \leq C_\eps (1 + |v_* - \u|)^{|\gamma|}$}, which means that {\small $ \tfrac{\chi_\eps (|v_* - v|)}{|v_* - v|^{- \gamma} (1 + |v|)^\gamma } e^{- \tfrac{|v_* - \u|^2}{8 T}} \leq C_\eps $}. Then
	\begin{equation}
		\begin{aligned}
			\widetilde{I\!I\!I}_{11} \leq C_\eps \iint_{\mathbb{R}^3 \times \mathbb{S}^2} (1 + |v' - v|)^\frac{|m| ( 1 - \gamma )}{2} e^{ (c \hbar + \vartheta ) \big| |v - \u|^2 - |v' - \u|^2 \big| } e^{ - \tfrac{1}{8 T} ( |v_* - \u|^2 + |v_*' - \u|^2 ) } \d \omega \d v_* \,.
		\end{aligned}
	\end{equation}
	By letting $V = v_* - v$, $V_\shortparallel = (V \cdot \omega) \omega$, $V_\perp = V - V_\shortparallel$ and $\omega = \frac{v' -v}{|v' - v|}$, one has $\d \omega \d v_* = 2 |V_\shortparallel|^{-2} \d V_\perp \d V_\shortparallel$, see Section 2 of \cite{Grad-1963}. Moreover, $v' = v + V_\shortparallel$, $v_*' = v + V_\perp$, $v_* = v + V$. Recall \eqref{zeta}, hence, $ \zeta = v + \tfrac{1}{2} V_\shortparallel \,,  \zeta_\shortparallel = [(\zeta - \u) \cdot \omega] \omega \,, \zeta_\perp = (\zeta - \u) - \zeta_\shortparallel $. Combining with \eqref{Trans-1} and \eqref{Trans-2}, one has
{\small
	\begin{equation*}
		\begin{aligned}
			\widetilde{I\!I\!I}_{11} \leq & 2 C_\eps \int_{\R^3} \int_{V_\perp \perp V_\shortparallel} |V_\shortparallel|^{-2} (1 + |V_\shortparallel|)^\frac{|m| ( 1 - \gamma )}{2} e^{- \tfrac{1}{4 T} (|\zeta_\shortparallel|^2 + \tfrac{1}{4} |V_\shortparallel|^2) + 2 (c \hbar + \vartheta) |V_\shortparallel \cdot \zeta_\shortparallel|} e^{- \tfrac{|V_\perp + \zeta_\perp|^2}{4 T}} \d V_\perp \d V_\shortparallel \leq C_\eps
		\end{aligned}
	\end{equation*}}
	provided that $\hbar, \vartheta \geq 0$ are both sufficiently small. It therefor follows that
	\begin{equation}\label{III-11}
		\begin{aligned}
			| \sigma_x^\frac{m}{2} (x,v) w_{- \gamma , \vartheta} (v) I\!I\!I_{11} | \leq C_\eps \tilde{\eps} \LL g \RR_{A; m, 0, \vartheta} \,.
		\end{aligned}
	\end{equation}

	{\em Step 3.1.2. Estimate the quantity $| \sigma_x^\frac{m}{2} (x,v) w_{- \gamma , \vartheta} (v) I\!I\!I_{12} |$.} It easily follows from \eqref{pro-sigma} that $\sigma_x^\frac{m}{2} (x,v) \sigma_x^{- \frac{m}{2}} (x,v') \leq C (1 + |v-v'|)^\frac{|m| ( 1 - \gamma )}{2}$. Then by Lemma \ref{Lmm-sigma},
	\begin{equation*}{\small
		\begin{aligned}
			| I\!I\!I_{12} | \leq C \iint_{\mathbb{R}^3 \times \mathbb{S}^2} \chi_\eps (|v_* - v|) e^{c \hbar ||v-\u|^2 - |v'-\u|^2|} \M^\frac{1}{2} (v_*') \M^\frac{1}{2} (v_*) \sigma_x^\frac{m}{2} (x,v') | g (x, v')| \\
			\times (1 + |v-v'|)^\frac{|m| ( 1 - \gamma )}{2} \mathbf{1}_{\{ |v' - v| > \tilde{\eps} |v_* - v| \}} |v_* - v|^\gamma \tfrac{|v' - v|}{|v_* - v|} \d \omega \d v_* \,.
		\end{aligned}}
	\end{equation*}
	Note that $ w_{- \gamma, \vartheta} (v) \leq C (1 + |v_* - \u|)^{|\gamma|} e^{2 \vartheta |v_* - \u|^2} \cdot (1 + |v_* - v|)^{|\gamma|} e^{2 \vartheta |v_* - v|^2} $, which means that
	\begin{equation*}
		\begin{aligned}
			w_{- \gamma, \vartheta} (v) \M^\frac{1}{4} (v_*) \leq C (1 + |v_* - v|)^{|\gamma|} e^{2 \vartheta |v_* - v|^2}
		\end{aligned}
	\end{equation*}
	for sufficiently small $\vartheta \geq 0$. Moreover, one has $\tfrac{\chi_\eps}{|v_* - v|^{1 - \gamma}} \leq C_\eps (1 + |v_* - v|)^{\gamma - 1}$. Then
{\footnotesize
	\begin{equation*}
		\begin{aligned}
			| \sigma_x^\frac{m}{2} (x,v) w_{- \gamma, \vartheta} (v) I\!I\!I_{12} | \leq C_\eps \iint_{\mathbb{R}^3 \times \mathbb{S}^2} |v'-v| (1 + |v-v'|)^\frac{|m| ( 1 - \gamma )}{2} (1 + |v_* - v|)^{|\gamma| + \gamma - 1} e^{2 \vartheta |v_* - v|^2} \\
			\times e^{c \hbar ||v-\u|^2 - |v'-\u|^2|} \sigma_x^\frac{m}{2} (x,v') (\delta x + l)^\n | g (x,v')| e^{- \tfrac{1}{8 T} (|v_*' - \u|^2 + |v_* - \u|^2)} \mathbf{1}_{\{ |v' - v| > \tilde{\eps} |v_* - v| \}} \d \omega \d v_* \,.
		\end{aligned}
	\end{equation*}}
	By employing the change variables $(\omega, v_*) \mapsto (V_\perp, V_\shortparallel)$, one gains
{\small
	\begin{equation*}
		\begin{aligned}
			| \sigma_x^\frac{m}{2} (x,v) w_{- \gamma, \vartheta} (v) I\!I\!I_{12} | \leq C_\eps \int_{\R^3} \int_{V_\perp \perp V_\shortparallel} |V_\shortparallel|^{-1} (1 + |V_\shortparallel|)^\frac{|m| ( 1 - \gamma )}{2} (1 + |V|)^{|\gamma| + \gamma - 1} e^{2 \vartheta |V|^2} \\
			\times e^{2 c \hbar |V_\shortparallel \cdot \zeta_\shortparallel|} |\sigma_x^\frac{m}{2} g (x, v+ V_\shortparallel)| e^{- \tfrac{\frac{1}{4} |V_\shortparallel|^2 + |\zeta_\shortparallel|^2 + |V_\perp + \zeta_\perp|^2}{4 T}} \mathbf{1}_{\{ |V_\shortparallel| > \tilde{\eps} |V| \}} \d V_\perp V_\shortparallel \,.
		\end{aligned}
	\end{equation*}}
	Observe that for sufficiently small $\hbar \geq 0$
	\begin{equation*}
		\begin{aligned}
			& e^{- \tfrac{\frac{1}{4} |V_\shortparallel|^2 + |\zeta_\shortparallel|^2 + |V_\perp + \zeta_\perp|^2}{4 T} + 2 c \hbar |V_\shortparallel \cdot \zeta_\shortparallel|} \mathbf{1}_{\{ |V_\shortparallel| > \tilde{\eps} |V| \}} \\
			\leq & e^{- C_0 ( |V_\shortparallel|^2 + |\zeta_\shortparallel|^2 + |V_\perp + \zeta_\perp|^2 )} \mathbf{1}_{\{ |V_\shortparallel| > \tilde{\eps} |V| \}} \leq e^{- \frac{C_0 \tilde{\eps}^2}{2} |V|^2} e^{- \frac{C_0}{2} |V_\shortparallel|^2}  e^{- C_0 |V_\perp + \zeta_\perp|^2} \,,
		\end{aligned}
	\end{equation*}
	and $ (1 + |V|)^{|\gamma| + \gamma - 1} e^{- \frac{C_0 \tilde{\eps}^2}{2} |V|^2} \leq C_{\tilde{\eps}} $, $ \int_{V_\perp \perp V_\shortparallel} e^{- C_0 |V_\perp + \zeta_\perp|^2} \d V_\shortparallel \leq C $. Then one has
	\begin{equation*}{\small
		\begin{aligned}
			| \sigma_x^\frac{m}{2} (x,v) w_{- \gamma, \vartheta} (v) I\!I\!I_{12} | \leq & C_{\eps, \tilde{\eps}} \int_{\R^3} |V_\shortparallel|^{-1} (1 + |V_\shortparallel|)^\frac{|m| ( 1 - \gamma )}{2} e^{- \frac{C_0}{2} |V_\shortparallel|^2} |\sigma_x^\frac{m}{2} g (x, v + V_\shortparallel)| \d V_\shortparallel \\
			\leq & C_{\eps, \tilde{\eps}} \big( \int_{\R^3} z_{\alpha'}^2 (v + V_\shortparallel) |\sigma_x^\frac{m}{2} g (x, v + V_\shortparallel)|^2 \d V_\shortparallel \big)^\frac{1}{2} \\
			& \times \big( \int_{\R^3} |V_\shortparallel|^{-2} (1 + |V_\shortparallel|)^{|m| ( 1 - \gamma )} z_{- \alpha'}^2 (v + V_\shortparallel) e^{- C_0 |V_\shortparallel|^2} \d V_\shortparallel \big)^\frac{1}{2} \,.
		\end{aligned}}
	\end{equation*}
	Due to $ \int_{\R^3} |V_\shortparallel|^{-2} (1 + |V_\shortparallel|)^{|m| ( 1 - \gamma )} z_{- \alpha'}^2 (v + V_\shortparallel) e^{- C_0 |V_\shortparallel|^2} \d V_\shortparallel \leq C $ uniformly in $v \in \R^3$ whence $0 \leq \alpha' < \frac{1}{2}$. It thereby infers that for $0 \leq \alpha' < \frac{1}{2}$,
	\begin{equation}\label{III-12}
		\begin{aligned}
			| \sigma_x^\frac{m}{2} (x,v) w_{- \gamma, \vartheta} (v) I\!I\!I_{12} | \leq C_{\eps, \tilde{\eps}} \| \sigma_x^\frac{m}{2} z_{\alpha'} g \|_{L^\infty_x L^2_v} \,.
		\end{aligned}
	\end{equation}

	Then \eqref{III-11} and \eqref{III-12} imply that
	\begin{equation}\label{III-1}
		\begin{aligned}
			| \sigma_x^\frac{m}{2} (x,v) w_{- \gamma, \vartheta} (v) I\!I\!I_1| \leq C_\eps \tilde{\eps} \LL g \RR_{A; m, 0, \vartheta} + C_{\eps, \tilde{\eps}} \| \sigma_x^\frac{m}{2} z_{\alpha'} w_{- \gamma, \vartheta} g \|_{L^\infty_x L^2_v}
		\end{aligned}
	\end{equation}
	for $0 \leq \alpha' < \frac{1}{2}$, $- 3 < \gamma \leq 1$, $ m \in \R $ and sufficiently small $\hbar, \vartheta \geq 0$, where $\eps, \tilde{\eps} > 0$ are both small to be determined later.
	
	\underline{\em Step 3.2. Estimate of $\LL \nu^{-1} I\!I\!I_2 \RR_{A; m, 0, \vartheta}$.} Recalling the definition of $I\!I\!I_2$ in \eqref{III1-III2}, we decompose it as
	\begin{equation*}
		\begin{aligned}
			I\!I\!I_2 = I\!I\!I_{21} + I\!I\!I_{22} \,,
		\end{aligned}
	\end{equation*}
	where
{\small
	\begin{equation*}
		\begin{aligned}
			I\!I\!I_{21} = \iint_{\mathbb{R}^3 \times \mathbb{S}^2} \chi_\eps (|v - v_*|) e^{\hbar \sigma (x,v) - \hbar \sigma (x, v_*')} \M^\frac{1}{2} (v') \M^\frac{1}{2} (v_*) g (x,v_*') \mathbf{1}_{\{ |v_*' - v| \leq \tilde{\eps} |v_* - v| \}} b (\omega, v_* - v) \d \omega \d v_* \,,
		\end{aligned}
	\end{equation*}}
	and
{\small
		\begin{align*}
			I\!I\!I_{22} = \iint_{\mathbb{R}^3 \times \mathbb{S}^2} \chi_\eps (|v - v_*|) e^{\hbar \sigma (x,v) - \hbar \sigma (x, v_*')} \M^\frac{1}{2} (v') \M^\frac{1}{2} (v_*) g (x,v_*') \mathbf{1}_{\{ |v_*' - v| > \tilde{\eps} |v_* - v| \}} b (\omega, v_* - v) \d \omega \d v_* \,.
		\end{align*}}
	
	{\em Step 3.2.1. Control the quantity $| \sigma_x^\frac{m}{2} (x,v) w_{- \gamma, \vartheta} (v) I\!I\!I_{21} |$.} By $b (\omega, v_* - v) \leq C |v_* - v|^\gamma \frac{|v' - v|}{|v_* - v|} \leq C |v_* - v|^\gamma$ and $\frac{\chi_\eps (|v_* - v|) (1 + |v|)^{- \gamma}}{|v_* - v|^{- \gamma}} \leq C_\eps (1 + |v_*|)^{|\gamma|}$, one has
	\begin{equation*}
		\begin{aligned}
			| \sigma_x^\frac{m}{2} (x,v) w_{- \gamma, \vartheta} (v) I\!I\!I_{21} | \leq & C_\eps \iint_{\mathbb{R}^3 \times \mathbb{S}^2} \tfrac{e^{\hbar \sigma (x,v) } \sigma_x^\frac{m}{2} (x,v) w_{0, \vartheta } (v)}{e^{\hbar \sigma (x,v_*') } \sigma_x^\frac{m}{2} (x,v_*') w_{0, \vartheta }(v_*')} (1 + |v_*|)^{|\gamma|} \mathbf{1}_{\{ |v_*' - v| \leq \tilde{\eps} |v_* - v| \}} \\
			& \quad \times | \sigma_x^\frac{m}{2} (x,v_*') w_{0, \vartheta }(v_*') g (x, v_*') | \M^\frac{1}{2} (v') \M^\frac{1}{2} (v_*) \d \omega \d v_* \\
			\leq & C_\eps \LL g \RR_{A; m, 0, \vartheta} \widetilde{I\!I\!I}_{21} \,,
		\end{aligned}
	\end{equation*}
	where
	\begin{equation*}
		\begin{aligned}
			\widetilde{I\!I\!I}_{21} = \iint_{\mathbb{R}^3 \times \mathbb{S}^2} (1 + |v_*' - v|)^\frac{|m| ( 1 - \gamma )}{2} e^{(c \hbar + \vartheta) ||v-\u|^2 - |v_*' - \u|^2|} e^{- \frac{|v_*-\u|^2 + |v' - \u|^2}{8 T}} \mathbf{1}_{\{ |v_*' - v| \leq \tilde{\eps} |v_* - v| \}} \d \omega \d v_* \,.
		\end{aligned}
	\end{equation*}
	Here we have used the bound $ (1 + |v_*|)^{|\gamma|} \M^\frac{1}{2} (v') \M^\frac{1}{2} (v_*) \leq C e^{- \frac{|v_*-\u|^2 + |v' - \u|^2}{8 T}} $ and the inequality \eqref{p0-bnd} (replacing $v_*$ by $v_*'$). By employing the change variables $(\omega, v_*) \mapsto (V_\shortparallel, V_\perp)$ and combining with the relations \eqref{L1}-\eqref{L2}, it holds that for sufficiently small $\hbar, \vartheta \geq 0$
{\footnotesize
	\begin{equation*}
		\begin{aligned}
			\widetilde{I\!I\!I}_{21} = & \int_{\R^3} \int_{V_\perp \perp V_\shortparallel} (1 + |V_\perp|)^\frac{|m| ( 1 - \gamma )}{2} e^{(c \hbar + \vartheta) |- |V_\perp + \zeta_\perp|^2 + |\zeta_\perp|^2 |} \mathbf{1}_{\{ |V_\perp| \leq \tilde{\eps} |V| \}} \\
& \qquad \qquad \qquad \times e^{- \tfrac{|\zeta_\shortparallel + \frac{1}{2} V_\shortparallel|^2}{4 T} - \frac{|\zeta_\perp|^2}{8 T} - \frac{|V_\perp + \zeta_\perp|^2}{8 T}} \cdot 2 |V_\shortparallel|^{-2} \d V_\perp \d V_\shortparallel \\
			\leq & C \int_{\R^3} \int_{V_\perp \perp V_\shortparallel} (1 + |V_\perp + \zeta_\perp|)^\frac{|m| ( 1 - \gamma )}{2} (1 + |\zeta_\perp|)^\frac{|m| ( 1 - \gamma )}{2} |V_\shortparallel|^{-2} \\
			& \qquad \qquad \times e^{- C_1 ( |\zeta_\shortparallel + \frac{1}{2} V_\shortparallel|^2 + |\zeta_\perp|^2 + |V_\perp + \zeta_\perp|^2 )} \mathbf{1}_{\{ |V_\perp| \leq \frac{\tilde{\eps}}{\sqrt{1 - \tilde{\eps}^2}} |V_\shortparallel| \}} \d V_\perp \d V_\shortparallel \\
			\leq & C \int_{\R^3} |V_\shortparallel|^{-2} e^{- C_1 |\zeta_\shortparallel + \frac{1}{2} V_\shortparallel|^2 } \int_{V_\perp \perp V_\shortparallel} \mathbf{1}_{\{ |V_\perp| \leq \frac{\tilde{\eps}}{\sqrt{1 - \tilde{\eps}^2}} |V_\shortparallel| \}} e^{- \frac{C_1}{2} |V_\perp + \zeta_\perp|^2 } \d V_\perp \d V_\shortparallel \,.
		\end{aligned}
	\end{equation*}}
	Thanks to $\int_{V_\perp \perp V_\shortparallel} e^{- \frac{C_1}{2} |V_\perp + \zeta_\perp|^2 } \d V_\perp \leq C < + \infty$ uniformly in $V_\shortparallel \in \R^3$, we know that for any $\hat{\eps} > 0$ there is a $\tilde{\eps}_{\hat{\eps}} > 0$ such that for any $0 < \tilde{\eps} < \tilde{\eps}_{\hat{\eps}}$, $ \int_{V_\perp \perp V_\shortparallel} \mathbf{1}_{\{ |V_\perp| \leq \frac{\tilde{\eps}}{\sqrt{1 - \tilde{\eps}^2}} |V_\shortparallel| \}} e^{- \frac{C_1}{2} |V_\perp + \zeta_\perp|^2 } \d V_\perp \leq \hat{\eps} $. Then the quantity $\widetilde{I\!I\!I}_{21}$ can be bounded by $ \widetilde{I\!I\!I}_{21} \leq C \hat{\eps} $, which means that
	\begin{equation}\label{III-21}
		\begin{aligned}
			| \sigma_x^\frac{m}{2} (x,v) w_{- \gamma, \vartheta} (v) I\!I\!I_{21} | \leq C_\eps \hat{\eps} \LL g \RR_{A; m, 0, \vartheta}
		\end{aligned}
	\end{equation}
	for any $0 < \tilde{\eps} < \tilde{\eps}_{\hat{\eps}}$.
	
	{\em Step 3.2.2. Control the quantity $| \sigma_x^\frac{m}{2} (x,v) w_{- \gamma, \vartheta} (v) I\!I\!I_{22}|$.} Observe that
	\begin{equation*}
		\begin{aligned}
			& w_{- \gamma, \vartheta} (v) \M^\frac{1}{2} (v') \M^\frac{1}{2} (v_*) \leq C (1 + |v_* - v|)^{|\gamma|} e^{2 \vartheta |v_* - v|^2} e^{- \tfrac{1}{8 T} (|v' - \u|^2 + |v_* - \u|^2)} \,, \\
			& \chi_\eps (|v_* - v|) b (\omega, v_* - v) \leq C_\eps |v' - v| (1 + |v_* - v|)^{\gamma - 1} \,, \\
			& \sigma_x^\frac{m}{2} (x,v) \sigma_x^{- \frac{m}{2}} (x,v_*') e^{\hbar \sigma (x,v) - \hbar \sigma (x,v_*')} \leq C (1 + |v-v_*'|)^\frac{|m| ( 1 - \gamma )}{2} e^{c \hbar ||v - \u|^2 - |v_*' - \u|^2|} \,.
		\end{aligned}
	\end{equation*}
	Then
		\begin{align*}
			| \sigma_x^\frac{m}{2} (x,v) w_{- \gamma, \vartheta} (v) I\!I\!I_{22} | \leq C_\eps \iint_{\mathbb{R}^3 \times \mathbb{S}^2} |v'-v| (1 + |v-v_*'|)^\frac{|m| ( 1 - \gamma )}{2} (1 + |v_*-v|)^{|\gamma| + \gamma - 1} e^{2 \vartheta |v_* - v|^2} \\
			\times e^{c \hbar ||v - \u|^2 - |v_*' - \u|^2|} | \sigma_x^\frac{m}{2} g (x,v_*')| e^{- \tfrac{1}{8 T} (|v' - \u|^2 + |v_* - \u|^2)} \mathbf{1}_{|v_*' - v| > \tilde{\eps} |v_* - v|} \d \omega \d v_* \,.
		\end{align*}
	By employing the change variables $(\omega, v_*) \mapsto (V_\shortparallel, V_\perp)$ and the similar computations of $\widetilde{I\!I\!I}_{21}$, one has
	\begin{equation*}
		\begin{aligned}
			| \sigma_x^\frac{m}{2} (x,v) w_{- \gamma, \vartheta} (v) I\!I\!I_{22} | \leq C_\eps \int_{\R^3} |V_\shortparallel|^{-1} e^{- C_2 |\zeta_\shortparallel + \frac{1}{2} V_\shortparallel|^2 } \int_{V_\perp \perp V_\shortparallel} (1 + |V|)^{|\gamma| + \gamma - 1} e^{2 \vartheta |V|^2} \\
			\times |\sigma_x^\frac{m}{2} g (x,v + V_\perp)| \mathbf{1}_{\{ |V_\perp| > \tilde{\eps} |V| \}} (1 + |V_\perp|)^\frac{|m| ( 1 - \gamma )}{2} e^{- C_2 (|\zeta_\perp|^2 + |V_\perp + \zeta_\perp|^2)} \d V_\perp \d V_\shortparallel \,.
		\end{aligned}
	\end{equation*}
	If $|V_\perp| > \tilde{\eps} |V|$, there holds $|\zeta_\perp|^2 + |V_\perp + \zeta_\perp|^2 \geq |V_\perp|^2 > \tilde{\eps}^2 |V|^2$. Together with $(1 + |V_\perp|)^\frac{|m| ( 1 - \gamma )}{2} e^{- \frac{C_2}{4} (|\zeta_\perp|^2 + |V_\perp + \zeta_\perp|^2)} \leq C$ and $(1 + |V|)^{|\gamma| + \gamma - 1} e^{- (\frac{C_2}{2} \tilde{\eps}^2 - 2 \vartheta) |V|^2} \leq C_{\tilde{\eps}}$ under $0 \leq \vartheta < \frac{C_2}{4} \tilde{\eps}^2$, one has
{\small
	\begin{equation*}
		\begin{aligned}
			| \sigma_x^\frac{m}{2} (x,v) w_{- \gamma, \vartheta} (v) I\!I\!I_{22} | \leq & C_{\eps, \tilde{\eps}} \int_{\R^3} |V_\shortparallel|^{-1} e^{- C_2 |\zeta_\shortparallel + \frac{1}{2} V_\shortparallel|^2 } \int_{V_\perp \perp V_\shortparallel} |\sigma_x^\frac{m}{2} g (x,v + V_\perp)| e^{- \frac{C_2}{4} |V_\perp + \zeta_\perp|^2 } \d V_\perp \d V_\shortparallel \\
			= & C_{\eps, \tilde{\eps}} \int_{\R^3} |\sigma_x^\frac{m}{2} g (x,v + V_\perp)| e^{- \frac{C_2}{4} |V_\perp + \zeta_\perp|^2 } \int_{V_\perp \perp V_\shortparallel} |V_\shortparallel|^{-1} e^{- C_2 |\zeta_\shortparallel + \frac{1}{2} V_\shortparallel|^2 } \d V_\shortparallel \d V_\perp \\
			\leq & C_{\eps, \tilde{\eps}} \int_{\R^3} |\sigma_x^\frac{m}{2} g (x,v + V_\perp)| e^{- \frac{C_2}{4} |V_\perp + \zeta_\perp|^2 } \d V_\perp \,.
		\end{aligned}
	\end{equation*}}
	Since $ \int_{\R^3} z_{- \alpha'}^2 (v + V_\perp ) e^{- \frac{C_2}{2} |V_\perp + \zeta_\perp|^2} \d V_\perp \leq C $ uniformly in $v, \zeta_\perp \in \R^3$ whence $0 \leq \alpha' < \frac{1}{2}$, we obtain
	\begin{equation}\label{III-22}
		\begin{aligned}
			| \sigma_x^\frac{m}{2} (x,v) w_{- \gamma, \vartheta} (v) I\!I\!I_{22} | \leq C_{\eps, \tilde{\eps}} \| \sigma_x^\frac{m}{2} z_{\alpha'} w_{- \gamma, \vartheta} g \|_{L^\infty_x L^2_v} \,.
		\end{aligned}
	\end{equation}

	Then \eqref{III-21} and \eqref{III-22} indicate that for any $\hat{\eps} > 0 $ there is an $\tilde{\eps}_{\hat{\eps}} > 0$ such that for all $\tilde{\eps} \in (0, \tilde{\eps}_{\hat{\eps}})$,
	\begin{equation}\label{III-2}
		\begin{aligned}
			| \sigma_x^\frac{m}{2} (x,v) w_{- \gamma, \vartheta} (v) I\!I\!I_2 | \leq C_\eps \hat{\eps} \LL g \RR_{A; m, 0, \vartheta} + C_{\eps, \tilde{\eps}} \| \sigma_x^\frac{m}{2} z_{\alpha'} w_{- \gamma, \vartheta} g \|_{L^\infty_x L^2_v} \,,
		\end{aligned}
	\end{equation}
	provided that $\hbar, \vartheta \geq 0$ are both small enough. It therefore follows from \eqref{III-1} and \eqref{III-2} that for any fixed $\eps, \hat{\eps} > 0$, $- 3 < \gamma \leq 1$, $ m \in \R $ and $0 \leq \alpha' < \frac{1}{2}$
	\begin{equation}\label{Hv-est}
		\begin{aligned}
			\LL \nu^{-1} K_{\hbar 2}^{\chi_\eps} g \RR_{A; m, 0, \vartheta} \leq C_\eps ( \tilde{\eps} + \hat{\eps} ) \LL g \RR_{A; m, 0, \vartheta} + C_{\eps, \tilde{\eps}} \| \sigma_x^\frac{m}{2} z_{\alpha'} w_{- \gamma, \vartheta} g \|_{L^\infty_x L^2_v}
		\end{aligned}
	\end{equation}
	for all $\tilde{\eps} \in (0, \tilde{\eps}_{\hat{\eps}})$, provided that $\hbar, \vartheta \geq 0$ are both small enough. Consequently, \eqref{Loss-Est}, \eqref{Lv-Est} and \eqref{Hv-est} show that
	\begin{equation*}
		\begin{aligned}
			\LL \nu^{-1} K_\hbar g \RR_{A; m, 0, \vartheta} \leq [ C \eps^{3 + \gamma} + C_\eps (\hat{\eps} + \tilde{\eps}) ] \LL g \RR_{A; m, 0, \vartheta} + C_{\eps, \tilde{\eps}} \| \sigma_x^\frac{m}{2} z_{\alpha'} w_{- \gamma, \vartheta} g \|_{L^\infty_x L^2_v} \,.
		\end{aligned}
	\end{equation*}
	
	For any fixed $\eta_1 > 0$, we first take $\eps > 0$ such that $C \eps^{3 + \gamma} \leq \frac{\eta_1}{3}$ and fixed. Then we take $\hat{\eps} > 0$ such that $C_\eps \hat{\eps} \leq \frac{\eta_1}{3}$ and fixed. At the end, we choose $\tilde{\eps} \in (0, \tilde{\eps}_{\hat{\eps}})$ such that $C_\eps \tilde{\eps} \leq \frac{\eta_1}{3}$ and fixed. Therefore, the above inequality finish the proof of Lemma \ref{Lmm-Kh-2-infty}.
\end{proof}

\subsection{Weighted $L^2_{x,v}$ property of $K$: Proof of Lemma \ref{Lmm-Kh-L2}}\label{Subsec:K-L2}

In this section,

\begin{proof}[Proof of Lemma \ref{Lmm-Kh-L2}]
	By \eqref{K-K1-K2}-\eqref{K1}-\eqref{K2}, one knows that
	\begin{equation*}
		\begin{aligned}
			K_\hbar g = - K_{\hbar 1} g + K_{\hbar 2} g + K_{\hbar 3} g \,,
		\end{aligned}
	\end{equation*}
	where
	\begin{align}
		K_{\hbar 1} g (x,v) & = e^{\hbar \sigma (x,v)} \M^\frac{1}{2} (v) \int_{\R^3} e^{- \hbar \sigma (x,v_*)} g (x, v_*) |v - v_*|^\gamma \M^\frac{1}{2} (v_*) \d v_* \,, \label{Kh1} \\
		K_{\hbar 2} g (x,v) & = e^{\hbar \sigma (x,v)} \iint_{\mathbb{R}^3 \times \mathbb{S}^2} e^{- \hbar \sigma (x, v')} g (x, v') \M^\frac{1}{2} (v_*') \M^\frac{1}{2} (v_*) b (\omega, v_* - v) \d \omega \d v_* \,, \label{Kh2} \\
		K_{\hbar 3} g (x,v) & = e^{\hbar \sigma (x,v)} \iint_{\mathbb{R}^3 \times \mathbb{S}^2} e^{- \hbar \sigma (x, v_*')} g (x, v_*') \M^\frac{1}{2} (v') \M^\frac{1}{2} (v_*) b (\omega, v_* - v) \d \omega \d v_* \,. \label{Kh3}
	\end{align}

	{\bf Step 1. Estimates for $K_{\hbar 1}$ part.}
	
	Claim that for $0 \leq \alpha < \min \{ \frac{1}{2}, \frac{\gamma + 3}{2} \}$,
	\begin{equation}\label{Claim-Kh1}
		\begin{aligned}
			\int_{\R^3} |\nu^{- \frac{1}{2}} z_{- \alpha} \sigma_x^\frac{m}{2} w_{\beta, \vartheta} K_{\hbar 1} g (x,v)|^2 \d v \lesssim \int_{\R^3} |\nu^\frac{1}{2} \sigma_x^\frac{m}{2} w_{\beta, \vartheta} g (x,v)|^2 \d v \,.
		\end{aligned}
	\end{equation}
	
	Indeed, we rewrite
	\begin{equation*}
		\begin{aligned}
			\nu^{- \frac{1}{2}} z_{- \alpha} \sigma_x^\frac{m}{2} w_{\beta, \vartheta} K_{\hbar 1} g (x,v) = \int_{\R^3} \mathfrak{k}_{\hbar 1} (x, v, v_*) \hat{g} (x, v_*) \d v_* \,,
		\end{aligned}
	\end{equation*}
	where $\hat{g} : = \nu^\frac{1}{2} \sigma_x^\frac{m}{2} w_{\beta, \vartheta} g$, and
	\begin{equation*}
		\begin{aligned}
			\mathfrak{k}_{\hbar 1} (x, v, v_*) = z_{- \alpha} (v) \nu^{- \frac{1}{2}} (v) \nu^{- \frac{1}{2}} (v_*) \tfrac{w_{\beta, \vartheta} (v) \sigma_x^\frac{m}{2} (x,v) e^{\hbar \sigma (x,v)}}{w_{\beta, \vartheta} (v_*) \sigma_x^\frac{m}{2} (x,v_*) e^{\hbar \sigma (x,v_*)}} |v - v_*|^\gamma \M^\frac{1}{2} (v) \M^\frac{1}{2} (v_*) \,.
		\end{aligned}
	\end{equation*}
	By \eqref{p-pro}, one has
	\begin{equation*}
		\begin{aligned}
			\tfrac{w_{\beta, \vartheta} (v) \sigma_x^\frac{m}{2} (x,v) e^{\hbar \sigma (x,v)}}{w_{\beta, \vartheta}(v_*) \sigma_x^\frac{m}{2} (x,v_*) e^{\hbar \sigma (x,v_*)}} \lesssim (1 + |v-v_*|)^{|\beta| + \frac{|m| ( 1 - \gamma )}{2}} e^{( c \hbar + \vartheta ) ||v-\u|^2 - |v_* - \u|^2|} \,.
		\end{aligned}
	\end{equation*}
	Then for sufficiently small $\hbar, \vartheta \geq 0$, $\nu^{- \frac{1}{2}} (v) \nu^{- \frac{1}{2}} (v_*) \tfrac{w_{\beta, \vartheta} (v) \sigma_x^\frac{m}{2} (x,v) e^{\hbar \sigma (x,v)}}{w_{\beta, \vartheta}(v_*) \sigma_x^\frac{m}{2} (x,v_*) e^{\hbar \sigma (x,v_*)}} \M^\frac{1}{4} (v) \M^\frac{1}{4} (v_*) \lesssim 1$ uniformly in $x$, $v$ and $v_*$, which yields that $ |\mathfrak{k}_{\hbar 1} (x,v,v_*)| \lesssim z_{- \alpha} (v) |v-v_*|^\gamma \M^\frac{1}{4} (v) \M^\frac{1}{4} (v_*) $ uniformly in $x$. Then
{\small
	\begin{equation*}
		\begin{aligned}
			& I : = \int_{\R^3} | \nu^{- \frac{1}{2}} z_{- \alpha} \sigma_x^\frac{m}{2} w_{\beta, \vartheta} K_{\hbar 1} g (x,v) |^2 \d v = \int_{\R^3} | \int_{\R^3} \mathfrak{k}_{\hbar 1} (x, v, v_*) \hat{g} (x, v_*) \d v_* |^2 \d v \\
			\lesssim & \int_{\R^3} \Big( \int_{\R^3} z_{- \alpha} (v) |v-v_*|^\gamma \M^\frac{1}{4} (v) \M^\frac{1}{4} (v_*) |\hat{g} (x,v_*)| \d v_* \Big)^2 \d v \\
			\lesssim & \int_{\R^3} \Big( \int_{\R^3} z_{- \alpha}^2 (v) |v_* - v|^\gamma \M^\frac{1}{4} (v) \M^\frac{1}{4} (v_*) \d v_* \Big) \Big( \int_{\R^3} |v_* - v|^\gamma \M^\frac{1}{4} (v) \M^\frac{1}{4} (v_*) |\hat{g} (x,v_*)|^2 \d v_* \Big) \d v \,.
		\end{aligned}
	\end{equation*}}
	Note that by Lemma 2.3 of \cite{LS-2010-KRM}, $ \int_{\R^3} |v-v_*|^\gamma \M^\frac{1}{4} (v_*) \d v_* \lesssim (1 + |v|)^\gamma $, which means that
	\begin{equation*}
		\begin{aligned}
			\int_{\R^3} z_{- \alpha}^2 (v) |v_* - v|^\gamma \M^\frac{1}{4} (v) \M^\frac{1}{4} (v_*) \d v_* \lesssim z_{- \alpha}^2 (v) \M^\frac{1}{4} (v) (1 + |v|)^\gamma \lesssim z_{- \alpha}^2 (v) \M^\frac{1}{8} (v) \,.
		\end{aligned}
	\end{equation*}
	It thereby holds
		\begin{align*}
			I \lesssim \iint_{\R^3 \times \R^3} z_{- \alpha}^2 (v) \M^\frac{3}{8} (v) \M^\frac{1}{4} (v_*) |v-v_*|^\gamma |\hat{g} (x,v_*)|^2 \d v_* \d v = \int_{\R^3} \M^\frac{1}{4} (v_*) |\hat{g} (x,v_*)|^2 \mathsf{h} (v_*) \d v_* \,,
		\end{align*}
	where
	\begin{equation*}
		\begin{aligned}
			\mathsf{h} (v_*) = \int_{\R^3} z_{- \alpha}^2 (v) |v-v_*|^\gamma \M^\frac{3}{8} (v) \d v \,.
		\end{aligned}
	\end{equation*}
	
	One asserts that
	\begin{equation}\label{h-bnd}
		\begin{aligned}
			\mathsf{h} (v_*) \lesssim (1 + |v_*|)^\gamma
		\end{aligned}
	\end{equation}
	for $- 3 < \gamma \leq 1$ and $0 \leq \alpha < \min \{ \frac{1}{2}, \frac{\gamma + 3}{3} \}$. Once the assertion \eqref{h-bnd} holds, one has
	\begin{equation*}
		\begin{aligned}
			I \lesssim \int_{\R^3} (1 + |v_*|)^\gamma \M^\frac{1}{4} (v_*) |\hat{g} (x, v_*)|^2 \d v_* \lesssim \int_{\R^3} |\hat{g} (x, v_*)|^2 \d v_* \,,
		\end{aligned}
	\end{equation*}
	which means the claim \eqref{Claim-Kh1} holds.
	
	It remains to show the assertion \eqref{h-bnd}. Note that
	\begin{equation*}
		\begin{aligned}
			\mathsf{h} (v_*) = \underbrace{ \int_{|v-v_*| \leq 1} (\cdots) \d v}_{:= \mathsf{h}_1 (v_*)} +  \underbrace{ \int_{|v-v_*| > 1} (\cdots) \d v }_{:= \mathsf{h}_2 (v_*)} \,.
		\end{aligned}
	\end{equation*}
	Recalling the definition of $z_{- \alpha}$ in \eqref{z-alpha}, one has
	\begin{equation*}
		\begin{aligned}
			\mathsf{h}_1 (v_*) \lesssim & \int_{|v-v_*| \leq 1} z_{- \alpha}^2 (v) |v - v_*|^\gamma \d v \\
			= & \int_{|v-v_*| \leq 1, |v_3| \leq 1} |v_3|^{- 2 \alpha} |v-v_*|^\gamma \d v + \int_{|v-v_*| \leq 1, |v_3| > 1} |v-v_*|^\gamma \d v \\
			\lesssim & \int_{|v-v_*| \leq 1} |v_3|^{- 2 \alpha} |v-v_*|^\gamma \d v + \int_{|v-v_*| \leq 1} |v-v_*|^\gamma \d v \,.
		\end{aligned}
	\end{equation*}
	By letting $ v_1 - v_{*1} = r \sin \varphi \cos \theta \,, \ v_2 - v_{* 2} = r \sin \varphi \sin \theta \,, \ v_3 - v_{*3} = r \cos \varphi $ with $0 \leq r \leq 1$, $0 \leq \theta \leq 2 \pi$ and $0 \leq \varphi \leq \pi$, a direct computation shows
	\begin{equation*}
		\begin{aligned}
			\mathsf{h}_1 (v_*) \lesssim 2 \pi \int_0^1 r^{2 + \gamma - 2 \alpha} \int_0^\frac{\pi}{2} \big( |\cos \varphi + \tfrac{|v_{*3}|}{r}|^{- 2 \alpha} + |\cos \varphi - \tfrac{|v_{*3}|}{r}|^{- 2 \alpha} \big) \sin \varphi \d \varphi \d r + \tfrac{4 \pi}{3 + \gamma} \,.
		\end{aligned}
	\end{equation*}
	Observe that for $0 \leq \alpha < \frac{1}{2}$,
	\begin{equation*}
		\begin{aligned}
			& \int_0^\frac{\pi}{2} \big( |\cos \varphi + \tfrac{|v_{*3}|}{r}|^{- 2 \alpha} + |\cos \varphi - \tfrac{|v_{*3}|}{r}|^{- 2 \alpha} \big) \sin \varphi \d \varphi \\
			= & \left\{
			\begin{aligned}
				\tfrac{1}{1 - 2 \alpha} \big[ ( \tfrac{|v_{*3}|}{r} + 1 )^{1 - 2 \alpha} - (\tfrac{|v_{*3}|}{r} - 1)^{1 - 2 \alpha} \big] \,, & \quad \tfrac{|v_{*3}|}{r} \geq 1 \\
				\tfrac{1}{1 - 2 \alpha} \big[ (1 + \tfrac{|v_{*3}|}{r})^{1 - 2 \alpha} - (1 - \tfrac{|v_{*3}|}{r})^{1 - 2 \alpha} \big] \,, & \quad 0 \leq \tfrac{|v_{*3}|}{r} < 1
			\end{aligned}
			\right. \quad \leq C_\alpha < \infty
		\end{aligned}
	\end{equation*}
	uniformly in $\tfrac{|v_{*3}|}{r} \geq 0$, which implies that
	\begin{equation}\label{h1-bnd}
		\begin{aligned}
			\mathsf{h}_1 (v_*) \lesssim 2 \pi C_\alpha \int_0^1 r^{2 + \gamma - 2 \alpha} \d r + \tfrac{4 \pi}{3 + \gamma} = \tfrac{2 \pi C_\alpha}{3 + \gamma - 2 \alpha} + \tfrac{4 \pi}{3 + \gamma} < \infty \,,
		\end{aligned}
	\end{equation}
	provided that $0 \leq \alpha < \min \{ \frac{1}{2}, \frac{3 + \gamma}{2} \}$.
	
	For the case $|v - v_* | > 1$, one can assert that $|v - v_*|^\gamma \lesssim (1 + |v|)^{|\gamma|} (1 + |v_*|)^\gamma$ for $- 3 < \gamma \leq 1$. Indeed, if $0 \leq \gamma \leq 1$, the bound $|v - v_*| \leq (1 + |v|) (1 + |v_*|)$ implies the assertion for the case $0 \leq \gamma \leq 1$. If $- 3 < \gamma < 0$, one has $ 1 + |v_*| \leq (1 + |v|) (1 + |v - v_*|) \leq 2 (1 + |v|) |v - v_*| $, hence, $|v-v_*|^{-1} \leq 2 (1 + |v|) (1 + |v_*|)^{-1}$, which infers the assertion for the case $-3 < \gamma < 0$. Then it follows
	\begin{equation*}
		\begin{aligned}
			\mathsf{h}_2 (v_*) \lesssim \int_{|v-v_*| > 1} z_{- \alpha}^2 (v) (1 + |v|)^{|\gamma|} (1 + |v_*|)^\gamma \M^\frac{3}{8} (v) \d v \lesssim (1 + |v_*|)^\gamma \int_{\R^3} z_{- \alpha}^2 (v) \M^\frac{1}{4} (v) \d v \,.
		\end{aligned}
	\end{equation*}
	Thanks to $0 \leq \alpha < \frac{1}{2}$, one has $ \int_{\R^3} z_{- \alpha}^2 (v) \M^\frac{1}{4} (v) \d v \lesssim 1 $, which implies
	\begin{equation}\label{h2-bnd}
		\begin{aligned}
			\mathsf{h}_2 (v_*) \lesssim (1 + |v_*|)^\gamma \,.
		\end{aligned}
	\end{equation}
	Then the bounds \eqref{h1-bnd} and \eqref{h2-bnd} conclude the assertion \eqref{h-bnd}.

	{\bf Step 2. Estimates for $K_{\hbar 2}$ part.}
	
	One asserts that for $0 \leq \alpha < \min \{ \frac{1}{2}, \frac{b_0 + \gamma + 1}{2} \}$ with $b_0 \in \mathcal{S}_\gamma$ given in Lemma \ref{Lmm-Kh-L2},
	\begin{equation}\label{Claim-Kh2}
		\begin{aligned}
			\int_{\R^3} |\nu^{- \frac{1}{2}} z_{- \alpha} \sigma_x^\frac{m}{2} w_{\beta, \vartheta} K_{\hbar 2} g (x, v)|^2 \d v \lesssim \int_{\R^3} |\nu^\frac{1}{2} \sigma_x^\frac{m}{2} w_{\beta, \vartheta} g (x,v)|^2 \d v \,.
		\end{aligned}
	\end{equation}
	
	Indeed, let $\chi (r)$ be a smooth monotone function satisfying $\chi (r) = 0$ for $0 \leq r \leq 1$ and $\chi (r) = 1$ for $r \geq 2$. Then the operator in \eqref{Kh2} can be decomposed as follows:
	\begin{equation}\label{Kh2-Decomp}
		\begin{aligned}
			K_{\hbar 2} g (x,v) = K_{\hbar 2}^\chi g (x,v) + K_{\hbar 2}^{1 - \chi} g (x,v) \,,
		\end{aligned}
	\end{equation}
	where for $a = \chi$ or $1 - \chi$
	\begin{equation*}
		\begin{aligned}
			K_{\hbar 2}^a g (x,v) = e^{\hbar \sigma (x,v)} \iint_{\mathbb{R}^3 \times \mathbb{S}^2} e^{- \hbar \sigma (x, v')} a (|v-v_*|) g (x, v') \M^\frac{1}{2} (v_*') \M^\frac{1}{2} (v_*) b (\omega, v_* - v) \d \omega \d v_* \,.
		\end{aligned}
	\end{equation*}
	
	{\em Step 2.1. Control of $K_{\hbar 2}^\chi$.}
	
	Denote by
	\begin{equation*}
		\begin{aligned}
			\widetilde{K}_{\hbar 2}^\chi g (x,v) = z_{- \alpha} (v) \nu^{- \frac{1}{2}} (v) \sigma_x^\frac{m}{2} (x,v) w_{\beta, \vartheta} (v) K_{\hbar 2}^\chi g (x,v) \,.
		\end{aligned}
	\end{equation*}
	By \eqref{p-pro},
	\begin{equation}\label{p-pro-0}
		\begin{aligned}
			\tfrac{w_{\beta, \vartheta} (v) \sigma_x^\frac{m}{2} (x, v) e^{\hbar \sigma (x,v)}}{w_{\beta, \vartheta} (v') \sigma_x^\frac{m}{2} (x, v') e^{\hbar \sigma (x, v')}} \lesssim (1 + |v - v'|)^{|\beta| + \frac{|m| ( 1 - \gamma )}{2} } e^{ (c \hbar + \vartheta ) ||v - \u|^2 - |v' - \u|^2|} \,.
		\end{aligned}
	\end{equation}
	Then
	\begin{equation*}
		\begin{aligned}
			| \widetilde{K}_{\hbar 2}^\chi g (x,v) | \lesssim z_{-\alpha} (v) \iint_{\mathbb{R}^3 \times \mathbb{S}^2} (1 + |v-v'|)^{|\beta| + \frac{|m| ( 1 - \gamma )}{2} } e^{ ( c \hbar + \vartheta ) ||v - \u|^2 - |v' - \u|^2|} \nu^{- \frac{1}{2}} (v) \nu^{- \frac{1}{2}} (v') \\
			\times |\nu^\frac{1}{2} \sigma_x^\frac{m}{2} g (x,v')| \chi (|v-v_*|) b (\omega, v_*-v) \d \omega \d v_* \,,
		\end{aligned}
	\end{equation*}
	which means
	\begin{equation*}{\small
		\begin{aligned}
			& |\int_{\R^3} \widetilde{K}_{\hbar 2}^\chi g (x,v) \cdot h (x,v) \d v | \\
			\lesssim & \iiint_{\R^3 \times \R^3 \times \mathbb{S}^2} z_{- \alpha} (v) \nu^{- \frac{1}{2}} (v') (1 + |v-v'|)^{|\beta| + \frac{|m| ( 1 - \gamma )}{2} } e^{ ( c \hbar + \vartheta ) ||v - \u|^2 - |v' - \u|^2|} |\nu^\frac{1}{2} \sigma_x^\frac{m}{2} w_{\beta, \vartheta} g (x,v')| \\
			& \qquad \quad \ \times |\nu^{- \frac{1}{2}} h (x,v)| \M^\frac{1}{2} (v_*') \M^\frac{1}{2} (v_*) \chi (|v_*-v|) b (\omega, v_* - v) \d \omega \d v_* \d v \\
			\lesssim & \sqrt{I_1 I_2} \,,
		\end{aligned}}
	\end{equation*}
	where
	\begin{equation*}
		\begin{aligned}
			I_1 = \iiint_{\R^3 \times \R^3 \times \mathbb{S}^2} |\nu^{- \frac{1}{2}} h (x,v)|^2 \M^\frac{1}{4} (v_*) b (\omega, v_* - v) \d \omega \d v_* \d v
		\end{aligned}
	\end{equation*}
	and
{\small
	\begin{equation*}
		\begin{aligned}
			I_2 = \iiint_{\R^3 \times \R^3 \times \mathbb{S}^2} z_{- \alpha}^2 (v) \nu^{-1} (v') (1 + |v-v'|)^{2 |\beta| + |m| ( 1 - \gamma )} e^{2 ( c \hbar + \vartheta ) ||v-\u|^2 - |v'-\u|^2|} |\nu^\frac{1}{2} \sigma_x^\frac{m}{2} w_{\beta, \vartheta} g (x,v')|^2 \\
			\times \M (v_*') \M^\frac{3}{4} (v_*) \chi^2 (|v-v_*|) b (\omega, v_* - v) \d \omega \d v_* \d v \,.
		\end{aligned}
	\end{equation*}}
	
	Due to $\iint_{\mathbb{R}^3 \times \mathbb{S}^2} \M^\frac{1}{4} (v_*) b (\omega, v_* - v) \d \omega \d v_* \lesssim \nu (v)$ derived from Lemma 2.1 of \cite{LS-2010-KRM}, the factor $I_1$ can be bounded by
	\begin{equation}\label{I1-bnd}
		\begin{aligned}
			I_1 \lesssim \int_{\R^3} |h (x,v)|^2 \d v \,.
		\end{aligned}
	\end{equation}
	Now we estimate the second factor $I_2$. Note that
	\begin{equation*}
		\begin{aligned}
			\nu^{-1} (v') \thicksim (1 + |v'|)^{- \gamma} \lesssim (1 + |v_*'|)^{|\gamma|} (1 + |v' - v_*'|)^{- \gamma} \lesssim \M^{- \frac{1}{4}} (v_*') (1 + |v' - v_*'|)^{- \gamma}
		\end{aligned}
	\end{equation*}
	for $- 3 < \gamma \leq 1$. Moreover, the measure $\chi^2 (|v_* - v|) b (\omega, v_*-v) \d \omega \d v_* \d v$ is invariant under the transform $(v,v_*) \mapsto (v', v_*')$. Then
	\begin{equation*}
		\begin{aligned}
			I_2 \lesssim \int_{\R^3} |\nu^\frac{1}{2} \sigma_x^\frac{m}{2} w_{\beta, \vartheta} g (x,v)|^2 \mathfrak{i}_2 (v) \d v \,,
		\end{aligned}
	\end{equation*}
	where
	\begin{equation*}
		\begin{aligned}
			\mathfrak{i}_2 (v) = \iint_{\mathbb{R}^3 \times \mathbb{S}^2} z_{- \alpha}^2 (v') (1 + |v-v_*|)^{- \gamma} (1 + |v-v'|)^{2 |\beta| + |m| ( 1 - \gamma )} e^{2 ( c \hbar + \vartheta ) ||v-\u|^2 - |v' - \u|^2|} \\
			\times \M^\frac{3}{4} (v_*') \M^\frac{3}{4} (v_*) \chi^2 (|v-v_*|) b (\omega, v_*-v) \d \omega \d v_* \,.
		\end{aligned}
	\end{equation*}
	
	Let $V = v_* -v$, $V_\shortparallel = (V \cdot \omega) \omega$, $V_\perp = V - V_\shortparallel $. Then $\d \omega \d v_* = 2 |V_\shortparallel|^{-2} \d V_\perp \d V_\shortparallel$, $v' = v + V_\shortparallel$, $v_*' = v + V_\perp $, $v_* = v + V$. By \eqref{Trans-1} and \eqref{Trans-2}, one has
	\begin{equation*}
		\begin{aligned}
			- \tfrac{|v_*' - \u|^2 + |v_* - \u|^2}{2 T} = - \tfrac{|\zeta_\shortparallel|^2 + \frac{1}{4} |V_\shortparallel|^2 + |V_\perp + \zeta_\perp|^2}{T} \,, \ |v-\u|^2 - |v' - \u|^2 =- 2 V_\shortparallel \cdot \zeta_\shortparallel \,,
		\end{aligned}
	\end{equation*}
	where $\zeta = v + \frac{1}{2} V_\shortparallel$, $\zeta_\shortparallel = [(\zeta - \u) \cdot \omega] \omega$, $\zeta_\perp = (\zeta - \u) - \zeta_\shortparallel$. Recall that
	\begin{equation*}
		\begin{aligned}
			b (\omega, v_* - v) = |v-v_*|^\gamma \tilde{b} (\cos \theta) \leq C |v-v_*|^\gamma |\cos \theta| = \tfrac{C |V_\shortparallel|}{|V|^{1 - \gamma}} \,.
		\end{aligned}
	\end{equation*}
	Then
	\begin{equation*}
		\begin{aligned}
			\mathfrak{i}_2 (v) \lesssim \int_{\R^3} \int_{V_\perp \perp V_\shortparallel} z_{- \alpha}^2 (v + V_\shortparallel) (1 + |V|)^{- \gamma} (1 + |V_\shortparallel|)^{2 |\beta| + |m| ( 1 - \gamma )} e^{4 ( c \hbar + \vartheta ) |V_\shortparallel \cdot \zeta_\shortparallel|} \\
			\times e^{- \frac{3}{4 T} ( |\zeta_\shortparallel|^2 + \frac{1}{4} |V_\shortparallel|^2 + |V_\perp + \zeta_\perp|^2 )} \tfrac{|V_\shortparallel| \chi^2 (|V|)}{|V|^{1 - \gamma}} 2 |V_\shortparallel|^{-2} \d V_\perp \d V_\shortparallel \,.
		\end{aligned}
	\end{equation*}
	Note that $(1 + |V|)^{- \gamma} \lesssim (1 + |V_\shortparallel|)^{|\gamma|} (1 + |V_\perp |)^{- \gamma}$. Then, for sufficiently small $\hbar, \vartheta \geq 0$,
	\begin{equation*}
		\begin{aligned}
			\mathfrak{i}_2 (v) \lesssim \int_{\R^3} \int_{V_\perp \perp V_\shortparallel} z_{- \alpha}^2 (v + V_\shortparallel) |V_\shortparallel|^{-1} e^{- c ( |\zeta_\shortparallel|^2 + |V_\shortparallel|^2 + |V_\perp + \zeta_\perp|^2 )} (1 + |V_\perp|)^{- \gamma} \tfrac{\chi^2 (|V|)}{|V|^{1 - \gamma}} \d V_\perp \d V_\shortparallel
		\end{aligned}
	\end{equation*}
	for some positive constant $c > 0$. Observe that $ \tfrac{\chi^2 (|V|)}{|V|^{1 - \gamma}} \lesssim (1 + |V_\shortparallel|^2 + |V_\perp|^2)^\frac{\gamma - 1}{2} $. If $|V_\perp + \zeta_\perp| > \frac{1}{2} |\zeta_\perp|$, one has $|V_\perp| \leq |V_\perp + \zeta_\perp| + |\zeta_\perp| < 3 |V_\perp + \zeta_\perp|$. Then $ (1 + |V_\perp|)^{- \gamma} \tfrac{\chi^2 (|V|)}{|V|^{1 - \gamma}} \lesssim e^{\frac{c}{2} |V_\perp + \zeta_\perp|^2} $. If $|V_\perp + \zeta_\perp| \leq \frac{1}{2} |\zeta_\perp|$, it holds $|V_\perp| \geq |\zeta_\perp| - |V_\perp + \zeta_\perp| \geq \tfrac{1}{2} |\zeta_\perp|$ and $ (1 + |V_\perp|)^{- \gamma} \leq (1 + |V_\perp + \zeta_\perp|)^{|\gamma|} (1 + |\zeta_\perp|)^{1 - \gamma} $, which imply that
	\begin{equation*}
		\begin{aligned}
			(1 + |V_\perp|)^{- \gamma} \tfrac{\chi^2 (|V|)}{|V|^{1 - \gamma}} \lesssim \tfrac{(1 + |V_\perp + \zeta_\perp|)^{|\gamma|} (1 + |\zeta_\perp|)^{1 - \gamma}}{(1 + |V_\shortparallel|^2 + \frac{1}{4} |\zeta_\perp|^2 )^\frac{1 - \gamma}{2}} \lesssim e^{\frac{c}{2} |V_\perp + \zeta_\perp|^2} \,.
		\end{aligned}
	\end{equation*}
	In summary, the inequality $(1 + |V_\perp|)^{- \gamma} \tfrac{\chi^2 (|V|)}{|V|^{1 - \gamma}} \lesssim e^{\frac{c}{2} |V_\perp + \zeta_\perp|^2}$ holds. Then
	\begin{equation*}
		\begin{aligned}
			\mathfrak{i}_2 (v) \lesssim & \int_{\R^3} |V_\shortparallel|^{-1} z_{- \alpha}^2 (v + V_\shortparallel) e^{- c (|\zeta_\shortparallel|^2 + |V_\shortparallel|^2)} \big( \int_{V_\perp \perp V_\shortparallel} e^{- \frac{c}{2} |V_\perp + \zeta_\perp|^2} \d V_\perp \big) \d V_\shortparallel \\
			\lesssim & \int_{\R^3} |V_\shortparallel|^{-1} z_{- \alpha}^2 (v + V_\shortparallel) e^{- c |V_\shortparallel|^2} \d V_\shortparallel \,.
		\end{aligned}
	\end{equation*}
	By the similar arguments in \eqref{h1-bnd}, $ \int_{\R^3} |V_\shortparallel|^{-1} z_{- \alpha}^2 (v + V_\shortparallel) e^{- c |V_\shortparallel|^2} \d V_\shortparallel \lesssim 1 $ uniformly in $v \in \R^3$, provided that $0 \leq \alpha < \frac{1}{2}$. Then one has proved that $\mathfrak{i}_2 (v) \lesssim 1$ uniformly in $v \in \R^3$. It therefore infers that $I_2 \lesssim \int_{\R^3} |\nu^\frac{1}{2} \sigma_x^\frac{m}{2} w_{\beta, \vartheta} g (x,v)|^2 \d v$, which, together with \eqref{I1-bnd}, implies that
	\begin{equation*}
		\begin{aligned}
			| \int_{\R^3} \widetilde{K}_{\hbar 2}^\chi g (x,v) \cdot h (x,v) \d v | \lesssim \Big( \int_{\R^3} |\nu^\frac{1}{2} \sigma_x^\frac{m}{2} w_{\beta, \vartheta} g (x,v)|^2 \d v \Big)^\frac{1}{2} \Big( \int_{\R^3} |h (x,v)|^2 \d v \Big)^\frac{1}{2} \,.
		\end{aligned}
	\end{equation*}
	By letting $h (x,v) = \widetilde{K}_{\hbar 2}^\chi g (x,v)$, one concludes that for $0 \leq \alpha < \frac{1}{2}$,
	\begin{equation}\label{Kh2-x}
		\begin{aligned}
			\int_{\R^3} |\nu^{- \frac{1}{2}} z_{- \alpha} \sigma_x^\frac{m}{2} w_{\beta, \vartheta} K_{\hbar 2}^\chi g (x, v)|^2 \d v \lesssim \int_{\R^3} |\nu^\frac{1}{2} \sigma_x^\frac{m}{2} w_{\beta, \vartheta} g (x,v)|^2 \d v \,.
		\end{aligned}
	\end{equation}

	{\em Step 2.2. Control of $K_{\hbar 2}^{1 - \chi}$.}
	
	Since $|v-v_*| \leq 2$ in $K_{\hbar 2}^{1 - \chi}$, the claim \eqref{Claim-M} shows $ \M (v_*') \lesssim \M^{1 - \delta_0} (v') \,, \ \M (v_*) \lesssim \M^{1 - \delta_0} (v) $ for any fixed $\delta_0 \in (0, 1)$. Then
	\begin{equation*}
		\begin{aligned}
			& |\nu^{- \frac{1}{2}} z_{- \alpha} \sigma_x^\frac{m}{2} w_{\beta, \vartheta} K_{\hbar 2}^{1 - \chi} g (x,v)| \\
			\lesssim & z_{- \alpha} (v) \iint_{\mathbb{R}^3 \times \mathbb{S}^2} \tfrac{ w_{\beta, \vartheta} (v) \sigma_x^\frac{m}{2} (x,v) e^{\hbar \sigma (x,v)}}{ w_{\beta, \vartheta} (v') \sigma_x^\frac{m}{2} (x,v') e^{\hbar \sigma (x,v')}} \M^\frac{1}{4} (v_*') \M^\frac{1}{4} (v_*) \M^\frac{1-\delta_0}{4} (v') \M^\frac{1-\delta_0}{4} (v) \nu^{- \frac{1}{2}} (v) \nu^{- \frac{1}{2}} (v') \\
			& \qquad \qquad \qquad \qquad \qquad \times (1-\chi) (|v-v_*|) |\nu^\frac{1}{2} \sigma_x^\frac{m}{2} w_{\beta, \vartheta} g (x,v')| b (\omega, v_*-v) \d \omega \d v_* \,.
		\end{aligned}
	\end{equation*}
	By \eqref{p-pro} and taking $\delta_0 = \frac{1}{2}$, one has
{\footnotesize
	\begin{equation*}
		\begin{aligned}
			\tfrac{w_{\beta, \vartheta} (v) \sigma_x^\frac{m}{2} (x,v) e^{\hbar \sigma (x,v)}}{ w_{\beta, \vartheta} (v') \sigma_x^\frac{m}{2} (x,v') e^{\hbar \sigma (x,v')}} \M^\frac{1}{4} (v_*') \M^\frac{1}{4} (v_*) \M^\frac{1-\delta_0}{4} (v') \M^\frac{1-\delta_0}{4} (v) \nu^{- \frac{1}{2}} (v) \nu^{- \frac{1}{2}} (v') \lesssim [ \M (v_*') \M (v_*) \M (v') \M (v) ]^\frac{1}{16}
		\end{aligned}
	\end{equation*}}
	for sufficiently small $\hbar, \vartheta \geq 0$. It thereby holds
	\begin{equation}\label{Kh2-1-x-1}
		\begin{aligned}
			& |\nu^{- \frac{1}{2}} z_{- \alpha} \sigma_x^\frac{m}{2} w_{\beta, \vartheta} K_{\hbar 2}^{1 - \chi} g (x,v)| \\
			\lesssim & \iint_{\mathbb{R}^3 \times \mathbb{S}^2} |\nu^\frac{1}{2} \sigma_x^\frac{m}{2} w_{\beta, \vartheta} g (x,v')| z_{- \alpha} (v) [ \M (v_*') \M (v_*) \M (v') \M (v) ]^\frac{1}{16} b (\omega, v_*-v) \d \omega \d v_* \,.
		\end{aligned}
	\end{equation}
	
	Let $V = v_* -v$, $V_\shortparallel = (V \cdot \omega) \omega$, $V_\perp = V - V_\shortparallel $, $\zeta = v + \frac{1}{2} V_\shortparallel$, $\zeta_\shortparallel = [(\zeta - \u) \cdot \omega] \omega$, $\zeta_\perp = (\zeta - \u) - \zeta_\shortparallel$. Then $\d \omega \d v_* = 2 |V_\shortparallel|^{-2} \d V_\perp \d V_\shortparallel$, $v' = v + V_\shortparallel$, $v_*' = v + V_\perp $, $v_* = v + V$. Denote by $\tilde{g} (x,v') = \nu^\frac{1}{2} \sigma_x^\frac{m}{2} w_{\beta, \vartheta} g (x, v')$. It then follows from \eqref{b}-\eqref{Cutoff}-\eqref{Trans-1} and $|v +V_\shortparallel - \u|^2 = |\zeta_\perp|^2 + |\zeta_\shortparallel + \frac{1}{2} V_\shortparallel|^2$ that
{\footnotesize
		\begin{align}\label{Kh2-1-x-2}
			\no & |\nu^{- \frac{1}{2}} z_{- \alpha} \sigma_x^\frac{m}{2} w_{\beta, \vartheta} K_{\hbar 2}^{1 - \chi} g (x,v)| \\
			\no \lesssim & \int_{\R^3} \int_{V_\perp \perp V_\shortparallel} z_{- \alpha} (v) \tilde{g} (x, v + V_\shortparallel) [ \M (v + V_\perp) \M (v + V) \M (v + V_\shortparallel) \M (v) ]^\frac{1}{16} \\
			\no & \qquad \qquad \qquad \qquad \qquad \qquad \qquad \times (1 - \chi) (|V|) b (\omega, V) \cdot 2 |V_\shortparallel|^{- 2} \d V_\perp \d V_\shortparallel \\
			\no \lesssim & z_{- \alpha} (v) \M^\frac{1}{16} (v) \int_{\R^3} \int_{V_\perp \perp V_\shortparallel} \tilde{g} (v + V_\shortparallel) \M^\frac{1}{32} (v + V_\shortparallel) \\
			\no & \qquad \quad \times e^{- \frac{1}{32 T} ( |\zeta_\shortparallel|^2 + \frac{1}{4} |V_\shortparallel|^2 + |V_\perp + \zeta_\perp|^2 + |\zeta_\perp|^2 + |\zeta_\shortparallel + \frac{1}{2} V_\shortparallel|^2 )} \tfrac{|V_\shortparallel|^{-1} (1 - \chi) (|V|)}{|V|^{1-\gamma}} \d V_\perp \d V_\shortparallel \\
			\no \lesssim & z_{- \alpha} (v) \M^\frac{1}{16} (v) \int_{\R^3} \int_{V_\perp \perp V_\shortparallel} \tilde{g} (x,v + V_\shortparallel) \M^\frac{1}{32} (v + V_\shortparallel) \\
			\no & \qquad \quad \times e^{- \frac{1}{32 T} ( \frac{1}{4} |V_\shortparallel|^2 + \frac{1}{2} |V_\perp|^2 )} \tfrac{|V_\shortparallel|^{-1} }{(|V_\shortparallel|^2 + |V_\perp|^2)^\frac{1 - \gamma}{2}} \d V_\perp \d V_\shortparallel \\
			= &  z_{- \alpha} (v) \M^\frac{1}{16} (v) \int_{\R^3} \tilde{g} (x,v + V_\shortparallel) \M^\frac{1}{32} (v + V_\shortparallel) |V_\shortparallel|^{-1} e^{- \frac{|V_\shortparallel|^2}{128 T}} \Phi (V_\shortparallel) \d V_\shortparallel \,,
		\end{align}}
	where
	\begin{equation*}
		\begin{aligned}
			\Phi (V_\shortparallel) = \int_{V_\perp \perp V_\shortparallel} (|V_\shortparallel|^2 + |V_\perp |^2 )^\frac{\gamma - 1}{2} e^{- \frac{|V_\perp|^2}{64 T}} \d V_\perp \,.
		\end{aligned}
	\end{equation*}
	By Lemma 2.3 of \cite{JLT-2022-arXiv}, $( |V_\shortparallel|^2 + |V_\perp|^2 )^\frac{\gamma - 1}{2} |V_\perp|^{b_0} \lesssim |V_\shortparallel|^{b_0 + \gamma - 1}$ holds for any fixed $b_0 \in [0, 1 - \gamma]$. Then, if $0 \leq b_0 \leq 1 -\gamma$ and $b_0 < 2$,
	\begin{equation*}
		\begin{aligned}
			\Phi (|V_\shortparallel|) \lesssim |V_\shortparallel|^{b_0 + \gamma - 1} \int_{V_\perp \perp V_\shortparallel} |V_\perp|^{- b_0} e^{- \frac{|V_\perp|^2}{64 T}} \d V_\perp \lesssim |V_\shortparallel|^{b_0 + \gamma - 1} \,.
		\end{aligned}
	\end{equation*}
	It therefore infers that
{\small
		\begin{align}\label{Kh2-1-x-3}
			\no | \nu^{- \frac{1}{2}} z_{- \alpha} \sigma_x^\frac{m}{2} w_{\beta, \vartheta} K_{\hbar 2}^{1 - \chi} g (x, v) | \lesssim & z_{- \alpha} (v) \M^\frac{1}{32} (v) \int_{\R^3} \tilde{g} (x,v + V_\shortparallel) \M^\frac{1}{32} (v + V_\shortparallel) |V_\shortparallel|^{b_0 + \gamma -2} e^{- \frac{|V_\shortparallel|^2}{128 T}} \d V_\shortparallel \\
			= & z_{- \alpha} (v) \int_{\R^3} \tilde{g} (x,\xi) \mathfrak{k}_{b_0} (v, \xi) \d \xi \,,
		\end{align}}
	where $\xi = v + V_\shortparallel$ has been used, and
	\begin{equation*}
		\begin{aligned}
			\mathfrak{k}_{b_0} (v, \xi) = [\M (v) \M (\xi)]^\frac{1}{32} |v - \xi|^{b_0 + \gamma - 2} e^{- \frac{|v - \xi|^2}{128 T}}
		\end{aligned}
	\end{equation*}
	with $b_0 < 2$ and $0 \leq b_0 \leq 1 - \gamma$. We further require $b_0 + \gamma - 2 > - 3$ such that the kernel $\mathfrak{k}_{b_0} (v, \xi)$ is integrable on the variables both $\xi$ and $v$. More precisely, one has
{\small
	\begin{equation}\label{kb0-1}
		\begin{aligned}
			\int_{\R^3} \mathfrak{k}_{b_0} (v, \xi) \d \xi = \M^\frac{1}{32} (v) \int_{\R^3} |v - \xi|^{b_0 + \gamma - 2} e^{- \frac{|v - \xi|^2}{128 T}} \M^\frac{1}{32} (\xi) \d \xi \lesssim (1 + |v|)^{b_0 + \gamma - 2} \M^\frac{1}{32} (v) \lesssim 1
		\end{aligned}
	\end{equation}}
	uniformly in $v$ by Lemma 2.1 of \cite{LS-2010-KRM}. Furthermore, by \eqref{h-bnd},
{\footnotesize
	\begin{equation}\label{kb0-2}
		\begin{aligned}
			\int_{\R^3} z_{- \alpha}^2 (v) \mathfrak{k}_{b_0} (v, \xi) \d v = \M^\frac{1}{32} (\xi) \int_{\R^3} z_{- \alpha}^2 (v) |v - \xi|^{b_0 + \gamma - 2} e^{- \frac{|v - \xi|^2}{128 T}} \M^\frac{1}{32} (v) \d \xi \lesssim \M^\frac{1}{32} (\xi) (1 + |\xi|)^{b_0 + \gamma - 2} \lesssim 1
		\end{aligned}
	\end{equation}}
	uniformly in $\xi$, provided that $b_0 + \gamma - 2 > - 3$ and $0 \leq \alpha < \min \{ \frac{1}{2}, \frac{b_0 + \gamma + 1}{2} \}$ with $b_0 \in S_\gamma$. It thereby follows from \eqref{kb0-1}-\eqref{kb0-2} that
		\begin{align}\label{Kh2-1-x}
			\no \int_{\R^3} |\nu^{- \frac{1}{2}} & z_{- \alpha} \sigma_x^\frac{m}{2} w_{\beta, \vartheta} K_{\hbar 2}^{1 - \chi} g (x,v)|^2 \d v \lesssim \int_{\R^3} z_{- \alpha}^2 (v) \big( \tilde{g} (x,\xi) \mathfrak{k}_{b_0} (v, \xi) \d \xi \big)^2 \d v \\
			\no & \lesssim \int_{\R^3} z_{- \alpha}^2 (v) \big( \int_{\R^3} \mathfrak{k}_{b_0} (v, \xi) \d \xi \big) \big( \int_{\R^3} | \tilde{g} (x,\xi) |^2 \mathfrak{k}_{b_0} (v, \xi) \d \xi \big) \d v \\
			\no & \xlongequal[]{\text{Fubini Theorem}} \int_{\R^3} \big( \int_{\R^3} \mathfrak{k}_{b_0} (v, \xi) \d \xi \big) \big( \int_{\R^3} z_{- \alpha}^2 (v) \mathfrak{k}_{b_0} (v, \xi) \d \xi \big) | \tilde{g} (x,\xi) |^2 \d \xi \\
			& \lesssim \int_{\R^3} | \tilde{g} (x,\xi) |^2 \d \xi = \int_{\R^3} |\nu^\frac{1}{2} \sigma_x^\frac{m}{2} w_{\beta, \vartheta} g (x,v)|^2 \d v
		\end{align}
	for $0 \leq \alpha < \min \{ \frac{1}{2}, \frac{b_0 + \gamma + 1}{2} \}$ with $- 3 < \gamma \leq 1$ and $b_0 \in S_\gamma$. Therefore, the bounds \eqref{Kh2-x} and \eqref{Kh2-1-x} conclude the claim \eqref{Claim-Kh2}.
	
	{\bf Step 3. Estimates for $K_{\hbar 3}$ part.}
	
	The goal is to show that for $0 \leq \alpha < \min \{ \frac{1}{2} , \frac{b_1 + \gamma}{2} \}$ with $b_1 \in T_\gamma$ given in Lemma \ref{Lmm-Kh-L2},
	\begin{equation}\label{Claim-Kh3}
		\begin{aligned}
			\int_{\R^3} |\nu^{- \frac{1}{2}} z_{- \alpha} \sigma_x^\frac{m}{2} w_{\beta, \vartheta} K_{\hbar 3} g (x, v)|^2 \d v \lesssim \int_{\R^3} |\nu^\frac{1}{2} \sigma_x^\frac{m}{2} w_{\beta, \vartheta} g (x,v)|^2 \d v \,.
		\end{aligned}
	\end{equation}
	As similar as in \eqref{Kh2-Decomp}, $K_{\hbar 3} g (x,v)$ can be split as
	\begin{equation*}
		\begin{aligned}
			K_{\hbar 3} g (x,v) = K_{\hbar 3}^\chi g (x,v) + K_{\hbar 3}^{1 - \chi} g (x,v) \,,
		\end{aligned}
	\end{equation*}
	where for $a = \chi$ or $1 - \chi$
	\begin{equation*}
		\begin{aligned}
			K_{\hbar 3}^a g (x,v) = \iint_{\mathbb{R}^3 \times \mathbb{S}^2} e^{\hbar \sigma (x,v) - \hbar \sigma (x, v_*')} a (|v-v_*|) g (x, v_*') \M^\frac{1}{2} (v') \M^\frac{1}{2} (v_*) b (\omega, v_* - v) \d \omega \d v_* \,.
		\end{aligned}
	\end{equation*}
	
	{\em Step 3.1. Control of $K_{\hbar 3}^\chi$.}
	
	By \eqref{p-pro-0} with $v'$ replaced by $v_*'$, one has
	\begin{equation*}
		\begin{aligned}
			& |\nu^{- \frac{1}{2}} z_{- \alpha} \sigma_x^\frac{m}{2} w_{\beta, \vartheta} K_{\hbar 3} g (x, v)| \\
			\lesssim & z_{- \alpha} (v) \iint_{\mathbb{R}^3 \times \mathbb{S}^2}     (1 + |v - v_*'|)^{ |\beta| + \frac{|m| ( 1 - \gamma )}{2} } e^{ ( c \hbar + \vartheta ) ||v - \u|^2 - |v_*' - \u|^2|} \nu^{- \frac{1}{2}} (v) \nu^{- \frac{1}{2}} (v_*') \\
			& \qquad \times |\nu^\frac{1}{2} \sigma_x^\frac{m}{2} w_{\beta, \vartheta} g (x,v_*')| \chi (|v-v_*|) \M^\frac{1}{2} (v') \M^\frac{1}{2} (v_*) b (\omega, v_* - v) \d \omega \d v_* \,,
		\end{aligned}
	\end{equation*}
	which implies that by the H\"older inequality
	\begin{equation*}
		\begin{aligned}
			| \int_{\R^3} \nu^{- \frac{1}{2}} z_{- \alpha} \sigma_x^\frac{m}{2} w_{\beta, \vartheta} K_{\hbar 3}^\chi g (x,v) \cdot h (x, v) \d v | \lesssim \sqrt{I\!I_1 I\!I_2} \,,
		\end{aligned}
	\end{equation*}
	where
	\begin{equation*}
		\begin{aligned}
			I\!I_1 = \iiint_{\R^3 \times \R^3 \times \mathbb{S}^2} |\nu^{- \frac{1}{2}} (v) h (x,v)|^2 \M^\frac{1}{4} (v_*) b (\omega, v_* - v) \d \omega \d v_* \d v \,,
		\end{aligned}
	\end{equation*}
	and
	\begin{equation*}
		\begin{aligned}
			I\!I_2 = \iiint_{\R^3 \times \R^3 \times \mathbb{S}^2} z_{- \alpha}^2 (v) \nu^{-1} (v_*') (1 + |v-v_*'|)^{2 |\beta| + |m| ( 1 - \gamma )} e^{2 ( c \hbar + \vartheta ) | |v-\u|^2 - |v_*' - \u|^2 |} \\
			\times |\nu^\frac{1}{2} \sigma_x^\frac{m}{2} w_{\beta, \vartheta} g (x,v_*')|^2 \M (v') \M^\frac{3}{4} (v_*) \chi^2 (|v-v_*|) b (\omega, v_* - v) \d \omega \d v_* \d v \,.
		\end{aligned}
	\end{equation*}
	
	By \eqref{I1-bnd}, one has
	\begin{equation}\label{II1}
		\begin{aligned}
			I\!I_1 \lesssim \int_{\R^3} |h (x,v)|^2 \d v \,.
		\end{aligned}
	\end{equation}
	From changing the variables $(v, v') \mapsto (v_*, v_*')$ and employing the Fubini Theorem, it follows that
	\begin{equation*}
		\begin{aligned}
			I\!I_2 = \iiint_{\R^3 \times \R^3 \times \mathbb{S}^2} z_{- \alpha}^2 (v_*) \nu^{-1} (v') (1 + |v_*-v'|)^{2 |\beta| + |m| ( 1 - \gamma )} e^{2 ( c \hbar + \vartheta ) | |v_*-\u|^2 - |v' - \u|^2 |} \\
			\times |\nu^\frac{1}{2} \sigma_x^\frac{m}{2} w_{\beta, \vartheta} g (x,v')|^2 \M (v_*') \M^\frac{3}{4} (v) \chi^2 (|v-v_*|) b (\omega, v_* - v) \d \omega \d v_* \d v \,,
		\end{aligned}
	\end{equation*}
	where we have used the fact that the measure $\chi^2 (|v-v_*|) b (\omega, v_* - v) \d \omega \d v_* \d v$ is invariant under the previous changing variables. Further by changing variables $(v, v_*) \mapsto (v', v_*')$ and the corresponding invariant of $\chi^2 (|v-v_*|) b (\omega, v_* - v) \d \omega \d v_* \d v$, one has
	\begin{equation*}
		\begin{aligned}
			I\!I_2 = \int_{\R^3} |\nu^\frac{1}{2} \sigma_x^\frac{m}{2} w_{\beta, \vartheta} g (x,v)|^2 \mathfrak{j}_2 (v) \d v \,,
		\end{aligned}
	\end{equation*}
	where the kernel
	\begin{equation*}
		\begin{aligned}
			\mathfrak{j}_2 (v) = \iint_{\R^3 \times \mathbb{S}^2} z_{- \alpha}^2 (v_*') \nu^{-1} (v) (1 + |v_*'-v|)^{2 |\beta| + |m| ( 1 - \gamma )} e^{2 ( c \hbar + \vartheta ) | |v_*'-\u|^2 - |v - \u|^2 |} \\
			\times \M (v_*) \M^\frac{3}{4} (v') \chi^2 (|v-v_*|) b (\omega, v_* - v) \d \omega \d v_* \,.
		\end{aligned}
	\end{equation*}
	We now want to show that $ \mathfrak{j}_2 (v) \lesssim 1 $ uniformly in $v \in \R^3$.
	
	Let $V = v_* -v$, $V_\shortparallel = (V \cdot \omega) \omega$, $V_\perp = V - V_\shortparallel $, $\zeta = v + \frac{1}{2} V_\shortparallel$, $\zeta_\shortparallel = [(\zeta - \u) \cdot \omega] \omega$, $\zeta_\perp = (\zeta - \u) - \zeta_\shortparallel$. Then $\d \omega \d v_* = 2 |V_\shortparallel|^{-2} \d V_\perp \d V_\shortparallel$, $v' = v + V_\shortparallel$, $v_*' = v + V_\perp $, $v_* = v + V$. Note that
	\begin{equation*}
		\begin{aligned}
			- \tfrac{|v' - \u|^2 + |v_* - \u|^2}{2 T} = - \tfrac{2 |\zeta_\shortparallel + \frac{1}{2} V_\shortparallel|^2 + |\zeta_\perp|^2 + |V_\perp + \zeta_\perp|^2}{2 T} \,, \ |v - \u|^2 - |v_*' - \u|^2 = |\zeta_\perp|^2 - |V_\perp + \zeta_\perp|^2 \,.
		\end{aligned}
	\end{equation*}
	Moreover, $ \nu^{-1} (v) \thicksim (1 + |v|)^{- \gamma} \lesssim (1 + |v'|)^{|\gamma|} (1 + |v-v'|)^{- \gamma} $ for $- 3 < \gamma \leq 1$, which means that $ \nu^{-1} (v) \M (v_*) \M^\frac{3}{4} (v') \lesssim (1 + |V_\shortparallel|)^{- \gamma} e^{- \frac{|v' - \u|^2 + |v_* - \u|^2}{4 T}} $. Together with \eqref{b}-\eqref{Cutoff}, one has
	\begin{equation*}
		\begin{aligned}
			\mathfrak{j}_2 (v) \lesssim \int_{\R^3} \int_{V_\perp \perp V_\shortparallel} z_{- \alpha}^2 (v + V_\perp) (1 + |V_\shortparallel|)^{- \gamma} (1 + |V_\perp|)^{2 |\beta| + |m| ( 1 - \gamma )} e^{2 ( c \hbar + \vartheta ) ||\zeta_\perp|^2 - |V_\perp + \zeta_\perp|^2|} \\
			\times e^{- \frac{2 |\zeta_\shortparallel + \frac{1}{2} V_\shortparallel|^2 + |\zeta_\perp|^2 + |V_\perp + \zeta_\perp|^2}{4 T}} \tfrac{\chi^2 (|V|) |V_\shortparallel|^{-1}}{|V|^{1 - \gamma}} \d V_\perp \d V_\shortparallel \,.
		\end{aligned}
	\end{equation*}
	For sufficiently small $\hbar \geq 0$, there is a $c' > 0$ such that
	\begin{equation*}
		\begin{aligned}
			(1 + |V_\perp|)^{2 |\beta| + |m| ( 1 - \gamma )} e^{2 ( c \hbar + \vartheta ) ||\zeta_\perp|^2 - |V_\perp + \zeta_\perp|^2|} e^{- \frac{|\zeta_\perp|^2 + |V_\perp + \zeta_\perp|^2}{4 T}} \lesssim e^{- c' ( |\zeta_\perp|^2 + |V_\perp|^2 )} \,.
		\end{aligned}
	\end{equation*}
	It also holds $ \tfrac{ (1 + |V_\shortparallel|)^{- \gamma} \chi^2 (|V|)}{|V|^{1 - \gamma}} \lesssim \tfrac{(1 + |V_\shortparallel|)^{- \gamma}}{(1 + |V_\shortparallel| + |V_\perp|)^{1 - \gamma}} \lesssim 1 $ for $- 3 < \gamma \leq 1$. Then
	\begin{equation*}
		\begin{aligned}
			\mathfrak{j}_2 (v) \lesssim \int_{\R^3} |V_\shortparallel|^{-1} e^{- \frac{|\zeta_\shortparallel + \frac{1}{2} V_\shortparallel|^2}{2 T}} \Big( \int_{V_\perp \perp V_\shortparallel} z_{- \alpha}^2 (v + V_\perp) e^{- c' ( |\zeta_\perp|^2 + |V_\perp|^2 )} \d V_\perp \Big) \d V_\shortparallel \,.
		\end{aligned}
	\end{equation*}
	Observe that for $0 \leq \alpha < \frac{1}{2}$, $ \int_{V_\perp \perp V_\shortparallel} z_{- \alpha}^2 (v + V_\perp) e^{- c' ( |\zeta_\perp|^2 + |V_\perp|^2 )} \d V_\perp \lesssim 1 $ uniformly in $v \in \R^3$. Then $ \mathfrak{j}_2 (v) \lesssim \int_{\R^3} |V_\shortparallel|^{-1} e^{- \frac{|\zeta_\shortparallel + \frac{1}{2} V_\shortparallel|^2}{2 T}} \d V_\shortparallel \lesssim 1 $, where the last inequality is derived from \eqref{W1}-\eqref{W2}. It thereby follows
	\begin{equation}\label{II2}
		\begin{aligned}
			I\!I_2 \lesssim \int_{\R^3} |\nu^\frac{1}{2} \sigma_x^\frac{m}{2} w_{\beta, \vartheta} g (x,v)|^2 \d v \,.
		\end{aligned}
	\end{equation}
	As a result, \eqref{II1} and \eqref{II2} indicate that
{\small
	\begin{equation*}
		\begin{aligned}
			| \int_{\R^3} \nu^{- \frac{1}{2}} z_{- \alpha} \sigma_x^\frac{m}{2} w_{\beta, \vartheta} K_{\hbar 3}^\chi g (x,v) \cdot h (x, v) \d v | \lesssim \Big( \int_{\R^3} |\nu^\frac{1}{2} \sigma_x^\frac{m}{2} w_{\beta, \vartheta} g (x,v)|^2 \d v \Big)^\frac{1}{2} \Big( \int_{\R^3} |h (x,v)|^2 \d v \Big)^\frac{1}{2} \,,
		\end{aligned}
	\end{equation*}}
	which, by taking $h (x,v) = \nu^{- \frac{1}{2}} z_{- \alpha} \sigma_x^\frac{m}{2} w_{\beta, \vartheta} K_{\hbar 3}^\chi g (x,v)$, implies that
	\begin{equation}\label{Kh3-x}
		\begin{aligned}
			\int_{\R^3} |\nu^{- \frac{1}{2}} z_{- \alpha} \sigma_x^\frac{m}{2} w_{\beta, \vartheta} K_{\hbar 3}^\chi g (x,v)|^2 \d v \lesssim \int_{\R^3} |\nu^\frac{1}{2} \sigma_x^\frac{m}{2} w_{\beta, \vartheta} g (x,v)|^2 \d v
		\end{aligned}
	\end{equation}
	for $- 3 < \gamma \leq 1$, $ m \in \R $ and $0 \leq \alpha < \frac{1}{2}$.
	
	{\em Step 3.2. Control of $K_{\hbar 3}^{1 - \chi}$.}
	
	Due to $|v-v_*| \leq 2$ in $K_{\hbar 3}^{1 - \chi} g (x,v)$, the claim \eqref{Claim-M} reads $ \M (v') \lesssim \M^{1 - \delta_0} (v_*') $ and $ \M (v_*) \lesssim \M^{1 - \delta_0} (v) $ for any fixed $\delta_0 \in (0, 1)$. From the similar arguments in \eqref{Kh2-1-x-1}, it follows that
	\begin{equation*}
		\begin{aligned}
			|\nu^{- \frac{1}{2}} & z_{- \alpha} \sigma_x^\frac{m}{2} w_{\beta, \vartheta} K_{\hbar 3}^{1 - \chi} g (x,v)| \\
			& \lesssim \iint_{\mathbb{R}^3 \times \mathbb{S}^2} |\nu^\frac{1}{2} \sigma_x^\frac{m}{2} w_{\beta, \vartheta} g (x,v_*')| z_{- \alpha} (v) [ \M (v_*') \M (v_*) \M (v') \M (v) ]^\frac{1}{16} b (\omega, v_*-v) \d \omega \d v_* \,.
		\end{aligned}
	\end{equation*}
	for sufficiently small $\hbar, \vartheta \geq 0$. Let $V = v_* -v$, $V_\shortparallel = (V \cdot \omega) \omega$, $V_\perp = V - V_\shortparallel $, $\zeta = v + \frac{1}{2} V_\shortparallel$, $\zeta_\shortparallel = [(\zeta - \u) \cdot \omega] \omega$, $\zeta_\perp = (\zeta - \u) - \zeta_\shortparallel$. Then one has $ \d \omega \d v_* = 2 |V_\perp|^{-2} \d V_\shortparallel \d V_\perp $. We remark that the derivation of the previous relation is similar to that of $\d \omega \d v_* = 2 |V_\shortparallel|^{-2} \d V_\perp \d V_\shortparallel$ given in (38), Page 35 of \cite{Grad-1963}. It therefore infers from the similar arguments in \eqref{Kh2-1-x-2} that
	\begin{equation*}
		\begin{aligned}
			|\nu^{- \frac{1}{2}} & z_{- \alpha} \sigma_x^\frac{m}{2} w_{\beta, \vartheta} K_{\hbar 3}^{1 - \chi} g (x,v)| \\
			& \lesssim z_{- \alpha} (v) \M^\frac{1}{16} (v) \int_{\R^3} |\nu^\frac{1}{2} \sigma_x^\frac{m}{2} w_{\beta, \vartheta} g (x,v + V_\perp)| \M^\frac{1}{16} (v + V_\perp) |V_\perp|^{- 2} e^{- \frac{|V_\perp|^2}{64 T}} \Psi (V_\perp) \d V_\perp \,,
		\end{aligned}
	\end{equation*}
	where
	\begin{equation*}
		\begin{aligned}
			\Psi (V_\perp) = \int_{V_\shortparallel \perp V_\perp} e^{- \frac{1}{32 T} ( |\zeta_\shortparallel|^2 + \frac{1}{4} |V_\shortparallel|^2 )} \tfrac{|V_\shortparallel| (1 - \chi) (|V|)}{(|V_\shortparallel|^2 + |V_\perp|^2)^\frac{1 - \gamma}{2}} \d V_\shortparallel \,.
		\end{aligned}
	\end{equation*}
	By Lemma 2.3 of \cite{JLT-2022-arXiv}, $ ( |V_\shortparallel|^2 + |V_\perp|^2 )^\frac{\gamma - 1}{2} |V_\shortparallel|^{b_1} \lesssim |V_\perp|^{b_1 + \gamma - 1} $ holds for any fixed $b_1 \in [0, 1 - \gamma]$. Then
	\begin{equation*}
		\begin{aligned}
			\Psi (V_\perp) \lesssim |V_\perp|^{b_1 + \gamma - 1} \int_{V_\shortparallel \perp V_\perp} e^{- \frac{|V_\shortparallel|^2}{128 T}} |V_\shortparallel|^{1 - b_1} \d V_\shortparallel \lesssim |V_\perp|^{b_1 + \gamma - 1}
		\end{aligned}
	\end{equation*}
	under the further constraint $b_1 < 3$. It thereby follows from the same arguments in \eqref{Kh2-1-x-3} that
	\begin{equation*}
		\begin{aligned}
			|\nu^{- \frac{1}{2}} z_{- \alpha} \sigma_x^\frac{m}{2} w_{\beta, \vartheta} K_{\hbar 3}^{1 - \chi} g (x,v)| \lesssim z_{- \alpha} (v) \int_{\R^3} |\nu^\frac{1}{2} \sigma_x^\frac{m}{2} w_{\beta, \vartheta} g (x, \varpi)| \mathfrak{k}_{b_1} (v, \varpi) \d \varpi \,,
		\end{aligned}
	\end{equation*}
	where $ \mathfrak{k}_{b_1} (v, \varpi) = [ \M (v) \M (\varpi) ]^\frac{1}{32} |v - \varpi|^{b_1 + \gamma - 3} e^{- \frac{|v - \varpi|^2}{128 T}} $ with $b_1 < 3$ and $0 \leq b_1 \leq 1 - \gamma$. We further assume that $b_1 + \gamma - 3 > - 3$. As similar as in \eqref{kb0-1}-\eqref{kb0-2}, one has
	\begin{equation*}
		\begin{aligned}
			\int_{\R^3} \mathfrak{k}_{b_1} (v, \varpi) \d \varpi \lesssim (1 + |v|)^{b_1 + \gamma - 3} \M^\frac{1}{32} (v) \lesssim 1
		\end{aligned}
	\end{equation*}
	uniformly in $v \in \R^3$, and
	\begin{equation*}
		\begin{aligned}
			\int_{\R^3} z_{- \alpha}^2 (v) \mathfrak{k}_{b_1} (v, \varpi) \d v \lesssim (1 + |\varpi|)^{b_1 + \gamma - 3} \M^\frac{1}{32} (\varpi) \lesssim 1
		\end{aligned}
	\end{equation*}
	uniformly in $\varpi \in \R^3$, provided that $0 \leq \alpha < \min \{ \frac{1}{2} , \frac{b_1 + \gamma}{2} \}$ with $- 3 < \gamma \leq 1$ and $b_1 \in \mathcal{T}_\gamma$. Then the similar arguments in \eqref{Kh2-1-x} shows that
	\begin{equation}\label{Kh3-1-x}
		\begin{aligned}
			\int_{\R^3} |\nu^{- \frac{1}{2}} z_{- \alpha} \sigma_x^\frac{m}{2} w_{\beta, \vartheta} K_{\hbar 3}^{1 - \chi} g (x,v)|^2 \d v \lesssim \int_{\R^3} |\nu^\frac{1}{2} \sigma_x^\frac{m}{2} w_{\beta, \vartheta} g (x,v)|^2 \d v
		\end{aligned}
	\end{equation}
	for $0 \leq \alpha < \min \{ \frac{1}{2} , \frac{b_1 + \gamma}{2} \}$ with $- 3 < \gamma \leq 1$ and $b_1 \in \mathcal{T}_\gamma$. As a result, the bounds \eqref{Kh3-x} and \eqref{Kh3-x} imply the claim \eqref{Claim-Kh3}.
	
	Finally, the inequalities \eqref{Claim-Kh1}, \eqref{Claim-Kh2} and \eqref{Claim-Kh3} conclude the bound \eqref{Kh-L2-Bnd}. Then the proof of Lemma \ref{Lmm-Kh-L2} is completed.
\end{proof}

\section{Properties of artificial damping operator $\mathbf{D}$}\label{Sec:WMDM}

The goal of this section is to study the properties of artificial damping operator $\mathbf{D}$, i.e., to prove Lemma \ref{Lmm-Dh}-\ref{Lmm-D-XY}-\ref{Lmm-Dh-L2}-\ref{Lmm-Dh-L2-L2}.

\subsection{Coercivity of $\mathbf{D}$}\label{Subsec:D-Coercivity}

We now study the coercivity of the artificial damping operator $\mathbf{D}$, hence, to verify Lemma \ref{Lmm-Dh}.

\begin{proof}[Proof of Lemma \ref{Lmm-Dh}]
	We first decompose the left quantity in \eqref{Dh-Coercivity} as
	\begin{equation}\label{D0D1D2}
		\begin{aligned}
			\int_{\R^3} w_{\beta, \vartheta}^2 g ( \mathbf{D}_\hbar g - \hbar \sigma_x v_3 g) \d v = D_0 + D_1 + D_2 \,,
		\end{aligned}
	\end{equation}
	where
	\begin{equation}\label{D012}
		\begin{aligned}
			D_0 & = \int_{\R^3} \big[ w_{\beta, \vartheta}^2 g \mathbf{D}_\hbar g - \hbar \sigma_x v_3 (\P w_{\beta, \vartheta} g)^2 \big] \d v \,, \ D_1 = - \hbar \int_{\R^3} \sigma_x v_3 (\P^\perp w_{\beta, \vartheta} g)^2 \d v \,, \\
			D_2 & = - 2 \hbar \int_{\R^3} \sigma_x v_3 \P^\perp w_{\beta, \vartheta} g \cdot \P w_{\beta, \vartheta} g \d v \,.
		\end{aligned}
	\end{equation}
	
	Due to $|v_3 \sigma_x| \leq c \nu (v)$ as in Lemma \ref{Lmm-sigma}, the quantity $D_1$ can be bounded by
	\begin{equation}\label{D1-bnd}
		\begin{aligned}
			| D_1 | \leq c \hbar \int_{\R^3} \nu (v) |\P^\perp w_{\beta, \vartheta} g|^2 \d v \,,
		\end{aligned}
	\end{equation}
	and $D_2$ can be dominated by
	\begin{equation*}{\small
		\begin{aligned}
			|D_2| \leq & c \hbar \int_{\R^3} |\sigma_x v_3| |\P^\perp w_{\beta, \vartheta} g| \, |\P w_{\beta, \vartheta} g| \d v \\
			\leq & c \hbar \big( \int_{\R^3} \nu^{- \frac{1}{2}} (v) |v_3 \sigma_x| |\P w_{\beta, \vartheta} g|^2 \d v \big)^\frac{1}{2} \big( \int_{\R^3} \nu (v) |\P^\perp w_{\beta, \vartheta} g|^2 \d v \big)^\frac{1}{2} \,.
		\end{aligned}}
	\end{equation*}
	Moreover, one has
	\begin{equation*}{\small
		\begin{aligned}
			\int_{\R^3} \nu^{- \frac{1}{2}} |v_3 \sigma_x| |\P w_{\beta, \vartheta} g|^2 \d v = \underbrace{ \int_{2 ( 1 + |v - \u| )^{3 - \gamma} \leq \delta x + l} (\cdots) \d v }_{:= D_{21}} + \underbrace{ \int_{2 ( 1 + |v - \u| )^{3 - \gamma} > \delta x + l} (\cdots) \d v }_{: = D_{22}} \,.
		\end{aligned}}
	\end{equation*}
	By \eqref{sigma-x}-\eqref{Omega123} below, if $2 ( 1 + |v - \u| )^{3 - \gamma} \leq \delta x + l$, it follows that $\sigma_x = \tfrac{10 \delta}{3 - \gamma} (\delta x + l)^{- \frac{1 - \gamma}{3 - \gamma}} $. Furthermore, the definition of $\P w_{\beta, \vartheta} g$ in \eqref{Projt-P} can be rewritten as
	\begin{equation}\label{Proj-1}
		\begin{aligned}
			\P w_{\beta, \vartheta} g = \sum_{j=0}^4 a_j \psi_j \,, \ \int_{\R^3} \psi_i (v) \psi_j (v) \d v = \delta_{ij} \quad (0 \leq i,j \leq 4) \,,
		\end{aligned}
	\end{equation}
	which means that
	\begin{equation}\label{Proj-2}
		\begin{aligned}
			\int_{\R^3} |\P w_{\beta, \vartheta} g|^2 \d v = \sum_{j=0}^4 a_j^2 \,.
		\end{aligned}
	\end{equation}
	It is easy to see that $\psi_j$ satisfy
	\begin{equation}\label{Proj-3}
		\begin{aligned}
			\sum_{j=0}^4 \nu^{- \frac{1}{2}} (v) |\psi_j (v)| \leq c'' e^{- c' |v - \u|^2}
		\end{aligned}
	\end{equation}
	for some $c', c'' > 0$. Then $D_{21}$ can be bounded by
{\small
	\begin{align*}
		D_{21} \leq C \delta (\delta x + l)^{- \frac{1 - \gamma}{3 - \gamma}} \int_{2 ( 1 + |v - \u| )^{3 - \gamma} \leq \delta x + l} |v_3| e^{- c' |v - \u|^2} \d v \sum_{j=0}^4 a_j^2 \leq C \delta (\delta x + l)^{- \frac{1 - \gamma}{3 - \gamma}} \int_{\R^3} |\P w_{\beta, \vartheta} g|^2 \d v \,.
	\end{align*}}
	Due to $|v_3 \sigma_x| \leq c \nu (v)$ in Lemma \ref{Lmm-sigma} and the relations \eqref{Proj-1}-\eqref{Proj-2}-\eqref{Proj-3}, $D_{22}$ can be controlled by
{\small
	\begin{align*}
		D_{22} \leq & C \int_{2 ( 1 + |v - \u| )^{3 - \gamma} > \delta x + l} \nu (v) e^{- c' |v - \u|^2} \d v \int_{\R^3} |\P w_{\beta, \vartheta} g|^2 \d v \\
		\leq & C e^{- \frac{c'}{8} ( \frac{\delta x + l}{2} )^\frac{2}{3 - \gamma}} \int_{\R^3} \nu (v) e^{- \frac{c'}{2} |v - \u|^2} \d v \int_{\R^3} |\P w_{\beta, \vartheta} g|^2 \d v \\
		\leq & C e^{ - \frac{c_0}{2} (l/2)^\frac{2}{3 - \gamma} } (\delta x + l)^{- \frac{1 - \gamma}{3 - \gamma}} \int_{\R^3} |\P w_{\beta, \vartheta} g|^2 \d v
	\end{align*}}
	for some $c_0 > 0$, where we have used the fact $|v - \u|^2 \geq \tfrac{1}{4} (\tfrac{\delta x + l}{2})^\frac{2}{3 - \gamma}$ derived from $2 ( 1 + |v - \u| )^{3 - \gamma} > \delta x + l$ when $l \geq 2^{4 - \gamma}$. Then one has $ \int_{\R^3} \nu^{- \frac{1}{2}} |v_3 \sigma_x| |\P w_{\beta, \vartheta} g|^2 \d v \leq C ( \delta + e^{ - \frac{c_0}{2} (l/2)^\frac{2}{3 - \gamma} } ) (\delta x + l)^{- \frac{1 - \gamma}{3 - \gamma}} \int_{\R^3} |\P w_{\beta, \vartheta} g|^2 \d v $, which yields that
{\small
	\begin{equation}\label{D2-bnd}
		\begin{aligned}
			|D_2| \leq C \hbar^\frac{3}{2} ( \delta + e^{ - \frac{c_0}{2} (l/2)^\frac{2}{3 - \gamma} } ) (\delta x + l)^{- \frac{1 - \gamma}{3 - \gamma}} \int_{\R^3} |\P w_{\beta, \vartheta} g|^2 \d v + C \hbar^\frac{1}{2} \int_{\R^3} \nu (v) |\P^\perp w_{\beta, \vartheta} g|^2 \d v \,.
		\end{aligned}
	\end{equation}}
	
	It turns to control the quantity $D_0$ in \eqref{D012}. Observe that
	\begin{equation}\label{D0EF}
		\begin{aligned}
			D_0 = E + G \,,
		\end{aligned}
	\end{equation}
	where
	\begin{equation}\label{EF}{\small
		\begin{aligned}
			E = \int_{\R^3} \big[ w_{\beta, \vartheta} g \mathbf{D} w_{\beta, \vartheta} g - \hbar \sigma_x v_3 ( \P w_{\beta, \vartheta} g )^2 \big] \d v \,, \ G = \int_{\R^3} w_{\beta, \vartheta} g ( w_{\beta, \vartheta} \mathbf{D}_\hbar w_{\beta, \vartheta}^{-1} - \mathbf{D} ) w_{\beta, \vartheta} g \d v \,.
		\end{aligned}}
	\end{equation}
	We split the quantity $E$ as
	\begin{equation}\label{EE1E2}
		\begin{aligned}
			E = \int_{2 ( 1 + |v - \u| )^{3 - \gamma} \leq \delta x + l} (\cdots) \d v + \int_{2 ( 1 + |v - \u| )^{3 - \gamma} > \delta x + l} (\cdots) \d v : = E_1 + E_2 \,.
		\end{aligned}
	\end{equation}
	We first control the quantity $E_1$. By utilizing the fact \eqref{psi*-basis}, one knows that
	\begin{equation}\label{a*}
		\begin{aligned}
			\P w_{\beta, \vartheta} g = \sum_{j=0}^4 a_j^* \psi_j^* : = \sum_{j=0}^4 \int_{\R^3} w_{\beta, \vartheta} g \psi_j^* \d v \psi_j^* \,.
		\end{aligned}
	\end{equation}
	By \eqref{sigma-x}, if $2 ( 1 + |v - \u| )^{3 - \gamma} \leq \delta x + l$, one has $\sigma_x = \tfrac{10 \delta}{3 - \gamma} (\delta x + l)^{- \frac{1 - \gamma}{3 - \gamma}} $. Together with the definition of $\mathbf{D}$ in \eqref{D-operator}, the quantity $E_1$ can be expressed by
	\begin{equation}\label{E1bnd0}
		\begin{aligned}
			E_1 = (\delta x + l)^{- \frac{1 - \gamma}{3 - \gamma}} \int_{2 ( 1 + |v - \u| )^{3 - \gamma} \leq \delta x + l} \big[ \ss_+ w_{\beta, \vartheta} g \P^+ ( v_3 w_{\beta, \vartheta} g ) + \ss_0 w_{\beta, \vartheta} g \P^0 w_{\beta, \vartheta} g \\
			- \tfrac{10 \delta \hbar}{3 - \gamma} v_3 \big\{  \big( \sum_{j=0}^4 a_j^* \psi_j^* \big)^2 - (a_4^* \psi_4^*)^2 \big\} + v_3 (a_4^* \psi_4^*)^2 \big] \d v \,.
		\end{aligned}
	\end{equation}
	Observe that
	\begin{equation}\label{E10-E11}
		\begin{aligned}
			& \int_{2 ( 1 + |v - \u| )^{3 - \gamma} \leq \delta x + l} \ss_+ w_{\beta, \vartheta} g \P^+ ( v_3 w_{\beta, \vartheta} g ) \d v \\
			= & \underbrace{ \ss_+ \int_{\R^3} w_{\beta, \vartheta} g \P^+ ( v_3 w_{\beta, \vartheta} g ) \d v }_{E_{10}} - \underbrace{ \ss_+ \int_{2 ( 1 + |v - \u| )^{3 - \gamma} > \delta x + l} \ss_+ w_{\beta, \vartheta} g \P^+ ( v_3 w_{\beta, \vartheta} g ) \d v }_{E_{11}} \,.
		\end{aligned}
	\end{equation}
	It then follows from the straightforward computation that
	\begin{equation}\label{E10-bnd}
		\begin{aligned}
			E_{10} = & \ss_+ \int_{\R^3} w_{\beta, \vartheta} g \P^+ ( v_3 \P w_{\beta, \vartheta} g ) \d v + \ss_+ \int_{\R^3} w_{\beta, \vartheta} g \P^+ ( v_3 \P^\perp w_{\beta, \vartheta} g ) \d v \\
			= & \ss_+ (a_3^*)^2 \int_{\R^3} \psi_3^* \P v_3 \P \psi_3^* \d v + \ss_+ a_3^* \int_{\R^3} \psi_3^* \P^+ ( v_3 \P^\perp w_{\beta, \vartheta} g ) \d v \\
			= & \sqrt{\tfrac{5}{3} T } \ss_+ (a_3^*)^2 + \ss_+ a_3^* \int_{\R^3} \psi_3^* \P^+ ( v_3 \P^\perp w_{\beta, \vartheta} g ) \d v \\
			\geq & \tfrac{1}{2} \sqrt{\tfrac{5}{3} T } \ss_+ (a_3^*)^2 - C \ss_+ \int_{\R^3} \nu | \P^\perp w_{\beta, \vartheta} g |^2 \d v \,.
		\end{aligned}
	\end{equation}
	By the definition of $\P^+$ in \eqref{P+}, one has
	\begin{equation*}
		\begin{aligned}
			E_{11} = \ss_+ \int_{\R^3} v_3 \psi_3^* w_{\beta, \vartheta} g \d v \int_{2 ( 1 + |v - \u| )^{3 - \gamma} > \delta x + l} \psi_3^* w_{\beta, \vartheta} g \d v \,.
		\end{aligned}
	\end{equation*}
	Notice that $|\psi_3^*| \leq C e^{ - c_0 ( 1 + |v - \u| )^2 }$ for some $C, c_0 > 0$. Then
	\begin{equation*}
		\begin{aligned}
			| \int_{2 ( 1 + |v - \u| )^{3 - \gamma} > \delta x + l} \psi_3^* w_{\beta, \vartheta} g \d v | \leq & e^{ - \frac{c_0}{2} (l / 2)^\frac{2}{3 - \gamma} } \int_{\R^3} e^{ \frac{c_0}{2} (1 + |v - \u|)^2 } |\psi_3^* w_{\beta, \vartheta} g | \d v \\
			\leq & C e^{ - \frac{c_0}{2} (l / 2)^\frac{2}{3 - \gamma} } \big( \int_{\R^3} \nu | w_{\beta, \vartheta} g |^2 \d v \big)^\frac{1}{2} \,.
		\end{aligned}
	\end{equation*}
	One obviously has $| \int_{\R^3} v_3 \psi_3^* w_{\beta, \vartheta} g \d v | \leq C \big( \int_{\R^3} \nu | w_{\beta, \vartheta} g |^2 \d v \big)^\frac{1}{2} $. Then the quantity $E_{11}$ can be bounded by
	\begin{equation}\label{E11-bnd}
		\begin{aligned}
			E_{11} \leq C \ss_+ e^{ - \frac{c_0}{2} (l / 2)^\frac{2}{3 - \gamma} } \big( \int_{\R^3} | \P w_{\beta, \vartheta} g |^2 \d v + \int_{\R^3} \nu | \P^\perp w_{\beta, \vartheta} g |^2 \d v \big) \,.
		\end{aligned}
	\end{equation}
	In summary, the relations \eqref{E10-E11}, \eqref{E10-bnd} and \eqref{E11-bnd} reduce to
		\begin{align}\label{E1bnd1}
			\no & \int_{2 ( 1 + |v - \u| )^{3 - \gamma} \leq \delta x + l} \ss_+ w_{\beta, \vartheta} g \P^+ ( v_3 w_{\beta, \vartheta} g ) \d v \\
			\no \geq & \tfrac{1}{2} \sqrt{\tfrac{5}{3} T } \ss_+ (a_3^*)^2 - C \ss_+ e^{ - \frac{c_0}{2} (l / 2)^\frac{2}{3 - \gamma} } \int_{\R^3} | \P w_{\beta, \vartheta} g |^2 \d v \\
			& - C \ss_+ ( 1 + e^{ - \frac{c_0}{2} (l / 2)^\frac{2}{3 - \gamma} } ) \int_{\R^3} \nu | \P^\perp w_{\beta, \vartheta} g |^2 \d v \,.
		\end{align}
	
	Moreover,
	\begin{equation*}
		\begin{aligned}
			& \int_{2 ( 1 + |v - \u| )^{3 - \gamma} \leq \delta x + l} \ss_0 w_{\beta, \vartheta} g \P^0 ( v_3 w_{\beta, \vartheta} g ) \d v \\
			= & \ss_0 \int_{\R^3} w_{\beta, \vartheta} g \P^0 ( v_3 w_{\beta, \vartheta} g ) \d v - \ss_0 \int_{2 ( 1 + |v - \u| )^{3 - \gamma} > \delta x + l} w_{\beta, \vartheta} g \P^0 ( v_3 w_{\beta, \vartheta} g ) \d v \,.
		\end{aligned}
	\end{equation*}
	By Lemma 3 of \cite{Golse-2008-BIMA}, the map $ w_{\beta, \vartheta} g \mapsto \big( \int_{\R^3} w_{\beta, \vartheta} g \P^0 w_{\beta, \vartheta} g \d v \big)^\frac{1}{2} $ defines a norm on $\mathrm{Span} \{ \psi_0^*, \psi_1^*, \psi_2^* \} $. It is well known that all norms on the finite dimensional space are equivalent each other. Then there is a constant $\mu_1 > 0$ such that
	\begin{equation*}{\small
		\begin{aligned}
			\int_{\R^3} w_{\beta, \vartheta} g \P^0 ( v_3 w_{\beta, \vartheta} g ) \d v \geq \mu_1 \sum_{j=0}^2 (a_j^*)^2 \,.
		\end{aligned}}
	\end{equation*}
	Following the similar arguments in \eqref{E11-bnd}, one has
{\small
	\begin{equation*}
		\begin{aligned}
			\int_{2 ( 1 + |v - \u| )^{3 - \gamma} > \delta x + l} w_{\beta, \vartheta} g \P^0 ( v_3 w_{\beta, \vartheta} g ) \d v \leq C e^{ - \frac{c_0}{2} (l / 2)^\frac{2}{3 - \gamma} } \big( \int_{\R^3} | \P w_{\beta, \vartheta} g |^2 \d v + \int_{\R^3} \nu | \P^\perp w_{\beta, \vartheta} g |^2 \d v \big) \,.
		\end{aligned}
	\end{equation*}}
	It therefore holds that
	\begin{equation}\label{E1bnd2}{\small
		\begin{aligned}
			& \int_{2 ( 1 + |v - \u| )^{3 - \gamma} \leq \delta x + l} \ss_0 w_{\beta, \vartheta} g \P^0 ( v_3 w_{\beta, \vartheta} g ) \d v \\
			\geq & \mu_1 \ss_0 \sum_{j=0}^2 (a_j^*)^2 - C \ss_0 e^{ - \frac{c_0}{2} (l / 2)^\frac{2}{3 - \gamma} } \big( \int_{\R^3} | \P w_{\beta, \vartheta} g |^2 \d v + \int_{\R^3} \nu | \P^\perp w_{\beta, \vartheta} g |^2 \d v \big) \,.
		\end{aligned}}
	\end{equation}
	
	Recall that $\psi_4^*$ satisfies $\A \psi_4^* = \P v_3 \P \psi_4^* = \lambda_4 \psi_4^*$ with $\lambda_4 = - \sqrt{\frac{5}{3} T} < 0$. Then
	\begin{equation}\label{Neg-Damp}
		\begin{aligned}
			- \tfrac{10 \delta \hbar}{3 - \gamma} \int_{\R^3} v_3 ( a_4^* \psi_4^* )^2 \d v = & \sqrt{\tfrac{3}{5 T}} \tfrac{10 \delta \hbar}{3 - \gamma} (a_4^*)^2 \int_{\R^3} v_3 \psi_4^* \P v_3 \P \psi_4^* \d v \\
			= & \sqrt{\tfrac{3}{5 T}} \tfrac{10 \delta \hbar}{3 - \gamma} (a_4^*)^2 \int_{\R^3} ( \P v_3 \psi_4^* )^2 \d v \geq \mu_2 \delta \hbar  (a_4^*)^2
		\end{aligned}
	\end{equation}
	with $\mu_2 = \sqrt{\tfrac{3}{5 T}} \tfrac{10}{3 - \gamma} \int_{\R^3} ( \P v_3 \psi_4^* )^2 \d v > 0$. Furthermore, the similar arguments in \eqref{E11-bnd} imply that
	\begin{equation*}
		\begin{aligned}
			| \tfrac{10 \delta \hbar}{3 - \gamma} \int_{2 ( 1 + |v - \u| )^{3 - \gamma} > \delta x + l} v_3 ( a_4^* \psi_4^* )^2 \d v | \leq C \delta \hbar e^{ - \frac{c_0}{2} (l / 2)^\frac{2}{3 - \gamma} } \int_{\R^3} | \P w_{\beta, \vartheta} g |^2 \d v \,.
		\end{aligned}
	\end{equation*}
	As a consequence,
	\begin{equation}\label{E1bnd3}{\small
		\begin{aligned}
			- \tfrac{10 \delta \hbar}{3 - \gamma} \int_{2 ( 1 + |v - \u| )^{3 - \gamma} \leq \delta x + l} v_3 ( a_4^* \psi_4^* )^2 \d v \geq \mu_2 \delta \hbar  (a_4^*)^2 - C \delta \hbar e^{ - \frac{c_0}{2} (l / 2)^\frac{2}{3 - \gamma} } \int_{\R^3} | \P w_{\beta, \vartheta} g |^2 \d v \,.
		\end{aligned}}
	\end{equation}
	Following the above similar arguments, one derives that
	\begin{equation}\label{E1bnd4}{\small
		\begin{aligned}
			& | \tfrac{10 \delta \hbar}{3 - \gamma} \int_{2 ( 1 + |v - \u| )^{3 - \gamma} \leq \delta x + l} v_3 \big\{  \big( \sum_{j=0}^4 a_j^* \psi_j^* \big)^2 - (a_4^* \psi_4^*)^2 \big\} \d v | \\
			\leq & \tfrac{1}{2} \mu_2 \delta \hbar (a_4^*)^2 + c_1 \delta \hbar \sum_{j=0}^3 (a_j^*)^2 + C \delta \hbar e^{ - \frac{c_0}{2} (l / 2)^\frac{2}{3 - \gamma} } \int_{\R^3} | \P w_{\beta, \vartheta} g |^2 \d v
		\end{aligned}}
	\end{equation}
	for some harmless constant $c_1 > 0$. Consequently, combining with the relations \eqref{E1bnd0}, \eqref{E1bnd1}, \eqref{E1bnd2}, \eqref{E1bnd3} and \eqref{E1bnd4}, one gains
{\small
		\begin{align*}
			E_1 \geq & (\delta x + l)^{- \frac{1 - \gamma}{3 - \gamma}} \big[ (\mu_1 \ss_0 - c_1 \delta \hbar) \sum_{j=0}^2 (a_j^*)^2 + ( \tfrac{1}{2} \sqrt{\tfrac{5}{3} T} \ss_+ - c_1 \delta \hbar ) (a_3^*)^2 + \tfrac{1}{2} \mu_2 \delta \hbar (a_4^*)^2 \big] \\
			& - C (\ss_0 + \ss_+ + \delta \hbar ) e^{ - \frac{c_0}{2} (l / 2)^\frac{2}{3 - \gamma} } (\delta x + l)^{- \frac{1 - \gamma}{3 - \gamma}} \int_{\R^3} | \P w_{\beta, \vartheta} g |^2 \d v \\
			& - C [ \ss_0 + \ss_+ (1 + e^{ - \frac{c_0}{2} (l / 2)^\frac{2}{3 - \gamma} } ) ] (\delta x + l)^{- \frac{1 - \gamma}{3 - \gamma}} \int_{\R^3} \nu | \P^\perp w_{\beta, \vartheta} g |^2 \d v \,.
		\end{align*}}
	We now take $\ss_0 = \ss_0' \delta \hbar$ and $\ss_+ = \ss_+' \delta \hbar$, where $\ss_0', \ss_+' > 0$ are large enough such that
	\begin{equation*}
		\begin{aligned}
			\mu_1 \ss_0' - c_1 \geq \tfrac{1}{2} \mu_2 \,, \ \tfrac{1}{2} \sqrt{\tfrac{5}{3} T} \ss_+' - c_1 \geq \tfrac{1}{2} \mu_2 \,.
		\end{aligned}
	\end{equation*}
	Then, together with the properties of the orthonormal basis $\{ \psi_j^* \}_{0 \leq j \leq 4}$ in \eqref{psi*-basis} and the definitions of $a_j^* (0 \leq j \leq 4)$ in \eqref{a*}, one obtains
	\begin{equation*}
		\begin{aligned}
			& (\mu_1 \ss_0 - c_1 \delta \hbar) \sum_{j=0}^2 (a_j^*)^2 + ( \tfrac{1}{2} \sqrt{\tfrac{5}{3} T} \ss_+ - c_1 \delta \hbar ) (a_3^*)^2 + \tfrac{1}{2} \mu_2 \delta \hbar (a_4^*)^2 \\
			\geq & \tfrac{1}{2} \mu_2 \delta \hbar \sum_{j=0}^4 (a_j^*)^2 = \tfrac{1}{2} \mu_2 \delta \hbar \int_{\R^3} | \P w_{\beta, \vartheta} g |^2 \d v \,,
		\end{aligned}
	\end{equation*}
	which means that
	\begin{equation}\label{E1-bnd}{\scriptsize
		\begin{aligned}
			E_1 \geq ( 5 \mu_0 - C e^{ - \frac{c_0}{2} (l / 2)^\frac{2}{3 - \gamma} } ) \delta \hbar (\delta x + l)^{- \frac{1 - \gamma}{3 - \gamma}} \int_{\R^3} |\P w_{\beta, \vartheta} g|^2 \d v - C \delta \hbar (\delta x + l)^{- \frac{1 - \gamma}{3 - \gamma}} \int_{\R^3} \nu | \P^\perp w_{\beta, \vartheta} g |^2 \d v \,,
		\end{aligned}}
	\end{equation}
	where $\mu_0 = \tfrac{1}{10} \mu_2 > 0$.
	
	By the similar arguments in \eqref{E11-bnd}, the quantity $E_2$ can be bounded by
	\begin{equation}\label{E2-bnd}
		\begin{aligned}
			|E_2| \leq C e^{ - \frac{c_0}{2} (l / 2)^\frac{2}{3 - \gamma} } \hbar (\delta x + l)^{- \frac{1 - \gamma}{3 - \gamma}} \int_{\R^3} |\P w_{\beta, \vartheta} g|^2 \d v \,.
		\end{aligned}
	\end{equation}
	Then \eqref{EE1E2}, \eqref{E1-bnd} and \eqref{E2-bnd} show that
	\begin{equation}\label{E-bnd}
		\begin{aligned}
			E \geq E_1 - |E_2| \geq \big[ 5 \mu_0 \delta - C (\delta + 1) e^{ - \frac{c_0}{2} (l / 2)^\frac{2}{3 - \gamma} } \big] \hbar (\delta x + l)^{- \frac{1 - \gamma}{3 - \gamma}} \int_{\R^3} |\P w_{\beta, \vartheta} g|^2 \d v \\
			- C \delta \hbar (\delta x + l)^{- \frac{1 - \gamma}{3 - \gamma}} \int_{\R^3} \nu | \P^\perp w_{\beta, \vartheta} g |^2 \d v \,.
		\end{aligned}
	\end{equation}
	
	We then control the quantity $G$. By the definition of the operator $\mathbf{D}$ in \eqref{D-operator}, one has
	\begin{equation}\label{G1G2}{\footnotesize
		\begin{aligned}
			( w_{\beta, \vartheta} \mathbf{D}_\hbar w_{\beta, \vartheta}^{-1} - \mathbf{D} ) w_{\beta, \vartheta} g = & (\delta x + l)^{- \frac{1 - \gamma}{3 - \gamma}} \Big\{ \ss_+ \big[ \underbrace{ w_{\beta, \vartheta} e^{\hbar \sigma} \P^+ ( v_3 w_{\beta, \vartheta}^{-1} e^{- \hbar \sigma} w_{\beta, \vartheta} g ) - \P^+ (v_3 w_{\beta, \vartheta} g ) }_{G_1} \big] \\
		    & \qquad \qquad \qquad \quad + \ss_0 \big[ \underbrace{ w_{\beta, \vartheta} e^{\hbar \sigma} \P^0 ( w_{\beta, \vartheta}^{-1} e^{- \hbar \sigma} w_{\beta, \vartheta} g ) - \P^0 ( w_{\beta, \vartheta} g ) }_{G_2} \big] \Big\} \,.
		\end{aligned}}
	\end{equation}
	Together with the definition of $\P^+$ in \eqref{P+} and the fact $w_{\beta, \vartheta} (v) w_{\beta, \vartheta}^{-1} (\tilde{v}) \leq C e^{2 \vartheta ( |v - \u|^2 + |\tilde{v} - \u|^2 )}$, one has
	\begin{equation*}
		\begin{aligned}
			| G_1 | \leq C | \psi_3^* (v) | \int_{\R^3} | \tilde{v}_3 | \mathfrak{q} (\hbar, \vartheta, v, \tilde{v}) | w_{\beta, \vartheta} (\tilde{v}) g (\tilde{v}) \psi_3^* (\tilde{v}) | \d \tilde{v} \,,
		\end{aligned}
	\end{equation*}
	where
	\begin{equation*}
		\begin{aligned}
			\mathfrak{q} (\hbar, \vartheta, v, \tilde{v}) = e^{ 2 \vartheta ( |v - \u|^2 + |\tilde{v} - \u|^2 ) + \hbar | \sigma (x,v) - \sigma (x, \tilde{v}) | } - 1 \,.
		\end{aligned}
	\end{equation*}
    By the similar arguments in Lemma 2 of \cite{Chen-Liu-Yang-2004-AA}, one knows that for $0 \leq \vartheta < \tfrac{1}{16 T}$,
	\begin{equation}\label{Claim-q}
		\begin{aligned}
			\sup_{v, \tilde{v} \in \R^3} \big\{ \mathfrak{q} (\hbar, \vartheta, v, \tilde{v}) e^{ - \frac{|v - \u|^2 + |\tilde{v} - \u|^2}{8 T} } \big\} \leq C \hbar
		\end{aligned}
	\end{equation}
	for sufficiently small $\hbar > 0 $ and for some harmless constant $C > 0$. Then $G_1$ can be further bounded by
	\begin{equation}\label{G1-bnd}
		\begin{aligned}
			| G_1 | \leq C \hbar | \psi_3^* (v) | \big( \int_{\R^3} \nu (\tilde{v}) | w_{\beta, \vartheta} (\tilde{v}) g (\tilde{v}) |^2 \d \tilde{v} \big)^\frac{1}{2} \,.
		\end{aligned}
	\end{equation}
	Similarly, $G_2$ can be bounded by
	\begin{equation}\label{G2-bnd}
		\begin{aligned}
			| G_2 | \leq C \hbar \sum_{j=0}^2 |\psi_j^* (v)| \big( \int_{\R^3} \nu (\tilde{v}) | w_{\beta, \vartheta} (\tilde{v}) g (\tilde{v}) |^2 \d \tilde{v} \big)^\frac{1}{2} \,.
		\end{aligned}
	\end{equation}
	From plugging \eqref{G1G2}, \eqref{G1-bnd} and \eqref{G2-bnd} into the definition of $G$ in \eqref{EF}, one follows that
	\begin{equation}\label{G-bnd}
		\begin{aligned}
			|G| \leq C \delta \hbar^2 (\delta x + l)^{- \frac{1 - \gamma}{3 - \gamma}} \big( \int_{\R^3} | \P w_{\beta, \vartheta} g |^2 \d v + \int_{\R^3} \nu | \P^\perp w_{\beta, \vartheta} g |^2 \d v \big)
		\end{aligned}
	\end{equation}
	for $0 \leq \vartheta < \tfrac{1}{16 T}$ and sufficiently small $\hbar > 0$. It then follows from \eqref{D0EF}, \eqref{E-bnd} and \eqref{G-bnd} that
{\small
	\begin{align}\label{D0-bnd}
		\no D_0 \geq & E - |G| \geq \big[ 5 \mu_0 \delta - C \delta \hbar - C (\delta + 1) e^{ - \frac{c_0}{2} (l / 2)^\frac{2}{3 - \gamma} } \big] \hbar (\delta x + l)^{- \frac{1 - \gamma}{3 - \gamma}} \int_{\R^3} |\P w_{\beta, \vartheta} g|^2 \d v \\
		& - C \delta \hbar (\delta x + l)^{- \frac{1 - \gamma}{3 - \gamma}} \int_{\R^3} \nu | \P^\perp w_{\beta, \vartheta} g |^2 \d v \,.
	\end{align}}
	Then, by plugging \eqref{D1-bnd}, \eqref{D2-bnd} and \eqref{D0-bnd} into \eqref{D0D1D2}, one has
	\begin{equation*}
		\begin{aligned}
			& \int_{\R^3} w_{\beta, \vartheta}^2 g ( \mathbf{D}_\hbar g - \hbar \sigma_x v_3 g) \d v \\
			\geq & \big[ 5 \mu_0 \delta - C \delta \hbar - C \delta \hbar^\frac{1}{2} - C (\delta + \hbar^\frac{1}{2} + 1) e^{ - \frac{c_0}{2} (l / 2)^\frac{2}{3 - \gamma} } \big] \hbar (\delta x + l)^{- \frac{1 - \gamma}{3 - \gamma}} \int_{\R^3} |\P w_{\beta, \vartheta} g|^2 \d v \\
			& - C ( \delta \hbar + \hbar + \hbar^\frac{1}{2} ) (\delta x + l)^{- \frac{1 - \gamma}{3 - \gamma}} \int_{\R^3} \nu | \P^\perp w_{\beta, \vartheta} g |^2 \d v \,.
		\end{aligned}
	\end{equation*}
	We first take $0 < \delta < 1$ and $0 < \hbar < \frac{2 \mu_0}{C}$, then take $l > 2 [ \frac{2}{c_0} ( \ln \frac{C}{\nu_0} + \ln \frac{1}{\delta} ) ]^\frac{3 - \gamma}{2}$ such that $ 2 C e^{ - \frac{c_0}{2} (l / 2)^\frac{2}{3 - \gamma} } < 2 \mu_0 \delta $. One thereby has
	\begin{equation*}
		\begin{aligned}
			5 \mu_0 \delta - C \delta \hbar - C \delta \hbar^\frac{1}{2} - C (\delta + \hbar^\frac{1}{2} + 1) e^{ - \frac{c_0}{2} (l / 2)^\frac{2}{3 - \gamma} } \geq \mu_0 \delta \,, \quad C ( \delta \hbar + \hbar + \hbar^\frac{1}{2} ) \leq 2 C \hbar^\frac{1}{2} \,.
		\end{aligned}
	\end{equation*}
	Therefore, the results in Lemma \ref{Lmm-Dh} hold and the proof is finished.
\end{proof}

\subsection{Weighted $L^\infty_{x,v}$ estimate of $\mathbf{D}$}\label{Subsec:D-L5}

We now study the weighted $L^\infty_{x,v}$ estimate of the artificial damping operator $\mathbf{D}$, hence, to verify Lemma \ref{Lmm-D-XY}.

\begin{proof}[Proof of Lemma \ref{Lmm-D-XY}]
	For the simplicity of notations, we employ the notation $\mathfrak{p} (x,v)$ given in \eqref{p}, i.e.,
	\begin{equation*}
		\begin{aligned}
			\mathfrak{p} (x,v) = \sigma_x^\frac{m}{2} (x,v) e^{\hbar \sigma (x,v)} w_{\beta, \vartheta} (v) \,.
		\end{aligned}
	\end{equation*}
	Then one has
	\begin{equation*}
		\begin{aligned}
			\| \sigma_x^\frac{m}{2} e^{\hbar \sigma} \mathbf{D} g \|_{\beta, \vartheta} = \| \mathfrak{p} \mathbf{D} g \|_{L^\infty_{x,v}} \,, \quad \| \sigma_x^\frac{m}{2} e^{\hbar \sigma} g \|_{\beta - \tilde{N}, \vartheta} = \| (1 + |v|)^{-\tilde{N}} \mathfrak{p} g \|_{L^\infty_{x,v}} \,.
		\end{aligned}
	\end{equation*}
	
	Now we control the norm $\| \mathfrak{p} \mathbf{D} g \|_{L^\infty_{x,v}}$. Recalling the definition of $\mathbf{D} g$ in \eqref{D-operator}, one has
	\begin{equation*}
		\begin{aligned}
			\mathfrak{p} \mathbf{D} g (x, v) = \sum_{j=0}^3 \underbrace{ \int_{\R^3} \mathfrak{d}_j (x,v, \tilde{v}) (1 + \tilde{v})^{-\tilde{N}} \mathfrak{p} (x, \tilde{v}) g (x, \tilde{v}) \d \tilde{v} }_{ \mathbf{P}_j } \,,
		\end{aligned}
	\end{equation*}
	where
		\begin{align*}
			& \mathfrak{d}_0 (x,v, \tilde{v}) = ( \int_{\R^3} \widehat{\mathbb{B}}_3 \mathbb{B}_3 \d v )^{-1} \ss_0 (\delta x + l)^{- \frac{1 - \gamma}{3 - \gamma}} \tfrac{\mathfrak{p} (x,v)}{\mathfrak{p} (x, \tilde{v})} \tilde{v}_3 (1 + |\tilde{v}|)^{\tilde{N}} \widehat{\mathbb{B}}_3 (\tilde{v}) \psi_0^* (v) \,, \\
			& \mathfrak{d}_j (x,v, \tilde{v}) = ( \int_{\R^3} \widehat{\mathbb{A}}_{j3} \mathbb{A}_{j3} \d v )^{-1} \ss_0 (\delta x + l)^{- \frac{1 - \gamma}{3 - \gamma}} \tfrac{\mathfrak{p} (x,v)}{\mathfrak{p} (x, \tilde{v})} \tilde{v}_3 (1 + |\tilde{v}|)^{\tilde{N}} \widehat{\mathbb{A}}_{j3} (\tilde{v}) \psi_j^* (v) \,, \ j = 1,2 \,, \\
			& \mathfrak{d}_3 (x,v, \tilde{v}) = \ss_+ (\delta x + l)^{- \frac{1 - \gamma}{3 - \gamma}} \tfrac{\mathfrak{p} (x,v)}{\mathfrak{p} (x, \tilde{v})} \tilde{v}_3 (1 + |\tilde{v}|)^{\tilde{N}} \psi_3^* (\tilde{v}) \psi_3^* (v) \,.
		\end{align*}
	
	We first control the quantity $\mathbf{P}_0$. More precisely, it holds
	\begin{equation*}
		\begin{aligned}
			| \mathbf{P}_0 | \leq \| \sigma_x^\frac{m}{2} e^{\hbar \sigma} g \|_{\beta - \tilde{N}, \vartheta} \int_{\R^3} | \mathfrak{d}_0 (x,v, \tilde{v}) | \d \tilde{v}
		\end{aligned}
	\end{equation*}
	By \eqref{p-pro}, it infers
	\begin{equation*}
		\begin{aligned}
			\tfrac{\mathfrak{p} (x,v)}{\mathfrak{p} (x, \tilde{v})} \leq C (1 + |v|)^k (1 + |\tilde{v}|)^k \exp \big[ (c \hbar + \vartheta) ( |v - \u|^2 + |\tilde{v} - \u|^2 ) \big] \,.
		\end{aligned}
	\end{equation*}
	Moreover, it is easy to see that
	\begin{equation*}
		\begin{aligned}
			|\widehat{\mathbb{B}}_3 (\tilde{v}) \psi_0^* (v)| \leq C \exp \big[ - c' (|v - \u|^2 + |\tilde{v} - \u|^2) \big] \,.
		\end{aligned}
	\end{equation*}
	Note that
	\begin{equation*}
		\begin{aligned}
			|\tilde{v}_3| (1 + |\tilde{v}|)^{k+\tilde{N}} (1 + |v|)^k \exp \big[ - \tfrac{c'}{2} (|v - \u|^2 + |\tilde{v} - \u|^2) \big] \leq C \,.
		\end{aligned}
	\end{equation*}
	It therefore follows that
	\begin{equation}\label{dj-bnd}
		\begin{aligned}
			|\mathfrak{d}_0 (x,v, \tilde{v})| \leq C \ss_0 (\delta x + l)^{- \frac{1 - \gamma}{3 - \gamma}} \exp \big[ - s_0 (|v - \u|^2 + |\tilde{v} - \u|^2) \big] \,,
		\end{aligned}
	\end{equation}
	where $s_0 = \tfrac{c'}{2} - c \hbar - \vartheta > 0$ provived that $\hbar, \vartheta \geq 0$ is sufficiently small. Consequently, one has
	\begin{equation*}
		\begin{aligned}
			\sum_{j=0}^4 \int_{\R^3} | \mathfrak{d}_j (x,v, \tilde{v}) | \d \tilde{v} \leq C \tau (\delta x + l)^{- \frac{1 - \gamma}{3 - \gamma}} \int_{\R^3} \exp \big[ - s_0 (|v - \u|^2 + |\tilde{v} - \u|^2) \big] \d \tilde{v} \leq C \tau
		\end{aligned}
	\end{equation*}
	uniformly in $(x, v) \in [0, + \infty) \times \R^3$, which concludes that
	\begin{equation}\label{P-n-bnd}
		\begin{aligned}
			| \mathbf{P}_0 | \leq C \ss_0 \| \sigma_x^\frac{m}{2} e^{\hbar \sigma} g \|_{\beta - \tilde{N}, \vartheta} \leq C \delta \hbar \| \sigma_x^\frac{m}{2} e^{\hbar \sigma} g \|_{\beta - \tilde{N}, \vartheta} \,.
		\end{aligned}
	\end{equation}
	Following the almost same arguments in \eqref{P-n-bnd}, one has
	\begin{equation}\label{P-j-bnd}
		\begin{aligned}
			| \mathbf{P}_j | \leq C \delta \hbar \| \sigma_x^\frac{m}{2} e^{\hbar \sigma} g \|_{\beta - \tilde{N}, \vartheta}
		\end{aligned}
	\end{equation}
	for $j = 1,2,3$. As a result, the bounds \eqref{P-n-bnd} and \eqref{P-j-bnd} conclude the estimate \eqref{D-inf-bnd}, and the proof of Lemma \ref{Lmm-D-XY} is completed.
\end{proof}

\subsection{Boundedness of $\mathbf{D}$ from weighted $L^\infty_x L^2_v$ to $L^\infty_{x,v}$}\label{Subsec:D-L5L2}

We now study the boundedness of $\mathbf{D}$ from weighted $L^\infty_x L^2_v$ to $L^\infty_{x,v}$, hence, to verify Lemma \ref{Lmm-Dh-L2}.

\begin{proof}[Proof of Lemma \ref{Lmm-Dh-L2}]
	Note that
	\begin{equation*}
		\begin{aligned}
			\LL \nu^{-1} \mathbf{D}_\hbar g \RR_{A; m, 0, \vartheta} \leq C \| \sigma_x^\frac{m}{2} (x,v) w_{- \gamma, \vartheta} (v) \mathbf{D}_\hbar g (x,v) \|_{L^\infty_{x,v}} \,,
		\end{aligned}
	\end{equation*}
	and the definition of $\mathbf{D}_h$ in \eqref{Dh} reads
	\begin{equation}\label{Dh-Exp}{\small
		\begin{aligned}
			\mathbf{D}_\hbar g (x,v) = & \ss_0 \kappa_0 (\delta x + l)^{- \frac{1 - \gamma}{3 - \gamma}} \int_{\R^3} e^{\hbar \sigma (x,v) - \hbar \sigma (x, \tilde{v})} g (x, \tilde{v} ) \psi_0^* (v) \tilde{v}_3 \widehat{\mathbb{B}}_3 (\tilde{v}) \d \tilde{v} \\
			& + \ss_0 (\delta x + l)^{- \frac{1 - \gamma}{3 - \gamma}} \sum_{j=1}^2 \kappa_j \int_{\R^3} e^{\hbar \sigma (x,v) - \hbar \sigma (x, \tilde{v})} g (x, \tilde{v} ) \psi_j^* (v) \tilde{v}_3 \widehat{\mathbb{A}}_{j3} (\tilde{v}) \d \tilde{v} \\
			& + \ss_+ (\delta x + l)^{- \frac{1 - \gamma}{3 - \gamma}} \int_{\R^3} e^{\hbar \sigma (x,v) - \hbar \sigma (x, \tilde{v})} g (x, \tilde{v} ) \psi_3^* (v) \tilde{v}_3 \psi_3^* (\tilde{v}) \d \tilde{v} \,,
		\end{aligned}}
	\end{equation}
	where the constants $\kappa_0 = ( \int_{\R^3} \widehat{\mathbb{B}}_3 \mathbb{B}_3 \d v )^{-1}$ and $\kappa_j = ( \int_{\R^3} \widehat{\mathbb{A}}_{j3} \mathbb{A}_{j3} \d v )^{-1}$ with $j =1,2$. It thereby follows that
	\begin{equation*}
		\begin{aligned}
			| \sigma_x^\frac{m}{2} (x,v) w_{- \gamma, \vartheta} (v) \mathbf{D}_\hbar g (x,v) | \leq \sum_{j=0}^3 \int_{\R^3} z_{- \alpha'} (\tilde{v}) \mathfrak{b}_j (x,v, \tilde{v}) \sigma_x^\frac{m}{2} (x,\tilde{v}) |z_{\alpha'} (\tilde{v}) g (x, \tilde{v})| \d \tilde{v} \,,
		\end{aligned}
	\end{equation*}
	where
		\begin{align*}
			\mathfrak{b}_0 (x,v, \tilde{v}) = & \ss_0 \kappa_0 (\delta x + l)^{- \frac{1 - \gamma}{3 - \gamma}} \sigma_x^\frac{m}{2} (x,v) \sigma_x^{- \frac{m}{2}} (x,\tilde{v}) e^{\hbar \sigma (x,v) - \hbar \sigma (x, \tilde{v})} | w_{- \gamma, \vartheta} (v) \psi_0^* (v) \tilde{v}_3 \widehat{\mathbb{B}}_3 (\tilde{v}) | \,, \\
			\mathfrak{b}_j (x,v, \tilde{v}) = & \ss_0 \kappa_j (\delta x + l)^{- \frac{1 - \gamma}{3 - \gamma}} \sigma_x^\frac{m}{2} (x,v) \sigma_x^{- \frac{m}{2}} (x,\tilde{v}) e^{\hbar \sigma (x,v) - \hbar \sigma (x, \tilde{v})} | w_{- \gamma, \vartheta} (v) \psi_j^* (v) \tilde{v}_3 \widehat{\mathbb{A}}_{j3} (\tilde{v}) | \,, \\
			\mathfrak{b}_3 (x,v, \tilde{v}) = & \ss_+ (\delta x + l)^{- \frac{1 - \gamma}{3 - \gamma}} \sigma_x^\frac{m}{2} (x,v) \sigma_x^{- \frac{m}{2}} (x,\tilde{v}) e^{\hbar \sigma (x,v) - \hbar \sigma (x, \tilde{v})} | w_{- \gamma, \vartheta} (v) \psi_3^* (v) \tilde{v}_3 \psi_3^* (\tilde{v}) | \,,
		\end{align*}
	for $j =1,2$. Together with Lemma \ref{Lmm-sigma}, \eqref{pro-sigma}, and the facts of exponential decay behavior of the functions $\psi_j^* (v) ( j =0,1,2,3)$ and $\widehat{\mathbb{A}}_{j3} (v) (j=1,2)$, $\widehat{\mathbb{B}}_3 (v)$, one gains
	\begin{equation}\label{b-bnd}
		\begin{aligned}
			\sum_{j=0}^3 \mathfrak{b}_j (x,v, \tilde{v}) \leq C e^{- \frac{c'}{2} (|v-\u|^2 + |\tilde{v} - \u|^2)}
		\end{aligned}
	\end{equation}
	for small enough $\hbar, \vartheta \geq 0$. Then
	\begin{equation*}{\small
		\begin{aligned}
			| \sigma_x^\frac{m}{2} (x,v) w_{- \gamma, \vartheta} (v) \mathbf{D}_\hbar g (x,v) | \leq & C \int_{\R^3} z_{- \alpha'} (\tilde{v}) e^{- \frac{c'}{2} (|v-\u|^2 + |\tilde{v} - \u|^2)} \sigma_x^\frac{m}{2} (x,\tilde{v}) |z_{\alpha'} (\tilde{v}) g (x, \tilde{v})| \d \tilde{v} \\
			\leq & C \Big( \int_{\R^3} z_{- \alpha'}^2 (\tilde{v}) e^{- c' (|v-\u|^2 + |\tilde{v} - \u|^2)} \d \tilde{v} \Big)^\frac{1}{2} \| \sigma_x^\frac{m}{2} z_{\alpha'} g (x, \cdot ) \|_{L^2_v} \\
			\leq & C \| \sigma_x^\frac{m}{2} z_{\alpha'} g \|_{L^\infty_x L^2_v} \leq C \| \sigma_x^\frac{m}{2} z_{\alpha'} w_{- \gamma, \vartheta} g \|_{L^\infty_x L^2_v} \,.
		\end{aligned}}
	\end{equation*}
	Here $0 \leq \alpha' < \frac{1}{2}$ is required. Then the proof of Lemma \ref{Lmm-Dh-L2} is completed.
\end{proof}

\subsection{Weighted $L^2_v$ boundedness of $\mathbf{D}$}\label{Subsec:D-L2L2}

We now study the weighted $L^2_v$ boundedness of $\mathbf{D}$, hence, to verify Lemma \ref{Lmm-Dh-L2-L2}.

\begin{proof}[Proof of Lemma \ref{Lmm-Dh-L2-L2}]
	Recalling \eqref{Dh-Exp}, one easily has
	\begin{equation*}
		\begin{aligned}
			\nu^{- \frac{1}{2}} z_{- \alpha} \sigma_x^\frac{m}{2} w_{\beta, \vartheta} \mathbf{D}_\hbar g (x,v) = \sum_{j=0}^3 \int_{\R^3} \mathfrak{n}_j (x,v, \tilde{v}) \nu^\frac{1}{2} \sigma_x^\frac{m}{2} w_{\beta, \vartheta} g (x, \tilde{v}) \d \tilde{v} \,,
		\end{aligned}
	\end{equation*}
	where
	\begin{equation*}
		\begin{aligned}
			\mathfrak{n}_0 (x,v, \tilde{v}) = & \ss_0 \kappa_0 (\delta x + l)^{- \frac{1 - \gamma}{3 - \gamma}} \tfrac{\mathfrak{p} (x,v)}{ \mathfrak{p} (x,\tilde{v}) } z_{- \alpha} (v) \nu^{- \frac{1}{2}} (v) \nu^{- \frac{1}{2}} (\tilde{v}) | \psi_0^* (v) \tilde{v}_3 \widehat{\mathbb{B}}_3 (\tilde{v}) | \,, \\
			\mathfrak{n}_j (x,v, \tilde{v}) = & \ss_0 \kappa_j (\delta x + l)^{- \frac{1 - \gamma}{3 - \gamma}} \tfrac{\mathfrak{p} (x,v)}{ \mathfrak{p} (x,\tilde{v}) } z_{- \alpha} (v) \nu^{- \frac{1}{2}} (v) \nu^{- \frac{1}{2}} (\tilde{v}) | \psi_j^* (v) \tilde{v}_3 \widehat{\mathbb{A}}_{j3} (\tilde{v}) | \,, \\
			\mathfrak{n}_3 (x,v, \tilde{v}) = & \ss_+ (\delta x + l)^{- \frac{1 - \gamma}{3 - \gamma}} \tfrac{\mathfrak{p} (x,v)}{ \mathfrak{p} (x,\tilde{v}) } z_{- \alpha} (v) \nu^{- \frac{1}{2}} (v) \nu^{- \frac{1}{2}} (\tilde{v}) | \psi_3^* (v) \tilde{v}_3 \psi_3^* (\tilde{v}) | \,,
		\end{aligned}
	\end{equation*}
	for $j=1,2$. Here the function $\mathfrak{p} (x,v)$ is defined in \eqref{p}. Following the similar arguments in \eqref{b-bnd}, and combining with the fact $(\delta x + l)^{- \frac{1 - \gamma}{3 - \gamma}} \lesssim 1$ uniformly in $x \geq 0$, one infers that
	\begin{equation*}
		\begin{aligned}
			\sum_{j=0}^3 |\mathfrak{n}_j (x,v,\tilde{v})| \lesssim z_{- \alpha} (v) e^{- \frac{c'}{2} (|v - \u|^2 + |\tilde{v} - \u|^2)}
		\end{aligned}
	\end{equation*}
	uniformly in $x \geq 0$. Thanks to $0 \leq \alpha < \frac{1}{2}$, it holds
	\begin{equation*}
		\begin{aligned}
			\sum_{j=0}^3 \iint_{\R^3 \times \R^3} |\mathfrak{n}_j (x,v,\tilde{v})|^2 \d \tilde{v} \d v \lesssim \iint_{\R^3 \times \R^3} z_{- \alpha}^2 (v) e^{- c' (|v - \u|^2 + |\tilde{v} - \u|^2)} \d \tilde{v} \d v \lesssim 1
		\end{aligned}
	\end{equation*}
	uniformly in $x \geq 0$. Consequently, together with the H\"older inequality, one has
{\small
		\begin{align*}
			& \int_{\R^3} |\nu^{- \frac{1}{2}} z_{- \alpha} \sigma_x^\frac{m}{3} w_{\beta, \vartheta} \mathbf{D}_\hbar g (x,v)|^2 \d v = \int_{\R^3} \Big( \sum_{j=0}^3 \int_{\R^3} \mathfrak{n}_j (x,v, \tilde{v}) \nu^\frac{1}{2} \sigma_x^\frac{m}{2} w_{\beta, \vartheta} g (x, \tilde{v}) \d \tilde{v} \Big)^2 \d v \\
			\lesssim & \int_{\R^3} \big( \sum_{j=0}^3 \int_{\R^3} | \mathfrak{n}_j (x,v, \tilde{v}) |^2 \d \tilde{v} \big) \big( \int_{\R^3} | \nu^\frac{1}{2} \sigma_x^\frac{m}{2} w_{\beta, \vartheta} g (x, \tilde{v}) |^2 \d \tilde{v} \big) \d v \\
			= & \sum_{j=0}^3 \iint_{\R^3 \times \R^3} |\mathfrak{n}_j (x,v,\tilde{v})|^2 \d \tilde{v} \d v \int_{\R^3} | \nu^\frac{1}{2} \sigma_x^\frac{m}{2} w_{\beta, \vartheta} g (x, v) |^2 \d v \lesssim \int_{\R^3} | \nu^\frac{1}{2} \sigma_x^\frac{m}{2} w_{\beta, \vartheta} g (x, v) |^2 \d v \,.
		\end{align*}}
	The proof of Lemma \ref{Lmm-Dh-L2-L2} is therefore finished.
\end{proof}


\section*{Acknowledgments}
This work was supported by National Key R\&D Program of China under the grant 2023YFA1010300. The author N. Jiang is supported by the grants from the National Natural Foundation of China under contract Nos. 11971360, 12371224 and 12221001. The author Y.-L. Luo is supported by grants from the National Natural Science Foundation of China under contract No. 12201220, the Guang Dong Basic and Applied Basic Research Foundation under contract No. 2021A1515110210 and 2024A1515012358.


\bibliography{reference}

\end{document}